\documentclass[12pt]{article}
\usepackage[all,2cell]{xy} \UseAllTwocells \SilentMatrices
\usepackage{amsfonts,amsmath}
\usepackage{amssymb,latexsym}
\usepackage[dvips]{epsfig}
\usepackage{amscd}
\usepackage{psfrag}
\usepackage{amsthm}
\usepackage[hmargin=3cm,vmargin=3.5cm]{geometry}
\usepackage{hyperref}
\usepackage{diagcat}

\title{\hspace{1.3in} {\it {\normalsize Dedicated to Igor Frenkel on the occasion of
his sixtieth birthday.}} \\
 \hspace{1in}  \\
{\bf An approach to categorification of some small quantum groups}}
\author{Mikhail Khovanov, You Qi}
\date{May 20, 2013}

\newcommand{\oplusop}[1]{{\mathop{\oplus}\limits_{#1}}}

\theoremstyle{plain}
\newtheorem{thm}{Theorem}[section]
\newtheorem{prop}[thm]{Proposition}
\newtheorem{lemma}[thm]{Lemma}
\newtheorem{cor}[thm]{Corollary}
\newtheorem{conj}[thm]{Conjecture}

\theoremstyle{definition}

\newtheorem{defn}[thm]{Definition}
\newtheorem{eg}[thm]{Example}
\newtheorem{rmk}[thm]{Remark}
\newtheorem{ntn}[thm]{Notation}

\def\Q{\mathbb Q}
\def\Z{\mathbb Z}
\def\N{\mathbb N}
\def\C{\mathbb C}
\def\F{\mathbb F}

\def\o{\otimes}
\def\lra{\longrightarrow}

\def\RHOM{\mathbf{R}\mathrm{HOM}}
\def\Id{\mathrm{Id}}
\def\mc{\mathcal}
\def\mf{\mathfrak}

\newcommand{\umod}{\!-\!\underline{\mathrm{mod}}}
\newcommand{\udmod}{\!-\!\underline{\mathrm{mod}}}  %% stable category
\newcommand{\ufmod}{\!-\!\underline{\mathrm{fmod}}}  %% stable, finite-dimensional
\newcommand{\dmod}{\!-\!\mathrm{mod}}
\newcommand{\dfmod}{\!-\!\mathrm{fmod}}
\newcommand{\dif}{\partial}   %% differential
\newcommand{\mker}{\mathrm{Ker}}
\newcommand{\mim}{\mathrm{Im}}
\newcommand{\mH}{\mathrm{H}}  %% homology
\newcommand{\NH}{\mathrm{NH}} %% nilHecke algebra
\newcommand{\pol}{\mathrm{Pol}}
\newcommand{\bpol}{\mathcal{P}}
\newcommand{\sym}{\mathrm{Sym}}

\newcommand{\Hom}{{\rm Hom}}
\newcommand{\HOM}{{\rm HOM}}

\newcommand{\Ind}{{\rm Ind}}
\newcommand{\Res}{{\rm Res}}

\newcommand{\drawing}[1]{
\begin{center}{\psfig{figure=fig/#1}}\end{center}}

\newcommand{\Mn}{\mathrm{M}(n,\Bbbk)} %% n*n matrix algebra

\begin{document}
\maketitle

\baselineskip 14pt

\tableofcontents

\begin{abstract}
We categorify one half of the small quantum sl(2) at a prime root of unity.
An extension of this construction to an arbitrary simply-laced case is proposed.
\end{abstract}

%%%%%%%%%%%%%%%%%%%%%%%%%%%%%%%

\psfrag{Z1}{$\partial_a $}
\psfrag{Z2}{$(a+1)$}
\psfrag{Z3}{$(a-1)$}
\psfrag{Z4}{$a$}
\psfrag{Z5}{$\mH_{/ 0}(\widetilde{V})$}
\psfrag{Z6}{$\mH_{/ p-2}(\widetilde{V})$}
\psfrag{T1}{$R\dmod$}
\psfrag{T2}{$A_{\dif}\dmod$}

%%%%%%%%%%%%%%%%%%%%%%%%%%%%
%%%%%%%%%%%%%%%%%%%%%%%%%%%%
%%
%%   INTRODUCTION
%%
%%%%%%%%%%%%%%%%%%%%%%%%%%%%
%%%%%%%%%%%%%%%%%%%%%%%%%%%%

\section{Introduction}

The equation $d^2=0$, together with the related notion of a chain
complex and its (co)homology, plays a fundamental role in modern
mathematics and physics. Shortly after the discovery of simplicial
homology, Mayer~\cite{May1, May2} suggested to remove the minus
signs in the definition of the differential in the chain complex of
a simplicial complex while working in characteristic $p$. The Mayer
differential satisfies $d^p=0$ and gives rise to Mayer homology
groups. Spanier~\cite{Spa} showed that the Mayer homology of a
topological space can be expressed via the usual homology groups.
Spanier's result closed the topic for almost half a century. The
subject was brought back to life by  Kapranov~\cite{Kap},
Sarkaria~\cite{Sar1, Sar2}, Dubois-Violette~\cite{DV1, DV2,DV3},
Kerner~\cite{DVK} and others, also see~\cite{AD, CSW, Jon, KW, Wa} and the
references therein. These papers study a characteristic zero analogue
of Mayer's generalized complexes, called $N$-complexes,
with the generalized differential
satisfying $d^N=0$. Defining such differentials often requires the
use of a primitive $N$th root of unity $q$ to twist classical
constructions of the homological algebra. The subject is sometimes
referred to as $q$-homological algebra.

\vspace{0.06in}

Complexes of vector spaces constitute a symmetric monoidal category.
A complex of vector spaces can be thought of as a graded module over
the exterior algebra on one generator. The existence of a natural
tensor product on complexes can be explained by the Hopf algebra
structure of the exterior algebra, when viewed as an algebra in the
category of super-vector spaces. Algebraic structures in the super
world can be thickened using Majid's bosonisation procedure
\cite{Maj} to produce corresponding structures in the category of vector
spaces. Grading (ubiquitous in homological algebra)
can be interpreted in the
language of comodules: a comodule over the group ring $\Bbbk[G]$ is
the same as a $G$-graded $\Bbbk$-vector space.

\vspace{0.06in}

Furthermore, as originally observed by Pareigis~\cite{Par}, the
category of complexes is equivalent to the category of comodules
over a suitable Hopf algebra,  equipped with a cotriangular
structure. This approach explains the symmetric monoidal tensor
product of complexes via the Hopf algebra framework.
Bichon~\cite{Bi} characterized $N$-complexes as comodules over a
suitable Hopf algebra (the Borel subalgebra of the small quantum
$\mf{sl}_2$ at an $N$-th root of unity), generalizing Pareigis' work.

\vspace{0.06in}

Given a  Frobenius algebra $H$ over a field $\Bbbk$, its stable
category $H\udmod$ is triangulated~\cite{Hap} (a morphism of $H$-modules
is zero in the stable category if it factors through a projective module). Any
finite-dimensional Hopf algebra is a Frobenius algebra, and its
stable category is triangulated monoidal. Several years ago one of
us suggested~\cite{Kh} to study module-categories over $H\udmod$ for
suitable finite-dimensional Hopf algebras $H$. This was motivated by
the observation that the Grothendieck ring of the stable category of
finite-dimensional graded modules over the Hopf algebra
$H=\Bbbk[x]/(x^p)$ with $\Delta(x) = x\otimes 1 + 1\otimes x$,
$\deg(x)=1$ for a field $\Bbbk$ of characteristic $p$ is naturally
isomorphic to the ring of integers $\mathcal{O}_p$ in the cyclotomic field
$\Q[\zeta_p]$ for the $p$-th primitive root of unity $\zeta_p=e^{\frac{2\pi
i }{p}}$:
\begin{equation}\label{ring-iso}
 K_0(H\udmod) \cong \mathcal{O}_p\cong \Z[q]/(1+q+\dots + q^{p-1}),
\end{equation}
where $q$ is a formal variable (the ring of integers $\mathcal{O}_N$ in
$\Q[\zeta_N]$ is generated by $\zeta_N=e^{\frac{2\pi
i }{N}}$ over $\Z$).
 Witten-Reshetikhin-Turaev invariants of 3-manifolds take
values in the ring of cyclotomic integers $\mathcal{O}_N$
(in favorable cases; in general one
needs to extend the ring to $\mathcal{O}_N[\frac{1}{N}]$). Categorified WRT
invariants, if exist, might take values in a tensor triangulated
category with the Grothendieck ring $\mathcal{O}_N$ or, even better,
$\mathcal{O}_N[\frac{1}{N}].$

\vspace{0.06in}

We do not know any instances of monoidal triangulated categories whose
Grothen\-dieck ring contains $\Z[\frac{1}{N}]$ as a subring (it is an exciting
problem to find such examples), and so should restrict to
a simpler problem of categorifying $\mathcal{O}_N$.
The ring $\mathcal{O}_N$ is isomorphic to $ \Z[q]/(\Psi_N(q))$, where
$q$ is a formal variable and $\Psi_N$ the $N$-th cyclotomic polynomial. It seems that
examples of monoidal triangulated categories with Grothendieck ring isomorphic to
$\mathcal{O}_N$ are only known for $N$ a prime
power $p^r$ or $2p^r$. The cyclotomic
polynomials for these $N$ are
\begin{equation}\label{eq-psi-pr}
\Psi_{p^r}(q) = 1 + q^{p^{r-1}} + q^{2p^{r-1}} + \dots +
q^{(p-1)p^{r-1}},
\end{equation}
and, for odd $p$,
\begin{equation}\label{eq-psi-2pr}
\Psi_{2p^r}(q) = \Psi_{p^r}(-q)=
1 + (-q)^{p^{r-1}} + (-q)^{2p^{r-1}} + \dots + (-q)^{(p-1)p^{r-1}}.
\end{equation}
The fields $\Q[\zeta_{p^r}]$ and $\Q[\zeta_{2p^r}]$ coincide, for odd $p$, and thus
share the same ring of integers. The above example of categorification of $\mathcal{O}_p$
via the Hopf algebra $H$ can be easily modified, by changing the grading of generator
$x$ from $1$ to $p^{r-1}$, to a categorification of $\mathcal{O}_{p^r}$.

\vspace{0.06in}

It would be thrilling to lift
this example to characteristic $0$ and to get rid of the $p^r$ restriction.
There is a problem with doing it via
$N$-complexes - for $N>2$ the base category of $N$-complexes modulo chain
homotopies is neither symmetric nor braided,
and its Grothendieck ring is isomorphic to $\Z[q]/(1+q+\dots + q^{N-1})$ rather
than $\Z[q]/(\Psi_N(q))$. We do not know how to circumvent these difficulties
and construct a triangulated monoidal  category
with Grothendieck ring $\Z[q]/(\Psi_N(q))$. For this reason we work in
characteristic $p$, with Mayer's $p$-complexes, and so restrict our categorification
attempts to $p$-th and $2p$-th roots of unity.

\vspace{0.06in}

To do this $p$-homological algebra, we need a supply of module-categories over
the base symmetric triangulated category $H\udmod$. Given a graded $\Bbbk$-algebra
$A$ with a degree one derivation $\partial$ (so that $\partial(ab)=\partial(a)b +
a \partial(b)$) subject to the condition $\partial^p=0$, one can form
the category $(A,\partial)\dmod$ of $p$-DG modules over $A$, by analogy
with the category of DG-modules over a
DG-algebra, then pass to the homotopy category and, finally, localize by
quasi-isomorphisms to get the analogue $\mc{D}(A,\partial)$
of the derived category. To have an
interesting Grothendieck group, one can restrict to the subcategory $\mc{D}^c(A,\partial)$
of compact modules. For several reasons (mostly compatibility with prior work
on categorification of quantum groups at generic $q$)
we make our derivations have degree two. This results in the Grothendieck group
of  $\mc{D}^c(A,\partial)$ being a module over the ground ring
\begin{equation}\label{eq-gr-ring}
\mathbb{O}_p=\Z[q]/(\Psi_{p}(q^2)),
\end{equation}
the Grothendieck ring of the stable
category of finitely-dimensional $H$-modules. We denote the Grothendieck group
by $K_0(A,\partial)$ or simply by $K_0(A)$. Derivation being of degree two
translates into having $q^2$ in the formula. This polynomial is reducible,
$\Psi_{p}(q^2)= \Psi_p(q) \Psi_{2p}(q)$, so that $\mathbb{O}_p$ is not an
integral domain, although it is a subring of the product $\mc{O}_p\times \mc{O}_{2p}$
of two integral domains.

\vspace{0.06in}

Categorification of tensor products of quantum group representations was obtained
geometrically by Zheng~\cite{Zheng}, utilizing earlier ideas and constructions of
Nakajima~\cite{Na1,Na2} and Malkin~\cite{Mal}, and by Webster~\cite{Web,Web2}
in an algebro-combinatorial fashion, who also extended it to a categorification of
Reshetikhin-Turaev link and tangle invariants~\cite{RT}.
One of the inputs in Webster's construction is a categorification of positive halves of
quantum groups via diagrammatically described KLR rings~\cite{KL1,KL2,Rou}.
To test whether $p$-complexes are relevant to categorification at prime roots of unity, we
can try to integrate them with the categorification of quantum groups and their
representations. A first step would be to look for nilpotent ($\partial^p=0$)
derivations with interesting properties on KLR rings $R(\nu)$.

\vspace{0.06in}

We start with the $\mf{sl}_2$ case, when the KLR rings $R(\nu)$ specialize to nilHecke
algebras $\NH_n$. Categorification of multiplication and comultiplication of
the positive half of quantum $\mf{sl}_2$ at generic $q$ is achieved via
induction and restriction functors for inclusions of algebras $\NH_n\otimes \NH_m
\subset \NH_{n+m}$, see~\cite{KL1, Lau1}. The nilHecke algebra $\NH_n$ is isomorphic
to the matrix algebra of size $n!$ with coefficients in the ring of symmetric
functions $\mathrm{Sym}_n$ in $n$ variables, and the $n$-th divided power
$E^{(n)}= \frac{E^n}{[n]!}$ is categorified
by the symbol of an indecomposable
projective module $\bpol_n$ over $\NH_n$. The latter comes from realizing $\NH_n$
as the endomorphism ring of the space $\pol_n$ of polynomials in $n$ variables
over the subring $\sym_n$ of symmetric polynomials, and $\bpol_n$ is identified
with $\pol_n$, up to a grading shift.

\vspace{0.06in}

Working over a field $\Bbbk$ of characteristic $p$, we construct
a family of derivations $\partial_a$, parametrized by $a\in \F_p$, i.e., by
residues modulo $p$. These satisfy $\partial_a^p=0$ and have a
local description, given on generators by

\vspace{0.04in}

\[
\dif_a~\left(
\begin{DGCpicture}
\DGCstrand(1,0)(1,1) \DGCdot{0.5}
\end{DGCpicture} \right)
~ =~
\begin{DGCpicture}
\DGCstrand(1,0)(1,1) \DGCdot{0.5}[ur]{$2$}
\end{DGCpicture}
\ , \ \ \ \
\dif_a\left(~
\begin{DGCpicture}
\DGCstrand(1,0)(0,1) \DGCstrand(0,0)(1,1)
\end{DGCpicture}~\right)
= a~
\begin{DGCpicture}
\DGCstrand(0,0)(0,1) \DGCstrand(1,0)(1,1)
\end{DGCpicture}
-(a+1)
\begin{DGCpicture}
\DGCstrand(1,0)(0,1)\DGCdot{0.75} \DGCstrand(0,0)(1,1)
\end{DGCpicture}
+(a-1)
\begin{DGCpicture}
\DGCstrand(1,0)(0,1) \DGCstrand(0,0)(1,1)\DGCdot{0.75}
\end{DGCpicture}.
\]

\vspace{0.03in}

The derivation $\dif_a$ on $\NH_n$ comes from a suitable derivation $\dif_{\alpha}$ on
the polynomial space $\pol_n$. Inclusions
\begin{equation}\label{eq-inclusions}
\NH_n\otimes \NH_m \subset \NH_{n+m}
\end{equation}
 commute with
$\partial_a$ and give induction and restriction functors
between categories of $p$-DG modules over these algebras and corresponding
derived categories.

The positive half of the small quantum $\mf{sl}_2$ has one generator $E$ and
one defining relation $E^p=0$.
In our categorification $E$ becomes the image in the Grothendieck group of an
object $\mc{P}$,
and the hom spaces $\HOM(\mc{P}^{\otimes n},\mc{P}^{\otimes m})=0$ for $n\not= m$, while
$\HOM(\mc{P}^{\otimes n},\mc{P}^{\otimes n}) =\mathrm{END}(\mc{P}^{\otimes n})= \NH_n,$
equipped with derivation $\partial_a$.
We show that when $n=p$ and $a\not= 0$, the endomorphism algebra
$\mathrm{END}(\mc{P}^{\otimes p})=\NH_p$ is acyclic, as an $H$-module,
and its derived category is trivial, implying that the object $\mc{P}^{\otimes p}$
is isomorphic to the zero object in the derived category.
We view this result as a categorification of the defining relation $E^p=0$ in
the small quantum group.
When $a=0$, the $p$-DG algebra $\NH_p$ is not acyclic,
preventing the derived category from being trivial and leading us to exclude this
case by restricting to $a\in \F_p^{\ast}$.

\vspace{0.06in}

A twisted version of the restriction functor for the inclusion (\ref{eq-inclusions}) of $p$-DG algebras
induces a homomorphism of Grothendieck groups
\begin{equation}\label{eq-homo-NH}
K_0( \NH_{n+m} ) \lra K_0 (\NH_n \otimes \NH_m)
\end{equation}
(we write $K_0(\NH_{n+m})$ instead of $K_0(\NH_{n+m},\dif_a)$, etc.)
For this map to categorify the comultiplication in a bialgebra requires an
isomorphism
\begin{equation}\label{eq-req-isom}
K_0 (\NH_n \otimes \NH_m) \cong K_0 (\NH_n) \otimes_{\mathbb{O}_p} K_0(\NH_m),
\end{equation}
where $\mathbb{O}_p$ is the ground ring (the Grothendieck ring of the stable
category of finite-dimensional $H$-modules).

\vspace{0.06in}

Only for $a=\pm 1$ are we able to establish this isomorphism as well as to
show that the Grothendieck group of the $p$-DG algebra $\NH_n$ (for $n<p$)
equipped with derivation $\dif_a$ is a free rank one $\mathbb{O}_p$-module.
This $\mathbb{O}_p$-module is generated
by the symbol $[\bpol_n(\alpha)]$ of the indecomposable projective $\NH_n$-module
carrying a derivation $\dif_{\alpha}$.
The $p$-DG-module $\bpol_n(\alpha)$ categorifies
the $n$-th divided power $E^{(n)} = \frac{E^n}{[n]!}$ of the generator $E$.
We do not know the structure of
$K_0(\NH_n,\dif_a)$ for other values of $a$.

\vspace{0.06in}

For $a= \pm 1$ induction and twisted restriction functors lead to a
$q$-bialgebra (or twisted bialgebra) structure on
the direct sum of Grothendieck groups $K_0(\NH_n,\partial_a)$.
We identify this $q$-bialgebra
with an integral form of the positive half of the small quantum
$\mf{sl}_2$ at a $2p$-th root of unity
(Theorem \ref{thm-p-DG-nilHecke-categorifies-small-sl2}):

\vspace{0.04in}

\[ K_0(\NH, \dif_{\pm 1}) \cong u^+_{\mathbb{O}_p}(\mf{sl}_2).  \]

\vspace{0.04in}

The integral form is a free $\mathbb{O}_p$-module with basis elements
$E^{(n)}$ for $0\le n\le p-1$ and the usual multiplication and comultiplication
rules, see Section~\ref{subsec-gr-ring}.

\vspace{0.06in}

To an oriented graph $\Gamma$ there is assigned~\cite{KL1,KL2,Rou}
a family of graded $\Bbbk$-algebras
$R(\nu)$, over positive integral weights $\nu$. The direct sum of Grothendieck groups
of finitely-generated projective graded $R(\nu)$-modules, over all $\nu$, can be canonically
identified~\cite{KL1,KL2,VV}
with an integral form of the positive half of the quantum group associated to $\Gamma$:
\begin{equation*}
 \mathrm{U}^+_{q,\Z}(\Gamma) \cong \oplusop{\nu \in \N^I} K_0(R(\nu)).
\end{equation*}
In particular, Serre relations hold on the categorical level, lifting to isomorphisms
between certain direct sums of projective $R(\nu)$-modules.

\vspace{0.06in}

Each vertex $i$ of $\Gamma$ generates nilHecke algebras $R(ni)\cong \NH_n$ as
 $R(\nu)$ for $\nu= ni$. Assigning $a_i\in \F_p^*$ to a vertex $i$ gives us derivations
$\partial_{a_i}$ on $R(ni)$. These derivations extend in a multi-parameter way,
 written down in Section~\ref{sec-KLR},
to derivations $\partial$ on $R(\nu)$ for all weights $\nu$.
We show in Section~\ref{sec-KLR} that having Serre relations on the categorical level
requires $a_i = a_j$ for $i,j$ in the same connected component of $\Gamma$
and $a_i\in \{ 1, -1\}$ for each $i$. Assuming $\Gamma$ is connected, this
forces all $a_i$ to be equal $1$ or $-1$. We can restrict to $a_i=1,$ for all $i$;
the opposite case follows by applying an automorphism to rings $R(\nu)$.
Furthermore, having Serre relations on the Grothendieck group level
requires a unique choice for all other parameters
as well, once one of the two cases ($1$ or $-1$) is chosen for the $a_i$'s.
We work with the case $a_i=1, $  $i\in I$, and obtain a canonical derivation
$\partial_1$ on all $R(\nu)$, given in Definition~\ref{def-correct-differential-on-KLR}
and extending derivation $\partial_1$ of $\NH_n= R(ni)$.
Our results on  categorification of the Serre relations are summarized in
Theorem~\ref{thm-QSR-holds-only-for-special-paramters}. Conjecture~\ref{conj-cat}
conveys our hopes that $p$-DG rings $R(\nu)$, over all weights $\nu$,
categorify a positive half of the
small quantum group associated with the graph $\Gamma$.

\vspace{0.06in}

Other algebraic objects that appear in categorification also carry $p$-nilpotent
derivations and have a potential to categorify corresponding quantum structures at prime
roots of unity. For instance, in~\cite{EQ}, it is shown
that Lauda's two-category ${\mc{U}}$ admits a $p$-nilpotent derivation
leading to a categorification of an integral form of the small dotted quantum group
$\dot{u}_{\zeta_{2p}}(\mf{sl}_2)$. For another example, Webster's algebras~\cite{Web},
which categorify tensor products of simple modules over quantum Kac-Moody
algebras, carry a multi-parameter family of nilpotent derivations; the
$\mathfrak{sl}_2$ case is briefly discussed at the end of this paper.
Thick calculus of~\cite{KLMS} coupled with suitable derivations promises
a categorification of the big quantum $\mathfrak{sl}_2$ at a prime root of unity~\cite{EQ2}.

\vspace{0.06in}

In 1994 Louis Crane and Igor Frenkel conjectured~\cite{CF} that there exists a
categorification of the small quantum group $\mathfrak{sl}_2$ at roots of unity
and that a categorical lift of Kuperberg's 3-manifold invariant~\cite{Ku} should give
a quantum invariant of four-manifolds. Igor Frenkel developed the ideas of categorification
much further, in unpublished notes on structural constraints in categorified
quantum $\mathfrak{sl}_2$,
in many of his follow-up joint papers on categorification and geometrization of
representation theory,  and while advising several graduate students who went on to
further grow the categorification program, including the senior author of the present
paper. The junior author is a second-generation student of Igor Frenkel.
We are delighted to dedicate this work to Igor Frenkel, on the occasion of his sixtieth birthday.

\vspace{0.1in}

{\bf Acknowledgments.} We are grateful to Ben Elias, Alexander Ellis, Aaron Lauda, Joshua Sussan and the
anonymous referees for carefully reading early
drafts of the paper and suggesting many corrections.
Krzysztof Putyra wrote a latex package which allowed us to produce all diagrams
in Sections \ref{sec-nilHecke} and \ref{sec-KLR}. The authors were partially supported by
NSF grants DMS-1005750 and DMS-0739392 while working on the paper.

%%%%%%%%%%%%%%%%%%%%%%%%
%%%%%%%%%%%%%%%%%%%%%%%%
%%
%%  GENERALITIES
%%
%%%%%%%%%%%%%%%%%%%%%%%%
%%%%%%%%%%%%%%%%%%%%%%%%

\section{Homological algebra of \texorpdfstring{$p$}{p}-derivations} \label{sec-pder}

%%%%%%%%%%%%%%%%%%%%%%%%%
%%
%%   BASE CATEGORY
%%
%%%%%%%%%%%%%%%%%%%%%%%%%

\subsection{The base category}\label{sec-base-category}
\paragraph{The Hopf algebra $H$ and its category of modules.} We
work over a field $\Bbbk$ of prime characteristic $p$, so that there is
a natural inclusion of fields $\F_p =\Z/(p)\subset \Bbbk$. The algebra
$H= \Bbbk[\dif]/(\dif^p)$ is a Hopf algebra, with comultiplication
$\Delta(\dif)= 1\otimes \dif + \dif \otimes 1$, antipode $S(\dif) =
-\dif$, and counit $\epsilon(\dif^i)= \delta_{0,i}$. We make $H$
into a $\Z$-graded Hopf algebra via $\deg(\dif)=2$.
Let $H\dmod$ be the category of graded $H$-modules
(morphisms are grading-preserving module maps) and $H\dfmod$ the
subcategory of finite-dimensional graded $H$-modules. We also call
objects of $H\dmod$ \emph{$p$-complexes} and denote the vector space of
(grading-preserving) homomorphisms between $p$-complexes $U$ and $W$ by
$\Hom_H(U,W)$.

Any indecomposable object of $H\dmod$ is
isomorphic, up to a grading shift,  to $V_i := H/(\dif^{i+1})$ for some
$0\le i \le p-1$. Define the balanced indecomposable by
$\widetilde{V}_i = V_i\{ -i\}$.  It has a copy of the ground field in degrees
$-i, -i+2, \dots, i$, and we denote the basis elements by $\tilde{v}_k$.
so that $\widetilde{V}_i=\oplus_{k=0}^i\Bbbk\tilde{v}_k$,
with $\dif(\tilde{v}_k)=\tilde{v}_{k+1}$ for $k<i$ and $\dif(\tilde{v}_i)=0$.

The categories $H\dmod$ and $H\dfmod$ have the
Krull-Schmidt unique decomposition property. The cocommutative Hopf algebra
structure of $H$ turns $H\dmod$ and $H\dfmod$ into symmetric monoidal
categories, with $\dif$ acting on $U\otimes W$ via $\dif(u\otimes w) =
\dif (u)\otimes w + u \otimes \dif (w).$ The tensor product of indecomposable
objects is given by
\begin{equation} \label{eqn-product-rule-for-balanced-indecomposables}
\widetilde{V}_i \otimes \widetilde{V}_j \cong (\oplusop{m\in I} \widetilde{V}_m)
 \oplus \widetilde{V}_{p-1}^{[\max(0,i+j+2-p)]} ,
\end{equation}
where $I=\{ |i-j|, |i-j|+2, \dots, \mathrm{min}(i+j, 2p-i-j-4)\}$
and $[n] = q^{n-1} + q^{n-3}+\dots + q^{1-n}\in \N[q,q^{-1}]$ for
$n=\max(0,i+j+2-p)$ is the multiplicity, with grading shifts, of the
free module $\widetilde{V}_{p-1}$ in the tensor product. Notice that
$I$ is the empty set if $i=p-1$ or $j=p-1$, since the tensor product
of a free module over a Hopf algebra with any module is free. We
have
$$ V_i \otimes V_{p-1} \cong V_{p-1}\oplus V_{p-1}\{2\}\oplus \dots \oplus V_{p-1}
 \{2i\}.$$

 The antipode $S$ gives us a contravariant functor on $H\dmod$ taking $U$ to
$U^{\ast} =\oplus_{i \in \Z} \Hom_{\Bbbk}(U^i,\Bbbk)$ with $(\dif f)(u) = - f(\dif (u))$, where
$U^i$ stands for the homogeneous degree $i$ part of $U$.
This functor restricts to a contravariant involutive self-equivalence on
$H\dfmod$.  Internal homs in $H\dmod$ are given by
$${\HOM}_\Bbbk(U,W)  := \oplus_{i,j\in \Z}\Hom_\Bbbk(U^i, W^j),
\ \ \ \  (\dif f)(u) = \dif (f(u))  - f (\dif (u)).$$
Internal homs are graded whose homogeneous terms are given by
$${\HOM}_{\Bbbk}^k(U,W) := \oplus_{j-i=k}\Hom_\Bbbk(U^i,W^j)$$
for any $k\in \Z$. In the smaller category $H\dfmod$ the above direct sum coincides with
$\oplus_{i\in Z}\Hom_{\Bbbk}(U\{i\},W)$. We also define
$$\HOM_{H}(U,W):= \oplusop{i\in \Z} \Hom_H(U\{ i\}, W) .$$

Balanced modules are self-dual: $\widetilde{V}_i \cong
\HOM_\Bbbk(\widetilde{V}_i,\Bbbk)$.

The Grothendieck ring $K_0(H\dfmod)$ of the symmetric monoidal abelian category
$H\dfmod$ is naturally isomorphic to the ring of Laurent polynomials
$\Z[q,q^{-1}]$, with $1=[V_0]$,
grading shift corresponding to multiplication by $q$:
$[U\{m\}] = q^m [U]$, and the symbol $[U]$ of any $p$-complex $U$ given
by its graded dimension. In particular,
$$ [V_i] = 1+ q^2 +\dots + q^{2(i-1)}, \quad [\widetilde{V}_i] = [i]:=
 q^{i-1}+q^{i-3}+\dots +q^{1-i}.$$

\paragraph{Stable category and triangular structure.} Any graded projective
$H$-module is a free graded $H$-module, isomorphic to a direct sum of
copies of $H$ with shifts. Taking the graded dual
$\widetilde{V}_{p-1}^*\cong \widetilde{V}_{p-1}$, we see that
projective modules are also injective and vice versa. Therefore the
category $H\dmod$ is Frobenius (\cite{Hap}). Define the stable
category $H\umod$ to have the same objects as $H\dmod$ and morphisms
$\Hom_{H\umod}(U,W)$ to be the quotient of the space of morphisms
$\Hom(U,W)$ in $H\dmod$ by the subspace of morphisms that factor
through a projective object. Define $H\ufmod$ and $\HOM_{H\ufmod}(U,W)$
likewise.

A morphism $f:U\lra W$ factors through a projective object if and only
if it factors through the canonical $H$-module map
\begin{equation}\label{eqn-canonical-embedding}
\lambda_U:U{\lra} H\{2-2p\}\otimes U, \ \ \ \ \lambda_U(u) =
\dif^{p-1}\otimes u,
\end{equation}
see~\cite[Lemma 1]{Kh}. A map $g:
H\{2-2p\}\otimes U\lra W$ is determined by $h: U \{2-2p\} \lra W$,
where $h(u) := g(1 \otimes u).$ Define $g_i(u)  := g(\dif^i
\o u).$ Then $h(u) = g_0(u)$ and $ g_{i+1}(u) = \dif (g_i
(u)) - g_i(\dif (u)),$ implying\footnote{Here we use ``$\circ$''
to denote the composition of maps. In what follows, to avoid
potential confusion, we will use ``$\cdot$'' to denote the $H$-action on
$\Hom$-spaces when necessary.}
$$ g_i  =  \sum_{j=0}^i (-1)^j {i \choose j} \dif^{i-j} \circ g_0 \circ \dif^j .$$
In particular ,
$$ g_{p-1} =  \sum_{j=0}^{p-1} \dif^{p-1-j} \circ h  \circ \dif^j, $$
since
$$ (-1)^j {p-1 \choose  j} \equiv 1 \ (\mathrm{mod}~p).$$
Therefore, $f: U\lra W$ factors through a projective $p$-complex if
and only if there is $\Bbbk$-linear map $h: U\lra W$ of degree
$2-2p$ such that
\begin{equation}
f = \sum_{j=0}^{p-1} \dif_W^{p-1-j} \circ h  \circ \dif_U^j .
\end{equation}
Here ``$\circ$'' denotes composition of $\Bbbk$-linear maps.
The following diagram illustrates an $f$ and $h$ of the formula. One
should compare this with the usual notion of a null-homotopic map between
complexes.
$$\xymatrix@C=1.5em
{\cdots\ar[r]^-{\partial_U} & U^{i-2(p-1)}\ar[r]^-{\partial_U}
 \ar[d]_{f} & U^{i-2(p-2)}
\ar[d]_{f}\ar[r]^-{\partial_U} &\cdots\ar[r]^-{\partial_U} &
U^i\ar[dlll]|*+<1ex,1ex>{\scriptstyle h}
\ar[d]_{f}\ar[r]^-{\partial_U}
&U^{i+2}\ar[dlll]|*+<1ex,1ex>{\scriptstyle h}
\ar[d]_{f}\ar[r]^-{\partial_U}& \cdots \ar[r]^-{\partial_U} &
U^{i+2(p-1)}\ar[dlll]|*+<1ex,1ex>{\scriptstyle h}
\ar[d]_{f}\ar[r]^-{\partial_U} & \cdots
\\
\cdots\ar[r]_-{\partial_W} & W^{i-2(p-1)}\ar[r]_-{\partial_W} &
W^{i-2(p-2)}\ar[r]_-{\partial_W} &\cdots\ar[r]_-{\partial_W}
&W^i\ar[r]_-{\partial_W} &W^{i+2}\ar[r]_-{\partial_W}& \cdots
\ar[r]_-{\partial_W} & W^{i+2(p-1)}\ar[r]_-{\partial_W} & \cdots .}
$$

\begin{defn}\label{def-contractible-p-complex}
We call a $p$-complex $U$ \emph{contractible} if it is isomorphic to
the zero complex in $H\umod$.
\end{defn}

Equivalently, from the discussion above, $U$ is contractible if and
only if the identity morphism $\Id_U:U\lra U$ factors through the
canonical injection (\ref{eqn-canonical-embedding}). We summarize the
definition and the above equational characterization into the
following proposition.

\begin{prop} \label{prop-contractible-H-module-condition}
The following conditions on a p-complex $U$ are equivalent:
\begin{enumerate}
\item[i)] $U$ is contractible.
\item[ii)] $U$ is a graded projective $H$-module.
\item[iii)] $U$ is a graded injective $H$-module.
\item[iv)] $U$ is isomorphic to a direct sum of graded shifts $V_{p-1}\{i\}$ of the free
rank one module $V_{p-1}$.
\item[v)] For any $x\in U$ with $\dif (x)=0$ there exists an $y\in U$ such that
$\dif^{p-1}(y) =x$.
\item[vi)] There exists a linear endomorphism $h$ of $U$ of degree $2-2p$ such
that
$$ \Id_U = \sum_{j=0}^{p-1} \dif^{p-1-j} \circ h \circ \dif^j .$$
\end{enumerate}
\end{prop}
\begin{proof} Obvious from the definition and discussion above.
\end{proof}

Once and for all,
we fix an inclusion of balanced indecomposable modules
\begin{equation} \label{eqn-fixed-p-2}
\tilde{\iota}: \widetilde{V}_0 \lra \widetilde{V}_{p-2}\o \widetilde{V}_{p-2}, \ \ \ \
\tilde{\iota}(\tilde{v}_0):=\sum_{j=0}^{p-2}(-1)^{j}\tilde{v}_{j}\o \tilde{v}_{p-2-j}.
\end{equation}
The decomposition (\ref{eqn-product-rule-for-balanced-indecomposables})
shows that $\tilde{\iota}$ induces an isomorphism in the stable category
$H\udmod$.

We define the shift functor $[1]$ on $H\umod$ and its full
subcategory $H\ufmod$ as follows. For any $U \in H\umod$, let
\begin{equation}\label{eq-shift}
U[1]:= (H\o U)\{2-2p\}/(\textrm{Im}\lambda_U)=\widetilde{V}_{p-2}\{-p\}\o U.
\end{equation}
Then $[1]$ is invertible and its inverse functor is given by
\begin{equation}\label{eq-minus-shift}
[-1]\ : \  U \mapsto \widetilde{V}_{p-2}\{p\}\o U = V_{p-2}\{2\} \o U .
\end{equation}
The map $\tilde{\iota}$ fixes functor isomorphisms $\mathrm{Id}\cong [-1][1]$ and $\{-2p\} \cong [2]$.

The class of exact triangles in $H\umod$ (resp. $H\ufmod$) is
constructed as follows. Let $\underline{f}: U\lra V$ be a morphism
in $H\udmod$ and let $f$ be a lift of $\underline{f}$ in $H\dmod$.
Consider the push-out of $\lambda_U$ and $f$, which fits into
the following commutative diagram:
$$\xymatrix{ 0 \ar[r]\ar[d] & U \ar[r]^-{\lambda_U}\ar[d]_{f} & H\o U
\ar[r]^{\overline{\lambda_U}}\ar[d] & U[1] \ar[r]\ar@{=}[d]& 0\ar[d]\\
0 \ar[r] & V\ar[r]^g & C_f \ar[r]^h & U[1]\ar[r] & 0 .}$$ Then we
declare the sextuple $\xymatrix{U \ar[r]^{\underline{f}}& V
\ar[r]^{\underline{g}} & C_f \ar[r]^{\underline{h}} & U[1]}$ to be a
\emph{standard exact triangle}, where $\underline{g}$,
$\underline{h}$ denotes the image of $g$ and $h$ in $H\umod$.
Moreover, any sextuple $\xymatrix{U \ar[r]^{\underline{f}}& V
\ar[r]^{\underline{g}} & W \ar[r]^{\underline{h}} & U[1]}$ that is
isomorphic in $H\umod$ to a standard exact triangle is called an \emph{exact
triangle}. The notation $\underline{f}$, $\overline{\lambda_U}$,
etc. is taken from Happel \cite{Hap}.

\begin{prop}\label{thm-p-complexes-triangulated}The category of
$p$-complexes $H\umod$ (resp. $H\ufmod$) equipped with the shift
functor $[1]$ and the class of exact triangles as above is
triangulated.
\end{prop}
\begin{proof}Omitted. This is a special case of Theorem 2.6 in
\cite[Section 1.2]{Hap}
\end{proof}

The shift $U[1]$ of an object $U$ can be computed more economically
by embedding $U$ into a contractible complex and taking the quotient.
For instance,
the balanced indecomposable $\widetilde{V}_i$ is a submodule of the contractible module
$\widetilde{V}_{p-1}\{i+1-p\}$, and $\widetilde{V}_i[1]\cong
\widetilde{V}_{p-1}\{i+1-p\}/\widetilde{V}_{i}\cong \widetilde{V}_{p-2-i}\{-p\}$. The
picture below illustrates this for $p=5$ and $i=2$. Each one-dimensional weight
space of is denoted by a dot and the $\dif$
action by horizontal arrows. The degree zero spaces are labeled by black dots.
$$\xymatrix{\widetilde{V}_i:\ar[d]& & & \circ \ar[r] \ar[d] & \bullet \ar[r]\ar[d] &
\circ\ar[d] \\
 \widetilde{V}_{p-1}\{i+1-p\}:\ar[d] & \circ \ar[r]\ar[d]& \circ\ar[r]\ar[d]& \circ\ar[r]
& \bullet\ar[r] &\circ \\
\widetilde{V}_i[1]: & \circ \ar[r] & \circ & & &.
 }$$
Reading these diagrams bottom-up with an appropriate grading shift is a convenient way
to realize the functor $[-1]$ on balanced indecomposable $H$-modules.

We also record the following useful result on how short exact
sequences of $H$-modules give rise to exact triangles in $H\umod$.
Suppose $0\lra U \stackrel{f}{\lra} V \stackrel{g}{\lra} W \lra 0$
is a short exact sequence of $H$-modules. Consider the diagram
$$\xymatrix{& 0\ar[d] & 0\ar[d] & & \\
0\ar[r] & U\ar[r]^-{\lambda_U}\ar[d]_f & H \o U \ar[r]
\ar[d]_{\overline{f}} & U[1]\ar[r]\ar@{=}[d]&0\\
0\ar[r]& V\ar[r]^-{h}\ar[d]_{g} & C_f \ar[r]^{e}
\ar[d]_{\overline{g}} & U[1]\ar[r] & 0\\
& W\ar[d]\ar@{=}[r] & W\ar[d] & & \\
& 0 & 0 & &. }$$ Here $C_f$ is the push-out of $f$ and $\lambda_U$,
and the rest of the diagram is completed by the push-out property.
We claim the $\overline{g}:C_f \lra W$ is an isomorphism in
$H\umod$. In fact, the submodule $H\o U$ is injective, and thus the
middle vertical exact sequence splits and gives us $C_f\cong (H\o U)\oplus W$.
The claim follows.

\begin{lemma} \label{lemma-ses-lead-to-dt-H-mod-case}
In the above notation,
$$\xymatrix{
U\ar[r]^{\underline{f}}& V\ar[r]^{\underline{h}}
&W\ar[r]^{\underline{e}\cdot\overline{g}^{-1}}& U[1] }$$ is an exact
triangle. Conversely any exact triangle in $H\umod$ is isomorphic to
one of this form.
\end{lemma}

\begin{proof}By the remarks made before the lemma, this
sextuple is isomorphic to the standard one
$$\xymatrix{U \ar[r]^{\underline{f}} & V \ar[r]^{\underline{h}} &
C_f \ar[r]^{\underline{e}}  & U[1]}.$$ The result follows.
\end{proof}

\paragraph{The Grothendieck ring of the stable category.}
Since the tensor product of any $p$-complex with a contractible
$p$-complex is contractible, the tensor product bifunctor $\o :H\dmod
\times H\dmod \lra H\dmod$ descends to the stable categories
$H\umod$ and $H\ufmod$ and turns them into symmetric monoidal
categories.

The tensor product of balanced indecomposables in the stable category is described by
$$ \tilde{V}_i \otimes \tilde{V}_j \cong  \oplusop{m\in I} \widetilde{V}_m$$
where $I=\{ |i-j|, |i-j|+2, \dots, \mathrm{min}(i+j, 2p-i-j-4)\}$, see earlier.
Notice that $U\otimes W\cong 0$ in $H\umod$ if and only if either $U\cong 0$ or
$W\cong 0$.

\begin{lemma}\label{lemma-compatibility-monoidal-triangulated-structures}
The symmetric monoidal structure $\o: H\umod\times H\umod\lra
H\umod$ is compatible with the triangulated structure on $H\umod$,
in the sense that for any $U,~W \in H\umod$,
$$(U\o W)[1]\cong U[1]\o W \cong U \o (W[1]).$$
\end{lemma}
\begin{proof} Immediate, since $U[1]=\widetilde{V}_{p-2}\{-p\}\o U$ and $H$ is
cocommutative.
\end{proof}

Later we will see that the natural inclusion $H\ufmod \subset H\umod$
induces an equivalence of $H\ufmod$ with the subcategory $H\udmod^c$ consisting of
\emph{compact} objects.

\begin{defn}\label{def-Grothendieck-ring-H-stable-fmod} We define
the Grothendieck group $K_0(H\ufmod)$ to be the abelian group
generated by the symbols $[U]$ of isomorphism classes of objects
$U\in H\ufmod$, subject to the relations that $[V]=[U]+[W]$ whenever
$U\lra V\lra W \lra U[1]$ is an exact triangle.
\end{defn}

The exact triangle $U \lra 0 \lra
U[1]\xrightarrow{-\Id_{U[1]}} U[1]$ gives the relation $[U[1]]=-[U]$
in $K_0(H\ufmod)$. The grading on $H\ufmod$ makes $K_0(H\ufmod)$
into a $\Z[q,q^{-1}]$-module by setting $q[U]:=[U\{1\}]$, while the
exact symmetric monoidal structure on $H\ufmod$ equips
$K_0(H\ufmod)$ with a commutative ring structure, $1:=[V_0]$ being the
multiplicative unit.

Clearly, $K_0(H\ufmod)$ is generated as a $\Z[q,q^{-1}]$-module by the symbol
$[V_0]$, subject to the only relation that
$(1+q^2+\cdots+q^{2(p-1)})[V_0]=0$. Thus, we identify
it with the ring $\Z[q]/(1+q^2 + \cdots + q^{2(p-1)})$, also denoted $\mathbb{O}_{p}$:
\begin{equation}\label{def-mathbb-O-p}
\mathbb{O}_p:=\Z[q]/(1+q^2 + \cdots + q^{2(p-1)}) \cong K_0(H\ufmod).
\end{equation}
For $p\neq 2$, the
polynomial $1+q^2 + \cdots + q^{2(p-1)}$ decomposes into the product
of the $p$-th and $2p$-th cyclotomic polynomials:
\begin{equation}\label{eqn-decomp-cyclomtic-polynomial}
1+q^2+\cdots+q^{2(p-1)}=\Psi_{p}(q)\Psi_{2p}(q).
\end{equation}
This is readily seen from the relation $(1-q^2)\Psi_{p}(q)\Psi_{2p}(q)=1-q^{2p}$, which
in turn also shows that
\begin{equation}
\Psi_{2p}(q)=\sum_{i=0}^{p-1}(-q)^i=\Psi_p(-q).
\end{equation}
It follows that $\Z[q]/(\Psi_p(q))\cong \Z[q]/(\Psi_{2p}(q))$. These rings,
which are also known as the rings of the $p$-th and $2p$-th cyclotomic integers, are
isomorphic to each other and usually denoted by $\mc{O}_p$ and $\mc{O}_{2p}$:
\begin{equation}
\mc{O}_p:=\Z[q]/(\Psi_p(q)), \ \ \ \ \mc{O}_{2p}:=\Z[q]/(\Psi_{2p}(q)), \ \ \ \ \mc{O}_p\cong \mc{O}_{2p}.
\end{equation}
By mapping $q$ to a primitive $p$-th root of unity $\zeta_p\in \C$ (resp. a primitive $2p$-th
root of unity $\zeta_{2p}\in \C$),
we identify $\mc{O}_p\cong \Z[\zeta_p]$ (resp. $\mc{O}_{2p} \cong \Z[\zeta_{2p}]$), and
the corresponding fields of fractions $\Q[q]/(\Psi_p(q))\cong \Q[\zeta_p]$ (resp.
$\Q[q]/(\Psi_{2p}(q))\cong \Q[\zeta_{2p}]$). Equation (\ref{eqn-decomp-cyclomtic-polynomial})
induces quotient maps
\[
\xymatrix{ & \mathbb{O}_p = \Z[q]/(\Psi_{p}(q^2)) \ar @{->>}[dl]\ar @{->>}[dr]&\\
\mc{O}_p= \Z[q]/(\Psi_p(q)) & & \mc{O}_{2p}= \Z[q]/(\Psi_{2p}(q)),}
\]
with the injective product map $\mathbb{O}_p\lra \mc{O}_p\times\mc{O}_{2p}$.
Inverting $2\in \Z$, the product map induces an isomorphism
$\mathbb{O}_p[\frac{1}{2}]\cong \mc{O}_p[\frac{1}{2}]\times \mc{O}_{2p}[\frac{1}{2}].$
We have
$\mathbb{O}_{p} \o_{\Z[q,q^{-1}]}\mathcal{O}_{p}\cong \mathcal{O}_{p}$ and
$\mathbb{O}_{p} \o_{\Z[q,q^{-1}]}\mathcal{O}_{2p}\cong \mathcal{O}_{2p}$.
Likewise, $\mathbb{O}_{p}\o_{\Z[q,q^{-1}]}\F_p \cong \F_p$, via the
homomorphism $\Z[q,q^{-1}]\lra \F_p$ taking $q$ to $1$.

\begin{ntn} \label{notation-symbol-as-O-dim}In what follows, we
will refer to the symbol $[U]\in K_0(H\ufmod)\cong \mathbb{O}_p$ of an
object $U\in H\ufmod$ as the \emph{$\mathbb{O}_p$-dimension}
of $U$, while $[U]_p:=[U]~(\mathrm{mod}~p)\in \F_p$ as its
\emph{$\mathbb{F}_p$-dimension}.
\end{ntn}

Our choice of having $\deg(\dif)=2$ rather than $1$ is mainly due to
the standard grading of the KLR algebras $R(\nu)$ (see Section 4)
and other objects that categorify quantum
groups and their representations. This results in a slight annoyance that
$K_0(H\ufmod)$ is no longer an integral domain, as shown by equation (\ref{eqn-decomp-cyclomtic-polynomial})
(c.f. \cite[Section 3]{Kh}). However, this choice of degree has an additional, if rather modest, merit,
as we will see shortly
that the existence of balanced indecomposable
$H$-modules makes the cohomology theories more symmetric with respect to duality.

%%%%%%%%%%%%%%%%%%%%%%%%%%%%%%%%%%

\paragraph{Cohomology of a $p$-complex.} When $p=2$, a graded $H$-module
$U$ is just a complex of $\Bbbk$-vector spaces with a degree two
differential, and has associated cohomology groups $\mH(U) =
\mathrm{ker}\dif / \mathrm{im}\dif$.

One possible analogue of these groups for larger $p$'s is given by the
following construction. For each $0\le k \le p-2$ form
the graded vector space
$$ \mH_{/ k}(U) = \mker(\dif^{k+1})/(\mim  (\dif^{p-k-1})+\mker (\dif^{k})).$$
The original $\Z$-grading of $U$ gives a decomposition
$$ \mH_{/ k}(U) = \bigoplus_{i\in \Z} \mH_{/ k}^i(U) .$$
The differential $\dif$ induces a map $\mH_{/ k} \lra \mH_{/k-1}$,
also denoted $\dif$, which takes  $\mH_{/ k}^i(U)$ to $\mH_{/
k-1}^{i+2}(U)$. Define the \emph{slash} cohomology of $U$ as
\begin{equation}
 \mH_{/ }(U) =  \bigoplus_{k=0}^{p-2} \mH_{/ k}(U).
\end{equation}
Let also
$$\mH^i_{/}(U) :=
\bigoplus_{k=0}^{p-2}\mH_k^i(U).$$
We have the decompositions
\begin{equation}
\mH_{/}(U) = \bigoplus_{i\in \Z} \mH^i_{/}(U)  = \bigoplus_{k=0}^{p-2}
\mH_{/k}(U)=\bigoplus_{i\in \Z}\bigoplus_{k=0}^{p-2} \mH^i_{/ k}(U).
\end{equation}
$\mH_{/}(U)$ is a bigraded $\Bbbk$-vector space, equipped with an operator
$\dif$ of bidegree $(2,-1)$, $\dif:\mH^i_{/k}\lra \mH^{i+2}_{/k-1}$.
The $/$-grading (slash-grading) is nontrivial only in degrees between $0$ and $p-2$.
Necessarily $\dif^{p-1}=0$.

Forgetting about the $/$-grading gives us a graded vector space
 $\mH_{/}(U)$ with the differential $\dif$, which we can view as a graded $H$-module.
$\mH_{/}(U)$ is isomorphic to $U$ in the stable category $H\umod$, and
we can decompose
\begin{equation}\label{eqn-decomp-of-M-into-cohomology-and-free}
U\cong \mH_{/}(U)\oplus P(U)
\end{equation}
in the abelian category of $H$-modules,
where $P(U)$ is a maximal projective direct summand of $U$. The cohomology group $\mH_{/}(U)$, viewed
as an $H$-module, does not contain any direct summand isomorphic to a free $H$-module.

This assignment $U\mapsto \mH_{/}(U)$ is functorial in $U$ and can be viewed as a functor
$H\dmod \lra H\umod$ or as a functor $H\umod\lra H\umod$; the latter
functor is isomorphic to the identity functor.

\begin{eg} When $p=3$,
$$ \mH(U) = (\mker (\dif^2)/(\mim\dif + \mker \dif)) \ \oplus \
   \mker (\dif) / \mim (\dif^2) ,$$
with the differential going from the first summand to the second.
\end{eg}

The following way to think about these cohomology groups is helpful.
Since a $p$-complex splits into a direct sum of indecomposable
$p$-complexes $V_i\{j\}$, we can reduce our consideration to
the cohomology of these complexes.
We depict $V_i$ by an oriented chain with $i+1$ vertices and
$i$ oriented edges; each vertex denotes a basis vector in the
corresponding weight space of $V_i$ and oriented edges -- the
action of $\dif$.

$$\overset{0}{\circ}\lra  \overset{1}{\circ} {\lra}  \circ
\overset{\cdots}{\lra}  \circ \lra  \overset{i}{\circ}$$

\vspace{0.1in}

%\drawing{chain1.eps}

\vspace{0.1in}

For an even simpler picture, keep vertices only and switch to
balanced indecomposables. Assemble diagrams for
$\widetilde{V}_{p-1}, \widetilde{V}_{p-2}, \dots, \widetilde{V}_0$
into a column to form a triangle with $p$ points on
each side, and let $\widetilde{V}=\widetilde{V}_{p-1}\oplus
\widetilde{V}_{p-2}\oplus \dots \oplus \widetilde{V}_0$.

%$$\xymatrix@C=.35em@R=.5em{ \circ && \circ && \circ && \circ && \circ &&
%\circ && \circ \\
%&\circ && \circ && \circ && \circ && \circ && \circ & \\
%&& \circ && \circ && \circ && \circ && \circ && \\
%&&& \circ && \circ && \circ && \circ &&&\\
%&&&& \circ && \circ && \circ &&&&\\
%&&&&&\circ &&\circ &&&&&\\
%&&&&&&\circ &&&&&& }$$

\vspace{0.1in}

\drawing{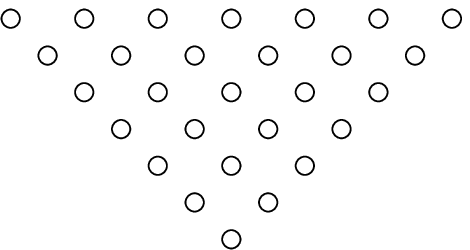}

\vspace{0.1in}

We think of points in the triangle diagram as basis elements of $\widetilde{V}$, with
the differential acting one step to the right.
The $k$-th slash cohomology groups $\mH_{/ k}(\widetilde{V})$ for $0\le k \le p-2$
pick out the terms in the $k$-th diagonal in the triangle, except for the vertex in
the top row (the top row corresponds to the unique projective indecomposable
module in the direct sum), for the total of $p-1-k$ vertices.

\vspace{0.1in}

\drawing{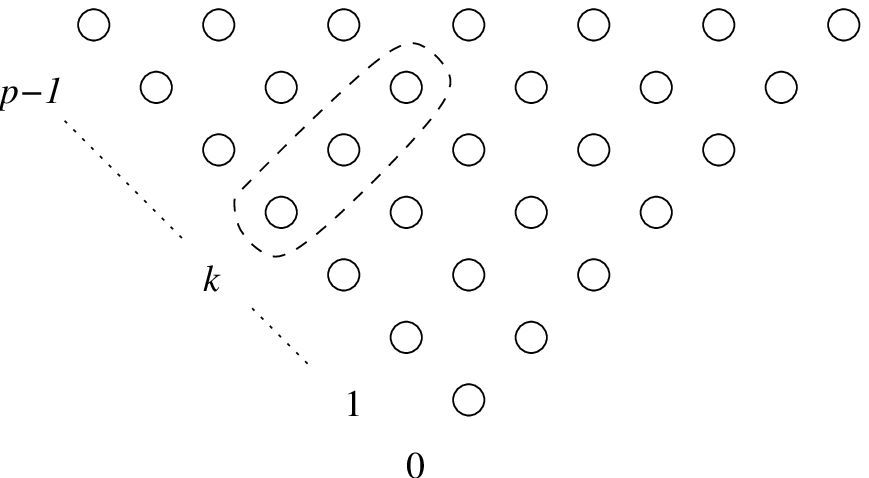}

\vspace{0.1in}

The zeroth cohomology $\mH_{/ 0}(\widetilde{V})$ collects vertices on the
southwest-northeast edge of the triangle save for the rightmost top one.
The last nonzero term $\mH_{/ p-2}(\widetilde{V})$ picks out
only one vertex, below and to the right of the upper left one.

\vspace{0.1in}
\drawing{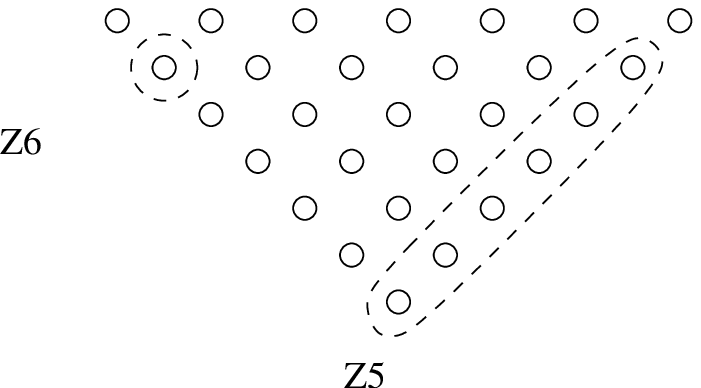}

\vspace{0.1in}

Our diagram of dots in a triangle has a $y$-axis symmetry, and
it is natural to introduce a second family of cohomology groups,
obtained by ``reflecting'' the original groups about the symmetry
axis. We call these the \emph{backslash cohomology} and define them
by
\begin{equation}
\mH_{\backslash k} = (\mim (\dif ^k)\cap \mker (\dif^{p-k-1}))/\mim (\dif^{k+1}).
\end{equation}
The dotted curve below surrounds basis elements of $\mH_{\backslash k}(\widetilde{V})$.

\vspace{0.1in}
\drawing{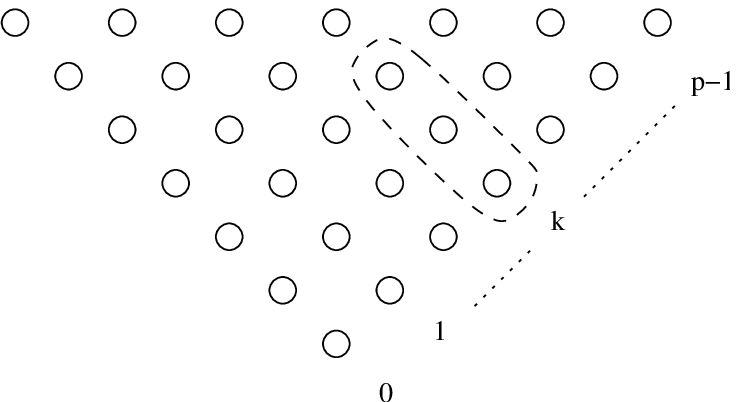}
\vspace{0.1in}

We have
$$\mH_{/ k}(U)^{\ast} \cong \mH_{\backslash k}(U^{\ast}), \quad \quad
    \mH_{\backslash k}(U)^{\ast} \cong \mH_{/k}(U^{\ast}),$$
where ``$\ast$'' denotes the graded dual $H$-module.

\vspace{0.1in}

In comparison, the original cohomology groups of Mayer~\cite{May1, May2}
assigned to a $p$-complex and generalized by Kapranov~\cite{Kap}
and Sarkaria~\cite{Sar1,Sar2}, see also Dubois-Violette \cite{DV2},
are
\begin{equation}
{}_k\mH \ = \ \mker{\dif^k}/\mim{\dif^{p-k} }
\end{equation}
for $1\le k \le p-1$.

Of course ${}_1 \mH(U) = \mH_{/ 0}(U)$  for any $p$-complex $U$, but
the two types of groups are otherwise different.
Diagrammatically, the $k$-th Mayer cohomology ${}_k\mH$ picks out
a parallelogram with $k$ vertices in the northwest direction and $p-k$ in the
northeast direction:

\vspace{0.1in}

\drawing{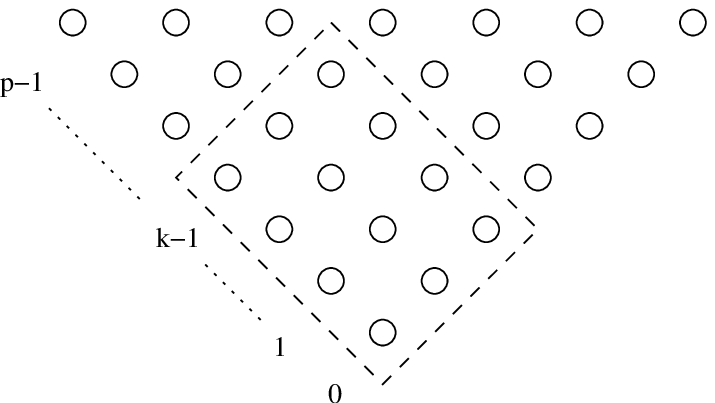}

\vspace{0.1in}

Thus, $\mH_{/ k-1}$ is the cokernel of the natural homomorphism
$ {}_{k-1}\mH \lra {}_k \mH$, while the kernel is precisely the backslash
cohomology group $\mH_{\backslash  p-k}$. See the picture below.

\vspace{0.1in}
\drawing{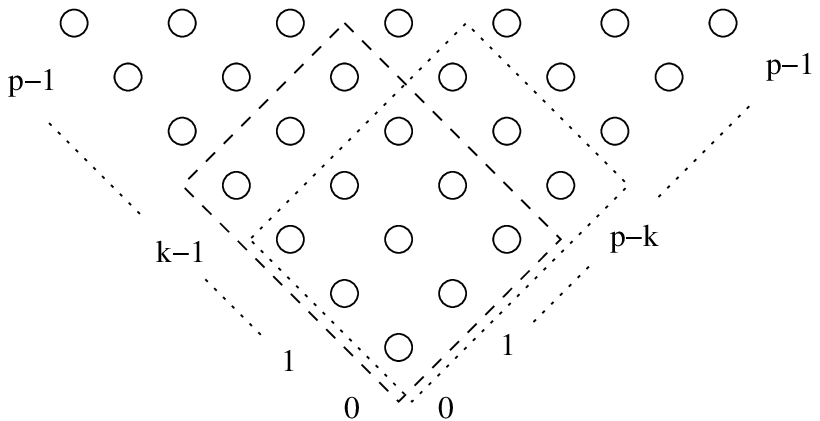}
\vspace{0.1in}

Overall, the various cohomology groups fit into exact sequences
\begin{equation}
0 \lra \mH_{\backslash  p-k}(U) \lra  {}_{k-1}\mH(U) \lra {}_k \mH(U)
\lra \mH_{/ k-1}(U) \lra 0
\end{equation}
functorially in $U$.

The two types of groups together with connecting homomorphisms carry
the same amount of information; either one determines the isomorphism type of
the $p$-complex in the stable category. They differ only by packaging.
The new packaging via (back)slash cohomology describes a  minimal
(unique up to isomorphism) representative of a $p$-complex $U$ in the
stable category as $\mH_{ / }(U)$.
Jonsson's ``train cohomology'' groups~\cite{Jon} pick out single vertices
on the longest anti-diagonal (northwest-southeast direction) in the triangle diagram.

\subsection{\texorpdfstring{$p$}{p}-DG algebras and \texorpdfstring{$p$}{p}-DG modules}\label{subsec-p-DG-algebra-and-module}

\paragraph{$p$-DG algebras.}
By analogy with DG algebras and DG modules \cite[Section 10.1]{BL}, we now define
their counterparts: $p$-DG algebras and $p$-DG modules. Recall that $\Bbbk$ is a
field of characteristic $p>0$.

\begin{defn} A $p$-DG algebra $A$ is graded $\Bbbk$-algebra $A$
equipped with a $\Bbbk$-linear derivation (also called a
differential) $\dif_A: A\lra A$ of degree $2$ (i.e. $\dif_A:A^k\lra
A^{k+2}$ for all $k\in \Z$), which satisfies the $p$-nilpotency condition
\begin{equation}
\partial_A^p=0
\end{equation}
and the Leibniz rule
\begin{equation}
\partial_A(ab) = \partial_A(a) b + a \partial_A(b),
\end{equation}
for any elements $a,~b\in A$ .
\end{defn}

When the algebra $A$ is generated by a set of elements $\{a_i|i\in
I\}$, to check that $\dif^p=0$ on the whole algebra, it is enough to
check that $\dif^p(a_i)=0$ for all $i\in I$, since
$\dif^p(ab)=\dif^p(a)b + a\dif^p(b)$ (in characteristic $p$, $\dif^p$ is
a derivation if $\dif$ is). A $\dif$-stable graded two-sided
ideal $I$ of $A$ gives rise to a $p$-DG algebra $A/I$. The direct
product $A\times B$ and the tensor product $A\o B$ of $p$-DG algebras $A$ and $B$ are again
$p$-DG algebras with the obvious gradings and differentials.

\begin{defn}\label{p-DG-module} A \emph{left
$p$-DG module} $(M,\partial_M)$ over a $p$-DG algebra $A$ is a graded
left $A$-module $M$ with a $\Bbbk$-linear endomorphism $\dif_M$ of
degree $2$ (i.e.~$\dif_M:M^k\lra M^{k+2}$ for all $k\in \Z$) such
that it is $p$-nilpotent:
\begin{equation}
\partial_M^{p}=0,
\end{equation}
and for any $a\in A$, $m\in M$
\begin{equation}
\partial_M(am)=\partial_A(a)m+a\partial_M(m).
\end{equation}
A morphism of (left) $p$-DG-modules is a morphism of (left) $A$-modules of degree
zero which commutes with the differentials. We will denote the
category of left $p$-DG modules by $A_\dif\dmod$, or, alternatively, by
$(A,\dif)\dmod$. Likewise, \emph{a right $p$-DG
module} $N$ over $A$ consists of a right $A$-module $N$ equipped with a
$\Bbbk$-linear endomorphism $\dif_N$ of degree $2$ such that:
\begin{equation}
\partial_N^{p}=0,
\end{equation}
and for any $a\in A$, $n\in N$
\begin{equation}
\partial_N(na)=\partial_N(n)a+n\partial_A(a).
\end{equation}
\end{defn}

We will omit the subscripts in $\dif_A$, $\dif_M$, $\dif_N$ whenever no confusion can arise.

\begin{rmk}[$p$-DG algebra as $H$-module algebra]
\label{rmk-p-DG-as-hopfological} A $p$-DG algebra is a
graded $H$-module algebra and can also be described as $H^\prime$-comodule
algebra, for a suitable infinite-dimensional Hopf algebra $H^\prime$ related to
the dual of $H$. The $p$-derivation on
$A$ allows us to define the smash product ring $A_{\dif}$ as
follows. As a $\Bbbk$-vector space, $A_{\dif}\cong A\o H$, and
the multiplication on $A_\dif$ is given by $(a_1\o 1)\cdot(a_2\o 1)=(a_1a_2)\o
1$, $(1\o h_1)\cdot (1\o h_2)=1\o(h_1h_2)$, $(a_1\o 1)(1\o
h_1)=a_1\o h_1$ for any $a_1,~a_2\in A$, $h_1, h_2 \in H$, and by the
rule for commuting elements of the form $a\o 1$, $1\o \dif$:
$$(1\o \dif)(a\o 1)=(a\o 1)(1\o \dif)+\dif(a)\o 1.$$
Note that $1\o H \subset A_\dif$ is a subalgebra. Moreover, since
$A$, $H$ are compatibly graded, i.e.
$\mathrm{deg}(\dif(a))=\textrm{deg}(a)+2$ for any homogeneous $a\in
A$, the commutator equation is homogeneous so that $A_\dif$ is
graded. The category of left (resp. right) $p$-DG modules is equivalent
to the category of graded left (resp. right) $A_\dif$-modules. This
explains our choice of notation for the category of left (resp. right)
$p$-DG modules above and shows that it is an abelian category with
arbitrary coproducts.

The algebras $A_\dif$ are examples of smash product algebras, see~\cite{Mo, Qi}
for instance.
\end{rmk}

If $A=\Bbbk$ is the ground field, a $p$-DG
module over $A$ is just a $p$-complex of $\Bbbk$-vector spaces. A
$2$-DG algebra is the same as a
differential graded algebra over a field of characteristic two, with
the differential of degree two rather than one.

\paragraph{The homotopy category.}

\begin{defn}\label{def-null-homotopy-and-homotopy-category}
Two morphisms of $p$-DG modules $f,g:M\lra N$ in $A_\dif\dmod$ are
said to be \emph{homotopic} if there exists a morphism $h$ of
$A$-modules of degree
$2-2p$ such that:
\begin{equation}
f-g=\sum_{i=0}^{p-1}\partial_N^{p-1-i}\circ h \circ \partial_M^{i}.
\end{equation}
It is readily seen that the collection of all morphisms that are homotopic to
zero (\emph{null-homotopic morphisms}) forms an ideal
in $A_\dif\dmod$. A $p$-DG module $M$ is called \emph{contractible} if $\Id_M$
is null-homotopic.
The homotopy category $\mc{C}(A,\dif)$ of $p$-DG complexes is the
categorical quotient of $A_\dif\dmod$ by the ideal of
null-homotopic morphisms.
\end{defn}

We often abbreviate $\mc{C}(A,\dif)$ to $\mc{C}(A)$ when the context is clear.

The tensor product of a $p$-complex $V$ and $p$-DG module $M$
is naturally a $p$-DG module. The $A$ action on
$M$ is extended to the whole of $V\o M$ by letting it act trivially
on $V$, and $\dif$ acts by $\dif(x\o m):=\dif(x)\o m+x\o \dif(m)$ for any
$x\in V$ and $m\in M$. This gives a bifunctor
$$\o:H\dmod\times A_\dif\dmod \lra A_\dif\dmod, \ \ \ \ (V,~M)\mapsto V\o M.$$
%If $f-g=\sum_{i=0}^{p-1}\dif^{i}\circ h\circ
%\dif^{p-1-i}$ so that $f$ and $g$ are homotopic morphisms, then for any $V\in H\dmod$,
%$\Id_V\o f-\Id_V\o g=\sum_{i=0}^{p-1} \dif^{i} \circ (\Id_V\o
%h)\circ \dif^{p-1-i}$, which in turn follows from $\dif\circ
%\Id_V-\Id_V\circ \dif=\dif(\Id_V)=0$. Similarly, if $f^\prime,~g^\prime:V\lra V$
%are homotopic morphisms of the $p$-complex $V$, $f^\prime\o
%\Id_M,~g^\prime\o\Id_M$ are homotopic morphisms of the $p$-DG module $V\o
%M$.
The tensor product descends to homotopy categories:
$$\o: H\umod\times \mc{C}(A) \lra \mc{C}(A)$$
and equips $\mc{C}(A)$ with a
``triangulated module-category'' structure over $H\umod$. The shift
functor on $\mc{C}(A)$ is inherited from $H\umod$.

\begin{defn}\label{def-shift-functor-p-DG-module} The \emph{shift
endo-functor} $[1]$ on $\mc{C}(A)$ is defined by
$$ M\mapsto \widetilde{V}_{p-2}\{-p\}\o M.$$
\end{defn}
The inverse of $[1]$ is
$[-1] \ : \ M\mapsto \widetilde{V}_{p-2}\{p\}\o M$, and moreover
$[2]:=[1]\circ[1]\cong \{-2p\}$, with the isomorphisms coming from
the map $\tilde{\iota}$ of (\ref{eqn-fixed-p-2}).

\begin{defn}\label{def-exact-triangle-homotopy-category} Let
$\underline{u}: M\lra N$ be a morphism in $\mc{C}(A)$, and $u:M\lra
N$ be a lift of $\underline{u}$ in $A_\dif\dmod$. Let $C_u$ be the
push-out of $u: M \lra N$ and $\lambda_M: M \lra H\o M$, so that it
fits into the following commutative diagram:
$$\xymatrix{0\ar[r]\ar[d] & M\ar[r]^-{\lambda_M}\ar[d]_u & H \o M
\ar[r]^{\overline{\lambda_M}}\ar[d]_{\overline{u}} & M[1]\ar[r]\ar@{=}
[d]&0\ar[d]\\
0\ar[r]& N \ar[r]^v & C_u \ar[r]^w & M[1]\ar[r] & 0.}$$
The \emph{standard exact triangle} associated with $\underline{u}$
in $\mc{C}(A)$ is the sextuple
$$\xymatrix{M \ar[r]^{\underline{u}}&
N\ar[r]^{\underline{v}} & C_u \ar[r]^{\underline{w}}& M[1]}$$ where
the underlined morphisms denote the image of the corresponding
un-underlined ones in $\mc{C}(A)$. We say that a sextuple
$M\stackrel{u}{\lra}N\stackrel{v}{\lra}W\stackrel{w}{\lra}M[1]$ is
an \emph{exact triangle} in $\mc{C}(A)$ if it is isomorphic to a
standard one in $\mc{C}(A)$.
\end{defn}

\begin{prop}\label{prop-homotopy-category-triangulated} The homotopy
category $\mc{C}(A)$ together with the shift functor $[1]$ and the
class of exact triangles described above is triangulated.
\hfill$\square$
\end{prop}

\begin{proof} See~\cite[Theorem 1]{Kh}.  \end{proof}

The class of exact triangles in $\mc{C}(A)$ can also be
characterized as consisting of the images of some short exact sequences
of $A_\dif$-modules.

\begin{lemma} \label{lemma-ses-lead-to-dt-homotopy-category}
Let $0\lra M\stackrel{u}{\lra}N \stackrel{v}{\lra} W \lra 0$ be a
short exact sequence of $A_\dif$-modules which splits
as a sequence of $A$-modules. Then associated with it there is an exact triangle
$$\xymatrix{M \ar[r]^{\underline{u}}&
N\ar[r]^-{\underline{v}} & W \ar[r]^-{\underline{w}}& M[1]}$$ in the
homotopy category.
\end{lemma}
\begin{proof} See \cite[Lemma 4.3]{Qi}.    \end{proof}

\paragraph{The derived category.}Since $A_\dif$ contains $H$ as a
subalgebra, the natural forgetful functor from the
category of $p$-DG modules to the base category of $p$-complexes allows
us to define the cohomologies $\mH_{\backslash}(M)$, $\mH_{/}(M)$ of a
$p$-DG module $M$ to be the cohomologies of the $p$-complex $M$. A morphism
$u:M\lra N$ of $p$-DG modules naturally induces a morphism of the
underlying $p$-complexes, and gives rise to morphisms on
cohomologies:
$$u_{\backslash}: \mH_{\backslash}(M)\lra \mH_{\backslash}(N),
\ \ \ \ u_/: \mH_{/}(M)\lra \mH_{/}(N).$$
This allows us to define the notion of quasi-isomorphisms.

\begin{defn}\label{def-quasi-isomorphism} A morphism of $p$-DG modules
$u:M\lra N$ is called a \emph{quasi-isomorphism} if it induces an
isomorphism of the underlying $p$-complexes up to homotopy. This is
equivalent to either of the following conditions on cohomology:
\begin{itemize}
 \item[a)] $u_{\backslash}:
\mH_{\backslash}(M)\cong \mH_{\backslash}(N)$,
 \item[b)] $u_/:
\mH_{/}(M)\cong \mH_{/}(N)$.
\end{itemize}
A $p$-DG module $M$ is called \emph{acyclic} if $0\lra M$, or
equivalently $M \lra 0$, is a quasi-isomorphism.
\end{defn}
Notice that a contractible $p$-DG module is automatically acyclic
but not vice versa. However, these notions coincide for the ground
field $\Bbbk$ viewed as a $p$-DG algebra.

It follows from the definition that a morphism of $p$-DG modules
which is homotopic to a quasi-isomorphism is also a
quasi-isomorphism, leading to a well-defined notion of quasi-isomorphism
in the homotopy category $\mc{C}(A)$.

\begin{prop}\label{prop-qis-localizing-class}Quasi-isomorphisms in
$\mc{C}(A)$ constitute a localizing class.
\end{prop}
\begin{proof}This is Proposition 4 of \cite[Section 1]{Kh}.
We also refer the reader to \cite[Section III.2]{GM} for the standard
notion of a localizing class and definition of localization
with respect to a localizing class.
\end{proof}

\begin{defn}Let $A$ be a $p$-DG algebra. We define the derived
category $\mc{D}(A,\dif)$ of $p$-DG modules over $A$ to be the
localization of $\mc{C}(A,\dif)$ with respect to the class of
quasi-isomorphisms. When no confusion can arise, we will write
$\mc{D}(A)$ for short.
\end{defn}

$\mc{D}(A)$ inherits a natural triangulated structure from that of
$\mc{C}(A)$. The class of exact triangles in $\mc{D}(A)$ consists of
sextuples
$M\stackrel{u}{\lra}N\stackrel{v}{\lra}W\stackrel{w}{\lra}M[1]$ that
are isomorphic to the images of exact triangles in
$\mc{C}(A)$. Moreover, it also inherits a triangulated
module-category structure over $H\ufmod$ from that of $\mc{C}(A)$.

\begin{lemma}\label{lemma-ses-lead-to-dt-derived-category}Let
$0\lra M\stackrel{u}{\lra}N \stackrel{v}{\lra} W \lra 0$ be a short
exact sequence of $A_\dif$-modules. Then associated with it there is
an exact triangle
$$\xymatrix{M \ar[r]^{\underline{u}}&
N\ar[r]^-{\underline{v}} & W \ar[r]^-{\underline{w}}& M[1]}$$ in the
derived category.
\end{lemma}
\begin{proof}Omitted. This together with the explicit construction of
$\underline{w}$ is Lemma 4.4 of \cite[Section 4]{Qi}.
\end{proof}

In one very particular situation the derived category is easy to
determine, namely when it is equivalent to the zero category.

\begin{prop}\label{prop-contractible-p-DG-algebra-condition}
Let $A$ be a $p$-DG algebra. The following statements are
equivalent.
\begin{itemize}
\item[i)]$\mc{D}(A)\cong 0$.
\item[ii)] $A$ is acyclic.
\item[iii)] There exists an element $a \in A$ such that
$\dif^{p-1}(a)=1$.
\item[iv)] There exists an element $b \in A$ such that
$\dif(b)=1$.
\end{itemize}
Furthermore, if $b$ is central in $A$, then $\mc{C}(A)\cong 0$.
\end{prop}
\begin{proof} $i)\Rightarrow ii)$. $A\cong 0$.

$ii)\Rightarrow iii)$. Since $\dif(1)=0$ and $A$ is acyclic, Proposition
\ref{prop-contractible-H-module-condition} $v)$  implies that there is an element
$a$ with $\dif^{p-1}(a)=1$.

$iii)\Rightarrow iv)$. Take $b=\dif^{p-2}(a)$.

$iv)\Rightarrow i)$. To do this we use the following elementary
identity: if $D$ and $L$ are $\Bbbk$-linear maps on a vector
space satisfying $[D,L]=\Id$, then
$$\sum_{i=0}^{p-1}D^{p-1-i}L^{p-1}D^{i}=-\Id.$$
If this is proven, let $L=L_b$ be the left multiplication on any
$p$-DG module $M$ by the element $b$. Then $[\dif_M,
L_b]=\dif(b)=\Id_M$ and $i)$ follows by taking the homotopy to be
$-(L_b)^{p-1}$. Now we prove the identity. Let
$\textrm{ad}_{D}(L):=[D,L]$, and note that $\textrm{ad}_D$ is a derivation on the
space of all linear operators. On one hand,
$$\begin{array}{rcl}
(\textrm{ad}_{D})^{p-1}(L^{p-1})& =
&(\textrm{ad}_D)^{p-2}((p-1)L^{p-2}
\textrm{ad}_D(L))=(\textrm{ad}_D)^{p-2}((p-1)L^{p-2}\cdot \Id)\\
& = & \textrm{ad}_D^{p-3}((p-1)(p-2)L^{p-3}\textrm{ad}_D(L) )\cdots
= (p-1)!\Id \equiv -\Id ~(\textrm{mod}~p).
\end{array}$$
On the other hand, expanding the iterated commutators directly
results in
$$\begin{array}{rcl}
\textrm{ad}_D^{p-1}(L^{p-1})& = &[D,[D,\cdots,[D,L^{p-1}]\cdots]] =
\sum_{i=0}^{p-1}(-1)^i{p-1 \choose i}D^iL^{p-1}D^{p-1-i}\\
& = & \sum_{i=0}^{p-1}D^iL^{p-1}D^{p-1-i}.
\end{array}$$
The claimed identity follows.

 For the last statement, notice that the proof of
$iv)\Rightarrow i)$ says that $-(L_b)^{p-1}$ on any $p$-DG module
$M$ is a homotopy between $\Id_M$ and $0_M$, and thus any $M$ is
acyclic. If $b$ is moreover central, then $L_b^{p-1}$ is an
$A$-module map. The result follows.
\end{proof}

\paragraph{Morphism spaces as cohomology.} Similar to the usual DG case,
the morphism space in the homotopy category $\mc{C}(A)$ is the degree
zero part of a certain cohomology $p$-complex $\RHOM$.

First off, recall that in the abelian category $A_\dif\dmod$, the
morphism space between two $p$-DG modules $M$, $N$ can be computed
as
\[
\begin{array}{rcl}
\Hom_{A_\dif}(M,N) &\cong &\{u\in \Hom_A(M,N)|\dif \circ u-u \circ
\dif =0\}\\
& \cong & \{u\in \Hom_A(M,N)|\dif\cdot u=0\}\\
& \cong & \textrm{Ker}(\dif:\HOM^0_A(M,N)\lra \HOM_A^{2}(M,N)),
\end{array}
\]
where ``$\cdot$'' denotes the $H$-action on $\HOM_A$, and
$\HOM_A^{2}(M,N)$ stands for the space of $A$-module maps
homogeneous of degree two. Similarly by Definition
\ref{def-null-homotopy-and-homotopy-category}, there are canonical
isomorphisms of $\Bbbk$-vector spaces
\begin{eqnarray}
\Hom_{\mc{C}(A)}(M,N) &\cong &\dfrac{\Hom_{A_\dif}(M,N)}{\dif^{p-1}\cdot\Hom_A^{2-2p}(M,N)}\vspace{0.1in} \nonumber \\
& \cong & \dfrac{\textrm{Ker}(\dif:\HOM^0_A(M,N)\lra
\HOM_A^{2}(M,N))}{\textrm{Im}(\dif^{p-1}:\HOM_A^{2-2p}(M,N)\lra
\HOM_A^0(M,N))}\vspace{0.1in} \nonumber \\
& \cong & \mH_{/0}^0(\HOM_A(M,N)),
\end{eqnarray}
i.e.~this is just the degree zero part of the slash (or backslash) cohomology
of the $p$-complex $\HOM_A(M,N)$. It forgets a lot of information
about the total $p$-complex $\HOM_A(M,N)$
even if $A_\dif$-module maps of all degrees are taken into account:
$$\HOM_{\mc{C}(A)}(M,N):=\oplus_{i\in \Z}\Hom_{\mc{C}(A)}(M,N\{i\})=
\mH_{/0}(\HOM_A(M,N)).$$
%This discussion motivates the following
%useful principle about $p$-DG modules :\emph{We should work with
%$\HOM_A(M,N)$ rather than the actual morphism space
%$\Hom_{\mc{C}(A)}(M,N)$ in the homotopy category.}
We summarize the
above discussion in the following proposition, whose proof in the
more general case of module-algebras over any finite dimensional
Hopf algebras can be found in \cite[Section 5]{Qi}.

\begin{prop}\label{prop-morphism-as-cohomology}The following diagram
commutes
$$\xymatrix{ A_\dif\dmod\times A_\dif\dmod\ar[rr]^-{\HOM_A(-,-)}\ar[rrdd]
\ar@{..>}@/_4.6pc/[ddrrrr]_(.25){\Hom_{\mc{C}(A)}(-,-)}&&
H\dmod \ar[dd]^{Q} \ddrruppertwocell^{<2.5>{(\mathrm{Ker}\dif)^0}}{<3>{\pi}}  \\
&& && \\
&& H\udmod \ar[rr]^{\mH_{/0}^0} && \Bbbk\dmod.\\
&& && &&}
$$
Here $Q$ is the natural quotient functor which is the identity on objects\footnote{We usually just omit writing it when the context is clear.} from the abelian category
of $p$-complexes to its homotopy category, and $\pi$ is the natural
transformation which on a $p$-complex $V$ is given by
$\pi(V):(\mathrm{Ker}\dif)^0(V):=\mathrm{Ker}(\dif|_{V^0})\lra \mH^0_{/0}(V)$. \hfill$\square$
\end{prop}

\paragraph{Cofibrant modules and bar resolution.}
In the DG case, morphism spaces between ``nice enough'' objects in
the derived category coincide with the morphism spaces of the same
objects in the homotopy category. The analogous result also holds in
the $p$-DG case, which we now make precise.

Following Keller \cite{Ke1}, we make the following definitions.

\begin{defn}\label{def-cofibrant-p-DG-modules} Let $A$ be a $p$-DG
algebra and $P$ a $p$-DG module over $A$.
\begin{itemize}
\item[i)] We say that $P$ is \emph{cofibrant} if for any surjective
quasi-isomorphism of $p$-DG modules $M\twoheadrightarrow N$, the
induced map of graded $\Bbbk$-vector spaces $\HOM_{A_\dif}(P,M)\lra
\HOM_{A_\dif}(P,N)$ is surjective.
\item[ii)]We say that $P$ satisfies \emph{property} (P) if the
following two conditions holds:
\begin{itemize}
\item[(P 1)] There is an exhaustive (possibly infinite) filtration
of $P$ by $A_\dif$-submodules:
$$0=F_{-1}\subset F_{0}\subset F_1\subset \cdots \subset F_r\subset F_{r+1}
\subset \cdots \subset P,$$ Here being exhaustive means that
$P=\cup_{r=0}^{\infty}F_r$.
\item[(P 2)]The associated graded
modules of the filtration $F_{r+1}/F_r$ for all $r\in \N$ are
isomorphic to (possibly infinite) direct sums of free $A$-modules of
the form $A\{s\}$, $s\in \Z$.
\end{itemize}
\end{itemize}
\end{defn}
The cofibrance property is equivalent to requiring that the induced map on
the homogeneous degree zero part of $\HOM_{A_\dif}$, which is just
$\Hom_{A_\dif}(P,M)\lra \Hom_{A_\dif}(P,N),$
be surjective, since $M\twoheadrightarrow N$ is a surjective quasi-isomorphism if and only if
$M\{r\}\twoheadrightarrow N\{r\}$ for any $r\in \Z$ is.

We list the main properties of these two types of modules without
proof. The reader can find the proofs in the
more general framework of hopfological algebra in \cite[Section 6]{Qi}.

\begin{prop}\label{prop-easy-properties-cofibrant-modules}
Let $A$ be a $p$-DG algebra. Cofibrant and property (P) modules over
$A$ enjoy the following properties:
\begin{itemize}
\item[i)] Property (P) modules are cofibrant.
\item[ii)] If $P$ is cofibrant and $M$ is acyclic, then
$\HOM_{A}(P,M)$ is a contractible $p$-complex. In particular
$\Hom_{\mc{C}(A)}(P,M)=0$.
\item[iii)] A $p$-DG module $P$ is cofibrant
if and only if $P$ is projective as an $A$-module, and for any
acyclic $p$-DG $A$-module $M$, $\HOM_A(P,M)$ is an acyclic $p$-complex.
\item[iv)] If $P$ is cofibrant and $M$ is any $p$-DG module, then
there is an isomorphism of $\Bbbk$-vector spaces
$\Hom_{\mc{C}(A)}(P,M)\cong \Hom_{\mc{D}(A)}(P,M).$
\item[v)] A $p$-DG module is cofibrant if and only if it is a direct
summand of a property (P) module in the abelian category $A_\dif\dmod$.

\end{itemize}
\end{prop}
\begin{proof} Omitted. See the results 6.4--6.10 of \cite{Qi}.
\end{proof}

The following theorem states that there are always ``enough''
property (P) modules.

\begin{thm}[Bar resolution]\label{thm-bar-resolution} Let $A$ be a
$p$-DG algebra. For any $p$-DG module $M$, there is a surjective
quasi-isomorphism of $p$-DG modules
$$\mathbf{p}(M)\twoheadrightarrow M,$$
where $\mathbf{p}(M)$ satisfies property (P). Moreover, the assignment
$M\mapsto \mathbf{p}(M)$ is functorial in $M$.
\end{thm}
\begin{proof} See Theorem 6.6 and Corollary 6.7 of \cite{Qi}. \end{proof}

We will refer to the functorial cofibrant module $\mathbf{p}(M)$ as
the \emph{bar resolution} of $M$.

\begin{cor} Let $\mc{CF}(A)$ (resp. $\mc{P}(A)$) denote the
full-subcategory of $\mc{C}(A)$ consisting of cofibrant (resp.
property (P)) modules. Then the composition of functors
$$\mc{CF}(A)\subset \mc{C}(A)\stackrel{Q}{\lra} \mc{D}(A),$$
$$(\textrm{resp.}~\mc{P}(A)\subset \mc{C}(A)\stackrel{Q}{\lra} \mc{D}(A)),$$
where $Q$ is the localization functor, is an equivalence of
categories. \hfill$\square$
\end{cor}

\paragraph{Compact modules and Grothendieck groups.} When working with the
abelian or derived category of modules over some ring, allowing
infinitely generated modules makes the Grothendieck group of the module category zero.
One needs to restrict the size of (projective) modules to define
the Grothendieck group. Similar consideration applies in the $p$-DG
context, and an appropriate size restriction on $p$-DG modules is the
\emph{compactness} condition.

\begin{defn}\label{def-compact-p-DG-module} Let $A$ be a $p$-DG
algebra. An object $M\in \mc{D}(A)$ is called \emph{compact} if the
functor $\Hom_{\mc{D}(A)}(M,-):\mc{D}(A)\lra \Bbbk\dmod$ commutes
with arbitrary direct sums.
\end{defn}

For instance, the module $A\{r\}$ is compact for any $r\in \Z$.

\begin{prop}\label{prop-derived-category-compactly-generated} Let
$A$ be a $p$-DG algebra. The derived category $\mc{D}(A)$ is
generated by the set of compact objects $\{A\{r\}|r\in \Z\}$, in
the sense that if $M\in \mc{D}(A)$ satisfies
$$\Hom_{\mc{D}(A)}(A\{r\}[s],M)=0$$
for all $r$, $s\in \Z$, then $M\cong 0$ in $\mc{D}(A)$.
\end{prop}
\begin{proof}Omitted. See \cite[Proposition 7.6]{Qi}.
\end{proof}
The general machinery of Ravenel-Neeman can be applied to the compactly generated
category $\mc{D}(A)$ to characterize the class of compact objects in it.

\begin{defn}\label{def-finite-cell-modules} A $p$-DG module $M\in
A_\dif\dmod$ is called a \emph{finite cell
module} if there is a finite exhaustive filtration of
$A_\dif$-modules
$$0=F_{-1}\subset F_0 \subset \cdots \subset F_{r-1} \subset F_{r} \subset \cdots \subset F_N=M$$
such that the associated graded modules $F_{r}/F_{r-1}$ for all $0
\leq r \leq N$ are isomorphic, as $A$-modules, to  finite direct
sums of the free $A$-modules $A\{s\}$ for various $s\in \Z$.
\end{defn}

In other words, finite cell modules are just property (P) modules that are finitely generated
as $A$-modules.

\begin{thm}[Characterization of compact modules]
\label{thm-characterizing-compact-modules} The compact
objects in $\mc{D}(A)$ are those which are isomorphic in the
derived category to a direct summand of a finite cell module.
\end{thm}
\begin{proof}Omitted. See Theorem 7.14 and Corollary
7.15 of \cite{Qi}.
\end{proof}

Informally, we think of the notions introduced
in this section as analogues of the usual ring theoretic
concepts below. In the comparison table, $R$ stands for some ring, ``f.g.'' is short
for ``finitely generated'', and the arrows indicate inclusion relations.
Dotted line surrounds
types of objects in the abelian category $A_{\dif}\dmod$, while the bottom right arrow
indicates that any finite cell module is a compact object in the derived category
$\mc{D}(A)$.

\drawing{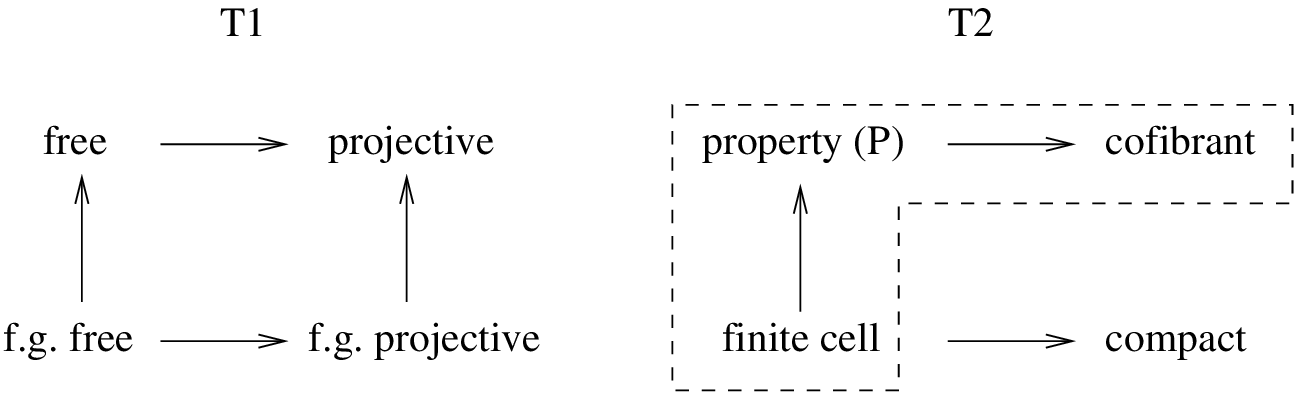}

Denote by $\mc{D}^c(A)$ the strictly full subcategory of $\mc{D}(A)$
consisting of compact objects. The above theorem implies that
$\mc{D}^c(A)$ is \emph{Karoubian}, or equivalently \emph{idempotent complete}.
In the simplest case $A\cong \Bbbk$
there is an equivalence of categories $H\ufmod \cong \mc{D}^c(\Bbbk)$.

\begin{defn}\label{def-Grothendieck-group-p-DG-rings}
The \emph{Grothendieck group}  $K_0(\mc{D}^c(A))$ of a $p$-DG algebra
$A$ (or $K_0(A)$ if no confusion can arise) is the abelian group
generated by the symbols of the isomorphism classes of objects in
$\mc{D}^c(A)$, subject to the relations that
$$[M]=[L]+[N]$$
whenever there is an exact triangle $L\lra M \lra N \lra L[1]$ in
$\mc{D}^c(A)$.
\end{defn}

The exact tensor bi-functor $\o: H\udmod \times \mc{D}(A)\lra
\mc{D}(A)$ restricts to an exact bi-functor
$$\o: H\ufmod \times \mc{D}^c(A)\lra \mc{D}^c(A),$$
which in turn equips $K_0(A)$ with a module structure over the ring
$\mathbb{O}_{p}\cong K_0(H\ufmod)$.

\paragraph{Derived hom and tensor product.} Observe that if $A$ is a $p$-DG
algebra, its opposite algebra $A^{\textrm{op}}$ equipped with the
same differential $\dif$ is also a $p$-DG algebra,
called the \emph{opposite $p$-DG algebra} of $A$. The
category of \emph{right} $p$-DG modules over $A$
can be identified with that of \emph{left} modules over $A^{\textrm{op}}$.

The category of $p$-DG algebras has a natural monoidal structure as
follows. If $A$ and $B$ are two $p$-DG algebras, then the \emph{tensor
product $p$-DG algebra} $A\o B$ is the usual tensor product algebra
equipped with the differential $\dif(a\o b):=\dif (a)\o b+a\o\dif(
b)$. One readily checks that this differential is compatible with
the algebra structure, and furthermore $A\cong A\o \Bbbk$, $B\cong
\Bbbk \o B$ sit inside $A\o B$ naturally as $p$-DG subalgebras.

Now if $A$, $B$ are $p$-DG algebras, a \emph{$p$-DG
$(A,B)$-bimodule} $_AX_B$ is a left module $X$ over
the $p$-DG algebra $A\o B^{\textrm{op}}$. $p$-DG bimodules naturally
give rise to functors on $p$-DG module categories via the hom and
tensor functors
$$\HOM_{A}(_AX_B,-):A_\dif\dmod \lra B_\dif\dmod, \ \ \ \
{_AM}\mapsto \HOM_A(_AX_B,{_AM}),$$
$$_AX_B\o_B(-): B_\dif\dmod \lra A_\dif\dmod, \ \ \ \
{_BN} \mapsto {_AX_B}\o_B N.$$

\begin{defn}\label{def-derived-hom-and-tensor} Let $A$, $B$ be $p$-DG
algebras, and $_AX_B$ be a $p$-DG $(A,B)$-bimodule.
\begin{itemize}
\item[i).]The \emph{derived tensor product} $_AX\o^{\mathbf{L}}_B:
\mc{D}(B)\lra \mc{D}(A)$ is the composition of functors
$$\mc{D}(B)\stackrel{\mathbf{p}}{\lra} \mc{P}(B) \xrightarrow{_A X\o_B(-)}
\mc{C}(A)\stackrel{Q}{\lra}\mc{D}(A),$$ where $Q$ is the natural
localization functor.
\item[ii).] The \emph{derived hom functor} $\mathbf{R}\HOM(_AX_B,-)$
is the composition of functors $$\mc{D}(A)
\xrightarrow{\HOM_A(\mathbf{p_A}(X),-)} \mc{C}(B) \stackrel{Q}{\lra}
\mc{D}(B),$$ where $\mathbf{p_A}(X)$ denotes the bar resolution of
$X$ as a left $p$-DG $A$-module.\footnote{By the construction of the
bar resolution in \cite[Theorem 6.6]{Qi}, $\mathbf{p_A}(X)$ has a
natural right $p$-DG $B$-module structure.}
\end{itemize}
\end{defn}

The tensor product and hom functors satisfy the following adjunction
property.

\begin{prop}\label{prop-tensor-hom-adjunction}A p-DG $(A,B)$-bimodule $X$
gives an adjunction of functors in the derived category:
$$\Hom_{\mc{D}(A)}(X\o^{\mathbf{L}}_BN,{M})\cong
\Hom_{\mc{D}(B)}({N},\mathbf{R}\HOM_A(X,M)),$$
for any $M\in\mc{D}(A)$ and $N\in \mc{D}(B)$.
\end{prop}
\begin{proof} This is a special case of Proposition 8.13 in \cite{Qi}.
\end{proof}

A morphism $\mu:X\lra Y$ of $p$-DG $(A,B)$-bimodules descends to a natural
transformation between the derived tensor product functors
$$\mu^{\mathbf{L}}: X\o_B^{\mathbf{L}}(-)\Longrightarrow Y\o_B^{\mathbf{L}}(-):
 \mc{D}(B)\lra \mc{D}(A).$$
The following proposition gives a criterion on when a derived tensor
functor induces an equivalence of derived categories, and when such
a natural transformation is an isomorphism of functors.

\begin{prop}\label{prop-criterion-derived-equivalence}
\begin{itemize}
\item[i)] Let $X$ be a $p$-DG $(A,B)$-bimodule and suppose it is cofibrant
as a $p$-DG $A$-module. Then
$X\o^{\mathbf{L}}_B(-):\mc{D}(B)\lra \mc{D}(A)$ is an
equivalence of triangulated categories if and only if the following
two conditions hold:
\begin{itemize}
\item[1)] The natural map $B\lra \HOM_{A}(X,X)$ is a quasi-isomorphism.
\item[2)] $X$, when regarded as a p-DG $A$-module, is a compact generator of
$\mc{D}(A)$.
\end{itemize}
\item[ii)] Let $\mu:X\lra Y$ be a morphism of p-DG
$(A,B)$-bimodules. The natural transformation
$\mu^{\mathbf{L}}: X\o_B^{\mathbf{L}}(-)\Longrightarrow
Y\o_B^{\mathbf{L}}(-)$ is an isomorphism of functors if and only if
$\mu:X\lra Y$ is a quasi-isomorphism of p-DG bimodules.
\end{itemize}
\end{prop}
\begin{proof}Omitted. See Proposition 8.8 of \cite{Qi}.
\end{proof}

\begin{eg}\label{eg-cofibrance-in-Morita}We give an example of the above proposition related to the classical Morita equivalence. Let $A$ be a $p$-DG algebra and $U$ a finite dimensional $p$-complex. Then $A\o U$ is a $p$-DG $A$-module which satisfies property (P) (Definition \ref{def-cofibrant-p-DG-modules}). Form the algebra $B:=\mathrm{END}_A(A\o U)^{op}$. $B$ has a natural $p$-DG algebra structure where for any $f\in B$ and $x\in A\o U$,
\[
\dif(f)(x)=\dif(f(x))-f(\dif(x)).
\]
Moreover, $A\o U$ is a $p$-DG $(A,B)$-bimodule, and we define $(A\o U)^\vee:=\HOM_A(A\o U, A)\cong A\o U^\ast$, which is a $p$-DG $(B,A)$-bimodule. It is easy to see that
\[(A\o U)^\vee \o_A (A\o U) \cong B,\ \ \ (A\o U)\o_B (A\o U)^\vee \cong A, \]
as $p$-DG bimodules. Therefore, the functors between abelian categories
\[
(A\o U)\o_B(-): B_\dif\dmod \lra A_\dif\dmod,
\]
\[
(A\o U)^\vee \o_A(-): A_\dif\dmod \lra B_\dif\dmod,
\]
which are mutually inverse of each other descend to equivalences of homotopy categories
\[
(A\o U)\o_B(-): \mc{C}(B) \lra \mc{C}(A),
\]
\[
(A\o U)^\vee \o_A(-): \mc{C}(A) \lra \mc{C}(B).
\]
It is a natural question to ask whether these functors further pass to derived equivalences. Proposition \ref{prop-criterion-derived-equivalence} i) says that if $U$ is not a contractible $p$-complex, in which case $A\o U$ is a compact generator of $\mc{D}(A)$, then
\[
(A\o U)\o_B^{\mathbb{L}}(-): \mc{D}(B) \lra \mc{D}(A)
\]
is an equivalence of derived categories. On the other hand, $(A\o U)^\vee\cong A\o U^\ast$ is not a cofibrant $B$-module if $U$ is contractible. It is cofibrant if and only if $U$ is not acyclic. When the cofibrance condition is satisfied,
\[
(A\o U)^\vee \o_A^{\mathbb{L}}(-): \mc{D}(A) \lra \mc{D}(B)
\]
is then an equivalence by Proposition \ref{prop-criterion-derived-equivalence} i) again.
\end{eg}

\paragraph{Induction and restriction functors.} A very special case
of the previous discussion is when we have a map of $p$-DG algebras
$\phi:B\lra A$, and the bimodule is given by $_AX_B={_AA_B}$. We
will allow maps $\phi: B\lra A$ with $\dif_B\circ \phi= \phi\circ
\dif_A$ which are non-unital, with $\phi(1_B)$ only an
idempotent in $A$.

\begin{defn}\label{def-induction-restriction-functors} Let $\phi:B\lra
A$ be a map of $p$-DG algebras.
\begin{itemize}
\item[i).] The (derived) induction functor $\phi^*=\mathrm{Ind}_B^A$ is the
derived tensor functor associated with the bimodule $_AA_B$:
$$\phi^*=\mathrm{Ind}_B^A=A\o_B^{\mathbf{L}}(-):\mc{D}(B)\lra \mc{D}(A).$$
\item[ii).] The restriction functor $\phi_*=\mathrm{Res}^A_B$ is the
forgetful functor via the morphism $\phi$, i.e. it takes a $p$-DG
$A$-module $M$ to a $p$-DG $B$ module $\phi(1_B)\cdot M$ by letting $B$ act through $\phi$.
It is an exact functor on $A_\dif\dmod$, and therefore it descends
to an exact functor
$$\phi_*=\mathrm{Res}^A_B:\mc{D}(A)\lra\mc{D}(B)$$
\end{itemize}
\end{defn}
The restriction functor coincides with the (derived) hom functor
$\mathbf{R}\HOM_A(A_B,-)$, since as a left $A_\dif$-module, $A$
satisfies property (P), and thus for any graded $A$-module $_AM$,
$$\mathbf{R}\HOM_A(A_B,M)=\HOM_A(A_B,M)={_BM}.$$
Hence by Proposition
\ref{prop-tensor-hom-adjunction}, there is an adjunction
$$\Hom_{\mc{D}(A)}(\phi^*(N),M)\cong \Hom_{\mc{D}(B)}(N,\phi_*(M))$$
for any $N\in \mc{D}(B)$ and $M\in \mc{D}(A)$. It follows that
$\phi^*(\mc{D}^c(A))\subset\mc{D}^c(B)$ since $\phi_*$ obviously
commutes with taking arbitrary direct sums and direct summands. Alternatively, this can be directly
seen from the characterization of compact modules as finite cell modules
(Theorem \ref{thm-characterizing-compact-modules}). In this language, Proposition
\ref{prop-criterion-derived-equivalence} translates into the
following important special case, which is the p-DG analogue of
Theorem 10.12.5.1 of \cite{BL}. Functor $\phi_*$ does not necessarily
preserve compact objects. It does, for instance, when $A$ has a finite $p$-DG
resolution as a $(B,B)$-bimodule,
or when $A$ has finite dimensional cohomology.

\begin{thm}\label{thm-qis-algebra-equivalence-derived-categories}Let
$\phi:B\lra A$ be a morphism of $p$-DG algebras that is a
quasi-isomorphism. Then the induction and restriction functors
$$\phi^*:\mc{D}(B)\lra \mc{D}(A), \ \ \ \ \phi_*:\mc{D}(A)\lra \mc{D}(B)$$
are mutually inverse equivalence of categories.
\end{thm}
\begin{proof}That $\phi^*$ is an equivalence follows from Proposition
\ref{prop-criterion-derived-equivalence}, while that of $\phi_*$
follows by adjunction \ref{prop-tensor-hom-adjunction}. For details see
the proof in \cite[Corollary 8.17]{Qi}.
\end{proof}

\begin{cor}\label{cor-qis-algebra-isomorphic-K-0}If $\phi:B\lra A$
is a quasi-isomorphism of $p$-DG algebras, then $[\phi_*]:K_0(A)\lra
K_0(B)$ is an isomorphism of $\mathbb{O}_{p}$-modules.
\end{cor}
\begin{proof} Since being compact is a categorical concept, an
equivalence of categories preserves compactness property. The result
follows readily from this and Theorem
\ref{thm-qis-algebra-equivalence-derived-categories}.
\end{proof}

\subsection{Special cases}\label{subsec-base-category-special-cases}
\paragraph{Smooth artinian algebras.}
A graded $\Bbbk$-algebra $A$ is naturally a $p$-DG algebra
with the trivial differential ($\dif=0$).
An algebra $A$ is smooth if it has a finite projective resolution
as an $(A,A)$-bimodules.

For the moment denote by $K_0^\prime(A)$ the usual Grothendieck group
of $A$, which is a $\Z[q,q^{-1}]$-module, generated by symbols $[P]$ of finitely
generated graded projective $A$-modules.

\begin{prop}\label{prop-K0-of-smooth-basic-artinian-algebras}The Grothendieck
group of a smooth artinian algebra $A$, regarded as a $p$-DG algebra
with the trivial differential, is related to the usual Grothendieck group $K_0^\prime(A)$
by
\[
K_0(A,\dif)\cong K_0^\prime(A)\o_{\Z[q,q^{-1}]}\mathbb{O}_p.
\]
\end{prop}
\begin{proof}Omitted. See \cite[Proposition 9.10]{Qi}.
\end{proof}

\paragraph{Derivations on matrix algebras.} It is well known that
the space of derivations on an associative algebra $A$ modulo inner
derivations (i.e. those of the form $\partial(a):=[x,a]=xa-ax$, for
any $a \in A$ and a fixed element $x\in A$) is classified
by the first Hochschild cohomology group $\mathrm{HH}^{1}(A)$, and
this cohomology group is preserved by Morita equivalences.
Thus if $A=\Mn $ is the $n \times n$ matrix algebra over
the ground field $\Bbbk$,
$$\mathrm{Der}(A)/\mathrm{Inn}(A)\cong \mathrm{HH}^{1}(A) \cong
\mathrm{HH}^1(\Bbbk)=0,$$
since on the ground field, there are no
non-zero derivations. In particular this says that any derivation on
the matrix algebra $\Mn$ arises as taking commutator with some fixed
element.

Now we specialize to matrix algebras over fields of positive characteristic
$p$. We  assume that $\Bbbk$ is algebraically closed
until the end of this paragraph. Let $\partial_J
(M):=\textrm{ad}_J(M)=JM-MJ~,\forall M \in \Mn$ while $J$ is fixed.
Then we have:
$$\begin{array}{rcl}
(\textrm{ad}_J)^{p}(M)&=&[J,\cdots,[J,M]\cdots] =
\sum_{k=0}^{p}(-1)^k {p \choose k}J^{p-k}M J^{k}\\
& = & J^{p}M-MJ^{p} = \textrm{ad}_{J^p}(M)
\end{array}$$
To have
$\partial_J^p=0$ on $\Mn$, one needs to require $J^p \in Z(\Mn)\cong
\Bbbk\cdot I_{n\times n}$, say $J^p=\lambda \cdot I_{n\times n}$.
Since $\Bbbk$ is algebraically closed, $\lambda$ has a $p$-th root
$\mu \in \Bbbk$. Then we have $(J-\mu I_{n\times n})^p=0$.
There exists $g\in
\mathrm{GL}(n,\Bbbk)$ such that $g(J-\mu I_{n\times
n})g^{-1}=\mathrm{Diag}(J_{i_1},\cdots,J_{i_m})$, where $i_1+\cdots
+i_m=n$ and each $J_{i_r}~(1\leq r \leq m)$ is the standard
$i_r\times i_r$ Jordan matrix with $E_{i,i+1}=1$ and 0 everywhere
else. Thus $J=g^{-1}(\textrm{Diag}(J_{i_1},\cdots,J_{i_m})+\mu
I_{n\times n})g=g^{-1}(\textrm{Diag}(J_{i_1},\cdots,J_{i_m}))g+\mu
I_{n\times n}$, and
$\mathrm{ad}J=\mathrm{ad}(g^{-1}(\textrm{Diag}(J_{i_1},\cdots,J_{i_m}))g)$
allows us to get rid of $\mu$ and $\lambda$. Hence we can assume
from the beginning that $J^p=0$. Such
matrices are classified up to conjugation by partitions
$(i_1,\cdots, i_m)\vdash n$ with $i_r\leq p$ for all $1\leq r \leq m$, with each
$i_r$ corresponding to a Jordan block $J_{i_r}$ as above. In particular, $p$-nilpotent
derivations on $\Mn$ are classified by such partitions, and the classifications
holds over any $\Bbbk$, no necessarily algebraically closed, so this restriction
can be relaxed.

Next, we observe that $\Mn$ has an obvious $\Z$ grading. Indeed, the
multiplication rule $E_{i,j}E_{k,l}=\delta_{j,k}E_{i,l}$ gives us a
$\Z$-grading $\mathrm{deg}(E_{i,j})=j-i$. The space of degree $r$ homogeneous
elements consists of matrices whose non-zero
entries are concentrated on the shifted $r$-th diagonal, i.e. the
span of $E_{i,j}$, $j-i=r$. Thus the space of homogeneous
derivations on $\Mn$ of degree one with respect to this grading coincide with the
shifted first diagonal $\oplus_{0\leq i\leq n-1}\Bbbk E_{i,i+1}$.

\paragraph{A toy model.} We examine the graded
matrix algebra $\Mn$ as above, with the derivation of
the simplest Jordan type: for any $M\in \Mn$,
$$\dif_{n}(M):=[J_n, M],$$
where $J_n=\sum_{i=1}^{n-1}E_{i,i+1}$ has only one Jordan block. Necessarily $n \leq p$. It is easy to see that
$\dif_{n}(E_{k,n})=E_{k-1,n}=E_{k-1,n}E_{n,n}$, for $1\leq k\leq n$. To make
the degrees match our conventions, set $\mathrm{deg}(E_{i,j})=2(j-i)$
so that $\mathrm{deg}(\dif_{n})=2$.

\begin{lemma}\label{lemma-a-Morita-equivalence-condition} Let $R$ be
a ring and $e\in R$ an idempotent such that $R=ReR:=\{\sum_{r,r^\prime} rer^{\prime}|
~r,~r^\prime \in R\}$. Then $R$ is Morita equivalent to the ring
$eRe.$
\end{lemma}
\begin{proof}$(Re)\o_{eRe}(-):eRe\dmod\lra R\dmod$, and
$(eR)\o_R(-):R\dmod\lra eRe\dmod$ are readily seen to be mutually inverse functors
between the two categories.
\end{proof}

We will apply this lemma in the situation when $R=\Mn_{\dif_{n}}$ is the
smash product algebra (\ref{rmk-p-DG-as-hopfological}).

\begin{prop} \label{prop-smash-matrix-morita-equivalent-H}
When $1\leq n \leq p$, the algebra $\Mn_{\dif_{n}}$ is Morita equivalent to $H$.
\end{prop}
\begin{proof} $\Mn_{\dif_{n}}$ has a basis
$\{E_{i,j}\dif_{n}^r|1\leq i,~j\leq n,~0\leq r \leq p-1\}$. Thus
$$
\begin{array}{rcl}
\Mn_{\dif_{n}}\cdot E_{n,n}\cdot \Mn_{\dif_n} & \supset & \{\sum_{i,j}E_{i,n}E_{n,n}E_{n,j}\dif_{n}^{r}|1\leq i,~j\leq n,~0\leq r \leq p-1\}\\
 & = & \{\sum_{i,j}E_{i,j}\dif_{n}^{r}|1\leq i,~j\leq n,~0\leq r \leq p-1 \}\\
 & = &\Mn_{\dif_{n}},
\end{array}$$
and the above lemma
applies. We conclude that $\Mn_{\dif_{n}}$ is Morita equivalent to the subring
$E_{n,n}\cdot \Mn_{\dif_{n}}\cdot E_{n,n}$. This subring is spanned by elements of the form
$\{E_{n,i}\dif_{n}^r E_{n,n}|1\leq i \leq n,~0\leq r \leq p-1\}$. We claim that in
fact it has as a basis $\{E_{n,n}\dif_n^r E_{n,n}|0\leq r \leq p-1\}$. This is
readily checked using the commutator relation
$$[\dif_n^k, E_{n,n}]=\sum_{l=1}^k{k \choose l}E_{n-l,n}\dif_n^{k-l},$$
where we set $E_{s,n}=0$ if $s\leq 0$. The formula follows by an easy induction
and the $\dif_n$ action on $E_{k,n}$ above. Finally, the same commutator
relation also implies that, for any $0\leq r,s \leq p-1$,
$$(E_{n,n}\dif_n^rE_{n,n})(E_{n,n}\dif_n^sE_{n,n})=E_{n,n}\dif_{n}^{r+s}E_{n,n},$$
which in turn shows that $E_{n,n}\cdot \Mn_{\dif_{n}} \cdot E_{n,n}$ has the same ring structure
as $H$. The proposition follows.
\end{proof}

The abelian category of graded $\Mn_{\dif_n}$-modules is Krull-Schmidt.
As an immediate application of the previous result, the indecomposable modules
are classified as follows.

\begin{cor} When $1\leq n \leq p$, any indecomposable graded $\Mn_{\dif_n}$-module is isomorphic to
precisely one
of the form $\widetilde{V}_i\{m\}\o \Bbbk^n $ for $i\in
\{0,1,\cdots p-1\}$ and $m\in \Z$, where $\Bbbk^n=\oplus_{j=1}^n\Bbbk E_{j,n}$
is the column p-DG module and the differential acts by
$\dif_n(E_{j,n})=E_{j-1,n}$.
\end{cor}
\begin{proof}Follows directly from the proof of Lemma
\ref{lemma-a-Morita-equivalence-condition} and the previous proposition.
\end{proof}

In particular, if $n=p$, the column module is acyclic, and the corollary implies that
all $p$-DG modules over $\mathrm{M}(p,\Bbbk)$ are acyclic. Thus $\mc{D}(\mathrm{M}(p,\Bbbk),\dif_p)\cong 0$.
In fact an easy computation shows that
\[
\dif_p\left(\sum_{i=1}^{p-1}i E_{i+1,i}\right)=
\left[\left(\begin{array}{cccccc}
0 & 1 &&&& \\
 & 0 & 1 &&& \\
 && 0 & 1 && \\
 &&& \ddots & \ddots &\\
 &&&& 0 & 1 \\
  &&&&& 0
\end{array}
\right),
\left(\begin{array}{cccccc}
0 &&&&& \\
1 & 0 &&&& \\
 &2 & 0 &&& \\
 &&\ddots& \ddots &&\\
 &&& p-2 & 0 & \\
  &&&& p-1 & 0
\end{array}
\right)\right]=I_{p} \ ,
\]
so that Proposition
\ref{prop-contractible-p-DG-algebra-condition} iii) applies.

\begin{prop}\label{prop-matrix-column-module-cofibrant} If $1\leq n \leq p-1$, the column $p$-DG module
$\Bbbk^n$ over the $p$-DG algebra $(\Mn,\dif_n)$ is a compact
cofibrant generator of $\mc{D}(\Mn)$.
\end{prop}
\begin{proof}Let $K$ be any indecomposable acyclic $p$-DG module over $(\Mn,\dif_J)$.
By our classification of indecomposable modules, $K\cong \Bbbk^n\o V_{p-1}\{m\}$ for some $m\in \Z$.
Thus $\HOM_{\Mn}(\Bbbk^n, K)=\HOM_{\Mn}(\Bbbk^n,\Bbbk^n)\o V_{p-1}\{m\}\cong V_{p-1}\{m\}$, which
is a contractible $p$-complex. Therefore Proposition \ref{prop-easy-properties-cofibrant-modules} iii)
implies that $\Bbbk^n$ is cofibrant. The compactness of $\Bbbk^n$ is clear since it is finite dimensional.
Finally, it is a generator since the free left $p$-DG module $\Mn$ has a filtration by column submodules
$F_1\subset F_2 \subset F_3 \subset \cdots \subset F_{n-1}\subset F_n=\Mn$, such that the quotients
$F_r/F_{r-1}\cong \Bbbk^n\{2(r-1)\}$. The result follows.
\end{proof}

\begin{rmk}\label{rmk-matrix-algebra-convolution}
When $1\leq n \leq p-1$, the iterated extension of $\Mn$ in the abelian category $\Mn_{\dif_n}\dmod$ gives rise to
a convolution diagram in $\mc{D}(\Mn)$, using Lemma \ref{lemma-ses-lead-to-dt-derived-category},
\[
\xymatrix@C=1.5em{F_1 \ar[rr] && F_2 \ar[rr] \ar[dl] && F_3\ar[r]\ar[dl] &\cdots\ar[r] &F_{n-1}\ar[rr] && F_n \ar[dl]\\
& \Bbbk^n\{2\}\ar[ul]_{[1]} && \Bbbk^n\{4\} \ar[ul]_{[1]} && && \Bbbk^n\{2(n-1)\}\ar[ul]_{[1]} &
}
\]
where $F_1\cong \Bbbk^n$ and $F_n\cong \Mn$. Since $\mc{D}(\Mn)$ is generated by $\Mn$ (Proposition \ref{prop-derived-category-compactly-generated}), it follows from this diagram that $\Bbbk^n$ is a generator of $\mc{D}(\Mn)$, and $[\Mn]=(1+q^2+\cdots+q^{2(n-1)})[\Bbbk^n]$ in $K_0(\Mn)$.
\end{rmk}

\begin{cor}\label{cor-p-DG-morita-equivalence-matrix-ring} If $1\leq n \leq p-1$, the functor
$$\Bbbk^n\o_{\Bbbk}(-):\mc{D}(\Bbbk) \lra \mc{D}(\Mn)$$
is an equivalence of triangulated categories. Consequently, $K_0(\Mn)\cong \mathbb{O}_{p}.$
\end{cor}
\begin{proof} The cofibrance of $\Bbbk^n$ allows us to pass from derived tensor product to underived tensor product. The result follows directly from the previous result and Proposition~\ref{prop-criterion-derived-equivalence}.
\end{proof}

The method generalizes with essentially no change to the case when the differential on $\Mn$ has more than one Jordan block. The above Proposition~\ref{prop-matrix-column-module-cofibrant} and Corollary~\ref{cor-p-DG-morita-equivalence-matrix-ring} hold as long as all Jordan blocks are of size less or equal to $p$ and there is at least one Jordan block of size strictly less than $p$. If all Jordan blocks are of size $p$, the algebra is acyclic and the derived category collapses to zero.

%%%%%%%%%%%%%%%%%%%%%%%%%%%%%
%%%%%%%%%%%%%%%%%%%%%%%%%%%%%
%%
%%  p-DG nilHecke algebra
%%
%%%%%%%%%%%%%%%%%%%%%%%%%%%%%
%%%%%%%%%%%%%%%%%%%%%%%%%%%%%

\section{The \texorpdfstring{$p$}{p}-DG nilHecke algebra}\label{sec-nilHecke}

\subsection{\texorpdfstring{$p$}{p}-derivations on the nilHecke algebra}\label{subsec-p-der-on-NH}
\paragraph{A $p$-derivation on the polynomial ring.}
Define the derivation $\dif$ on the ring of polynomials $\Bbbk[x]$, $\deg(x)=2$,
by $\dif(x) = x^2.$ This implies $\dif(x^m) = m x^{m+1}$, and thus $\dif^p(x)=0$, therefore making
$\Bbbk[x]$ into an $H$-module algebra. The inclusion $\Bbbk \subset \Bbbk[x]$
is an isomorphism in the stable category $H\umod$, since under the action of $\dif$
the algebra of polynomials decomposes into the trivial representation $\Bbbk$ of
$H$ and free modules $H\{2+2pk\}$  spanned by $x^{1+pk}, x^{2+pk}, \cdots,
x^{p(k+1)}$ for $k\in \N$ (recall that we adopt the convention $\N=\{0,1,2,\dots\}$).

Let $\pol_n = \Bbbk[x_1, \dots, x_n]$ be the ring of polynomials in
$n$ variables, with $\deg(x_i)=2$. Define the derivation $\dif$ on
$\pol_n$ by $\dif(x_i) = x_i^2$. Viewed as a graded $H$-module,
$\pol_n$ is the tensor product of $n$ copies of the module $\Bbbk[x]$.
Since the latter is isomorphic to $V_0$ in $H\umod$, the inclusion
$V_0 \subset \pol_n$ taking $V_0$ to $\Bbbk\cdot 1$ is an isomorphism in
$H\umod$.

We denote by $\bpol_n$ the space $\pol_n$ with the grading
shifted down by $\frac{n(n-1)}{2}$, viewed as a graded left $\pol_n$-module,
$$ \bpol_n:=\pol_n\ \{ \textstyle{\frac{n(1-n)}{2}} \}.$$
The generator $1$ of $\bpol_n$ lives in  degree $\frac{n(1-n)}{2}$, and  $\bpol_n=\pol_n \cdot 1$.
We call graded module  $\bpol_n$  the \emph{balanced} free $\pol_n$-module, and sometimes denote
its elements $f\cdot 1$ instead of just $f$.

A $p$-DG module structure on the balanced module $\bpol_n$ is determined by
$\dif(1)= g_\alpha \cdot 1$, where $\alpha=(\alpha_1, \dots, \alpha_n)\in \Bbbk^n$,
$g_\alpha=\sum \alpha_i x_i$ is a linear function in $x_i$'s,
viewed as an element of $\bpol_n$. Then for any $f\in \pol_n$,
\begin{equation} \label{eq-dif-g}
\dif(f \cdot 1 ) = \dif(f)\cdot 1  + f \cdot \dif(1)=(\dif(f)+ fg_\alpha)\cdot 1.
\end{equation}
The condition $\dif^p=0$ translates into $\alpha_i \in \mathbb{F}_p$, i.e.~ the coefficients being
residues mod $p$. We denote this $p$-DG module structure on
$\bpol_n$ by $\bpol_n(\alpha)$ and its generator $1$ by $1_\alpha$ to stress their dependence on $\alpha=(\alpha_1,
\dots, \alpha_n )\in \mathbb{F}_p^n$.
In what follows, to each $\alpha_t\in \mathbb{F}_p$ we assign the corresponding
element of $\N$, via the inclusion $\{0, 1, \dots, p-1\}\subset \N$, and denote it by $\alpha_t$
(as long as no confusion is possible).

Up to an overall grading
shift, $\bpol_n(\alpha)$ is isomorphic (as a $p$-DG $\pol_n$-module)
to the ideal inside the $p$-DG algebra
$\pol_n$ generated by the element $x_1^{\alpha_1} \cdots x_n^{\alpha_n}$.
As an $H$-module, $\bpol_n(\alpha)$
decomposes into the tensor product of modules $\bpol_1(\alpha_t)$ over $\Bbbk[x_t]$
for $t$ ranging from $1$ to $n$. The $H$-module $\bpol_1(\alpha), \alpha\in \mathbb{F}_p$
contains the submodule spanned by $1_\alpha, x_11_\alpha, \dots, x_1^{p-\alpha}1_\alpha$
for $\alpha\in\mathbb{F}_p^{\ast}$ and by $1_\alpha$ for $\alpha=0$, and the inclusion of this
submodule into $\bpol_1(\alpha)$ is a quasi-isomorphism. Thus, $\bpol_1(0) \cong V_0$,
and $\bpol_1(\alpha)\cong V_{p-\alpha}$ for $\alpha \neq 0$.
The module $\bpol_1(\alpha )$ is contractible if and only if
$\alpha=1$, and, more generally, $\bpol_n(\alpha)$ is a contractible
$H$-module if and only if at least one of the coefficients $\alpha_t$ of the
linear function $g_\alpha$ is $1$.

\paragraph{The induced action on symmetric functions.} The symmetric
group $S_n$ acts on $\pol_n$ by permuting the variables. Denote by
$\sym_n$ the subalgebra of $\pol_n$ consisting of $S_n$-invariant
functions, $\sym_n = \pol_n^{S_n}$. The derivation $\dif$ commutes with
the action of $S_n$ and restricts to a derivation on $\sym_n$. This
subalgebra of symmetric functions is free with generators being
elementary symmetric polynomials $e_1, e_2, \dots, e_n$, where $e_m$
is the sum of $x_{i_1}\cdots x_{i_m}$ over all subsets $\{i_1,
\dots, i_m\}$ of $\{1, 2, \dots, n\}$ of cardinality $m$:
$$ e_m = \sum_{1\le i_1 < \dots < i_m \le n} x_{i_1}\cdots x_{i_m} .$$
Later, when varying the number $n$ of variables, we will denote $e_m$ by
$e_m^{(n)}$ to stress the dependence on $n$. For instance
$e_1^{(1)}=x_1$, and $e_2^{(3)}=x_1x_2+x_1x_3+x_2x_3$.

We have the following explicit formula describing the differential
on the elementary symmetric functions.
\begin{lemma}\label{lemma-dif-on-elemetary-symmetric-functions} The
derivation $\dif$ acts on the elementary symmetric functions as
follows:
\begin{equation*} \label{eq-diff-of-e}
\dif (e_m)  = e_1 e_m - (m+1)e_{m+1} \ \ (m<n),\ \ \ \ \dif (e_n) =
e_1 e_n .
\end{equation*}
\end{lemma}
\begin{proof}Consider the following generating function over $\Z[t]$ for the
elementary symmetric functions (we set $e_0:=1$)
$$\prod_{i=1}^n(1+x_it)=\sum_{m=0}^ne_mt^m, $$
and let $\dif$ act on it as a $\Z[t]$-linear derivation which is
determined by $\dif(x_i)=x_i^2$. Then differentiating both sides
gives us
$$\begin{array}{rcl}
\sum_{m=0}^{n}\dif(e_m)t^m & = & \sum_{i=1}^{n}
(x_i^2t\cdot\prod_{j\neq i}(1+x_jt))\\
& = & \sum_{i=1}^{n}\left((x_i^2t+x_i)\cdot\prod_{j\neq
i}(1+x_jt)-x_i\cdot\prod_{j\neq i}(1+x_jt)\right)\\
& = & \sum_{i=1}^n\left(x_i \prod_{j=1}^m(1+x_jt)\right) -
\frac{\dif}{\dif t}\prod_{j=1}^n(1+x_jt)\\
& = & e_1\cdot \prod_{j=1}^n(1+x_jt)-\frac{\dif}{\dif
t}\prod_{j=1}^n(1+x_jt)\\
& = &e_1 \cdot \sum_{m=0}^ne_mt^m
-\sum_{m=-1}^{n-1}(m+1)e_{m+1}t^{m}.
\end{array}
$$
Comparing coefficients of $t$ on both sides gives the claimed
formula.
\end{proof}

As a graded $\sym_n$-module, $\pol_n$ is free of rank $n!$ with a
basis $\{ x_1^{b_1}x_2^{b_2}\cdots x_{n-1}^{b_{n-1}}|0 \le b_i \le
n-i \}$. More generally, for any fixed permutation $s$ of
$\{0,1,\cdots,n-1\}$, the set
$\{x_{1}^{b_1}x_{2}^{b_2}\cdots x_{n}^{b_{n}}| 0\leq a_i \leq s(i-1)\}$
is a basis of the module.

When $n<p$, the group algebra $\Bbbk[S_n]$ is semisimple, and the $S_n$-representation
$\pol_n$ decomposes into a direct sum of its isotypic components $\pol_{n,\lambda}$,
over all partitions $\lambda$ of $n$, corresponding to irreducible representations
$L_{\lambda}$ of $S_n$. The algebra of invariants
$\sym_n =\pol_n^{S_n}= \pol_{n, (n)}$ is identified with the summand corresponding
to the trivial representation $L_{(n)}$.
There is only one possibility for the corresponding direct sum decomposition
of the trivial representation $V_0$ in the stable category ($V_0\cong \pol_n$ in $H\udmod$).

\begin{cor}\label{cor-V0-qis-Symn} When $n<p$, the natural inclusion $V_0 \subset \sym_n$ is
an equivalence in the stable category $H\umod$, while $\pol_{n,\lambda}$
is contractible for any partition $\lambda\not= (n)$. \hfill$\square$
\end{cor}

When $n=p$, the elementary symmetric function $e_p = x_1x_2\dots x_p
$ has the property that $e_p^p$ is in the kernel but not in the
image of $\partial$, as easily seen from Lemma \ref{lemma-dif-on-elemetary-symmetric-functions}.
The inclusion $V_0\{ 2p^2\}\subset
\pol_n$ taking a basis vector of $V_0$ to $e_p^p$ realizes
$V_0\{2p^2\}$ as a direct summand of the $H$-module $\pol_p.$ More
generally, for any $a_1, \dots, a_p \ge 0$ the inclusion $V_0\{ 2p
(a_1+ \dots + a_p)\}\subset \pol_p$ taking a basis vector to
$x_1^{pa_1}x_2^{p a_2}\dots x_p^{pa_p}$ realizes the former as a
direct summand of $\pol_p$.

\paragraph{The nilHecke algebra.}The nilHecke algebra $\NH_n$ on $n$
strands over the field $\Bbbk$, see \cite{KK, Lau1},
has generators $x_1, \dots, x_n$ and $\delta_1,\dots, \delta_{n-1}$,
subject to defining relations
\begin{eqnarray} \label{eq-nh-1}
& & \delta_i^2=0, \ \  \delta_i \delta_{i+1} \delta_i =
\delta_{i+1}\delta_i \delta_{i+1}, \ \  \delta_i\delta_j = \delta_j
\delta_i
\ \ \mathrm{if} \ \  |i-j|>1, \\   \label{eq-nh-2}
& &  x_i \delta_j = \delta_j x_i \ \ \mathrm{if} \ \  j\not= i, i+1, \ \
 x_i x_j = x_j x_i, \\
 & & x_i \delta_i - \delta_i x_{i+1} = 1, \ \
     \delta_i x_i - x_{i+1}\delta_i = 1.  \label{eq-nh-3}
\end{eqnarray}

We often use a graphical presentation for monomials in the
generators of $\NH_n$, with $x_i$, respectively $\delta_i$ depicted
as a dot on the $i$-th strand and the crossing of the $i$th and
$(i+1)$-st strands:

\vspace{0.1in}
\[x_i:=
\begin{DGCpicture}
\DGCstrand(0,0)(0,1.5)[`$1$]
\DGCstrand(2,0)(2,1.5)[`$i$]\DGCdot{0.75}
\DGCstrand(4,0)(4,1.5)[`$n$]
\DGCcoupon*(0.25,0.25)(1.75,1.25){$\cdots$}
\DGCcoupon*(2.25,0.25)(3.75,1.25){$\cdots$}
\end{DGCpicture} \quad \quad \quad \quad
\delta_i:=
\begin{DGCpicture}
\DGCstrand(0,0)(0,1.5)[`$1$]
\DGCstrand(1.5,0)(2.5,1.5)[`$i+1$]
\DGCstrand(2.5,0)(1.5,1.5)[`$i$]
\DGCstrand(4,0)(4,1.5)[`$n$]
\DGCcoupon*(0.25,0.25)(1.75,1.25){$\cdots$}
\DGCcoupon*(2.25,0.25)(3.75,1.25){$\cdots$}
\end{DGCpicture}
\]

%\drawing{nilgen.eps}

The enumerating labels on the strands will be omitted when they are clear from the context.
Multiplication in $\NH_n$ is given by vertical concatenations of diagrams, with
the product $xy $ represented by the picture

\[
\begin{DGCpicture}
\DGCstrand(0,0)(0,2) \DGCstrand(1,0)(1,2) \DGCstrand(3,0)(3,2)
\DGCcoupon(-0.25,0.5)(3.25,1.5){$x$}
\DGCcoupon*(1.25,1.65)(2.75,1.85){$\cdots$}
\DGCcoupon*(1.25,0.15)(2.75,0.35){$\cdots$}
\end{DGCpicture}
\cdot
\begin{DGCpicture}
\DGCstrand(0,0)(0,2)
\DGCstrand(1,0)(1,2)
\DGCstrand(3,0)(3,2)
\DGCcoupon(-0.25,0.5)(3.25,1.5){$y$}
\DGCcoupon*(1.25,1.65)(2.75,1.85){$\cdots$}
\DGCcoupon*(1.25,0.15)(2.75,0.35){$\cdots$}
\end{DGCpicture}
=
\begin{DGCpicture}
\DGCstrand(0,0)(0,2)
\DGCstrand(1,0)(1,2)
\DGCstrand(3,0)(3,2)
\DGCcoupon(-0.25,0.2)(3.25,0.85){$y$}
\DGCcoupon(-0.25,1.15)(3.25,1.8){$x$}
\DGCcoupon*(1.25,0.88)(2.75,1.1){$\cdots$}
\DGCcoupon*(1.25,0)(2.75,0.16){$\cdots$}
\DGCcoupon*(1.25,1.84)(2.75,2){$\cdots$}
\end{DGCpicture}
\]

The nilHecke algebra with zero strands $\NH_0\cong\Bbbk$ is one-dimensional and spanned by the empty diagram $\varnothing$. The defining relations say that far away generators commute and the following diagrammatic equalities hold.

$$
\begin{DGCpicture}
\DGCstrand(0,0)(1,1)(0,2) \DGCstrand(1,0)(0,1)(1,2)
\end{DGCpicture}
=0, \ \ \ \ \ \ \ \
\begin{DGCpicture}
\DGCstrand(0,0)(2,2) \DGCstrand(2,0)(0,2) \DGCstrand(1,0)(0,1)(1,2)
\end{DGCpicture}
=
\begin{DGCpicture}
\DGCstrand(0,0)(2,2) \DGCstrand(2,0)(0,2) \DGCstrand(1,0)(2,1)(1,2)
\end{DGCpicture},
$$
\\
$$\begin{DGCpicture}
\DGCstrand(0,0)(1,1)\DGCdot{.25} \DGCstrand(1,0)(0,1)
\end{DGCpicture}-
\begin{DGCpicture}
\DGCstrand(0,0)(1,1)\DGCdot{.75} \DGCstrand(1,0)(0,1)
\end{DGCpicture}=
\begin{DGCpicture}
\DGCstrand(0,0)(0,1) \DGCstrand(1,0)(1,1)
\end{DGCpicture}=
\begin{DGCpicture}
\DGCstrand(0,0)(1,1) \DGCstrand(1,0)(0,1)\DGCdot{.75}
\end{DGCpicture}-\begin{DGCpicture}
\DGCstrand(0,0)(1,1) \DGCstrand(1,0)(0,1)\DGCdot{.25}
\end{DGCpicture}.
$$

For any permutation $w$ we define $\delta_w := \delta_{i_1}\dots \delta_{i_r}$,
where $w=s_{i_1}\dots s_{i_r}$ is some reduced decomposition of $w$ into the
product of elementary transpositions $s_i= (i,i+1)$. The nilHecke algebra has
a basis with elements given by
$$ x_1^{b_1}\cdots x_n^{b_n} \delta_w$$
over all $b_1, \dots, b_n\in \N$ and $w\in S_n$. The element
$x_i^b$ for some $b\in \N$ is denoted by a dot with
a label $b$ next to it on the $i$-th strand. Each basis
element can then be represented by a diagram of the composition of
crossings describing permutation $w$ composed with dots on the strands
above the permutation diagram.

\[
\begin{DGCpicture}
\DGCstrand(0,0)(1,1)(1,2)\DGCdot{1.75}[r]{$b_2$}
\DGCstrand(1,0)(0,1)(0,2)\DGCdot{1.75}[r]{$b_1$}
\DGCstrand(2,0)(3,1.5)(3,2)\DGCdot{1.75}[r]{$b_4$}
\DGCstrand(3,0)(4,1.5)(4,2)\DGCdot{1.75}[r]{$b_5$}
\DGCstrand(4,0)(2,1.5)(2,2)\DGCdot{1.75}[r]{$b_3$}
\end{DGCpicture}
\]
Diagrams without dots span a subalgebra of $\NH_n$ called
the \emph{nilCoxeter algebra}. This subalgebra, denoted $\mathrm{NC}_n$, has a
basis $\{\delta_w|w\in S_n\}$.

The algebra $\NH_n$ possesses the following symmetries which are naturally described
using its diagrammatic presentation. Reflecting a diagram about a
horizontal axis is an algebra anti-automorphism of $\NH_n$, which we
will denote by $\psi$.
\begin{equation}\label{eqn-NH-symmetry-psi}
\psi\left(~
\begin{DGCpicture}
\DGCstrand(0,0)(1,2)
\DGCstrand(1,0)(2,2)
\DGCstrand(2,0)(0,2)\DGCdot{1.65}
\end{DGCpicture}
~\right)=
\begin{DGCpicture}
\DGCstrand(0,0)(2,2)\DGCdot{0.35}
\DGCstrand(1,0)(0,2)
\DGCstrand(2,0)(1,2)
\end{DGCpicture}.
\end{equation}
Reflecting a diagram about a vertical axis and simultaneously multiplying it
by $(-1)^s$, where $s$ is the number of crossings in the diagram, is an algebra
automorphism of
$\NH_n$, which will be denoted by $\sigma$.
\begin{equation}\label{eqn-NH-symmetry-sigma}
\sigma\left(~
\begin{DGCpicture}
\DGCstrand(0,0)(1,2)
\DGCstrand(1,0)(2,2)
\DGCstrand(2,0)(0,2)\DGCdot{1.65}
\end{DGCpicture}
~\right)=(-1)^2
\begin{DGCpicture}
\DGCstrand(0,0)(2,2)\DGCdot{1.65}
\DGCstrand(1,0)(0,2)
\DGCstrand(2,0)(1,2)
\end{DGCpicture}=
\begin{DGCpicture}
\DGCstrand(0,0)(2,2)\DGCdot{1.65}
\DGCstrand(1,0)(0,2)
\DGCstrand(2,0)(1,2)
\end{DGCpicture}.
\end{equation}

\vspace{0.06in}

More intrinsically, the nilHecke algebra can be defined as the algebra of endomorphisms
of $\bpol_n$ as a $\sym_n$-module:
\begin{equation} \label{eq-nh-as-sym-mod}
 \NH_n = \mathrm{END}_{\sym_n}(\bpol_n).
\end{equation}
The generator $x_t$ of $\NH_n$ acts as multiplication by $x_t\in \pol_n$ on $\bpol_n$ (the same notation is
used here for an element of $\pol_n$ and the corresponding endomorphism of $\bpol_n$).
The generator $\delta_i$ act as the divided difference operator
$$\delta_t(f) = \frac{f- {}^t f}{x_t - x_{t+1}},$$
where $^t f$  is the polynomial $f\in \bpol_n$ with $x_t, x_{t+1}$
transposed. Divided difference operators commute with
multiplications by symmetric functions. It is easy to check that the relations
(\ref{eq-nh-1})-(\ref{eq-nh-3}) hold on these endomorphisms, resulting in a homomorphism
from the algebra with these generators and relations to $\mathrm{END}_{\sym_n}(\bpol_n)$.
This homomorphism can be shown to be an isomorphism. We refer to $\bpol_n$ as the (left)
balanced polynomial representation of $\NH_n$.

It is well known that $\bpol_n$ is a free graded module over $\sym_n$ of rank $n!$ and
graded rank $[n]!$ (due to balancing, the graded rank is invariant under
$q\leftrightarrow q^{-1}$ symmetry).
For a sequence $\beta=(b_1, \dots, b_n)$ with
$b_k\in \N$ let $x^{\beta}=x_1^{b_1}x_2^{b_2}\dots x_n^{b_n}$.
The set
\begin{equation}\label{eq-set-Bn}
  B_n \ =  \  \{ x^{\beta}|  \ \  0\le b_t \le n-t, \  t=1, \dots, n\}
\end{equation}
is a homogeneous basis of graded free $\sym_n$-module $\bpol_n$. Necessarily $b_n=0$
for such $\beta$. Let $U_n=\Bbbk\{B_n\}$ be the $\Bbbk$-vector subspace of $\bpol_n$
with basis $B_n$. There is a canonical isomorphism
\begin{equation}\label{eq-nat-map}
  U_n \otimes_{\Bbbk} \sym_n \stackrel{\cong}{\lra} \bpol_n
\end{equation}
taking $x^{\beta} \otimes f$ to $fx^{\beta}$ and, via equation
(\ref{eq-nh-as-sym-mod}), producing an isomorphism
\begin{equation}\label{NH-iso-mat}
 \NH_n \cong \mathrm{END}_{\Bbbk}(U_n)\otimes_{\Bbbk} \sym_n \cong
U_n\otimes U_n^{\ast} \otimes \sym_n.
\end{equation}
Upon ordering elements of $B_n$, we get an isomorphism
\begin{equation}\label{nhn-isom-sym}
 \NH_n \cong \mathrm{M}(n!,\sym_n),
\end{equation}
which identifies $\NH_n$ with the matrix algebra of size $n!$ with
coefficients in the ring of symmetric functions in $x_1, \dots, x_n$.

More generally, for any permutation $s\in S_n$, let
\begin{equation}\label{eq-set-Bns}
  B_{n,s} \ :=  \  \{ x^{\beta}|  \ \  0\le b_{s(t)} \le n-t, \  t=1, \dots, n\}.
\end{equation}
Then $B_{n,s}$ is a homogeneous basis of graded free $\sym_n$-module $\bpol_n$.
Denoting by $U_{n,s}=\Bbbk\langle B_{n,s}\rangle$ the $\Bbbk$-vector subspace of $\bpol_n$
with basis $B_n$, there is a natural isomorphism of $\sym_n$-modules
\begin{equation}\label{eq-nat-map-s}
  U_{n,s} \otimes_{\Bbbk} \sym_n \stackrel{\cong}{\lra} \bpol_n
\end{equation}
taking $x^{\beta} \otimes f$ to $fx^{\beta}$ and giving an isomorphism
\begin{equation}\label{NH-iso-mat-s}
\NH_n \cong \mathrm{END}_{\Bbbk}(U_{n,s}) \otimes_{\Bbbk} \sym_n \cong
U_{n,s}\otimes U_{n,s}^{\ast} \otimes \sym_n.
\end{equation}
If $s=(1)\in S_n$ is the identity element, we also write $B_{n}^+$ instead of $B_n=B_{n,(1)}$
and $U_n^+$ instead of $U_n=U_{n,(1)}$. When $s$ is the maximal length permutation
$w_0=(1,n)(2,n-1)\cdots$, we denote $B_{n,w_0}$ and $U_{n,w_0}$ by $B_n^-$ and $U_n^-$,
correspondingly.

\vspace{0.05in}

Let $\overline{U}_n=\bpol_n/\bpol_n\cdot \sym_n^{\prime},$ where
$\sym_n^{\prime}$ is the codimension one ideal of $\sym_n$ consisting of symmetric functions
with zero constant term.
Isomorphisms (\ref{eq-nat-map}) and (\ref{eq-nat-map-s}),
upon modding out by $\sym_n^{\prime}$, induce natural isomorphisms of graded $\Bbbk$-vector
spaces
\begin{equation} \label{U-isoms}
U_n \cong \overline{U}_n, \quad \quad U_{n,s}\cong \overline{U}_n.
\end{equation}

Any $\sym_n$-linear endomorphism of $\bpol_n$ preserves the subspace
$\bpol_n\cdot \sym_n^{\prime}$. Consequently, an element of $\NH_n$ induces
an endomorphism of $\overline{U}_n$, and we get a surjective homomorphism
\begin{equation}\label{eq-surj-hom}
\NH_n = \mathrm{END}_{\sym_n}(\bpol_n) \lra \mathrm{END}_{\Bbbk}(
 \bpol_n/(\bpol_n\cdot \sym_n^{\prime})) =  \mathrm{END}_{\Bbbk}(\overline{U}_n)
\cong \NH_n/\NH_n\cdot \sym_n^{\prime}.
\end{equation}
Denote this homomorphism
\begin{equation}\label{eq-hom-pin}
\pi_n \ : \  \NH_n \lra \mathrm{END}_{\Bbbk}(\overline{U}_n).
\end{equation}
The quotient map $\bpol_n \lra \overline{U}_n$ admits many sections
$\overline{U}_n \lra \bpol_n.$
Any such section determines an injective homomorphism
$$ \jmath_n' \ : \ \mathrm{END}_{\Bbbk}(\overline{U}_n) \lra \NH_n $$
such that $\pi_n \circ \jmath_n'=\Id.$ We now choose a particular section.
The quotient map takes $U_n \subset \bpol_n$ bijectively onto $\overline{U}_n$, which
we use to identify $U_n$ and $\overline{U}_n$, obtaining an injective homomorphism
\begin{equation} \label{eq-jmathn}
\jmath_n \ : \ \mathrm{END}_{\Bbbk}(\overline{U}_n) \lra \NH_n
\end{equation}
as the composition
\begin{eqnarray*}
& & \mathrm{END}_{\Bbbk}(\overline{U}_n) \cong \mathrm{END}_{\Bbbk}(U_n) \hookrightarrow
\mathrm{END}_{\Bbbk}(U_n) \otimes \mathrm{END}_{\sym_n}(\sym_n)
\cong   \\
& & \mathrm{END}_{\sym_n}(U_n\otimes \sym_n) \cong \mathrm{END}_{\sym_n}(\bpol_n)=\NH_n.
\end{eqnarray*}
We have $\pi_n \circ \jmath_n = \Id$.
Likewise, for any permutation $s\in S_n$, the quotient map  $\bpol_n \lra \overline{U}_n$
induces an isomorphism between vector spaces $U_{n,s}$ and $ \overline{U}_n$, leading
to a section $\jmath_{n,s}:  \mathrm{END}_{\Bbbk}(\overline{U}_n) \lra \NH_n$ with
$\pi_n \circ \jmath_{n,s} = \Id$.

Let
\begin{equation}\label{eq-def-polcheck}
\bpol^{\vee}_n= \HOM_{\sym_n}(\bpol_n,\sym_n)
\end{equation}
 be the dual of $\bpol_n$ viewed as
a (graded) $\sym_n$-module. $\bpol^{\vee}_n$ is naturally a right graded $\NH_n$-module
and a graded $\sym_n$-module of
graded dimension $[n]!$. Isomorphism (\ref{eq-nat-map}) induces an isomorphism
\begin{equation}\label{eq-pncheck-iso}
 \bpol_n^{\vee} \cong U_n^{\ast} \otimes_{\Bbbk} \sym_n
\end{equation}
of $\sym_n$-modules. There is a canonical algebra isomorphism
\begin{equation} \label{eq-nhn-via-polndual}
\phi' \ : \ \bpol_n \otimes_{\sym_n} \bpol_n^{\vee}\stackrel{\cong}{\lra}\NH_n.
\end{equation}
coming from the definition of $\NH_n$ as the algebra of all $\sym_n$-linear endomorphisms
of $\bpol_n$.

Let $\delta(n) = \delta_{w_0}$ be the element of $\NH_n$
corresponding to the maximal length permutation in $S_n$.
\begin{equation} \label{eqn-def-delta-n}
\delta(n):= \underbrace{
\begin{DGCpicture}
\DGCstrand(0,0)(3,3) \DGCstrand(1,0)(0,1)(2,3)
\DGCstrand(2,0)(0,2)(1,3) \DGCstrand(3,0)(0,3)
\end{DGCpicture}}_{n~\mathrm{strands}}
=\begin{DGCpicture} \DGCstrand(0,0)(0,3) \DGCstrand(1,0)(1,3)
\DGCstrand(2,0)(2,3) \DGCstrand(3,0)(3,3)
\DGCcoupon(-0.25,1)(3.25,2){$\delta(n)$}
\end{DGCpicture}
\end{equation}
The element $\delta(n)$ spans the one-dimensional subspace of $\NH_n$ of lowest
degree (degree $n(1-n)$). The operator $\delta(n)$ on $\bpol_n$ takes any polynomial $f$
to a symmetric polynomial.

The symmetric $\sym_n$-bilinear pairing
\begin{equation}\label{eq-nat-pairing}
(-,-) \ : \ \bpol_n \otimes_{\sym_n} \bpol_n \lra \sym_n
\end{equation}
taking $f\otimes g$ to $\delta(n)(fg)$ is nondegenerate~\cite[Section 2.5]{Man}
and  has the invariance property
$$ (y f , g) = (f , \psi(y) g), \ \  y\in \NH_n, \ f,g\in \pol_n.$$
It induces an isomorphism
\begin{equation} \label{eq-iso-pol-check}
\phi^{\circ} \ : \  \bpol_n \stackrel{\cong}{\lra} \bpol_n^{\vee}
\end{equation}
of free graded $\pol_n$-modules taking $1\in\bpol_n$ to the $\sym_n$-linear map
$1^\vee: \bpol_n \lra \sym_n$ which sends $g\in \bpol_n$ to $\delta(n)(g)$.
In general, under $\phi^{\circ}$, $f\in \bpol_n$ goes to the $\sym_n$-linear map
$$f1^\vee : \bpol_n \lra \sym_n, \ \ \ \ g \mapsto f1^\vee(g):=\delta(n)(fg).$$
The map $\phi^{\circ}$ can also be viewed as an isomorphism of graded right $\NH_n$-modules,
where we turn $\bpol_n$ from a left to a right $\NH_n$-module via $\psi$.

Let $\phi_n = \phi' \circ (\Id\otimes \phi^{\circ})$ be the composition  map
\begin{equation} \label{eqn-comp-map}
\bpol_n \otimes_{\sym_n} \bpol_n  \xrightarrow{\Id\otimes \phi^{\circ}}
\bpol_n \otimes_{\sym_n}  \bpol_n^{\vee} \stackrel{\phi'}{\lra}  \NH_n.
\end{equation}
This map
\begin{equation} \label{eqn-multofpol}
\phi_n \ : \ \bpol_n \otimes_{\sym_n} \bpol_n  \lra \NH_n
\end{equation}
 is an isomorphism of graded $\NH_n$-bimodules which takes $f\otimes g$ to
$f \delta(n) g$.

Let $\epsilon_n:=(-1)^{\frac{n(n-1)}{2}}\delta(n)x_2x_3^2\cdots x_n^{n-1}\in \NH_n$.
It is not hard to check that
$$\delta(n)(x_2x_3^2\cdots x_n^{n-1})=(-1)^{\frac{n(n-1)}{2}},$$
so that $\epsilon_n(f)=f$ for any $f\in \sym_n$. Moreover, the image of $\epsilon_n$ acting on $\bpol_n$
is exactly $\sym_n$. This implies that $\epsilon_n$ is a primitive idempotent in
$\NH_n$. In our graphical notation, it is depicted as follows.
\[
\epsilon_n=(-1)^{\frac{n(n-1)}{2}}
\begin{DGCpicture}
\DGCstrand(0,0)(3,3)
\DGCstrand(1,0)(0,1)(2,3)\DGCdot{0.25}[r]{$^1$}
\DGCstrand(2,0)(0,2)(1,3)\DGCdot{0.25}[r]{$^2$}
\DGCstrand(3,0)(0,3)\DGCdot{0.25}[r]{$^{n-1}$}
\DGCdot.{0.15}[l]{$\cdots$}
\end{DGCpicture}
\]
The element $\epsilon_n$ is homogeneous of degree $0$, and $\delta_t \epsilon_n=0$ for any $1\le t \le n-1$.
Up to a degree shift,
the graded projective left $\NH_n$-module $\NH_n\epsilon_n$ is naturally isomorphic
to the polynomial module $\bpol_n$,
\begin{equation} \label{eq-two-left-mod-iso}
 \bpol_n \cong \NH_n\epsilon_n\{\textstyle{\frac{ n(1-n)}{2}}\}.
\end{equation}
The isomorphism, unique up to a nonzero constant, takes the generator $1$ of $\bpol_n$ to $\epsilon_n$.
The modules in (\ref{eq-two-left-mod-iso}) are also isomorphic to the module induced from
the one-dimensional representation placed in degree $n(1-n)/2$ of the nilCoxeter subalgebra
$\mathrm{NC}_n\subset\NH_n$.

Consider the idempotent
\begin{equation}\label{eq-epsi-ast}
\epsilon_n^{\ast}=\psi\sigma({\epsilon_n})= x_1^{n-1}x_2^{n-2}\cdots x_{n-1} \delta(n).
\end{equation}
Diagrammatically,
\[
\epsilon^\ast_n:=\psi\sigma({\epsilon_n})=
\begin{DGCpicture}
\DGCstrand(0,0)(3,3)
\DGCstrand(1,0)(0,1)(2,3)\DGCdot{2.85}[ul]{$\cdots^{1}$}
\DGCstrand(2,0)(0,2)(1,3)\DGCdot{2.85}[ul]{$^{n-2}$}
\DGCstrand(3,0)(0,3)\DGCdot{2.85}[ul]{$^{n-1}$}
\end{DGCpicture}.
\]
The projective right $\NH_n$-module $\epsilon_n^\ast\NH_n$ generated by
this idempotent is isomorphic, up to grading shift, to $\bpol_n^{\vee}$,
\begin{equation}\label{eq-right-mod-iso}
\epsilon_n^\ast\NH_n \{ \textstyle{\frac{n(1-n)}{2}} \} \cong \bpol_n^{\vee},
\end{equation}
via the isomorphism which
takes $\epsilon_n^{\ast}$ to the functional $f\longmapsto \delta(n)(f),$ $f\in \bpol_n$.
The left regular representation of $\NH_n$ decomposes into $n!$ copies of the
polynomial representation $\bpol_n$ with grading shifts, while the right regular representation
decomposes into the same number of copies of $\bpol_n^\vee$.

We recall an explicit basis of $\NH_n$ over its center $\sym_n$
as given in \cite[Section 2.5]{KLMS}. Denote the set
of sequences $\{\alpha=(\alpha_1,\dots, \alpha_{n-1})|~0\leq \alpha_t \leq
t,~t=1,\cdots, n-1\}$ by $\mathrm{Sq}(n)$. For any
$\alpha\in \mathrm{Sq}(n)$, set $|\alpha|=\sum_{t=1}^{n-1}\alpha_t$ and
$\hat{\alpha}=(\hat{\alpha}_1,\hat{\alpha}_2, \dots, \hat{\alpha}_{n-1}
):=(1-\alpha_1,2-\alpha_2,\dots,n-1-\alpha_{n-1})$. Define for
each $\alpha \in \mathrm{Sq}(n)$,
$$e_{\alpha}:=e_{\alpha_1}^{(1)}e_{\alpha_2}^{(2)}\cdots
e_{\alpha_{n-1}}^{(n-1)},\ \ \ \
x^{\hat{\alpha}}:=x_2^{\hat{\alpha}_1}x_3^{\hat{\alpha}_2}\cdots
x_{n}^{\hat{\alpha}_{n-1}}.$$
See the beginning of Section 3.1 for the notation of symmetric functions
adopted here.

\begin{prop}\label{prop-nil-hecke-basis-over-center} $\NH_n$
is isomorphic to the $n!\times n!$ matrix algebra $\mathrm{M}(n!,\sym_n)$ over its center
$\sym_n$. A particular homogeneous matrix basis of $\NH_n$ can be given by
$$\{E_{\alpha,\beta}=(-1)^{|\hat{\beta}|}
e_{\alpha}\delta(n)x^{\hat{\beta}}
~|~\alpha,~\beta \in \mathrm{Seq}(n)\}.$$
\end{prop}
\begin{proof} This is true because $\{e_{\alpha}\}$, $\{x^{\hat{\alpha}}\}$ for $\alpha \in \mathrm{Sq}(n)$ form dual bases under the $\sym_n$- bilinear pairing $(-,-)$. See \cite[Proposition 2.16]{KLMS}.
\end{proof}

This basis will be depicted diagrammatically as follows.
Abbreviating the polynomials $e_{\alpha}$,
$x^{\hat{\beta}}$ by coupons labeled by the same symbols,
the basis elements look like
\[
E_{\alpha,\beta}
=(-1)^{|\hat{\beta} |}~
\begin{DGCpicture}
\DGCstrand(0,-0.5)(0,0)(3,3)(3,3.5) \DGCstrand(1,-0.5)(1,0)(0,1)(2,3)(2,3.5)
\DGCstrand(2,-0.5)(2,0)(0,2)(1,3)(1,3.5) \DGCstrand(3,-0.5)(3,0)(0,3)(0,3.5)
\DGCcoupon(-0.25, -0.35)(3.25, 0.15){$x^{\hat{\beta}}$}
\DGCcoupon(-0.25, 2.85)(3.25, 3.35){$e_{\alpha}$}
\end{DGCpicture}=
(-1)^{|\hat{\beta}|}~
\begin{DGCpicture} \DGCstrand(0,0)(0,3) \DGCstrand(1,0)(1,3)
\DGCstrand(2,0)(2,3) \DGCstrand(3,0)(3,3)
\DGCcoupon(-0.25,0.15)(3.25,0.75){$x^{\hat{\beta}}$}
\DGCcoupon(-0.25,1)(3.25,2){$\delta(n)$}
\DGCcoupon(-0.25,2.25)(3.25,2.85){$e_{\alpha}$}
\end{DGCpicture}\ .
\]

\begin{eg}\label{eg-NH2-basis}
$\NH_2$ provides the simplest non-trivial example for the above proposition,
with basis elements given by the four diagrams below.
$$
\left(\begin{array}{ccc}
\begin{DGCpicture}
\DGCstrand(0,0)(1,1)
\DGCstrand(1,0)(0,1)\DGCdot{0.75}
\end{DGCpicture} &&
-\begin{DGCpicture}
\DGCstrand(0,0)(1,1)
\DGCstrand(1,0)(0,1)\DGCdot{0.25}\DGCdot{0.75}
\end{DGCpicture}\\
 && \\
\begin{DGCpicture}
\DGCstrand(0,0)(1,1)
\DGCstrand(1,0)(0,1)
\end{DGCpicture} &&
-
\begin{DGCpicture}
\DGCstrand(0,0)(1,1)
\DGCstrand(1,0)(0,1)\DGCdot{0.25}
\end{DGCpicture}
\end{array}\right) \ .
$$
\end{eg}

%%%%%%%%%%%%%%%%%%%%%%%%%%%%%%%%%%%

\paragraph{Local $p$-differentials on nilHecke algebras.}
We now look for $p$-differentials on nilHecke
algebras compatible with the differential $\dif$ on the subalgebra $\pol_n$ and local,
in the sense of being compatible with inclusions $\NH_n \subset \NH_{k_1+n+k_2}$, where
we add $k_1$, respectively $k_2$, vertical strands on the sides of a diagram in $\NH_n$.
The nilHecke algebra is the $\sym_n$-endomorphism algebra of the module $\bpol_n$, and $\bpol_n$
has a family of $p$-differentials $\dif_{\alpha}$ parametrized by
$g_{\alpha}=\sum \alpha_i x_i$, $\alpha_i \in \mathbb{F}_p$, see (\ref{eq-dif-g}).
The corresponding $p$-DG $\pol_n$-module was denoted $\bpol_n(\alpha)$, where
$\alpha=(\alpha_1, \dots, \alpha_n)$.

Before we move on to study $(\NH_n,\dif_a)$ as a $p$-DG algebra, we investigate
how the differential structure of the polynomial module changes under the
duality $\bpol_n(\alpha)^{\vee}$. Recall that $\dif_\alpha$ on $\bpol_n(\alpha)$
induces on $\bpol_n(\alpha)^\vee$ the differential $\dif_\alpha^\vee$ such that,
for any $z \in \bpol_n(\alpha)^\vee$ and $f\in \bpol_n(\alpha)$,
\begin{equation}\label{eqn-dif-on-dual-mod-elements}
\dif_\alpha^\vee(z)(f)=\dif(z(f))-z(\dif_\alpha(f)),
\end{equation}
where the undecorated operator $\dif$ is the natural differential on the ring
of symmetric polynomials (Lemma \ref{lemma-dif-on-elemetary-symmetric-functions}).
In this way $\bpol_n(\alpha)^\vee$ is naturally a $p$-DG $(\sym_n,\pol_n)$-bimodule
since $\bpol_n(\alpha)$ is a $p$-DG bimodule over $(\pol_n,\sym_n)$. Recall from
equation (\ref{eq-iso-pol-check}) that, under $\phi^{\circ}$, the generator
$1_\alpha\in \bpol_n(\alpha)$ is sent to the $\sym_n$-linear functional
$1_\alpha^{\vee}$ which acts on any element $f$ of $\bpol_n(\alpha)$ as
$$1_\alpha^\vee(f)=\phi^{\circ}(1_\alpha)(f)=\delta(n)(f).$$
In what follows, for any $\alpha=(\alpha_1, \dots, \alpha_n)\in \Bbbk^n$, we let
$\alpha^\vee:=(\alpha_1^\vee, \dots, \alpha_n^\vee)\in \Bbbk^n$ and
$\alpha_t^\vee:=1-n-\alpha_t$ for any $1\leq t \leq n$.

\begin{lemma}\label{lemma-stucture-of-dual-poly-mod}The differential
$\dif_\alpha^\vee$ acts on the generator $1_\alpha^\vee$ by
\begin{equation}\label{eqn-dif-on-dual-mod-generator}
\dif_\alpha^\vee(1_\alpha^\vee)=\sum_{t=1}^n\alpha_t^\vee x_t 1_\alpha^\vee.
\end{equation}
Equivalently, the dual polynomial module $\bpol_n(\alpha)^{\vee}$ with the
induced differential $\dif_\alpha^\vee$ is isomorphic to
$\bpol_n(\alpha^\vee)$ with the differential $\dif_{\alpha^\vee}$ under the
homomorphism $\phi^\circ$.
\end{lemma}
\begin{proof}
Since both sides of equation (\ref{eqn-dif-on-dual-mod-generator}) are $\sym_n$-linear, it suffices to check this formula on a set of $\sym_n$-basis elements for $\bpol_n(\alpha)$, for instance, the set $B_n$ from (\ref{eq-set-Bn}).

The left hand side of equation (\ref{eqn-dif-on-dual-mod-generator}), applied to any basis element $x^\beta:=x_1^{b_1}\cdots x_n^{b_n}$ where $0\leq b_t\leq n-t$, gives via equation (\ref{eqn-dif-on-dual-mod-elements})
%\begin{eqnarray*}
\[
\dif_\alpha^\vee(1_\alpha^\vee)(x^\beta) =  \dif(\delta(n)(x^\beta))-\delta(n)(\dif_{\alpha}(x^\beta))
 =  -\sum_{t=1}^n \delta(n)((b_t+\alpha_t)x_1^{b_1}\cdots x_t^{b_t+1}\cdots x_n^{b_n}).
\]
%\end{eqnarray*}
On the other hand,
%\begin{eqnarray*}
\[
(\sum_{t=1}^n\alpha_t^\vee x_t 1_\alpha^\vee)(x^\beta)  =  \sum_{t=1}^n \alpha_t^\vee \delta(n)(x_t\cdot x_\beta)
=  \sum_{t=1}^n (1-n-\alpha_t)\delta(n)(x_1^{b_1}\cdots x_t^{b_t+1}\cdots x_n^{b_n}).\]
%\end{eqnarray*}
Observe that $\delta(n)(x_1^{b_1}\cdots x_t^{b_t+1}\cdots x_n^{b_n})$ can be non-zero only in the following two cases:
\begin{itemize}
\item[(i)] Each $b_l=n-l$ and $t=1$;
\item[(ii)] Exactly one of, say $b_l=n-l-1$ ($1\leq l \leq n-1$), while the other $b_{l^\prime}=n-l^\prime$ for $l^\prime\neq l$, and either $t=l$ or $t=l+1$.
\end{itemize}
In case (i) the right hand sides of both equations are equal to $(1-n-\alpha_1)e_1$,
while in case (ii) both sides are equal to $\alpha_{l+1}-\alpha_{l}$. The lemma follows.
\end{proof}

This $p$-DG structure induces a differential on
$\NH_n = \mathrm{END}_{\sym_n}(\bpol_n(\alpha))$, via
\[
\dif_\alpha(\xi)(f):=\dif(\xi(f))-\xi(\dif(f)),
\]
for $\xi\in \mathrm{END}_{\sym_n}(\bpol_n(\alpha))$ and $f\in \bpol_n(\alpha)$.

Let us compute this differential on generators of $\NH_n$. On the subalgebra $\pol_n$
the differential will restrict to the original differential $\dif$, with $\dif(x_i) =x_i^2$.

Due to the local nature of generators $\delta_i$, to compute $\dif_\alpha(\delta_i)$
we can reduce to $n=2$ case. Then $g= \alpha_1 x_1 + \alpha_2 x_2$, and a short computation
yields
\begin{equation} \label{eq-dif-on-delta1}
\dif_\alpha(\delta_1) = a_1  - (a_1+1) x_1
\delta_1 + (a_1 -1) x_2 \delta_1,
\end{equation}
where $a_1 = \alpha_2 - \alpha_1$. In general,
\begin{equation} \label{eq-dif-on-deltai}
\dif_\alpha(\delta_i) = a_i \Id - (a_i+1) x_i
\delta_i + (a_i -1) x_{i+1} \delta_i,
\end{equation}
where $a_i = \alpha_{i+1}-\alpha_i$.
We would like $\dif_\alpha$ to be local, in the sense that the coefficients for its action
on generators should not depend on $i$. This is equivalent to  $a_1=a_2= \dots = a_{n-1}$.
Let $a=a_1$. Then $\alpha_2=\alpha_1+a$ and, in general, $\alpha_{k+1}=\alpha_1 + k a$ and
\begin{equation*}
g=g(\alpha_1,a)= \alpha_1 x_1 + (\alpha_1+a)x_2 + \dots + (\alpha_1+(n-1)a)x_n.
\end{equation*}
We denote the corresponding $p$-DG module structure on $\bpol_n$ by $\bpol_n(\alpha_1,a)$,
previously denoted $\bpol_n(\alpha)$ for
$\alpha=(\alpha_1, \alpha_1+a,\dots, \alpha_1+(n-1)a)$.
The induced differential on $\NH_n$ depends only on $a\in \mathbb{F}_p$, not on
$\alpha_1$, and will be denoted $\dif_a$.
It is given on generators by
\begin{eqnarray*}
 \dif_a(x_i) & = & x_i^2, \\
 \dif_a(\delta_i) & = &  a  - (a+1) x_i
\delta_i + (a -1) x_{i+1} \delta_i.
\end{eqnarray*}

Diagrammatically,
\begin{equation}\label{eqn-dif-dot}
\dif_a~\left(
\begin{DGCpicture}
\DGCstrand(1,0)(1,1) \DGCdot{0.5}
\end{DGCpicture} \right)
~ = ~
\begin{DGCpicture}
\DGCstrand(1,0)(1,1) \DGCdot{0.35}\DGCdot{0.65}
\end{DGCpicture}
~ = ~
\begin{DGCpicture}
\DGCstrand(1,0)(1,1) \DGCdot{0.5}[r]{$^2$}
\end{DGCpicture}
\end{equation}

\begin{equation}\label{eqn-dif-crossing}
\dif_a\left(~
\begin{DGCpicture}
\DGCstrand(1,0)(0,1) \DGCstrand(0,0)(1,1)
\end{DGCpicture}~\right)
= a~
\begin{DGCpicture}
\DGCstrand(0,0)(0,1) \DGCstrand(1,0)(1,1)
\end{DGCpicture}
-(a+1)
\begin{DGCpicture}
\DGCstrand(1,0)(0,1)\DGCdot{0.75} \DGCstrand(0,0)(1,1)
\end{DGCpicture}
+(a-1)
\begin{DGCpicture}
\DGCstrand(1,0)(0,1) \DGCstrand(0,0)(1,1)\DGCdot{0.75}
\end{DGCpicture}.
\end{equation}

\vspace{0.1in}

Thus, the $p$-DG $\pol_n$-module $\bpol_n(\alpha_1, a)$ for $\alpha_1,a\in \mathbb{F}_p$
induces a local differential $\dif_a$ on $\NH_n$, turning it into a $p$-DG algebra.
The differential extends to the entire $\NH_n$ via the Leibniz rule
$\dif(xy) = \dif(x)y + x\dif(y)$. Clearly, $\dif_a^p=0$, since this is true
for the differential $\dif_\alpha$ on  $\bpol_n(\alpha_1, a)$.

The differential $\dif_a$ takes a ``one-strand'' generator
$x_i$ to a one-strand diagram $x_i^2$ and a ``two-strand'' generator $\delta_i$
to a linear combination of two-strand diagrams. It
commutes with the obvious inclusions $\NH_n \subset \NH_{k_1+n+k_2}$
given by adding $k_1$ vertical lines to the left and $k_2$ vertical
lines to the right of a diagram in $\NH_n$. We say that $\dif_a$ is
a \emph{local} differential on the family of algebras $\NH_n$. Notice that the defining
equations (\ref{eqn-dif-dot}), (\ref{eqn-dif-crossing}) make sense for any $a\in \Bbbk$.

\begin{lemma}\label{lemma-dif-p-equals-zero-nilhecke}
Under any parameter $a\in \Bbbk$, the equation $\dif_a^p=0$ holds on $\NH_n$ if and only if $a\in \F_p$.
\end{lemma}
\begin{proof} The ``if'' part follows at once from the corresponding property ($\dif_g^p=0$)
of  the differential $\dif_{g(\alpha_1,a)}$ on $\bpol_n$.
Clearly, $\dif_a^p(x_i)=0$ for any $a\in \Bbbk$.
It suffices to show that
$\dif_a^p(\delta_1)=0$ precisely when $a \in \F_p$.

By a direct computation we have
$$\begin{array}{rcl}
\dif_a^2(\delta_1)& = &
\dif_a(a-(a+1)x_1\delta_1+(a-1)x_2\delta_1)\\
& = &
(a+1)ax_1^2\delta_1-2(a+1)(a-1)x_1x_2\delta_1+(a-1)ax_2^2\delta_1\\
& & -(a+1)ax_1+a(a-1)x_2.
\end{array}$$
Thus when $p=2$, $\dif_a^2(\delta_1)=0$ if and only if $a\in \F_2$.

Applying $\dif_a$ once more, we obtain
$$\dif_a^3(\delta_1)=(a+1)a(a-1)\left((x_1-x_2)^3\delta_1+(x_1-x_2)^2\right).$$
An induction shows that for any $k\geq 3$,
$$\begin{array}{rcl}
\dif_a^k(\delta_1)& = &(a+1)a(a-1)(x_2-x_1)^3\\
& & \cdot \left(\sum_{i=0}^{k-3}(-1)^i{k-3 \choose i}
\prod_{j=0}^{k-4-i}(a+2+j)\prod_{l=0}^{i-1}(a-2-l)x_1^ix_2^{k-3-i}\delta_1\right)\\
& & +(a+1)a(a-1)(x_2-x_1)^2\\
& & \cdot \left(\sum_{i=0}^{k-3}(-1)^i{k-3 \choose i}
\prod_{j=0}^{k-4-i}(a+2+j)\prod_{l=0}^{i-1}(a-2-l)x_1^ix_2^{k-3-i}\right).
\end{array}$$
The coefficient in front of each term, for a fixed $i$,
is ${k-3 \choose i}$ times the product of sums of $a$ with
some consecutive integers:
$$(a+1)a(a-1)\prod_{j=0}^{k-4-i}(a+2+j)\prod_{l=0}^{i-1}(a-2-l)=
\prod_{m=-i-1}^{k-2-i}(a+m).$$
When $k=p\geq 3$, the $p$ consecutive residues range over all elements of
$\F_p$. It follows that
$$(a+1)a(a-1)\prod_{j=0}^{p-2-i}(a+2+j)\prod_{l=0}^{i-1}(a-2-l)=
a^p-a,$$
and thus
$$
\dif_a^p(\delta_1) =  (a^p-a)(x_2-x_1)^2\sum_{i=0}^{p-3}
\left((-1)^i{p-3 \choose i}x_1^ix_2^{p-3-i}\left((x_2-x_1)\delta_1+1\right)\right),
$$
which is zero if and only if $a\in \F_p$. The lemma follows.
\end{proof}

\begin{rmk}\label{rmk-p=2-relation-with-LOT}
When $p=2$, the differential
$\dif_1$ preserves the nilCoxeter subalgebra $\mathrm{NC}_n$ of $\NH_n$ and
acts on a crossing by
\[
\dif_1\left(~
\begin{DGCpicture}
\DGCstrand(0,0)(1,1)
\DGCstrand(1,0)(0,1)
\end{DGCpicture}
~\right)=
\begin{DGCpicture}
\DGCstrand(0,0)(0,1)
\DGCstrand(1,0)(1,1)
\end{DGCpicture}~.
\]
$(\mathrm{NC}_n,\dif_1)$ is a differential graded algebra, and
appears in knot Floer homology~\cite{LOT1,LOT2} and in categorification
of quantum superalgebras~\cite{Kh2}.
\end{rmk}

\begin{rmk}\label{rmk-degree-minus-two-differential}
Any degree -2 $p$-nilpotent local differential $\dif^\prime$ acting on $\NH_n$
has the form

$$
{\dif^\prime}~\left(
\begin{DGCpicture}
\DGCstrand(1,0)(1,1) \DGCdot{0.5}
\end{DGCpicture} \right)
~ = ~\mu~
\begin{DGCpicture}
\DGCstrand(1,0)(1,1)
\end{DGCpicture} \ ,
\ \ \ \ \
{\dif^\prime}\left(~
\begin{DGCpicture}
\DGCstrand(1,0)(0,1)
\DGCstrand(0,0)(1,1)
\end{DGCpicture}
~\right)
= 0,
$$
for some $\mu\in \Bbbk$.
If $\mu \neq 0$, ${\dif^\prime}(\frac{x_1}{\mu})=1$, so that $\NH_n$
is acyclic for any $n\geq 1$, and its Grothendieck group vanishes.
For this reason, these differentials do not appear
interesting and we will not consider them.
\end{rmk}

By Lemma \ref{lemma-dif-p-equals-zero-nilhecke}, $\dif_a$ is a local degree two $p$-nilpotent
differential on the family of algebras $\NH_n$. It is easy to check that any local
degree two $p$-nilpotent differential on the family $\NH_n$ is given by
$\lambda \dif_a$ for some $\lambda\in \Bbbk$ and $a\in \F_p$.
Moreover, $\dif_a$ and $\dif_{-a}$ are related to each other by the
symmetries $\psi$, $\sigma$ of the nilHecke algebra.

\begin{prop}\label{prop-symmetry-between-a-and-minus-a}The following equalities
hold:
\begin{equation}  \label{eq-d-invol}
\psi\dif_a\psi=\dif_{-a}, \ \ \ \ \sigma\dif_a\sigma=\dif_{-a}.
\end{equation}
\end{prop}
\begin{proof} The symmetry $\sigma$ of $\NH_n$ is induced by the involution on
the algebra $\pol_n$ (also denoted $\sigma$)
transposing $x_t$ and $x_{n-t}$, $1\le t \le n-1$. This involution takes
linear function $g(\alpha_1, a)$ to $g(\alpha_1+(n-1)a, -a)$, which implies
the second relation in (\ref{eq-d-invol}).

It suffices to check the first relation on generators
$x_t$ and $\delta_t$, since both sides are derivations. On $x_t$, we
have
$$\psi\dif_a\psi
\left(
\begin{DGCpicture}
\DGCstrand(0,0)(0,1)\DGCdot{0.5}
\end{DGCpicture}
\right)=\psi\dif_a
\left(
\begin{DGCpicture}
\DGCstrand(0,0)(0,1)\DGCdot{0.5}
\end{DGCpicture}
\right)=
\psi\left(
\begin{DGCpicture} \DGCstrand(0,0)(0,1)\DGCdot{0.5}[r]{$^2$}
\end{DGCpicture}\right)=
\begin{DGCpicture} \DGCstrand(0,0)(0,1)\DGCdot{0.5}[r]{$^2$}
\end{DGCpicture}~=
\dif_{-a}\left(
\begin{DGCpicture}
\DGCstrand(0,0)(0,1)\DGCdot{0.5}
\end{DGCpicture}
\right),
$$
while on $\delta_t$
\begin{eqnarray*}
\psi\dif_a\psi\left( \
\begin{DGCpicture}
\DGCstrand(0,0)(1,1)
\DGCstrand(1,0)(0,1)
\end{DGCpicture} \ \right) & \hspace{-0.1in} = & \hspace{-0.1in} \psi\dif_a
\left(~
\begin{DGCpicture}\DGCstrand(0,0)(1,1)
\DGCstrand(1,0)(0,1)
\end{DGCpicture}
~\right) =\psi \left(
a~\begin{DGCpicture}\DGCstrand(0,0)(0,1) \DGCstrand(1,0)(1,1)
\end{DGCpicture} -(a+1)
\begin{DGCpicture}\DGCstrand(0,0)(1,1)
\DGCstrand(1,0)(0,1)\DGCdot{0.75}
\end{DGCpicture}+(a-1)
\begin{DGCpicture}
\DGCstrand(0,0)(1,1)\DGCdot{0.75}
\DGCstrand(1,0)(0,1)
\end{DGCpicture}~\right)\\
& \hspace{-0.1in} = & \hspace{-0.1in}a~
\begin{DGCpicture}
\DGCstrand(0,0)(0,1)
\DGCstrand(1,0)(1,1)
\end{DGCpicture}-(a+1)
\begin{DGCpicture}
\DGCstrand(0,0)(1,1)\DGCdot{0.25}
\DGCstrand(1,0)(0,1)
\end{DGCpicture}+(a-1)
\begin{DGCpicture}
\DGCstrand(0,0)(1,1)
\DGCstrand(1,0)(0,1)\DGCdot{0.25}
\end{DGCpicture}\\
& \hspace{-0.1in} = & \hspace{-0.1in} a~
\begin{DGCpicture}
\DGCstrand(0,0)(0,1)
\DGCstrand(1,0)(1,1)
\end{DGCpicture}-(a+1)\left(~
\begin{DGCpicture}
\DGCstrand(0,0)(1,1)\DGCdot{0.75}
\DGCstrand(1,0)(0,1)
\end{DGCpicture}+
\begin{DGCpicture}
\DGCstrand(0,0)(0,1)
\DGCstrand(1,0)(1,1)
\end{DGCpicture}~\right)+(a-1)\left(~
\begin{DGCpicture}
\DGCstrand(0,0)(1,1)
\DGCstrand(1,0)(0,1)\DGCdot{0.75}
\end{DGCpicture}-
\begin{DGCpicture}
\DGCstrand(0,0)(0,1)
\DGCstrand(1,0)(1,1)
\end{DGCpicture}~\right)\\
& \hspace{-0.1in} = & \hspace{-0.1in}(-a)
\begin{DGCpicture}
\DGCstrand(0,0)(0,1)
\DGCstrand(1,0)(1,1)
\end{DGCpicture}-(-a+1)
\begin{DGCpicture}
\DGCstrand(0,0)(1,1)
\DGCstrand(1,0)(0,1)\DGCdot{0.75}
\end{DGCpicture}+(-a-1)
\begin{DGCpicture}
\DGCstrand(0,0)(1,1)\DGCdot{0.75}
\DGCstrand(1,0)(0,1)
\end{DGCpicture}=\dif_{-a}\left(~
\begin{DGCpicture}
\DGCstrand(0,0)(1,1)
\DGCstrand(1,0)(0,1)
\end{DGCpicture}~\right).
\end{eqnarray*}
The result follows.
\end{proof}

\begin{prop}\label{prop-NHn-compact-H-module} For any $n \in \N$ and $a\in \F_p$, the
graded $H$-module $\NH_n$ is a compact object of $H\udmod$.
\end{prop}
\begin{proof}
We must
show that $\NH_n$ is quasi-isomorphic to a finite dimensional $H$-module.
To do this, fix a reduced expression $s_{t_1}\cdots s_{t_r}$ for each element $w\in S_n$,
where $s_t=(t,t+1)$, and let $\delta_w:=\delta_{t_1}\cdots\delta_{t_r}$
(the element $\delta_w$ of $\NH_n$ depends only on $w$ and not on a reduced expression).
Equation
(\ref{eqn-dif-crossing}) and the nilHecke relations imply that
\begin{equation} \label{eq-da-on-deltaw}
\dif_a(\delta_w)=\sum_{t=1}^nb_{t}x_t\delta_w+\sum_{w^\prime}f_{w^\prime}
\delta_{w^\prime},
\end{equation}
where $b_k\in \F_p$, permutations
$w^\prime\in S_n$ appearing in the second summand have strictly
fewer crossings than $w$, and $f_{w^\prime}$ are polynomials in $x_i$'s
(both $b_k$ and $f_{w^\prime}$ depend on $w$). The nilHecke
algebra $\NH_n$ is a free left $\pol_n$-module with a basis of elements $\delta_w$, over all
$w\in S_n$.
Choose
a total ordering $\leq$ on $S_n$ which refines the partial order $w_1 < w_2$ whenever $w_1$ has fewer
crossings than $w_2$. Let $\tau_1, \tau_2, \dots, \tau_{n!}$ be the list of
permutations in this order. Consider the filtration
$$ 0 = N_0 \subset N_1 \subset N_2 \subset \dots \subset N_{n!} = \NH_n$$
where $N_j$ is the free left $\pol_n$-submodule of $\NH_n$ spanned by
$\delta_{\tau_1}, \delta_{\tau_2}, \dots , \delta_{\tau_j}$.

Equations (\ref{eq-da-on-deltaw}) show that the derivation
$\dif_a$ preserves this filtration, $\dif_a(N_j) \subset N_j$, and that
each subquotient $N_j / N_{j-1}$ is isomorphic to the p-DG $\pol_n$-module
$\bpol_n(\alpha)\{m\}$ for $\alpha=(\alpha_{1}, \dots , \alpha_n)\in \F_p^n$ and some $m\in \Z$.

From the earlier discussion of rank one $\pol_n$ $p$-DG modules we know that
$\bpol_n(\alpha)$ is quasi-isomorphic to a finite-dimensional $H$-module for any $\alpha$
(equivalently, the graded $H$-module $\bpol_n(\alpha)$ is isomorphic to a direct sum
of a finite-dimensional $H$-module and a free $H$-module).
Therefore, $\NH_n$ has the same property.
\end{proof}

\begin{lemma}\label{lemma-dif-0-on-delta-n}The differential $\dif_0$
acts on the element $\delta(n)\in \NH_n$ by
\begin{equation*}\label{eqn-dif0-on-delta}
\dif_0(\delta(n))=-(n-1)e_1^{(n)}\delta(n),
\end{equation*}
which is depicted diagrammatically by
$$\dif_0\left(
\begin{DGCpicture}
\DGCstrand(0,0)(0,2)
\DGCstrand(1,0)(1,2)
\DGCstrand(2,0)(2,2)
\DGCstrand(3,0)(3,2)
\DGCcoupon(-0.25,0.65)(3.25,1.35){$\delta(n)$}
\end{DGCpicture}
\right)=-(n-1)~
\begin{DGCpicture}
\DGCstrand(0,0)(0,2)
\DGCstrand(1,0)(1,2)
\DGCstrand(2,0)(2,2)
\DGCstrand(3,0)(3,2)
\DGCcoupon(-0.25,0.15)(3.25,0.85){$\delta(n)$}
\DGCcoupon(-0.25,1.15)(3.25,1.85){$e_1$}
\end{DGCpicture}.
$$
\end{lemma}
\begin{proof}Here we give a relatively detailed proof of the formula as the same
method can be used in similar situations later on.
The lemma is proven by induction. The $n=1$ case is clear.
Suppose we have shown the formula when $n\leq k-1$.
When $n=k$, we have
\[
\begin{array}{l}
\dif_0\left(\begin{DGCpicture}
\DGCstrand(0,0)(0,2)
\DGCstrand(0.5,0)(0.5,2)
\DGCstrand(1,0)(1,2)
\DGCstrand(1.5,0)(1.5,2)
\DGCstrand(2,0)(2,2)
\DGCcoupon(-0.15,0.5)(2.15,1.5){$\delta(k)$}
\end{DGCpicture}\right)  =  \dif_0\left(~
\begin{DGCpicture}
\DGCstrand(0,0)(2,0.95)(2,2)
\DGCstrand(0.5,0)(0,0.75)(0,2)
\DGCstrand(1,0)(0.5,0.75)(0.5,2)
\DGCstrand(1.5,0)(1,0.75)(1,2)
\DGCstrand(2,0)(1.5,0.75)(1.5,2)
\DGCcoupon(-0.15,1)(1.65,1.75){$\delta(k-1)$}
\end{DGCpicture}~\right)\\
\\
 =  -(k-2)~
\begin{DGCpicture}
\DGCstrand(0,0)(2,0.95)(2,2)
\DGCstrand(0.5,0)(0,0.75)(0,2)
\DGCstrand(1,0)(0.5,0.75)(0.5,2)
\DGCstrand(1.5,0)(1,0.75)(1,2)
\DGCstrand(2,0)(1.5,0.75)(1.5,2)
\DGCcoupon(-0.15,0.85)(1.65,1.35){$\delta(k-1)$}
\DGCcoupon(-0.15,1.45)(1.65,1.85){$e_{1}$}
\end{DGCpicture} -
\begin{DGCpicture}
\DGCstrand(0,0)(2,0.95)(2,2)
\DGCstrand(0.5,0)(0,0.75)(0,2)\DGCdot{0.75}
\DGCstrand(1,0)(0.5,0.75)(0.5,2)
\DGCstrand(1.5,0)(1,0.75)(1,2)
\DGCstrand(2,0)(1.5,0.75)(1.5,2)
\DGCcoupon(-0.15,1)(1.65,1.75){$\delta(k-1)$}
\end{DGCpicture}
-
\begin{DGCpicture}
\DGCstrand(0,0)(2,0.95)(2,2)\DGCdot{0.35}
\DGCstrand(0.5,0)(0,0.75)(0,2)
\DGCstrand(1,0)(0.5,0.75)(0.5,2)
\DGCstrand(1.5,0)(1,0.75)(1,2)
\DGCstrand(2,0)(1.5,0.75)(1.5,2)
\DGCcoupon(-0.15,1)(1.65,1.75){$\delta(k-1)$}
\end{DGCpicture}\\
\\
-~
\begin{DGCpicture}
\DGCstrand(0,0)(2,0.95)(2,2)
\DGCstrand(0.5,0)(0,0.75)(0,2)
\DGCstrand(1,0)(0.5,0.75)(0.5,2)\DGCdot{0.75}
\DGCstrand(1.5,0)(1,0.75)(1,2)
\DGCstrand(2,0)(1.5,0.75)(1.5,2)
\DGCcoupon(-0.15,1)(1.65,1.75){$\delta(k-1)$}
\end{DGCpicture}
-\begin{DGCpicture}
\DGCstrand(0,0)(2,0.95)(2,2)\DGCdot{0.45}
\DGCstrand(0.5,0)(0,0.75)(0,2)
\DGCstrand(1,0)(0.5,0.75)(0.5,2)
\DGCstrand(1.5,0)(1,0.75)(1,2)
\DGCstrand(2,0)(1.5,0.75)(1.5,2)
\DGCcoupon(-0.15,1)(1.65,1.75){$\delta(k-1)$}
\end{DGCpicture}-\cdots-
\begin{DGCpicture}
\DGCstrand(0,0)(2,0.95)(2,2)
\DGCstrand(0.5,0)(0,0.75)(0,2)
\DGCstrand(1,0)(0.5,0.75)(0.5,2)
\DGCstrand(1.5,0)(1,0.75)(1,2)
\DGCstrand(2,0)(1.5,0.75)(1.5,2)\DGCdot{0.75}
\DGCcoupon(-0.15,1)(1.65,1.75){$\delta(k-1)$}
\end{DGCpicture}-
\begin{DGCpicture}
\DGCstrand(0,0)(2,0.95)(2,2)\DGCdot{0.75}
\DGCstrand(0.5,0)(0,0.75)(0,2)
\DGCstrand(1,0)(0.5,0.75)(0.5,2)
\DGCstrand(1.5,0)(1,0.75)(1,2)
\DGCstrand(2,0)(1.5,0.75)(1.5,2)
\DGCcoupon(-0.15,1)(1.65,1.75){$\delta(k-1)$}
\end{DGCpicture}\\
\\
%\end{array}
%\]
%\[
%\begin{array}{l}
 = -(k-2)\begin{DGCpicture}
\DGCstrand(0,0)(2,0.95)(2,2)
\DGCstrand(0.5,0)(0,0.75)(0,2)
\DGCstrand(1,0)(0.5,0.75)(0.5,2)
\DGCstrand(1.5,0)(1,0.75)(1,2)
\DGCstrand(2,0)(1.5,0.75)(1.5,2)
\DGCcoupon(-0.15,0.85)(1.65,1.35){$\delta(k-1)$}
\DGCcoupon(-0.15,1.45)(1.65,1.85){$e_1$}
\end{DGCpicture} -\left(~
\begin{DGCpicture}
\DGCstrand(0,0)(2,0.95)(2,2)
\DGCstrand(0.5,0)(0,0.75)(0,2)\DGCdot{0.75}
\DGCstrand(1,0)(0.5,0.75)(0.5,2)
\DGCstrand(1.5,0)(1,0.75)(1,2)
\DGCstrand(2,0)(1.5,0.75)(1.5,2)
\DGCcoupon(-0.15,1)(1.65,1.75){$\delta(k-1)$}
\end{DGCpicture}+
\begin{DGCpicture}
\DGCstrand(0,0)(2,0.95)(2,2)
\DGCstrand(0.5,0)(0,0.75)(0,2)
\DGCstrand(1,0)(0.5,0.75)(0.5,2)\DGCdot{0.75}
\DGCstrand(1.5,0)(1,0.75)(1,2)
\DGCstrand(2,0)(1.5,0.75)(1.5,2)
\DGCcoupon(-0.15,1)(1.65,1.75){$\delta(k-1)$}
\end{DGCpicture}
+\cdots
\begin{DGCpicture}
\DGCstrand(0,0)(2,0.95)(2,2)
\DGCstrand(0.5,0)(0,0.75)(0,2)
\DGCstrand(1,0)(0.5,0.75)(0.5,2)
\DGCstrand(1.5,0)(1,0.75)(1,2)
\DGCstrand(2,0)(1.5,0.75)(1.5,2)\DGCdot{0.75}
\DGCcoupon(-0.15,1)(1.65,1.75){$\delta(k-1)$}
\end{DGCpicture}~\right)\\
\\
 -\left(~\begin{DGCpicture}
\DGCstrand(0,0)(2,0.95)(2,2)\DGCdot{0.35}
\DGCstrand(0.5,0)(0,0.75)(0,2)
\DGCstrand(1,0)(0.5,0.75)(0.5,2)
\DGCstrand(1.5,0)(1,0.75)(1,2)
\DGCstrand(2,0)(1.5,0.75)(1.5,2)
\DGCcoupon(-0.15,1)(1.65,1.75){$\delta(k-1)$}
\end{DGCpicture}
+
\begin{DGCpicture}
\DGCstrand(0,0)(2,0.95)(2,2)\DGCdot{0.45}
\DGCstrand(0.5,0)(0,0.75)(0,2)
\DGCstrand(1,0)(0.5,0.75)(0.5,2)
\DGCstrand(1.5,0)(1,0.75)(1,2)
\DGCstrand(2,0)(1.5,0.75)(1.5,2)
\DGCcoupon(-0.15,1)(1.65,1.75){$\delta(k-1)$}
\end{DGCpicture}+\cdots+
\begin{DGCpicture}
\DGCstrand(0,0)(2,0.95)(2,2)\DGCdot{0.75}
\DGCstrand(0.5,0)(0,0.75)(0,2)
\DGCstrand(1,0)(0.5,0.75)(0.5,2)
\DGCstrand(1.5,0)(1,0.75)(1,2)
\DGCstrand(2,0)(1.5,0.75)(1.5,2)
\DGCcoupon(-0.15,1)(1.65,1.75){$\delta(k-1)$}
\end{DGCpicture}
~\right)\\
\\
\end{array}
\]
\[
\begin{array}{l}
%\end{array}
%\]
%\[
%\begin{array}{l}
 =
-(k-2)\begin{DGCpicture}
\DGCstrand(0,0)(2,0.95)(2,2)
\DGCstrand(0.5,0)(0,0.75)(0,2)
\DGCstrand(1,0)(0.5,0.75)(0.5,2)
\DGCstrand(1.5,0)(1,0.75)(1,2)
\DGCstrand(2,0)(1.5,0.75)(1.5,2)
\DGCcoupon(-0.15,0.85)(1.65,1.35){$\delta(k-1)$}
\DGCcoupon(-0.15,1.45)(1.65,1.85){$e_1$}
\end{DGCpicture}-
\begin{DGCpicture}
\DGCstrand(0,0)(2,0.95)(2,2)
\DGCstrand(0.5,0)(0,0.75)(0,2)
\DGCstrand(1,0)(0.5,0.75)(0.5,2)
\DGCstrand(1.5,0)(1,0.75)(1,2)
\DGCstrand(2,0)(1.5,0.75)(1.5,2)
\DGCcoupon(-0.15,0.85)(1.65,1.25){$e_{1}$}
\DGCcoupon(-0.15,1.35)(1.65,1.85){$\delta(k-1)$}
\end{DGCpicture}-
(k-1)\begin{DGCpicture}
\DGCstrand(0,0)(2,0.95)(2,2)\DGCdot{1.75}
\DGCstrand(0.5,0)(0,0.75)(0,2)
\DGCstrand(1,0)(0.5,0.75)(0.5,2)
\DGCstrand(1.5,0)(1,0.75)(1,2)
\DGCstrand(2,0)(1.5,0.75)(1.5,2)
\DGCcoupon(-0.15,0.85)(1.65,1.35){$\delta(k-1)$}
\end{DGCpicture}\\
\\
 = - (k-1)\begin{DGCpicture}
\DGCstrand(0,0)(2,0.95)(2,2)
\DGCstrand(0.5,0)(0,0.75)(0,2)
\DGCstrand(1,0)(0.5,0.75)(0.5,2)
\DGCstrand(1.5,0)(1,0.75)(1,2)
\DGCstrand(2,0)(1.5,0.75)(1.5,2)
\DGCcoupon(-0.15,0.75)(1.65,1.25){$\delta(k-1)$}
\DGCcoupon(-0.15,1.35)(2.15,1.75){$e_1$}
\end{DGCpicture}=-(k-1)\begin{DGCpicture}
\DGCstrand(0,0)(0,2)
\DGCstrand(0.5,0)(0.5,2)
\DGCstrand(1,0)(1,2)
\DGCstrand(1.5,0)(1.5,2)
\DGCstrand(2,0)(2,2)
\DGCcoupon(-0.15,0.25)(2.15,1.25){$\delta(k)$}
\DGCcoupon(-0.15,1.35)(2.15,1.75){$e_1$}
\end{DGCpicture} \ ,
\end{array}
\]
where in the fourth equality, all dots on the last strand in the second
parenthesized expression can slide up
without obstruction essentially because $\delta_i^2=0$, while in the fifth
equality we used that elements of $\sym_{k-1}$ commute with $\delta(k-1)$.
\end{proof}

\begin{lemma}\label{lemma-dif-1-on-delta-n}
The differential $\dif_1$ acts on $\delta(n)\in \NH_n$ as follows:
\begin{equation*}
\dif_1(\delta(n)) = -\sum_{t=1}^n(n-t)x_t\delta(n)-\sum_{t=1}^n(t-1)\delta(n)x_t,
\end{equation*}
which has the diagrammatic presentation
\begin{eqnarray*}
\dif_1\left(
\begin{DGCpicture}
\DGCstrand(0,0)(0,2) \DGCstrand(0.5,0)(0.5,2) \DGCstrand(1,0)(1,2)
\DGCstrand(1.5,0)(1.5,2)
\DGCcoupon(-0.15,0.5)(1.65,1.5){$\delta(n)$}
\end{DGCpicture}
\right)& = & - (n-1) {\begin{DGCpicture}
\DGCstrand(0,0)(0,2)\DGCdot{1.75} \DGCstrand(0.5,0)(0.5,2)
\DGCstrand(1,0)(1,2) \DGCstrand(1.5,0)(1.5,2)
\DGCcoupon(-0.15,0.5)(1.65,1.5){$\delta(n)$}
\end{DGCpicture}}-(n-2)
{\begin{DGCpicture} \DGCstrand(0,0)(0,2)
\DGCstrand(0.5,0)(0.5,2)\DGCdot{1.75} \DGCstrand(1,0)(1,2)
\DGCstrand(1.5,0)(1.5,2)
\DGCcoupon(-0.15,0.5)(1.65,1.5){$\delta(n)$}
\end{DGCpicture}}- \cdots -
{\begin{DGCpicture} \DGCstrand(0,0)(0,2) \DGCstrand(0.5,0)(0.5,2)
\DGCstrand(1,0)(1,2)\DGCdot{1.75} \DGCstrand(1.5,0)(1.5,2)
\DGCcoupon(-0.15,0.5)(1.65,1.5){$\delta(n)$}
\end{DGCpicture}} \\
&&\\
&& - {\begin{DGCpicture} \DGCstrand(0,0)(0,2)
\DGCstrand(0.5,0)(0.5,2)\DGCdot{0.25} \DGCstrand(1,0)(1,2)
\DGCstrand(1.5,0)(1.5,2)
\DGCcoupon(-0.15,0.5)(1.65,1.5){$\delta(n)$}
\end{DGCpicture}}
-\cdots-(n-2) {\begin{DGCpicture} \DGCstrand(0,0)(0,2)
\DGCstrand(0.5,0)(0.5,2) \DGCstrand(1,0)(1,2)\DGCdot{0.25}
\DGCstrand(1.5,0)(1.5,2)
\DGCcoupon(-0.15,0.5)(1.65,1.5){$\delta(n)$}
\end{DGCpicture}}
-(n-1) {\begin{DGCpicture} \DGCstrand(0,0)(0,2)
\DGCstrand(0.5,0)(0.5,2) \DGCstrand(1,0)(1,2)
\DGCstrand(1.5,0)(1.5,2)\DGCdot{0.25}
\DGCcoupon(-0.15,0.5)(1.65,1.5){$\delta(n)$}
\end{DGCpicture}} \ .
\end{eqnarray*}
\end{lemma}
\begin{proof}The proof is by an induction argument analogous to the one we used
in the proof of the previous lemma. We leave it as an exercise.
\end{proof}

\begin{cor}\label{cor-dif-a-on-delta-n}The differential $\dif_a$ acts on $\delta(n)\in \NH_n$ by
$$ \dif_a(\delta(n)) = \left(\sum_{t=1}^n (t-1)ax_t\right) \delta(n) -
\delta(n)\left(\sum_{t=1}^n  \left((t-1)a+n-1\right) x_t\right). $$
\end{cor}
\begin{proof}The result follows from the previous two lemmas together
with the formula $\dif_a = \dif_0 + a (\dif_1-\dif_0)$.
\end{proof}

We summarize the main results of this subsection in the following proposition.

\begin{prop}\label{prop-NH-as-p-DG-bimod}Let $\alpha=(\alpha_1,\dots, \alpha_n)$ be an
$n$-tuple of numbers in $\F_p$ with the property that $\alpha_t-\alpha_{t-1}=a$ for all
$2\leq t \leq n$. Set $\alpha^\vee=(1-n-\alpha_1, \dots, 1-n-\alpha_n)$.
\begin{itemize}
\item[(i)] There is a $\dif$-invariant, non-degenerate $\sym_n$-bilinear form
\[
(-,-):\bpol_n(\alpha)\o \bpol_n(\alpha^\vee)\lra \sym_n, \ \ \ \ f \o g \mapsto \delta(n)(fg),
\]
which is also compatible with the left and right $p$-DG $(\NH_n,\dif_a)$-module
structures. Here being $\dif$-invariant means that
$\dif(f,g)=(\dif_\alpha(f),g)+(f,\dif_{\alpha^\vee}(g))$. The compatibility
condition states that $(\xi(f),g)=(f,\psi(\xi)(g))$ and the Leibniz rule
$$\dif(\xi(f),g)=(\dif_a(\xi)(f),g)+(\xi(\dif_{\alpha}(f)),g)+(\xi(f),\dif_{\alpha^\vee}(g))$$
holds for any $\xi \in \NH_n$.
\item[(ii)] There is an isomorphism $\phi^\circ$ of right $(\NH_n,\dif_a)$-modules
\[
\phi^\circ: (\bpol_n(\alpha^\vee),\dif_{\alpha^\vee} )\lra (\bpol_n(\alpha)^\vee,
\dif_\alpha^\vee),
\]
where $\phi^\circ$ takes the generator $1_{\alpha^\vee}\in \bpol_n(\alpha^\vee)$
to the $\sym_n$-linear map $1_\alpha^\vee: f \mapsto \delta(n)(f)$.
\item[(iii)] There is an isomorphism of graded $p$-DG $(\NH_n,\dif_a)$-bimodules
\[
\phi_n: \bpol_n(\alpha)\o_{\sym_n}\bpol_n(\alpha^\vee)\lra \NH_n, \ \ \ \ f \o g
\mapsto f \delta(n) g
\]
which equals the composition
\[
\bpol_n(\alpha) \otimes_{\sym_n} \bpol_n(\alpha^\vee)  \xrightarrow{\Id\otimes \phi^{\circ}}
\bpol_n(\alpha) \otimes_{\sym_n}  \bpol_n(\alpha)^{\vee} \stackrel{\phi'}{\lra}  \NH_n.
\]
\end{itemize}
\end{prop}
\begin{proof}Follows from Lemma \ref{lemma-stucture-of-dual-poly-mod}, Proposition
\ref{prop-symmetry-between-a-and-minus-a} and the explicit form of $\dif_a(\delta(n))$
in Corollary \ref{cor-dif-a-on-delta-n}.
\end{proof}

\subsection{Specializations of \texorpdfstring{$a$}{a}}
The differential $\dif_a$ turns $\NH_n$ into a $p$-DG algebra.

\begin{prop} \label{prop-NHp-contractible} The $p$-DG algebra $(\NH_n, \dif_a)$ is acyclic for
any $n\ge p$ and $a\in \F_p^{\ast}$.
\end{prop}
\begin{proof}It suffices to check the result for $n=p$. Indeed if $\NH_p$ is acyclic,
by Proposition \ref{prop-contractible-p-DG-algebra-condition}, there exists an element
$y_p$ such that $\dif_a(y_p)=1_{\NH_p}$.
Then the element $y_p\o 1_{\NH_{n-p}} \in \NH_p\o\NH_{n-p}\subset \NH_n$ satisfies
$\dif_a(y_p\o 1_{\NH_{n-p}})=1_{\NH_n}$, and $\NH_n$ is acyclic by Proposition
\ref{prop-contractible-p-DG-algebra-condition}.

When $n=p$, pick any $\alpha_1\in \F_p$. The residues $\alpha_1, \alpha_1 + a, \dots,
\alpha_1+(p-1)a$ run over all elements in $\F_p$. Let $s\in S_p$ be the permutation
such that $s(k)=p-\alpha_1 - (k-1)a$ modulo $p$ for $1\le k\le p$. Then
$ \dif_{\alpha}(x_1^{s(1)}x_2^{s(2)}\cdots x_p^{s(p)}) =0 $, where we view
the argument as an element of $\bpol_p(\alpha)$ for $\alpha=(\alpha_1, \alpha_1+a,
\dots, \alpha_1+(p-1)a)$. Consequently, the subspace $U_{p,s}$ is $\dif_{\alpha}$-stable
and map (\ref{eq-nat-map-s}) is an isomorphism of $(\sym_p, \dif_a)$-modules.

As an $H$-module,
$ U_{p,s} \cong V_0 \otimes V_1 \otimes \dots \otimes V_{p-1}.$
The $H$-module $V_{p-1}$ is free, and $U_{p,s}$ is a free module as well.
Hence, the $p$-DG algebra $\mathrm{END}_{\Bbbk}(U_{p,s})$ is acyclic, and
\begin{equation}\label{eq-NH-contr}
\NH_p \cong \mathrm{END}_{\sym_p}(\bpol_p(\alpha)) \cong \mathrm{END}_{\sym_p}(U_{p,s}\otimes \sym_p)
\cong \mathrm{END}_{\Bbbk}(U_{p,s}) \otimes \sym_p
\end{equation}
is acyclic as well.
\end{proof}

\begin{eg}\label{eg-NH3-acyclic-char-3}
When $\mathrm{char}(\Bbbk)=3$, and $a=1$ one computes that
\[
\dif_1^2 \left(~
\begin{DGCpicture}[scale=0.5]
\DGCstrand(0,0)(1,2)
\DGCstrand(1,0)(2,2)
\DGCstrand(2,0)(0,2)
\end{DGCpicture}
~\right)=
\begin{DGCpicture}[scale=0.5]
\DGCstrand(0,0)(0,2)
\DGCstrand(1,0)(1,2)
\DGCstrand(2,0)(2,2)
\end{DGCpicture}
\ .
\]
Conjugating  by $\sigma$,
\[
\dif_{-1}^2 \left(~
\begin{DGCpicture}[scale=0.5]
\DGCstrand(0,0)(2,2)
\DGCstrand(1,0)(0,2)
\DGCstrand(2,0)(1,2)
\end{DGCpicture}
~\right)=
\begin{DGCpicture}[scale=0.5]
\DGCstrand(0,0)(0,2)
\DGCstrand(1,0)(1,2)
\DGCstrand(2,0)(2,2)
\end{DGCpicture}
\ .
\]
This proves acyclicity of $(\NH_3,\dif_{\pm 1})$ via
Proposition~\ref{prop-contractible-p-DG-algebra-condition}.
\end{eg}

An argument similar to the one in Proposition~\ref{prop-NHp-contractible} shows that,
for $n=p-1$ and $a\in \F_p^{\ast}$, the subspace
$U_{p-1,s}$ is $\dif_{\alpha}$-stable for a unique pair of $\alpha_1$ and $s$.
As a $p$-DG algebra,
\begin{equation}\label{eq-NH-p-1}
\NH_{p-1} \cong \mathrm{END}_{\Bbbk}(U_{p-1,s}) \otimes \sym_{p-1}.
\end{equation}
Since $\sym_{p-1}$ is quasi-isomorphic to the ground field $\Bbbk$, the inclusion
$ \mathrm{END}_{\Bbbk}(U_{p-1,s}) \subset \NH_{p-1}$ is a quasi-isomorphism of $p$-DG algebras.
For $n<p-1$, in general, there are no permutations $s$ such that $U_{n,s}$ is
$\dif_{\alpha}$-stable unless $a=\pm 1$, singling out these values of $a$, as discussed below
after the proof of Proposition \ref{prop-a-equals-0-no-derived-category-vanishes}.

Proposition~\ref{prop-NHp-contractible} yields the following corollary.

\begin{cor}\label{cor-NHp-trivial-K0}
For any $a\in \mathbb{F}_p^*$ and $n\geq p$, $\mc{D}(\NH_n,\dif_a)\cong 0$, and
consequently $K_0(\NH_n, \dif_a)=0$.
\end{cor}
\begin{proof}Combine Proposition \ref{prop-NHp-contractible} with Proposition
\ref{prop-contractible-p-DG-algebra-condition}.
\end{proof}

When $a=0$, the derived category $\mc{D}(\NH_n,\dif_0)$ does not vanish for any $n\geq 0$.
This is the reason why we will disregard this case in what follows when
categorifying the small quantum group $u_{\mathbb{O}_{2p}}^+(\mathfrak{sl}_2)$.

\begin{prop}\label{prop-a-equals-0-no-derived-category-vanishes}
For any $n\in \N$, $\mc{D}(\NH_n,\dif_0)\ncong 0$.
\end{prop}
\begin{proof} By Proposition \ref{prop-contractible-p-DG-algebra-condition}
we need to show that $\NH_n$ is not acyclic. Acyclicity of $\NH_n$ implies acyclicity
of $\NH_m$ for any $m>n$.
If $n=kp+1$ where $k\in \N$, Lemma \ref{lemma-dif-0-on-delta-n} shows that
$\dif_0(\delta(n))=0$. Being of the lowest degree in $\NH_n$, the element $\delta(n)$
is not in the image of $\dif_0$. Hence $\Bbbk\delta(n)\cong V_0$ spans a non-trivial
summand in $\mH_{/0}(\NH_n)$, and $\NH_n$ is not acyclic for $n=kp+1$ and, therefore,
for any $n$.
\end{proof}

\paragraph{Why specialize.} The main point of this subsection is that, under the
specialization $a=\pm 1$, the $p$-DG algebras $\NH_n$ behave extremely well: they are
quasi-isomorphic to matrix algebras (see
Proposition~\ref{prop-nilhecke2-qis-2by2-matrix-algebra} and
Corollary~\ref{cor-NHn-qis-matrix-algebra}), and the column modules are compact cofibrant
(Proposition~\ref{prop-NHn-column-module-compact}), which allows one to compute the
Grothendieck group $K_0(\mc{D}(\NH_n,\dif_{\pm 1}))$, see
Corollary~\ref{cor-K0-of-derived-categories-NHn}.
Another reason will become clear upon generalizing from $\mf{sl}_2$
to arbitrary simply-laced simple Lie algebras, see
Theorem~\ref{thm-QSR-holds-only-for-special-paramters}.

Recall from the beginning of Section \ref{subsec-p-der-on-NH} that the
$(\pol_2,\dif)$-module $(\bpol_2(\alpha),\dif_\alpha)$ has a generator $1_\alpha$
on which $\dif_\alpha$ acts by $\dif_\alpha(1_\alpha)=(\alpha_1x_1+\alpha_2x_2)1_\alpha$,
where $\alpha_2-\alpha_1=a$ is non-zero in $\F_p$. We would like to determine
for which values of $a$ does there exist a two-step filtration on $\bpol_2(\alpha)$ whose
subquotients are rank one $(\sym_2,\dif)$-modules. This amounts to asking an equivalent
question: when does $(\beta_1x_1+\beta_2x_2)1_{\alpha}$ generate a $(\sym_2,\dif)$-stable
submodule, where $\beta_1,\beta_2\in \Bbbk$ are constants that are not both
zero?

\begin{prop}\label{prop-sub-stab} The $\sym_2$-module
$\sym_2\cdot(\beta_1x_1+\beta_2x_2)1_{\alpha}$
is $\dif_\alpha$-stable only when $a=\pm 1$ (among $a$ in $\F_p^{\ast}$). Furthermore,
\begin{itemize}
\item if $a=1$, then $\beta_2=0$ and $\beta_1\neq 0$;
\item if $a=-1$, then $\beta_1=0$ and $\beta_2\neq 0$.
\end{itemize}
\end{prop}
\begin{proof}Exercise.
\end{proof}

We see that such a filtration on $\bpol_2(\alpha)$ exists for $a\in \F_p^{\ast}$
iff $a=\pm 1$.  If $a=0$, there is a unique submodule as above,
generated by $(x_1-x_2)1_{\alpha}$.

\vspace{0.06in}

Recall that $\NH_2$ is isomorphic to the matrix algebra $\mathrm{M}(2,{\sym_2})$
(see Proposition \ref{prop-nil-hecke-basis-over-center} and Example \ref{eg-NH2-basis}).
Such an identification is not unique. Indeed, as a $\sym_2$-module,
\begin{equation}\label{eqn-decomp-pol2-over-sym2}
\bpol_2=\Bbbk[x_1, x_2]\cong \sym_2\{2\} \oplus \sym_2 \cong \sym_2v_1(b)\oplus \sym_2v_2,
\end{equation}
where the second summand is canonically generated by the degree zero element
$v_2=1\in \pol_2$,
while there is a 1-parameter family of choices for the degree two generator
\[
v_1(b)=x_1+b(x_1+x_2), \ \ \ b\in \Bbbk.
\]
Let us also set $v_2(b)=v_2$ for $b\in \Bbbk$.
Under the identification $\NH_2\cong \mathrm{End}_{\sym_2}(\bpol_2)$,
we obtain two primitive homogeneous idempotents of $\NH_2$ depending on the parameter $b$:
\[
\epsilon_{11}(b):=
\begin{DGCpicture}
\DGCstrand(0,0)(1,1)
\DGCstrand(1,0)(0,1)\DGCdot{0.75}
\end{DGCpicture}
+b
\left(
\begin{DGCpicture}
\DGCstrand(0,0)(1,1)
\DGCstrand(1,0)(0,1)\DGCdot{0.75}
\end{DGCpicture}
+
\begin{DGCpicture}
\DGCstrand(0,0)(1,1)\DGCdot{0.75}
\DGCstrand(1,0)(0,1)
\end{DGCpicture}
\right)\ , \ \ \
\epsilon_{22}(b):=-
\begin{DGCpicture}
\DGCstrand(0,0)(1,1)
\DGCstrand(1,0)(0,1)\DGCdot{0.25}
\end{DGCpicture}
-b
\left(
\begin{DGCpicture}
\DGCstrand(0,0)(1,1)
\DGCstrand(1,0)(0,1)\DGCdot{0.75}
\end{DGCpicture}
+
\begin{DGCpicture}
\DGCstrand(0,0)(1,1)\DGCdot{0.75}
\DGCstrand(1,0)(0,1)
\end{DGCpicture}
\right) \ ,
\]
such that $\epsilon_{ii}(b)v_j(b)=\delta_{ij}v_j(b)$, $i,j\in\{1,2\}$.
Thus, $\epsilon_{11}(b)$ is a
projection from $\pol_2$ onto a summand isomorphic to $\sym_2\{2\}$, while
$\epsilon_{22}(b)$ is the projection onto the unique summand isomorphic to $\sym_2$.
When $b=0$, we recover the idempotents of Example \ref{eg-NH2-basis}. The elements
$\epsilon_{11}(b), \epsilon_{22}(b)$ over $b\in \Bbbk$ are the only homogeneous minimal idempotents
in $\NH_2$.

We next check for what values of $a$ there is an idempotent
$\epsilon_{ii}(b)$ ($i=1,2$) which generates a $\dif_a$-stable left submodule of $\NH_2$.
This property is related to the categorification of the second divided power $E^{(2)}$, see
the next subsection.

\begin{prop}\label{prop-when-dif-a-preserves-idem}The $\NH_2$-module
$\NH_2\cdot\epsilon_{ii}(b)$ is $\dif_a$-stable, where $i\in \{1,2\}$, $b\in \Bbbk$ and
$a\in \F_p^\ast$ if and only if either $i=2$, $a=1$ and $b=-1$, in which case
\[
\epsilon_{22}(-1)=
\begin{DGCpicture}
\DGCstrand(0,0)(1,1)\DGCdot{0.25}
\DGCstrand(1,0)(0,1)
\end{DGCpicture}\ , \ \ \
\dif_{-1}\left(
\begin{DGCpicture}
\DGCstrand(0,0)(1,1)\DGCdot{0.25}
\DGCstrand(1,0)(0,1)
\end{DGCpicture}
~\right)=-~
\begin{DGCpicture}
\DGCstrand(0,0)(1,1)\DGCdot{0.25}\DGCdot{0.75}
\DGCstrand(1,0)(0,1)
\end{DGCpicture}~,
\]
or $i=2$, $a=-1$ and $b=0$, in which case
\[
\epsilon_{22}(0)=-~
\begin{DGCpicture}
\DGCstrand(0,0)(1,1)
\DGCstrand(1,0)(0,1)\DGCdot{0.25}
\end{DGCpicture}\ , \ \ \
\dif_1\left(-
\begin{DGCpicture}
\DGCstrand(0,0)(1,1)
\DGCstrand(1,0)(0,1)\DGCdot{0.25}
\end{DGCpicture}
~\right)=
\begin{DGCpicture}
\DGCstrand(0,0)(1,1)
\DGCstrand(1,0)(0,1)\DGCdot{0.25}\DGCdot{0.75}
\end{DGCpicture}
.
\]
\end{prop}
\begin{proof}Exercise.
\end{proof}

Notice that Propositions~\ref{prop-sub-stab} and \ref{prop-when-dif-a-preserves-idem}
both single out values $1,-1$ for $a\in \F_p^{\ast}$. The two propositions are closely
related. Given a two-step $\dif_a$-stable filtration $0\subset L \subset \bpol_2$,
the left ideal of $\mathrm{END}_{\sym_2}(\bpol_2)$ consisting of maps that annihilate $L$
is $\dif_a$-stable.

\vspace{0.1in}

\paragraph{A matrix presentation.}
Recall that $\sym_n^{\prime}$ is the maximal ideal of $\sym_n$ consisting of
all symmetric polynomials without the constant term.
Then $\sym_n/\sym_n^{\prime}\cong\Bbbk$, and moreover $\sym_n^{\prime}\cdot \NH_n$
is a two-sided ideal of $\NH_n$ whose quotient ring is a matrix ring:
\begin{equation}
\NH_n/\sym_n^{\prime}\cdot\NH_n = \mathrm{END}(\overline{U}_n) \cong \mathrm{M}(n!,\Bbbk).
\end{equation}
This ideal is stable under the $\dif_a$ action: for any $e\in \sym_n^{\prime}$ and $x\in \NH_n$,
$$\dif_a(e\cdot x)=\dif_a (e)\cdot x+e \cdot\dif_a(x)\in \sym_n^{\prime} \cdot\NH_n,$$
promoting the map (\ref{eq-hom-pin}) to a homomorphism of $p$-DG algebras.

\begin{ntn}
In what follows, we set
$$\left(\mathrm{M}(n!,\Bbbk),\dif_a\right) \ := \ \left(\NH_n/(\sym_n^{\prime} \cdot \NH_n),\dif_a\right)$$
to be the quotient matrix algebra over $\Bbbk$ with the induced differential, and denote by
$\pi_n: (\NH_n,\dif_a) \lra (\mathrm{M}(n!,\Bbbk),\dif_a)$ the canonical surjective homomorphism of $p$-DG algebras.
\end{ntn}

It is obvious that
$\pi_1:\NH_1\lra \NH_1/(\sym_1^{\prime})=\Bbbk$ and $\jmath_1:\Bbbk\hookrightarrow \NH_1$
are mutually inverse quasi-isomorphisms of $p$-DG algebras. Similar properties
hold for $\NH_n$ ($n\geq 2$) under the $\dif_{\pm 1}$-action.
We start with the first non-trivial case.
The $\dif_1$ action on the basis elements of $\NH_2$ from
Example~\ref{eg-NH2-basis} is given by

\[
\begin{CD}
\begin{DGCpicture}
\DGCstrand(0,0)(1,1)
\DGCstrand(1,0)(0,1) \DGCdot{0.75}
\end{DGCpicture}
@>{1}>>
-\begin{DGCpicture}
\DGCstrand(0,0)(1,1)
\DGCstrand(1,0)(0,1)\DGCdot{0.25} \DGCdot{0.75}
\end{DGCpicture}\\
@A{-1}AA @A{-1}AA\\
\begin{DGCpicture} \DGCstrand(0,0)(1,1)
\DGCstrand(1,0)(0,1)
\end{DGCpicture} @>{1}>>
-\begin{DGCpicture}
\DGCstrand(0,0)(1,1)
\DGCstrand(1,0)(0,1)\DGCdot{0.25}
\end{DGCpicture}
\end{CD}
\]

As a $(\sym_2,\dif_1)$-module, the $\dif_1$-stable ideal $\sym_2^{\prime}\cdot\NH_2$
admits an increasing filtration by $(\sym_2,\dif_1)$-submodules:
$$
\sym_2^{\prime}
\begin{DGCpicture}
\DGCstrand(0,0)(1,1)
\DGCstrand(1,0)(0,1)\DGCdot{0.25} \DGCdot{0.75}
\end{DGCpicture}
\subset\sym_2^{\prime}
\begin{DGCpicture}
\DGCstrand(0,0)(1,1)
\DGCstrand(1,0)(0,1)\DGCdot{0.25} \DGCdot{0.75}
\end{DGCpicture}
+\sym_2^{\prime}
\begin{DGCpicture}
\DGCstrand(0,0)(1,1)
\DGCstrand(1,0)(0,1)\DGCdot{0.75}
\end{DGCpicture}
\subset
\sym_2^{\prime}
\begin{DGCpicture}
\DGCstrand(0,0)(1,1)
\DGCstrand(1,0)(0,1)\DGCdot{0.25} \DGCdot{0.75}
\end{DGCpicture}
+\sym_2^{\prime}
\begin{DGCpicture}
\DGCstrand(0,0)(1,1)
\DGCstrand(1,0)(0,1)\DGCdot{0.75}
\end{DGCpicture}
+\sym_2^{\prime}
\begin{DGCpicture}
\DGCstrand(0,0)(1,1)
\DGCstrand(1,0)(0,1)\DGCdot{0.25}
\end{DGCpicture}
$$
$$\subset
\sym_2^{\prime}
\begin{DGCpicture}
\DGCstrand(0,0)(1,1)
\DGCstrand(1,0)(0,1)\DGCdot{0.25} \DGCdot{0.75}
\end{DGCpicture}
+\sym_2^{\prime}
\begin{DGCpicture}
\DGCstrand(0,0)(1,1)
\DGCstrand(1,0)(0,1)\DGCdot{0.75}
\end{DGCpicture}
+\sym_2^{\prime}
\begin{DGCpicture}
\DGCstrand(0,0)(1,1)
\DGCstrand(1,0)(0,1)\DGCdot{0.25}
\end{DGCpicture}
+\sym_2^{\prime}
\begin{DGCpicture}
\DGCstrand(0,0)(1,1)
\DGCstrand(1,0)(0,1)
\end{DGCpicture}
=\sym_2^{\prime}\cdot\NH_2.
$$
The subquotients of this filtration, as $(\sym_2,\dif_1)$-modules, are all
isomorphic to $\sym_2^{\prime}$, and hence contractible. Therefore,
$\sym_2^{\prime}\cdot\NH_2$ is a contractible ideal of $\NH_2$, and
$\NH_2\twoheadrightarrow \NH_2/\sym_2^{\prime}\cdot \NH_2\cong \mathrm{M}(2,\Bbbk)$
is a quasi-isomorphism of $p$-DG algebras. The same result holds for $\dif_{-1}$
by conjugation by $\sigma$.

\begin{prop}\label{prop-nilhecke2-qis-2by2-matrix-algebra} The canonical projection
$\pi_2:\NH_2\lra \NH_2/\sym_2^{\prime}\cdot\NH_2\cong
\mathrm{M}(2,\Bbbk)$ is a quasi-isomorphism if and only if $a=\pm
1$.
\end{prop}
\begin{proof} The ``if'' part follows from the discussion before the
proposition and the symmetry between $\dif_1$ and $\dif_{-1}$
(Proposition \ref{prop-symmetry-between-a-and-minus-a}). We now
prove the converse. As a $\dif_a$-module, $\NH_2$ fits into the
short exact sequence
\[
0\lra
\left\langle ~
\begin{DGCpicture}
\DGCstrand(0,0)(0,1)\DGCdot{0.75}[r]{$^{k_1}$}
\DGCstrand(1,0)(1,1)\DGCdot{0.75}[r]{$^{k_2}$}
\end{DGCpicture}
~\bigg| k_1,k_2 \in \N
\right\rangle
\lra \NH_2 \lra
\left\langle ~
\begin{DGCpicture}
\DGCstrand(0,0)(1,1)\DGCdot{0.75}[r]{$^{k_2}$}
\DGCstrand(1,0)(0,1)\DGCdot{0.75}[r]{$^{k_1}$}
\end{DGCpicture}
~\bigg|k_1,k_2 \in \N
\right\rangle
\lra 0.
\]
Here and in what follows, we use obtuse-angle brackets ``$\langle~\rangle$''
to stand for the $\Bbbk$-linear span of the enclosed elements. The $\dif_a$-submodule
\[
\left\langle
\begin{DGCpicture}
\DGCstrand(0,0)(0,1)\DGCdot{0.75}[r]{$^{k_1}$}
\DGCstrand(1,0)(1,1)\DGCdot{0.75}[r]{$^{k_2}$}
\end{DGCpicture}~\bigg| k_1,k_2 \in \N
\right\rangle
\cong \pol_2 \cong V_0 \oplus F,
\]
decomposes into a direct sum of the trivial module $V_0$ and a
free graded $\dif_a$-module $F$. On the quotient module $\dif_a$ acts by
\[\dif_a\left(\begin{DGCpicture}
\DGCstrand(0,0)(1,1)\DGCdot{0.75}[r]{$^{k_2}$}
\DGCstrand(1,0)(0,1)\DGCdot{0.75}[l]{$^{k_1}$}
\end{DGCpicture}\right)=
(k_1-a-1)\begin{DGCpicture}
\DGCstrand(0,0)(1,1)\DGCdot{0.75}[r]{$^{k_2}$}
\DGCstrand(1,0)(0,1)\DGCdot{0.75}[l]{${}^{k_1+1}$}
\end{DGCpicture}+(k_2+a-1)
\begin{DGCpicture}
\DGCstrand(0,0)(1,1)\DGCdot{0.75}[r]{$^{k_2+1}$}
\DGCstrand(1,0)(0,1)\DGCdot{0.75}[l]{$^{k_1}$}
\end{DGCpicture}.
\]
Hence,
\[
\left\langle~\begin{DGCpicture}
\DGCstrand(0,0)(1,1)\DGCdot{0.75}[r]{$^{k_2}$}
\DGCstrand(1,0)(0,1)\DGCdot{0.75}[r]{$^{k_1}$}
\end{DGCpicture}~\bigg|k_1\leq a+1~,k_2 \leq p+1-a
\right\rangle
\]
is a $\dif_a$-submodule of the quotient, and the quotient module decomposes
\[
\left\langle
~\begin{DGCpicture}
\DGCstrand(0,0)(1,1)\DGCdot{0.75}[r]{$^{k_2}$}
\DGCstrand(1,0)(0,1)\DGCdot{0.75}[r]{$^{k_1}$}
\end{DGCpicture}~\bigg|k_1~,k_2 \in \N
\right\rangle
\cong
\left\langle
~\begin{DGCpicture}
\DGCstrand(0,0)(1,1)\DGCdot{0.75}[r]{$^{k_2}$}
\DGCstrand(1,0)(0,1)\DGCdot{0.75}[r]{$^{k_1}$}
\end{DGCpicture}~\bigg|k_1\leq a+1~,k_2 \leq p+1-a \right\rangle
\oplus F^\prime,
\]
where $F^\prime$ is a free $\dif_a$-module. It follows that in the stable
category $H\udmod$ there is an exact triangle
\[V_0\lra \NH_2 \lra V_{a+1}\o V_{p+1-a}\{-2\}\lra V_0[1].\]
Therefore, the $\mathbb{O}_p$-dimension (see Notation~\ref{notation-symbol-as-O-dim})
of $(\NH_2,\dif_a)$ equals
$$[\NH_2]=1+q^{-2}(\sum_{i=0}^{a+1}q^{2i})(\sum_{j=0}^{p+1-a}q^{2j}).$$
Reduction mod
$p$ gives the $\mathbb{F}_p$-dimension $[\NH_2]_p=1+(a+2)(p+2-a)\equiv 5-a^2$, which
equals $\mathrm{dim}(\mathrm{M}(2,\Bbbk))=4$ modulo $p$ if and only if $a=\pm 1$.
The proposition follows.
\end{proof}

The general case of $(\NH_n,\dif_{\pm 1})$ for $n \geq 3$ is similar, namely the $p$-DG
algebras are all quasi-isomorphic to some $(n!)\times (n!)$-matrix $p$-DG algebras.
This will be shown using the $p$-DG polynomial
representations of $\NH_n$. We first make some simplification of notations.

\begin{ntn}\label{notation-P-plus-minus} When $a=1$, let $\bpol_n^+:=\bpol_n(\alpha^+)$, where $\alpha^+$ is the $n$-tuple $\alpha^+=(1-n, 2-n, \dots, 0)\in \F_p^n$. Likewise for $a=-1$, define $\bpol_n^-:=\bpol_n(\alpha^-)$, where $\alpha^- =(0,-1,\cdots, 1-n )$.
\end{ntn}

Observe that $\bpol_n^+$ is a left $p$-DG $(\NH_n,\dif_1)$-module while $\bpol_n^-$ is a left module over $(\NH_n,\dif_{-1})$. By applying $\psi$ (Proposition \ref{prop-symmetry-between-a-and-minus-a}), $\bpol_n^-$ becomes a right $p$-DG module over $(\NH_n,\dif_1)$. By Lemma \ref{lemma-dif-1-on-delta-n}, one can easily show that $\NH_n\epsilon_n$ (resp. $\epsilon_n^*\NH_n$) is a left (resp. right) $p$-DG ideal in $(\NH_n,\dif_1)$, and the isomorphism (\ref{eq-two-left-mod-iso}) (resp. (\ref{eq-right-mod-iso})) gives rise to an isomorphism of left (resp. right) $p$-DG modules:
$$
\big(\bpol_n^+,\dif_{\alpha^+}\big)\cong \big(\NH_n\epsilon_n\{\textstyle{\frac{n(1-n)}{2}}\}, \dif_1\big)
$$
$$
\left(\textrm{resp.}~\big(\bpol_n^-,\dif_{\alpha^-}\big)\cong \big(\epsilon_n^*\NH_n\{\textstyle{\frac{n(1-n)}{2}}\},\dif_1\big)\right).
$$
This follows from a straightforward computation that
\begin{equation}
\dif_1(\epsilon_n)=-\sum_{t=1}^n(n-t)x_t\epsilon_n, \ \ \ \ \left(\textrm{resp.}~ \dif_1(\epsilon_n^*)=-\sum_{t=1}^n(t-1)\epsilon_n^*x_t,\right)
\end{equation}
which in turn is a direct consequence of Lemma \ref{lemma-dif-1-on-delta-n}.

From now on till the end of this section, we will state the main results for $a=\pm 1$ while only giving the proof for $a=1$.
The $a=-1$ case follows by applying the symmetry $\sigma$ (Proposition \ref{prop-symmetry-between-a-and-minus-a}).
By Proposition \ref{prop-NH-as-p-DG-bimod}, there
is a homogeneous $(\NH_n,\dif_1)$-bimodule map
\begin{equation}\label{eqn-p-DG-bimodule-map-pol-and-NHn}
\phi_n:\bpol_n^+ \o_{\sym_n} \bpol_n^- \lra  \NH_n,
\end{equation}
which is an isomorphism. Furthermore, notice that as $p$-DG modules over $(\sym_n,\dif_1)$, $\bpol_n^+$
and $\bpol_n^-$ have respectively as bases
\begin{equation}\label{eqn-basis-of-Un-plus}
B_n^+:=
\left\{x^\beta := x_1^{b_1}x_2^{b_2}\cdots x_n^{b_n}| \beta=(b_1,b_2,\dots, b_n), ~0 \leq b_t \leq n-t, ~t=1,\dots, n
\right\},
\end{equation}
and
\begin{equation}\label{eqn-basis-of-Un-minus}
B_n^-:=
\left\{x^\beta := x_1^{b_1}x_2^{b_2}\cdots x_n^{b_n}| \beta=(b_1,b_2,\dots, b_n), ~0 \leq b_t \leq t-1 , ~t=1,\dots, n\right\}.
\end{equation}
The $\Bbbk$-vector subspace of $\bpol_n^+$ (resp. $\bpol_n^-$)
spanned by $B_n^+$ (resp. $B_n^-$) is stable under $\dif_1$. Therefore $B_n^+$
(resp. $B^-$) is also a basis for the $\dif_1$-module $U_n^+=\Bbbk\langle B_n^+\rangle$
(resp. $U_n^-=\Bbbk\langle B_n^-\rangle$ ). Furthermore, there are isomorphisms of $p$-DG modules over $\sym_n$
\[
U_n^+ \o \sym_n \lra \bpol_n^+, \ \ \ \ U_n^- \o \sym_n \lra \bpol_n^-,
\]
lifting the isomorphism (\ref{eq-nat-map-s}). The image of the product of these basis
elements under $\phi_n$ consists of
\[
\left\{
\begin{DGCpicture}
\DGCstrand(0,0.5)(0,2.5)\DGCdot{0.75}[r]{$^{c_1}$}\DGCdot{2.25}[r]{$^{b_1}$}
\DGCstrand(1,0.5)(1,2.5)\DGCdot{0.75}[r]{$^{c_2}$}\DGCdot{2.25}[r]{$^{b_2}$}
\DGCstrand(2,0.5)(2,2.5)\DGCdot{0.75}[r]{$^{c_3}\cdots$}\DGCdot{2.25}[r]{$^{b_3}\cdots$}
\DGCstrand(3,0.5)(3,2.5)\DGCdot{0.75}[r]{$^{c_n}$}\DGCdot{2.25}[r]{$^{b_n}$}
\DGCcoupon(-0.25,1.15)(3.25,1.85){$\delta(n)$}
\end{DGCpicture}
~\Bigg|~b_t \leq n-t ,~c_t \leq t-1,~t=1,\dots, n
\right\}.
\]
The matrix basis of $\NH_n$ given in Proposition \ref{prop-nil-hecke-basis-over-center}
is obviously contained in the $\Bbbk$-span of the above set. Since both sets have the
same cardinality $(n!)^2$, we conclude that they span the same subalgebra of $\NH_n$
isomorphic to the matrix algebra $\mathrm{M}(n!,\Bbbk)$. This gives rise to a $p$-DG lifting
of the homomorphism (\ref{eq-jmathn})
$$\jmath_n: (\mathrm{M}(n!,\Bbbk),\dif_1)\hookrightarrow (\NH_n,\dif_1)$$
as the composition
$$(\mathrm{END}(U_n^+),\dif_1) \hookrightarrow (\mathrm{END}_{\sym_n}(U_n^+\o \sym_n),\dif_1)\cong (\mathrm{END}(\bpol_n^+), \dif_1).$$
Similar results hold for the differential $\dif_{-1}$ as well.

\begin{prop}\label{prop-pm1-nilhecke-qis-matrix-algebra} When $a=\pm 1$, the canonical surjection of $p$-DG algebras $\pi_n: \NH_n\lra \NH_n/(\sym^{\prime}_n\cdot \NH_n)\cong \mathrm{M}(n!,\Bbbk)$ is a quasi-isomorphism. Furthermore, under $\dif_1$, a quasi-inverse of $\pi_n$ is given by the inclusion $\jmath_n:\mathrm{M}(n!,\Bbbk)\lra \NH_n$.
\end{prop}
\begin{proof} We only prove the $a=1$ case.  As a $p$-DG $\sym_n$-module, $\bpol_n^+\cong \sym_n \o U_n^+$. Introduce an increasing filtration on $U_n^+$ in the order of descending degrees of elements in $B_n^+$. The subquotients of this filtration on $U_n^+$ are all isomorphic to degree shifted copies of $\widetilde{V}_0$. Therefore $\bpol_n^+$ inherits a filtration whose subquotients are rank one free $p$-DG $\sym_n$-modules. It follows that the map of $p$-DG algebras $\pi_n$ factorizes as
$$\pi_n: (\mathrm{END}(\bpol_n^+),\dif_1)\cong (\mathrm{END}(U_n^+)\o \sym_n,\dif_1)\twoheadrightarrow (\mathrm{END}(U_n^+),\dif_1)\cong (\mathrm{M}(n!,\Bbbk),\dif_1) $$
with $\jmath_n$ as a section.
Here the middle arrow comes from the reduction of coefficient $\sym_n \twoheadrightarrow \Bbbk$.
Now if $n\geq p$, both sides are acyclic and the proposition is true vacuously. On the other hand, when $0\leq n< p$, $\sym_n$ is a quasi-isomorphic to $\Bbbk$ and the result follows.
\end{proof}

\begin{cor}\label{cor-NHn-qis-matrix-algebra} For any $n\in \N$, the functors
$$\jmath_n^*\cong {\pi_n}_\ast:\mc{D}(\NH_n,\dif_{1})\lra \mc{D}(\mathrm{M}(n!,\Bbbk),\dif_{1}),\ \ \ \ \pi_n^*\cong{\jmath_n}_\ast: \mc{D}(\mathrm{M}(n!,\Bbbk),\dif_{1})\lra \mc{D}(\NH_n,\dif_1)$$
are quasi-inverse equivalences of triangulated categories.
\end{cor}
\begin{proof}This is an easy consequence of the previous Proposition \ref{prop-pm1-nilhecke-qis-matrix-algebra} and Theorem \ref{thm-qis-algebra-equivalence-derived-categories}.
\end{proof}

Under the differential $\dif_1$, the isomorphism $\phi_n$ of Proposition \ref{prop-NH-as-p-DG-bimod} restricts to an isomorphism $(U_n^+)^*\cong U_n^-$. Combined with the above Proposition we summarize this subsection by exhibiting the following commutative diagram
\begin{equation}\label{eqn-iso-of-NHn-as-bimodule}
\xymatrix{
 U_n^+\o U_n^- \ar[r]^-{\phi_n} \ar@<-0.5ex>[d] &  \mathrm{END}(U_n^+)\ar[r]^-{\cong}\ar@<-0.5ex>[d] & \mathrm{M}(n!,\Bbbk)\ar@<-0.5ex>[d] \\
\sym_n\o U_n^+\o U_n^- \ar[r]^-{\cong} \ar[d]_{\cong} \ar@<-0.5ex>[u] & \sym_n \o \mathrm{END}(U_n^+) \ar[r] \ar[d]_{\cong} \ar@<-0.5ex>[u] &\mathrm{M}(n!,\sym_n)\ar[d]^-{\cong} \ar@<-0.5ex>[u] \\
\bpol_n^+ \o_{\sym_n} \bpol_n^-  \ar[r]^-{\phi_n} & \mathrm{END}_{\sym_n}(\bpol_n^+) \ar[r]^-{\cong} &~\NH_n.}
\end{equation}
The vertical arrows in the upper half part of the diagram, given by base change with the maps of $p$-DG algebras $\Bbbk\hookrightarrow \sym_n $ and $\sym_n\twoheadrightarrow \Bbbk$, are all quasi-isomorphisms; while the rest of the diagram is described by Proposition \ref{prop-NH-as-p-DG-bimod}.

\paragraph{Compactness of the column module.} Corollary \ref{cor-NHn-qis-matrix-algebra} allows us
to compute the Grothendieck group $K_0(\NH_n,\dif_{\pm 1})$ in the same way as for the
toy model of matrix algebras (Corollary \ref{cor-p-DG-morita-equivalence-matrix-ring}).
However, we will show that the column module over $\NH_n$ is compact by exhibiting it
as a direct summand of a finite cell module, which in turn leads to a more explicit
classification of indecomposable modules in $\mc{D}^c(\NH_n,\dif_{\pm 1})$. Again,
the symmetries $\psi$ and $\sigma$ allow one to focus on the differential $\dif_1$.

We start with the first non-trivial case $n=2$ as an example. Assume without loss of
generality that $p\geq 3$. The center $\sym_2$ of $\NH_2$ is a $p$-DG subalgebra. As a
left $p$-DG module over this subalgebra, $\NH_2$ admits the following filtration by
$p$-DG submodules
\[
0 \subset \sym_2\cdot
\begin{DGCpicture}
\DGCstrand(0,0)(1,1)
\DGCstrand(1,0)(0,1) \DGCdot{0.75}
\end{DGCpicture}\oplus
\sym_2\cdot
\begin{DGCpicture}
\DGCstrand(0,0)(1,1)
\DGCstrand(1,0)(0,1)\DGCdot{0.25} \DGCdot{0.75}
\end{DGCpicture}
\subset \NH_2,
\]
where the differential $\dif_1$ acts on the basis elements by
\[
\dif_1\left(
\begin{DGCpicture}
\DGCstrand(0,0)(1,1)
\DGCstrand(1,0)(0,1)\DGCdot{0.75}
\end{DGCpicture}\right)
=-\begin{DGCpicture}
\DGCstrand(0,0)(1,1)
\DGCstrand(1,0)(0,1)\DGCdot{0.25} \DGCdot{0.75}
\end{DGCpicture}, \ \ \ \
\dif_1\left(
\begin{DGCpicture}
\DGCstrand(0,0)(1,1)
\DGCstrand(1,0)(0,1)\DGCdot{0.25} \DGCdot{0.75}
\end{DGCpicture}
\right)=0.
\]
Denote the middle term by $\bpol_2^-\{1\}$. Then
$$\bpol_2^-\cong \mathrm{Res}_{\sym_2}^{\NH_2}(\bpol_2^-),$$
and the filtration gives rise to a short exact sequence of $p$-DG modules over $(\sym_2,\dif_1)$:
\begin{equation}\label{eqn-NH2-ses-sym2-submodules}
0\lra \bpol_2^-\{1\}\lra {_{\sym_2}\NH_2}\lra \bpol_2^-\{-1\}\lra 0.
\end{equation}
The short exact sequence (\ref{eqn-NH2-ses-sym2-submodules}) comes from restricting to
$\sym_2$ the two-step filtration of $\NH_2$ as a right $p$-DG module over itself by
submodules isomorphic to degree-shifts of $\bpol_2^-$.
Furthermore, the left $p$-DG module $_{\sym_2}\NH_2$ over $(\sym_2,\dif_1)$ decomposes
into a direct sum (since $p\geq 3$)
\begin{equation} \label{eqn-decomposition-NH2-over-sym2}
{_{\sym_2}\NH_2}\cong \sym_2\cdot
\begin{DGCpicture}
\DGCstrand(0,0)(0,1)
\DGCstrand(1,0)(1,1)
\end{DGCpicture}
\bigoplus \left({_{\sym_2}\NH_2^\prime}\right),
\end{equation}
where $_{\sym_2}\NH_2^\prime$ is the rank three $p$-DG $\sym_2$-submodule
\[
\sym_2\cdot
\begin{DGCpicture}
\DGCstrand(0,0)(1,1)
\DGCstrand(1,0)(0,1)
\end{DGCpicture}
\oplus \sym_2\cdot\left(~
\begin{DGCpicture}
\DGCstrand(0,0)(1,1)
\DGCstrand(1,0)(0,1) \DGCdot{0.75}
\end{DGCpicture}+
\begin{DGCpicture}
\DGCstrand(0,0)(1,1)
\DGCstrand(1,0)(0,1)\DGCdot{0.25}
\end{DGCpicture}~\right)
\oplus \sym_2\cdot
\begin{DGCpicture}
\DGCstrand(0,0)(1,1)
\DGCstrand(1,0)(0,1)\DGCdot{0.25} \DGCdot{0.75}
\end{DGCpicture}.
\]
Recall that $(\NH_2,\dif_1)$ has the polynomial representation $\bpol_2^+$.
Applying $\bpol_2^+\o_{\sym_2}(-)$ to the short exact sequence (\ref{eqn-NH2-ses-sym2-submodules}) gives the short exact sequence of $p$-DG modules
\begin{equation}\label{eqn-ses-two-step-filtration-of-NH2}
0\lra \bpol_2^+\o_{\sym_2}\bpol_2^-\{1\}\lra \bpol_2^+\o_{\sym_2}\NH_2 \lra \bpol_2^+\o_{\sym_2}\bpol_2^-\{-1\}\lra 0.
\end{equation}
We identify the terms in this short exact sequence. First off, $\bpol_2^+\o_{\sym_2}\bpol_2^- \cong \NH_2$ as a left $p$-DG module over $(\NH_2,\dif_1)$ by Proposition \ref{prop-NH-as-p-DG-bimod}. Thus the two end terms of the sequence (\ref{eqn-ses-two-step-filtration-of-NH2}) are respectively $\NH_2\{1\}$ and $\NH_2\{-1\}$. Next, the direct sum decomposition (\ref{eqn-decomposition-NH2-over-sym2}) gives rise to a direct sum decomposition of $\NH_2$-modules
\[
\bpol_2^+\o_{\sym_2}\NH_2\cong \bpol_2^+\bigoplus \left(\bpol_2^+\o_{\sym_2} \NH_2^\prime\right),
\]
which contains $\bpol_2^+$ as a $p$-DG direct summand. Using the characterization of compact modules (Theorem \ref{thm-characterizing-compact-modules}), we conclude that the module $\bpol_2^+$ is compact cofibrant. The general situation is proven in a similar way.

\begin{prop}\label{prop-NHn-column-module-compact}
When $0\leq n \leq p-1$, the left $p$-DG module $\bpol_n^+$ over $(\NH_n,\dif_{1})$ is compact cofibrant.
\end{prop}
\begin{proof}We claim that the module in the proposition is a direct summand of a finite cell module. The isomorphism (\ref{eqn-iso-of-NHn-as-bimodule}) gives rise to
\[
{_{\sym_n}}(\NH_n)\stackrel{\cong}{\lra}  \sym_n\o U_n^+\o U_n^- \stackrel{\cong}{\lra} U_n^+\o \bpol_n^-,
\]
as left $p$-DG modules over $\sym_n$. Applying $\bpol_n^+\o_{\sym_n}(-)$ to both sides, we have
\begin{equation}\label{eqn-general-case-bimodule-decomposition}
\bpol_n^+\o_{\sym_n}\NH_n \stackrel{\cong}{\lra} U_n^+\o \bpol_n^+ \o \bpol_n^-.
\end{equation}
From a finite step filtration of $U_n^+$ whose associated graded modules are isomorphic to degree-shifted copies of $V_0$, the term on the right hand side inherits a finite step filtration whose subquotients are isomorphic to $\bpol_n^+\o V_0\{r\} \o \bpol_n^-\cong \NH_n\{r\}$ by Proposition \ref{prop-NH-as-p-DG-bimod}. Hence the right hand side is a finite cell module. On the other hand, when $0\leq n \leq p-1$, $n!$ is coprime to $p$ and $\NH_n$ splits as a direct sum of $p$-DG $(\sym_n,\dif_1)$-modules:
\[
{_{\sym_n}\NH_n}\cong \sym_n\cdot I_{n!}\oplus \NH_n^\prime,
\]
where $\NH_n^{\prime}$ denotes the space of traceless $n!\times n!$-matrices over $\sym_n$, which is the kernel of the canonical projection
$$\NH_n \lra \NH_n/[\NH_n,\NH_n]\cong \sym_n.$$
Therefore the left hand side of (\ref{eqn-general-case-bimodule-decomposition}) decomposes as
\[
\bpol_n^+\o_{\sym_n}\NH_n \cong \bpol_n^+\oplus (\bpol_n^+\o_{\sym_n}\NH_n^\prime ),
\]
from which the claim follows.
\end{proof}

Corollary \ref{cor-NHn-qis-matrix-algebra} implies that $K_0(\NH_n,\dif_{ 1})\cong K_0(\mathrm{M}(n!,\Bbbk),\dif_{ 1})$. Moreover, since $\dif_{1}$ acts by taking the commutator with a fixed matrix in $\mathrm{M}(n!,\Bbbk)$, one can directly show that $K_0(\mathrm{M}(n!,\Bbbk),\dif_{ 1})\cong \mathbb{O}_{2p}$ by the same argument as in Section 2.3. We will now give a direct proof of this using that $\bpol_n^+$ is a compact generator of $\mc{D}(\NH_n,\dif_1)$, which is similar to the one used in Corollary \ref{cor-p-DG-morita-equivalence-matrix-ring}.

The isomorphism (\ref{eqn-iso-of-NHn-as-bimodule}) gives rise to an isomorphism of left $p$-DG modules over $(\NH_n,\dif_1)$:
\[
\bpol_n^+\o U_n^- \cong {_{\NH_n}\NH_n}.
\]
Therefore, $\NH_n$ has a filtration whose subquotients are grading-shifted copies of $\bpol_n^+$. It follows that $\bpol_n^+$ is a compact generator of $\mc{D}^c(\NH_n,\dif_1)$. Now, since $\NH_n\cong \mathrm{M}(n!,\sym_n)$, the natural map $\sym_n\lra \mathrm{END}_{\NH_n}(\bpol_n^+)$ is an isomorphism of $p$-DG algebras.

\begin{cor}\label{cor-K0-of-derived-categories-NHn}Let $0\leq n \leq p-1$. Then the (underived) tensor functor
\[
\bpol_n^+\o_{\sym_n}(-): \mc{D}(\sym_n,\dif) \lra \mc{D}(\NH_n,\dif_1)
\]
is an equivalence of triangulated categories whose quasi-inverse is given by
\[
\HOM_{\NH_n}(\bpol_n^+,-): \mc{D}(\NH_n,\dif_1) \lra \mc{D}(\sym_n,\dif).
\]
Consequently, the Grothendieck group $K_0(\NH_n,\dif_{\pm 1})\cong \mathbb{O}_p$.
\end{cor}
\begin{proof}
Proposition \ref{prop-criterion-derived-equivalence} implies that
$\bpol_n^+\o^{\mathbb{L}}_{\sym_n}(-)$ is an equivalence of triangulated
categories.
Since $\bpol_n^+$ is cofibrant, the derived tensor product functor
coincides with the underived one
\[
\bpol_n^+\o^{\mathbb{L}}_{\sym_n}(-)\cong \bpol_n^+\o_{\sym_n}(-),
\]
and likewise for the hom functor.
 The second claim follows from the adjunction in
Proposition~\ref{prop-tensor-hom-adjunction}. Lastly, the inclusion
$\Bbbk \lra \sym_n$ is a quasi-isomorphism
of $p$-DG algebras when $0\leq n \leq p-1$. The result follows from Corollary
\ref{cor-qis-algebra-isomorphic-K-0}.
\end{proof}

The proof of Proposition~\ref{prop-NHn-column-module-compact} together with
Corollary~\ref{cor-K0-of-derived-categories-NHn}
actually gives us more information about the idempotent complete triangulated
category $\mc{D}^c(\NH_n,\dif_{\pm 1})$. For instance, the free
module $\NH_n$ fits into a convolution diagram built out of shifted
copies of the polynomial representation $\bpol_n^+$,
similar to the toy model case (see Remark~\ref{rmk-matrix-algebra-convolution}).
Modules of the form $\bpol_n^+ \o \tilde{V}_i$ for all $0\leq i\leq p-2$ and their
grading shifts form a complete list of isomorphism classes of
indecomposables in $\mc{D}^c(\NH_n,\dif_1)$.

\subsection{Grothendieck ring as \texorpdfstring{$u_{\mathbb{O}_p}^+(\mf{sl}_2)$}{u(sl(2))}}
\label{subsec-gr-ring}

\paragraph{The small quantum group $u_{\zeta_{2l}}^+(\mf{sl}_2)$.} Let $l\geq 1$ be $2$ or an odd integer. Let $\zeta_{2l}$ be
a fixed primitive $2l$-th root of unity in $\C$. Then $\zeta_l:=\zeta_{2l}^2$ is a primitive $l$-th root of unity.
Following Lusztig \cite[Section 5]{Lu1}, \cite[Chapter 36]{Lu} (see also \cite[Section V]{CF}), we
define the small quantum group $u_{\zeta_{2l}}^+(\mf{sl}_2)$ ($u^+_{\zeta_{2l}}$ for short) to be the $\Q[\zeta_{2l}]$-algebra
with one generator $E$ subject to the relation $E^l=0$. Equipped with
the comultiplication
\[
\Delta(E)=E\o 1 +1 \o E,
\]
$u_{\zeta_{2l}}^+$ becomes a twisted bialgebra in the category of $\Z/l$-graded
vector spaces. Here being twisted means that the multiplication in $u_{\zeta_{2l}}^+ \o u_{\zeta_{2l}}^+$ is given by
\[(E^n\o 1)(1\o E^{m})=E^n\o E^m\]
\[(1\o E^{n})(E^m\o 1)=\zeta_{2l}^{nm}E^m\o E^n\]
For $n\leq l-1$, the $n$-th divided power of $E$ is by definition $E^{(n)}:=\frac{E^n}{[n]!}$ ($E^{(0)}:=1$), where $[n]:=\dfrac{\zeta_{2l}^n-\zeta_{2l}^{-n}}{\zeta_{2l}-\zeta_{2l}^{-1}}=\sum_{i=1}^{n}\zeta_{2l}^{n+1-2i}$, and $[n]!:=\prod_{i=1}^n[i]$. In what follows we also denote by ${m \brack n}$ the quantum binomial coefficient defined as ${m \brack n}:=\dfrac{[m]!}{[n]![m-n]!}$ whenever $n \leq m$, $0\leq n,~m-n \leq l-1$. Since $[m]!=0$ if $m\geq l$, ${l \brack n}=0$.
The set $\{E^{(r)}|r=0,1,\cdots,l-1\}$ forms an integral basis of $u_{\zeta_{2l}}^+$ over the ring of the $2l$-th cyclotomic
integers $\mc{O}_{2l}=\Z[\zeta_{2l}]\cong \Z[q]/(\Psi_{2l}(q))$, with the relation
\[
E^{(n)}\cdot E^{(m)}=\left\{
\begin{array}{ll}
{n+m \brack n} E^{(n+m)} & \textrm{if $n+m\leq l-1$,}\\
& \\
0 & \textrm{otherwise}.
\end{array}
\right.
\]
This integral form will be denoted by $u^+_{\mc{O}_{2l}}$. The comultiplication acts on on the divided power elements by
\begin{equation}\label{eqn-comult-on-divided-powers}
\Delta(E^{(n)})=\sum_{t=0}^{n}\zeta_{2l}^{-t(n-t)}E^{(t)}\o E^{(n-t)} \in u^+_{\mc{O}_{2l}} \o u^+_{\mc{O}_{2l}}.
\end{equation}

We now
specialize to $l=p$. To do this we also define an $\mathbb{O}_p$-integral small quantum
group $u^+_{\mathbb{O}_p}(\mf{sl}_2)$ ($u^+_{\mathbb{O}_p}$ for short).  The
following elementary lemma guarantees that the divided powers $E^{(n)}$ for
$0\leq n\leq p-1$ can be defined over $\mathbb{O}_p$.

\begin{lemma}In the ring $\mathbb{O}_p=\Z[q]/(\Psi_{p}(q^2))$, the element
$$[n]_{\mathbb{O}}:=\frac{q^n-q^{-n}}{q-q^{-1}}=\sum_{t=1}^{n}q^{n+1-2t}$$
is a unit if $1\leq n\leq p-1$.
\end{lemma}
\begin{proof}It suffices to show that $(1-q^{2n})/(1-q^2)$ is invertible. Since
$q^{2}$ generates multiplicatively a cyclic group of order $p$, $q^{2n}$ is another
generator of the group if $1\leq n\leq p-1$. Hence $q^2=(q^{2n})^v$ for some
$v\in \N$ and
$$\frac{1-q^2}{1-q^{2n}}=\frac{1-q^{2nv}}{1-q^{2n}}\in \mathbb{O}_p.$$
The result follows.
\end{proof}
The lemma implies that ${m \brack n}_{\mathbb{O}}:=\dfrac{[m]_{\mathbb{O}}!}
{[n]_{\mathbb{O}}![m-n]_{\mathbb{O}}!}$ is zero if $0\leq n, m-n\leq p-1$ and $m\geq p$.
Indeed by definition ${m \brack n}_{\mathbb{O}}[n]_{\mathbb{O}}![m-n]_{\mathbb{O}}!=
[m]_{\mathbb{O}}!$ and the last term vanishes.

The $\mathbb{O}_p$-integral form $u^+_{\mathbb{O}_p}$ is defined similar to
$u^+_{\mc{O}_{2p}}$: one simply replaces all the quantum integers $[n]$ in the
definition of $u^+_{\mc{O}_{2p}}$ by $[n]_{\mathbb{O}}$. Recall that there is a
surjective map of rings
$$\mathbb{O}_p \cong \Z[q]/(1+q^2+\cdots+q^{2(p-1)})
\twoheadrightarrow \mc{O}_{2p}=\Z[\zeta_{2p}]$$
by sending $q$ to $\zeta_{2p}$. Base changing by this ring map gives rise
to an isomorphism of twisted bialgebras
\[
u^+_{\mathbb{O}_p}\o_{\mathbb{O}_p}\mc{O}_{2p}\cong u^+_{\zeta_{2p}}.
\]

In the following diagram we summarize the relations between various forms of
the small quantum $\mf{sl}_2$ introduced above and their base rings:
\[
\xymatrix{
u^+_{\mathbb{O}_{p}} \ar@{->>}[rr]^{\o_{\mathbb{O}_p}\mc{O}_{2p}} && u^+_{\mc{O}_{2p}}
\ar@{^(->}[rr]^{\o_{\mc{O}_{2p}}\Q[\zeta_{2p}]}  && u^+_{\zeta_{2p}} \\
\mathbb{O}_p \ar[u] \ar@{->>}[rr]^{\o_{\mathbb{O}_p}\mc{O}_{2p}} && \mc{O}_{2p}\ar[u]
\ar@{^(->}[rr]^{\o_{\mc{O}_{2p}}\Q[\zeta_{2p}]} && \Q[\zeta_{2p}]. \ar[u]
}
\]
The main goal of this section is to categorify the small quantum group $u^+_{\mathbb{O}_{p}}$.

\paragraph{Induction and restriction functors.} There is an obvious inclusion of algebras
\begin{equation}\label{eqn-side-ways-inclusion}
\iota_{n,m}:\NH_n\o \NH_m \lra \NH_{n+m}
\end{equation}
coming from putting diagrams sideways next to each other: for any $x \in \NH_n$,
$y \in \NH_m$,
\[
\begin{DGCpicture}
\DGCstrand(0,0)(0,2)[`$1$]
\DGCstrand(2,0)(2,2)[`$n$]
\DGCcoupon(-0.25,0.5)(2.25,1.5){$x$}
\DGCcoupon*(0.25,1.65)(1.75,1.85){$\cdots$}
\DGCcoupon*(0.25,0.15)(1.75,0.35){$\cdots$}
\end{DGCpicture}
\ \o \
\begin{DGCpicture}
\DGCstrand(0,0)(0,2)[`$1$]
\DGCstrand(2,0)(2,2)[`$m$]
\DGCcoupon(-0.25,0.5)(2.25,1.5){$y$}
\DGCcoupon*(0.25,1.65)(1.75,1.85){$\cdots$}
\DGCcoupon*(0.25,0.15)(1.75,0.35){$\cdots$}
\end{DGCpicture}
\ \mapsto \
\begin{DGCpicture}
\DGCstrand(0,0)(0,2)[`$1$]
\DGCstrand(2,0)(2,2)[`$n$]
\DGCstrand(3,0)(3,2)[`$n+1$]
\DGCstrand(5,0)(5,2)[`$n+m$]
\DGCcoupon*(0.25,1.65)(1.75,1.85){$\cdots$}
\DGCcoupon*(0.25,0.15)(1.75,0.35){$\cdots$}
\DGCcoupon(-0.25,0.5)(2.25,1.5){$x$}
\DGCcoupon*(3.25,1.65)(4.75,1.85){$\cdots$}
\DGCcoupon*(3.25,0.15)(4.75,0.35){$\cdots$}
\DGCcoupon(2.75,0.5)(5.25,1.5){$y$}
\end{DGCpicture}.
\]
Recall from the discussion before Definition \ref{def-derived-hom-and-tensor} that the tensor product $p$-DG algebra $\NH_n\o \NH_m$ has the differential $\dif_a(x\o y)= \dif_a(x)\o y +x\o \dif_a(y)$ for any $x\in \NH_n$ and $y\in \NH_m$. The locality of the differentials $\dif_{a}$ implies that $\iota_{n,m}$ is an inclusion of $p$-DG algebras.

The map $\iota_{n,m}$ gives rise to induction and restriction functors (Definition \ref{def-induction-restriction-functors})
\[
\Ind_{n,m}:={\iota^*_{n,m}}: \mc{D}(\NH_n\o \NH_m, \dif_a) \lra \mc{D}(\NH_{n+m}, \dif_a),
\]
\[
\Res_{n,m}:={\iota_{n,m}}_*: \mc{D}(\NH_{n+m},\dif_a) \lra \mc{D}(\NH_n\o \NH_m, \dif_a).
\]
The functor $\Ind_{n,m}$ preserves the subcategories of compact modules:
\[
\Ind_{n,m}:={\iota^*_{n,m}}: \mc{D}^c(\NH_n\o \NH_m, \dif_a) \lra \mc{D}^c(\NH_{n+m}, \dif_a).
\]
Likewise, there are inclusions of $p$-DG algebras under the $\dif_a$-action
$$\iota_{n_1,n_2,\cdots,n_k}:\NH_{n_1}\o\NH_{n_2}\o \cdots \o \NH_{n_k} \lra \NH_{n_1+n_2+\cdots+n_k}$$
for any sequence of $n_i\in \N$. The induction functor generalizes in an obvious way. However, it is not clear at the moment whether
the restriction functor
$$\Res_{n_1, \cdots, n_k}:={\iota_{n_1,\cdots,n_k}}_*: \mc{D}(\NH_{n_1+\cdots + n_k}, \dif_a) \lra \mc{D}(\NH_{n_1}\o \cdots \o \NH_{n_k},\dif_a)$$
preserves the subcategories of compact objects for arbitrary $a\in \F_p^*$.

From now on, we will focus on the special values $a=\pm 1$ while only giving the proof for the $a=1$ case. The $\dif_{-1}$ case follows by applying
the symmetry $\psi$ (Proposition~\ref{prop-symmetry-between-a-and-minus-a}). Form the direct sum of $p$-DG algebras
\[
(\NH, \dif_{\pm 1}) := \bigoplus_{n \in \N}(\NH_n, \dif_{ \pm 1})
\]
which is no longer a unital $p$-DG algebra. The inclusion of the unital $p$-DG subalgebra $\oplus_{0\leq n\leq p-1}(\NH_n,\dif_{\pm 1})$ into $(\NH, \dif_{\pm 1})$ is a quasi-isomorphism by Proposition \ref{prop-NHp-contractible}. By a compact module over $\NH$
 we mean a finite direct sum of compact modules over $(\NH_n, \dif_{\pm 1})$ for various $n$\footnote{Following Keller \cite{Ke1}, one can say that $(\NH, \dif_{\pm 1})$ is a $p$-DG category, while objects in $\mc{D}^c(\NH, \dif_{\pm 1})$ are compact modules over this category.}, so that
$$\mc{D}^c(\NH,\dif_{\pm 1}):=\bigoplus_{n \in \N} \mc{D}^c(\NH_n,\dif_{\pm 1}).$$
On the level of Grothendieck groups, we have
$$K_0(\NH,\dif_{\pm 1})=\bigoplus_{n \in \N} K_0(\NH_n,\dif_{\pm 1}).$$
Summing over all $n, m \in \N$, $\iota_{n,m}$ gives rise to
to an induction functor between the compact derived categories
\begin{equation}\label{eqn-ind-res-category}
 \Ind  :=\sum_{n,m\in \N} \Ind_{n,m}: \oplus_{n,m\in \N}\mc{D}^c(\NH_n\o \NH_m, \dif_{\pm 1}) \lra \oplus_{r\in \N}\mc{D}^c(\NH_{r}, \dif_{\pm 1}).
\end{equation}

\begin{lemma} The exterior tensor product
\[
\begin{array}{cccc}
\boxtimes: & \NH_n\dmod\times \NH_m\dmod & \lra &(\NH_n\o \NH_m)\dmod\\
& (N, M) & \mapsto & N\boxtimes M
\end{array}
\]
induces an isomorphism of Grothendieck groups
$$
\begin{array}{rcl}
K_0(\NH_n,\dif_{\pm 1}) \o_{\mathbb{O}_p}K_0(\NH_m,\dif_{\pm 1}) & \cong & K_0(\NH_n\o \NH_m,\dif_{\pm 1})\\
& \cong &
\left\{
\begin{array}{cc}
\mathbb{O}_p & \textrm{if $n\leq p-1$ and $m\leq p-1$}\\
0 & \textrm{otherwise.}
\end{array}
\right.
\end{array}$$
\end{lemma}
\begin{proof}We show the $\dif_1$ case.
If $n\geq p$, both $\NH_n$ and $\NH_n\o\NH_m$ are acyclic and the result follows,
likewise for $m\geq p$. Otherwise, we use the derived equivalence of
Proposition~\ref{prop-pm1-nilhecke-qis-matrix-algebra}
\[
\jmath_n^*: \mc{D}(\NH_n, \dif_1 )\cong \mc{D}(\mathrm{M}(n!,\Bbbk), \dif_{1}),
\]
which takes the polynomial representation $\bpol_n^+$ to the column module $\Bbbk^{n!}$
with the induced differential. It follows that
\[
\mc{D}(\NH_n\o \NH_m, \dif_1 )\cong \mc{D}(\mathrm{M}(n!,\Bbbk)\o \mathrm{M}(m!,\Bbbk),
\dif_1) \cong \mc{D}(\mathrm{M}(n!\cdot m!,\Bbbk),\dif_1),
\]
and similarly for the subcategory of compact modules. Moreover, since
$K_0(\NH_n,\dif_1)$ is generated as an $\mathbb{O}_p$-module by the symbol
$[\bpol_n^+]$, using the above derived equivalence, we have
\[
K_0(\NH_n\o \NH_m,\dif_1)\cong K_0(\mathrm{M}(n!,\Bbbk)\o\mathrm{M}(m!,\Bbbk),
\dif_1)\cong K_0(\mathrm{M}(n!\cdot m!,\Bbbk),\dif_1),
\]
and, under this isomorphism, the exterior tensor product of polynomial
representations is sent to
\[
{[\bpol_n^+\boxtimes \bpol_m^+] \mapsto [\Bbbk^{n!}\boxtimes \Bbbk^{m!}]
\mapsto[\Bbbk^{n!\cdot m!}]},
\]
which is the rank one free $\mathbb{O}_p$-module generator of
$K_0(\mathrm{M}(n!\cdot m!,\Bbbk),\dif_1)$. Hence $K_0(\NH_n\o \NH_m,\dif_1)$ is
of rank one and generated by $[\bpol_n^+\boxtimes \bpol_m^+]$.  Since the groups
$K_0(\NH_n,\dif_1)$, $K_0(\NH_m,\dif_1)$ are also of rank one, the lemma follows.
\end{proof}

Although one can see that the restriction functor
\begin{equation}\label{eqn-res-category}
 \Res  :=\sum_{n,m\in \N} \Res_{n,m}: \oplus_{r\in \N}\mc{D}(\NH_{r}, \dif_{\pm 1}) \lra \oplus_{n,m\in \N}\mc{D}(\NH_n\o\NH_m, \dif_{\pm 1})
\end{equation}
sends compact objects to compact objects, an easy computation on the Grothendieck level shows that, unlike the abelian case in \cite[Proposition 2.18]{KL1}, the symbol of the restriction functor $[\Res]$ can not categorify the coproduct structure of $u^+_{\mathbb{O}_p}$. Indeed, by the adjunction between induction and restriction (Proposition \ref{prop-tensor-hom-adjunction}), one has
\begin{eqnarray*}
\NH_{n+m} & \cong & \RHOM_{\NH_{n+m}}(\NH_{n+m}, \NH_{n+m}) \\
& \cong &\RHOM_{\NH_{n+m}}(\Ind_{n,m}(\NH_n\boxtimes\NH_m), \NH_{n+m})\\
& \cong & \RHOM_{\NH_{n}\o \NH_m}(\NH_n\boxtimes\NH_m, \Res_{n,m}(\NH_{n+m})).
\end{eqnarray*}
In particular, when $n+m<p$, we have the equality in the Grothendieck group
\begin{equation}
([n+m]!)^2=[\NH_{n+m}]=[\RHOM_{\NH_{n}\o \NH_m}(\NH_n\boxtimes\NH_m, \Res_{n,m}(\NH_{n+m}))],
\end{equation}
which in turn gives rise to
\begin{equation}
[\Res_{n,m}(\mc{P}_{n+m}^+)]={m+n \brack n}_{\mathbb{O}}[\mc{P}^+_n\boxtimes \mc{P}^+_m].
\end{equation}
Here we used that $[\NH_n]=[\mc{P}_n^+\o U_n^-]=[n]_{\mathbb{O}}![\mc{P}_n^+]=([n]_{\mathbb{O}}!)^2[\sym_n]=([n]_{\mathbb{O}}!)^2$ in $K_0(H\udmod)=\mathbb{O}_p$ (see equation (\ref{eqn-iso-of-NHn-as-bimodule})). This mismatches with the comultiplication action given by equation (\ref{eqn-comult-on-divided-powers}).

\paragraph{A twisted restriction functor.} To categorify the comultiplication of $u^+_\mathbb{O}$, we introduce a twisted version of the restriction functor.

\begin{defn}\label{def-twisted-restriction-functor}Let $n,m\in \N$. We define the \emph{twisted $p$-DG bimodule} $^{\phantom{n,}X}_{n,m}\NH_{n+m}$ over $(\NH_n\o \NH_m, \NH_{n+m})$ as follows. As an $(\NH_n\o \NH_m,\NH_{n+m})$-bimodule, it is isomorphic to $\NH_{n+m}$, with the bimodule generator $^X1_{n+m}$ sitting in degree zero. It is twisted in the sense that the differential acts on the bimodule generator $^X1_{n+m}\in {^{\phantom{n,}X}_{n,m}\NH_{n+m}}$ by
\[
\dif({^X1_{n+m}})=m\sum_{t=1}^n x_t\cdot{^X1_{n+m}}=me_1^{(n)}\cdot {^X1_{n+m}}.
\]
\end{defn}

\begin{rmk}\label{rmk-twisted-p-DG-bimod}It is not hard to see that the twisted bimodule $^{\phantom{n,}X}_{n,m}\NH_{n+m}$ is isomorphic, up to a grading shift and as a $p$-DG bimodule, to the submodule $(e_n^{(n)})^m\cdot \NH_{n+m} \subset \NH_{n+m}$, where $(e_n^{(n)})^m=x_1^m\cdots x_n^m$, with the inherited differential. Notice that $(e_n^{(n)})^m$ is a central element in $\NH_n\o \NH_m$.

More generally, if $_AM_B$ is a $p$-DG bimodule over $(A,B)$ and $z$ is a central element of $A$ that is a non-zero-divisor on $M$, then $^zM:=z\cdot M$ is again a bimodule with the same underlying $(A,B)$-bimodule structure, but with the induced differential action from the inclusion $^zM\subset M$. We say that $^zM$ is obtained as a $p$-DG bimodule from $M$ by twisting the differential.
\end{rmk}

For any element $y\in\NH_{n+m}$, we will denote by $^Xy\in {^{\phantom{n,}X}_{n,m}\NH_{n+m}}$ the corresponding element in the twisted bimodule. Consider the element ${^X\delta(n+m)}\in{^{\phantom{n,}X}_{n,m}\NH_{n+m}}$ arising from the longest length element of $S_{n+m}$. From the differential formula on $\delta(n+m)$ (Lemma \ref{lemma-dif-1-on-delta-n}), one computes that
\begin{eqnarray}\label{eqn-dif-on-twisted-delta}
\dif({^X\delta(n+m)})& = & -\sum_{t=1}^n(n-t)x_t\cdot {^X\delta(n+m)}-\sum_{t=1}^m(m-t)x_{n+t}\cdot {^X\delta(n+m)}\nonumber \\
&& -\sum_{t=1}^{n+m}(t-1){^X\delta(n+m)}\cdot x_t.
\end{eqnarray}
It follows from this equation that the right $\pol_{n+m}$-span of the elements in the set
$
\left\{x_1^{b_1}\cdots x_n^{b_n} x_{n+1}^{c_1}\cdots x_{n+m}^{c_{m}}{^X\delta(n+m)}\right\},
$
where $0\leq b_t \leq n-t, ~t=1, \dots, n$, and $0\leq c_v \leq m-v,~v=1,\dots, m$, constitutes
a right $p$-DG module over $\NH_{n+m}$ which will be denoted $^{X}N_{n+m}$. Using the structure of $\mc{P}_{n+m}^-$ (see Notation \ref{notation-symbol-as-O-dim} and the isomorphism (\ref{eqn-iso-of-NHn-as-bimodule})), one readily sees that $^{X}N_{n+m}$ is cofibrant as a right $p$-DG module whenever $0\leq n+m\leq p-1$: it has an
$n!m!$-step filtration whose subquotients are isomorphic to grading shifted copies of $\mc{P}_{n+m}^-$.

\begin{lemma}\label{lemma-property-twisted-bim}Let $n,m$ be natural numbers with $n+m\leq p-1$.
\begin{itemize}
\item[(i)] The inclusion of $^{X}N_{n+m}$ into $^{\phantom{n,}X}_{n,m}\NH_{n+m}$ is a quasi-isomorphism of right $p$-DG modules over $\NH_{n+m}$.
\item[(ii)] The module $^{\phantom{n,}X}_{n,m}\NH_{n+m}$ is cofibrant as a left $p$-DG module over $\NH_n\o \NH_m$.
\end{itemize}
\end{lemma}
\begin{proof}To prove the first statement, we compute the cohomology of both modules. Abbreviating the element $x_1^{n-1}x_2^{n-2}\cdots x_{n-1} x_{n+1}^{m-1}x_{n+2}^{m-2}\cdots x_{n+m-1}{^X\delta(n+m)}$ by ${^XD_{n+m}}$, we see from equation (\ref{eqn-dif-on-twisted-delta}) that
\[
\dif(^XD_{n+m})=-\sum_{t=1}^{n+m}(t-1){^XD_{n+m}}\cdot x_t\in {^XD_{n+m}}\cdot \pol_{n+m}
\]
Then it is easy to identify the cohomology of $^{X}N_{n+m}$ with the finite dimensional $\dif$-stable space
\begin{equation}\label{eqn-cohomology-of-twisted-NH}
\Bbbk\langle x_1^{b_1}\cdots x_n^{b_n} x_{n+1}^{c_1}\cdots x_{n+m}^{c_{m}}{^X\delta(n+m)}x_1^{d_1}\cdots x_{n+m}^{d_{n+m}}\rangle,
\end{equation}
where $0 \leq b_t \leq n-t,~t=1, \dots, n$, $0 \leq c_t \leq m-t,~t=1,\dots, m$, and $0 \leq d_t \leq t-1, 1\leq t \leq n+m$. On the other hand, regarded as a left $p$-DG module over $\NH_n\o \NH_m$, $^{\phantom{n,}X}_{n,m}\NH_{n+m}$ has an $(n+m)!$-step filtration whose subquotients are all isomorphic to graded shifts of the module $\pol_{n+m}\cdot{^X\delta_{n+m}}x_1^{0}x_2^1\cdots x_{n+m}^{n+m-1}$, which is isomorphic to the cofibrant $\NH_n\o \NH_m$ module $\mc{P}_n^+\boxtimes \mc{P}_m^+$. Therefore, the cohomology of $^{\phantom{n,}X}_{n,m}\NH_{n+m}$ is also isomorphic to the above finite dimensional $p$-complex (\ref{eqn-cohomology-of-twisted-NH}). The first result follows.

The second claim follows from the proof of the first, since as a left $p$-DG module over $\NH_{n}\o \NH_m$, the subquotients of the filtration on $^{\phantom{n,}X}_{n,m}\NH_{n+m}$, isomorphic to $\mc{P}_n^+\boxtimes \mc{P}_m^+$, are cofibrant when $n+m <p$ (Proposition \ref{prop-NHn-column-module-compact}).
\end{proof}

\begin{defn}\label{defn-twisted-restriction} The (derived) \emph{twisted restriction} $^X\Res_{n,m}$ is the functor
\begin{eqnarray*}\label{eqn-def-twisted-restriction}
{^X\Res}_{n,m}: \mc{D}(\NH_{n+m},\dif_1)& \lra &\mc{D}(\NH_n\o \NH_m,\dif_1) \nonumber \\
M & \mapsto & ^{\phantom{n,}X}_{n,m}\NH_{n+m}\o^{\mathbb{L}}_{\NH_{n+m}}M
\end{eqnarray*}
\end{defn}
Note that by Lemma \ref{lemma-property-twisted-bim}, there is a functorial isomorphism of $p$-complexes
\[{^X\Res}_{n,m}(M)\cong {^XN_{n+m}}\otimes_{\NH_{n+m}}M.\]

In what follows, we will abbreviate the $p$-DG algebra $\NH_n\o\NH_m$ by $\NH_{n,m}$, $^{\phantom{n,}X}_{n,m}\NH_{n+m}$ by $^X\NH_{n+m}$ if no confusion can be made.
When we have three integers $k,n, m$ such that $0\leq k+n+m\leq p-1$, there are two ways of composing functors, namely $({^X\Res}_{k,n}\o \Id_m)\circ ({^X\Res}_{k+n,m})$ or $(\Id_k\o{^X\Res}_{n,m})\circ ({^X\Res}_{k,n+m})$, which are in fact both isomorphic to derived tensoring with the $(\NH_{k,n,m}, \NH_{k+n+m})$-bimodule
\[
((e_k^{(k)})^n (e_{k+n}^{(k+n)})^m)\cdot\NH_{k+n+m} = (x_1\cdots x_k)^{n+m} (x_{k+1}\cdots x_{k+n})^m \cdot \NH_{k+n+m}.
\]
To see this one uses Lemma \ref{lemma-property-twisted-bim} to identify the composition functor, for instance, as
$$
{^X\Res}_{k,n}\o \Id_m ({^X\Res}_{k+n,m}(M))\cong
\left(({{^X}\NH_{k+n}}\o \NH_m)\otimes_{\NH_{k+n,m}} ({^X\NH_{k+n+m}})\right)\o^{\mathbb{L}}(M),
$$
where the cofibrance of ${^X\NH_{k+n+m}}$ as a left $p$-DG module over $\NH_{k+n,m}$ allows us to replace the middle derived tensor by the usual one.

\begin{lemma}\label{lemma-twisted-res-respect-compact} Let $n, m$ be natural numbers with $n+m\leq p-1$.  Then:
\begin{itemize}
\item[(i)] the twisted restriction functor acts on the cofibrant module $\mc{P}^+_{n+m}$ by
\[
{^X\Res}_{n,m}(\mc{P}_{n+m}^+)\cong \mc{P}_n^+\boxtimes \mc{P}_m^+\{-mn\},
\]
\item[(ii)] the twisted restriction functor sends compact modules to compact modules.
\end{itemize}
\end{lemma}
\begin{proof}When acting on a cofibrant module, the twisted restriction functor need not be derived:
$${^X\Res_{n,m}}(\mc{P}^+_{n+m})\cong (e_n^{(n)})^m\cdot\mc{P}^+_{n+m} = (x_1\cdots x_n)^m\cdot \mc{P}^+_{n+m}.$$
The differential acts on the twisted module generator ${^X1_{\alpha^+}}$, which is identified with $(x_1\cdots x_n)^m\cdot 1_{\alpha^+}$ under the above isomorphism, by
\[
\dif({^X1_{\alpha^+}})=-\sum_{t=1}^n(n-t)x_t\cdot {^X1_{\alpha^+}} -\sum_{t=1}^m(m-t)x_{n+m-t}\cdot{^X1_{\alpha^+}}.
\]
The first statement follows by comparing with the differentials on $\mc{P}^+_{n}$, $\mc{P}^+_{m}$ (Notation \ref{notation-P-plus-minus}). 

The second statement follows from the classification of compact modules (see the discussion after Corollary \ref{cor-K0-of-derived-categories-NHn}) in $\mc{D}(\NH_{n+m},\dif_{\pm 1})$.
\end{proof}

\paragraph{Categorification.} As for the induction functor, we sum $^X\Res_{n,m}$ over all $n, m \in \N$ and
obtain the twisted restriction functor between the derived categories
\begin{equation}\label{eqn-ind-res-category}
 {^X\Res} :=\sum_{n,m\in \N} {^X\Res_{n,m}}: \oplus_{r\in \N}\mc{D}(\NH_{r}, \dif_{\pm 1}) \lra \oplus_{n,m\in \N}\mc{D}(\NH_{n}\o \NH_{m}, \dif_{\pm 1}).
\end{equation}

Together with earlier discussion, Lemma \ref{lemma-twisted-res-respect-compact} allows us to pass from the induction and twisted restriction functors on the derived category of compact $\NH$-modules to maps of Grothendieck groups, as $\mathbb{O}_p$-modules,
\begin{equation}
{[\Ind]}: K_0(\NH,\dif_{\pm 1})\o_{\mathbb{O}_p} K_0(\NH, \dif_{\pm 1}) \lra K_0(\NH, \dif_{\pm 1}),
\end{equation}
\begin{equation}
{[^X\Res]}: K_0(\NH,\dif_{\pm 1}) \lra K_0(\NH, \dif_{\pm 1}) \o_{\mathbb{O}_p} K_0(\NH,\dif_{\pm 1}).
\end{equation}

On Grothendieck groups, it follows from the associativity of $\Ind$ that $[\Ind]$
equips $K_0(\NH,\dif_{\pm 1})$ with the structure of an $\mathbb{O}_p$-algebra
whose unit is given by the symbol of $[\Bbbk]$ over $(\NH_0,\dif_{\pm 1})\cong
(\Bbbk, \dif_{\pm 1})$. Likewise, $[^X\Res]$ gives rise to a coassociative coalgebra
structure on $K_0(\NH, \dif_{\pm 1})$, whose counit is given by projection onto the
$p$-DG subalgebra $\NH_0 \subset \NH$. Furthermore, a direct computation shows that
\[[^X\Res]:K_0(\NH,\dif_{\pm 1})\lra K_0(\NH,\dif_{\pm 1})\o_{\mathbb{O}_p}K_0(\NH,\dif_{\pm 1})\]
is a map of algebras, where the algebra structure on
$K_0(\NH,\dif_{\pm 1})\o_{\mathbb{O}_p}K_0(\NH,\dif_{\pm 1})$ is twisted in the
sense that for homogeneous elements $x_1, x_2, y_1, y_2\in K_0(\NH,\dif_{\pm 1})$,
\[
(x_1\o x_2)\cdot(y_1\o y_2)=q^{|x_2||y_1|}x_1y_1\o x_2y_2.
\]

\begin{thm}\label{thm-p-DG-nilHecke-categorifies-small-sl2}
There is an isomorphism of twisted bialgebras
\[
K_0(\NH,\dif_{\pm 1})\ \cong u^+_{\mathbb{O}_{p}},
\]
under which the symbols $[\Ind]$, $[{^X\Res}]$ of the functors $\Ind$, $^X\Res$ are
identified with the multiplication and comultiplication in $u^+_{\mathbb{O}_{p}}(\mathfrak{sl}_2)$,
and the symbol of the polynomial representation $[\bpol_n^+]$ is identified with
the $n$-th divided power $E^{(n)}$.
$\hfill\square$
\end{thm}

\begin{rmk}[Degree one and zero differentials] Base changing $\mathbb{O}_p$ to
$\mc{O}_{2p}$ in Theorem \ref{thm-p-DG-nilHecke-categorifies-small-sl2} gives rise
to an isomorphism of twisted $\mc{O}_{2p}$-bialgebras
\[
K_0(\NH,\dif_{\pm 1})\o_{\mathbb{O}_p}\mc{O}_{2p}\cong u^+_{\mc{O}_{2p}},
\]
Here we sketch how one can categorify $u^+_{\mc{O}_{2p}}$ directly by rescaling the gradings involved.

Since the nilHecke algebra is concentrated in even degrees, one may restrict the
category of modules we considered before to the full subcategory of $\NH$-modules
which are only concentrated in even degrees. Alternatively, rescale the gradings of
$x_i$ and $\delta_i$ to be $1$ and $-1$ respectively, and make the differential
$\dif_a$ to be of degree one. All the earlier results are valid except that we do not
have balanced indecomposable $H$-modules $\widetilde{V_k}$ any more. The stable
Grothendieck ring of the Hopf algebra $H=\Bbbk[\dif]/(\dif^p)$ is isomorphic to the
ring of $p$-th cyclotomic integers $\mc{O}_p\cong \mc{O}_{2p}$, and the construction of this section
goes through without essential change, leading to a categorification of the integral
form $u^+_{\mc{O}_{2p}}$ over $\mc{O}_{2p}$.

Suppose that, instead, we forget about all the gradings involved. The stable category
of finite dimensional ungraded $H$-modules has its Grothendieck ring isomorphic to
$\F_p$. In parallel with the above story, the ungraded compact module category over
the ungraded $p$-DG algebra $(\NH,\dif_{\pm 1})$  categorifies the restricted universal
enveloping algebra $u^+(\mf{sl}_2)\cong \F_p[E]/(E^p)$.
\end{rmk}

%%%%%%%%%%%%%%%%%%%%%%%%%%%%%%%%%%%%%
%%%%%%%%%%%%%%%%%%%%%%%%%%%%%%%%%%%%%
%%
%%  p-derivations on KLR algebras
%%
%%%%%%%%%%%%%%%%%%%%%%%%%%%%%%%%%%%%%
%%%%%%%%%%%%%%%%%%%%%%%%%%%%%%%%%%%%%

\section{The \texorpdfstring{$p$}{p}-DG KLR algebra}\label{sec-KLR}

\newstrandstyle{red}{color=red}
\newstrandstyle{green}{color=green}
\newstrandstyle{blue}{color=blue}

\subsection{\texorpdfstring{$p$}{p}-derivations on the KLR algebra}
\paragraph{The KLR algebras.}The KLR algebras associated with any Cartan datum were
introduced in \cite{KL1, KL2, Rou} as a generalization of nilHecke algebras to
categorify one-half of the corresponding quantum group.

Let $\Gamma$ be an oriented simply-laced quiver whose vertices are indexed by $I$
($I$ is also referred to as the set of \emph{colors}). Such a quiver determines
a Cartan inner product $\cdot: \Z^I \times \Z^I \lra \Z$ defined by $i\cdot i=2 $
for all $i\in I$, $i\cdot j=-1$ if $i$ and $j$ are connected by one edge, and
$i\cdot j = 0$ otherwise. The KLR
$\Bbbk$-algebra $R(\Gamma)$ associated to $\Gamma$ has the following local diagrammatic
presentation. It is generated by braid-like planar diagrams with strands
colored by $I$ and carrying dots, subject to the following list of local relations.
\begin{itemize}
\item The usual nilHecke relations among same-color crossings and dots.
\item Dots slide through crossings of different colors, i. e. for $i, j\in I$ with $i \neq j$,

\[
\begin{DGCpicture}
\DGCstrand(0,0)(1,1)[$i$]\DGCdot{.25} \DGCstrand(1,0)(0,1)[$j$]
\end{DGCpicture}=
\begin{DGCpicture}
\DGCstrand(0,0)(1,1)[$i$]\DGCdot{.75} \DGCstrand(1,0)(0,1)[$j$]
\end{DGCpicture},
\ \ \ \
\begin{DGCpicture}
\DGCstrand(0,0)(1,1)[$i$] \DGCstrand(1,0)(0,1)[$j$]\DGCdot{.75}
\end{DGCpicture}=\begin{DGCpicture}
\DGCstrand(0,0)(1,1)[$i$] \DGCstrand(1,0)(0,1)[$j$]\DGCdot{.25}
\end{DGCpicture}.
\]

\item Reidemeister II type relations between strands,
\[
\begin{DGCpicture}
\DGCstrand(0,0)(1,1)(0,2)[$i$]
\DGCstrand(1,0)(0,1)(1,2)[$j$]
\end{DGCpicture}
=
\left\{
\begin{array}{cc}
0 & \textrm{if $i=j$,}\\
& \\
\begin{DGCpicture}
\DGCstrand(0,0)(0,1)[$i$]
\DGCstrand(1,0)(1,1)[$j$]
\end{DGCpicture}
&
\textrm{if $i\cdot j=0$,}\\
& \\
\begin{DGCpicture}
\DGCstrand(0,0)(0,1)[$i$]
\DGCdot{0.5}
\DGCstrand(1,0)(1,1)[$j$]
\end{DGCpicture}-
\begin{DGCpicture}
\DGCstrand(0,0)(0,1)[$i$]
\DGCstrand(1,0)(1,1)[$j$]
\DGCdot{0.5}
\end{DGCpicture}&
\textrm{if $i\lra j$,}\\
& \\
\begin{DGCpicture}
\DGCstrand(0,0)(0,1)[$i$]
\DGCstrand(1,0)(1,1)[$j$]
\DGCdot{0.5}
\end{DGCpicture}-
\begin{DGCpicture}
\DGCstrand(0,0)(0,1)[$i$]
\DGCdot{0.5}
\DGCstrand(1,0)(1,1)[$j$]
\end{DGCpicture}
&
\textrm{if $i\longleftarrow j$.}
\end{array}
\right.
\]
\item Reidemeister III type relations
\[
\begin{DGCpicture}[scale=0.75]
\DGCstrand(0,0)(2,2)[$i$]
\DGCstrand(2,0)(0,2)[$k$]
\DGCstrand(1,0)(0,1)(1,2)[$j$]
\end{DGCpicture}
-
\begin{DGCpicture}[scale=0.75]
\DGCstrand(0,0)(2,2)[$i$]
\DGCstrand(2,0)(0,2)[$k$]
\DGCstrand(1,0)(2,1)(1,2)[$j$]
\end{DGCpicture}
=
\left\{
\begin{array}{cc}
0 & \textrm{if $i\neq k$ or $i\cdot j\neq -1$,}\\
& \\
~~\begin{DGCpicture}[scale=0.75]
\DGCstrand(0,0)(0,2)[$i$]
\DGCstrand(2,0)(2,2)[$i$]
\DGCstrand(1,0)(1,2)[$j$]
\end{DGCpicture} &
\textrm{if $i=k$ and $i \lra j$,}\\
& \\
-~\begin{DGCpicture}[scale=0.75]
\DGCstrand(0,0)(0,2)[$i$]
\DGCstrand(2,0)(2,2)[$i$]
\DGCstrand(1,0)(1,2)[$j$]
\end{DGCpicture} &
\textrm{if $i=k$ and $i \longleftarrow j$.}
\end{array}
\right.
\]
\end{itemize}
$R(\Gamma)$ naturally decomposes into blocks $R(\Gamma)=\oplus_{\nu \in \N [I]}R(\nu)$,
where $\nu=\sum_{i \in I} \nu_i i$ are called weights. Each $R(\nu)$ is spanned by
diagrams consisting of $|\nu|:=\sum_{i\in I} \nu_i$ strands, $\nu_i$ of which are colored
by $i$. For each weight $\nu=\sum_{i\in I}\nu_ii\in \N[I]$, let $\mathrm{Seq}(\nu)$
be the set of sequences of colors $\mathbf{i}\in I^{m}$, where $m=|\nu|$. There are
idempotents
$1_{\mathbf{i}}\in R(\nu)$ for each $\mathbf{i}=(i_1, i_2,\dots, i_m)\in \mathrm{Seq}(\nu)$,
depicted by
\[
1_{\mathbf{i}}:= \
\begin{DGCpicture}
\DGCstrand(0,0)(0,1)[$i_1$]
\DGCstrand(1,0)(1,1)[$i_2$]
\DGCstrand(3,0)(3,1)[$i_m$]
\DGCcoupon*(1.25,0.35)(2.75,0.65){$\cdots$}
\end{DGCpicture},
\]
and,  as a graded vector space,
$R(\nu)= \oplus_{\mathbf{i},\mathbf{j}\in \mathrm{Seq}(\nu)}{_{\mathbf{i}}R(\nu)_{\mathbf{j}}}$,
where $_{\mathbf{i}}R(\nu)_\mathbf{j}:=1_{\mathbf{i}}R(\nu)1_{\mathbf{j}}$.

For more details about these rings, see \cite{KL1, KL2, Rou}. Notice that in the
previous section we labelled strands by numbers $1, 2, \dots$  starting from top
left of a diagram. Now we suppress this notation and only keep track of
the strands' colors, elements of $I$, written at the bottom endpoints of strands.

The algebra $R(\Gamma)$ has the following symmetries extending those of nilHecke
subalgebras. Reflecting a diagram about a horizontal axis induces
an algebra anti-involution of $R(\Gamma)$, denoted $\psi$:
\begin{equation}\label{eqn-KLR-symmetry-psi}
\psi\left(~
\begin{DGCpicture}[scale=0.75]
\DGCstrand[green](0,0)(1,2)[$j$]
\DGCstrand[blue](1,0)(2,2)[$k$]
\DGCstrand[red](2,0)(0,2)[$i$]\DGCdot{1.65}
\end{DGCpicture}
~\right)=
\begin{DGCpicture}[scale=0.75]
\DGCstrand[red](0,0)(2,2)[$i$]\DGCdot{0.35}
\DGCstrand[green](1,0)(0,2)[$j$]
\DGCstrand[blue](2,0)(1,2)[$k$]
\end{DGCpicture}.
\end{equation}
Reflecting a diagram about a vertical axis and simultaneously multiplying it by $(-1)$
for each same-color crossing induces an algebra involution of $R(\Gamma)$, denoted $\sigma$:
\begin{equation}\label{eqn-KLR-symmetry-sigma}
\sigma\left(~
\begin{DGCpicture}[scale=0.75]
\DGCstrand[green](0,0)(1,2)[$j$]
\DGCstrand[red](1,0)(2,2)[$i$]
\DGCstrand[red](2,0)(0,2)[$i$]\DGCdot{1.65}
\end{DGCpicture}
~\right)=-
\begin{DGCpicture}[scale=0.75]
\DGCstrand[red](0,0)(2,2)[$i$]\DGCdot{1.65}
\DGCstrand[red](1,0)(0,2)[$i$]
\DGCstrand[green](2,0)(1,2)[$j$]
\end{DGCpicture}.
\end{equation}

The KLR algebra $R(\Gamma)$ acts faithfully on a direct sum of polynomial spaces,
see \cite[Section 2.3]{KL1} for more details. The symmetric group
$S_m$ acts on $I^m$ by permutations of sequences. Define a vector space
\[
\bpol_\nu := \bigoplus_{\mathbf{i}\in \mathrm{Seq}(\nu)}\bpol_\mathbf{i},
\ \ \ \ \bpol_{\mathbf{i}}:=\Bbbk[x_1(\mathbf{i}),\dots,
x_m(\mathbf{i})]\cdot 1_{\mathbf{i}}.
\]
$\bpol_\nu$ inherits an action of $S_m$ by setting
$w(x_t(\mathbf{i}))=x_{w(t)}(w(\mathbf{i}))$ for any $1\leq t \leq m$.
The algebra $R(\nu)$ acts on $\bpol_\nu$ as follows.
First off, $_{\mathbf{j}}R(\nu)_{\mathbf{i}}$ acts on $\bpol_\mathbf{k}$
by zero unless $\mathbf{i}=\mathbf{k}$ and takes $\bpol_\mathbf{k}$
to $\bpol_\mathbf{j}$ when $\mathbf{i}=\mathbf{k}=(i_1,\dots,i_t,i_{t+1},\dots,i_m)$.
The local generators of $R(\nu)$
\[
x_t(\mathbf{i}):= \
\begin{DGCpicture}
\DGCstrand(0,0)(0,1)[$i_1$]
\DGCstrand(1.5,0)(1.5,1)[$i_t$]
\DGCdot{0.5}
\DGCstrand(3,0)(3,1)[$i_m$]
\DGCcoupon*(0.25,0.35)(1.25,0.65){$\cdots$}
\DGCcoupon*(1.75,0.35)(2.75,0.65){$\cdots$}
\end{DGCpicture},  \ \ \ \
\delta_t(\mathbf{i}):= \
\begin{DGCpicture}
\DGCstrand(0,0)(0,1)[$i_1$]
\DGCstrand(1,0)(2,1)[$i_t$]
\DGCstrand(2,0)(1,1)[$i_{t+1}$]
\DGCstrand(3,0)(3,1)[$i_m$]
\DGCcoupon*(0.15,0.35)(0.85,0.65){$\cdots$}
\DGCcoupon*(2.15,0.35)(2.85,0.65){$\cdots$}
\end{DGCpicture},
\]
act respectively on $f \in \bpol_{\mathbf{i}}$ by multiplying $f$ by $x_t(\mathbf{i})$ and
\[
\delta_t(\mathbf{i})(f):= \left\{
\begin{array}{ll}
{^tf} & \textrm{if $i_t \cdot i_{t+1}=0$},\\
& \\
\dfrac{f-{^tf}}{x_t(\mathbf{i})-x_{t+1}(\mathbf{i})} & \textrm{if $i_t = i_{t+1}$},\\
& \\
(x_{t+1}(s_k(\mathbf{i}))-x_t(s_k(\mathbf{i})))\cdot({^tf}) & \textrm{if $i_t\lra i_{t+1}$},\\
& \\
{^tf} & \textrm{if $i_t\longleftarrow i_{t+1}$},
\end{array}
\right.
\]
where $^tf$ denotes $f$ with the $t$-th and $(t+1)$-st colored-variables switched. One
should not confuse the polynomial representation with the polynomial subalgebra sitting
naturally inside $R(\nu)$. In particular, the lowest degree element of $ \bpol_\mathbf{i}$,
unique up to a non-zero constant and denoted by $1_\mathbf{i}$, should not be confused
with the idempotent $1_\mathbf{i}\in {_\mathbf{i}R(\nu)_\mathbf{i}}$.

\begin{eg}[$A_2$]\label{eg-KLR-relation-A2-case} For the quiver
\[
\xymatrix{\underset{i}{\textcolor[rgb]{1.00,0.00,0.00}{\bullet}}\ar[r]&\underset{j}{\textcolor[rgb]{0.00,1.00,0.00}{\bullet}}},
\]
the $R(\nu)$-relations translate into
\begin{itemize}
\item The usual nilHecke relations among strands and dots of the same color.

\item Dots slide through $i-j$ or $j-i$ crossings:
\[
\begin{DGCpicture}
\DGCstrand[red](0,0)(1,1)[$i$]\DGCdot{.25}
\DGCstrand[green](1,0)(0,1)[$j$]
\end{DGCpicture}=
\begin{DGCpicture}
\DGCstrand[red](0,0)(1,1)[$i$]\DGCdot{.75}
\DGCstrand[green](1,0)(0,1)[$j$]
\end{DGCpicture},
\ \ \ \
\begin{DGCpicture}
\DGCstrand[green](0,0)(1,1)[$j$]
\DGCstrand[red](1,0)(0,1)[$i$]\DGCdot{.75}
\end{DGCpicture}=
\begin{DGCpicture}
\DGCstrand[green](0,0)(1,1)[$j$]
\DGCstrand[red](1,0)(0,1)[$i$]\DGCdot{.25}
\end{DGCpicture},
\]

\[
\begin{DGCpicture}
\DGCstrand[green](0,0)(1,1)[$j$]\DGCdot{.25}
\DGCstrand[red](1,0)(0,1)[$i$]
\end{DGCpicture}=
\begin{DGCpicture}
\DGCstrand[green](0,0)(1,1)[$j$]\DGCdot{.75}
\DGCstrand[red](1,0)(0,1)[$i$]
\end{DGCpicture},
\ \ \ \
\begin{DGCpicture}
\DGCstrand[red](0,0)(1,1)[$i$]
\DGCstrand[green](1,0)(0,1)[$j$]\DGCdot{.75}
\end{DGCpicture}=
\begin{DGCpicture}
\DGCstrand[red](0,0)(1,1)[$i$]
\DGCstrand[green](1,0)(0,1)[$j$]\DGCdot{.25}
\end{DGCpicture}.
\]

\item Two additional Reidemeister II type relations for the new crossings:
\[
\begin{DGCpicture}
\DGCstrand[red](0,0)(1,1)(0,2)[$i$]
\DGCstrand[green](1,0)(0,1)(1,2)[$j$]
\end{DGCpicture}
\ = \
\begin{DGCpicture}
\DGCstrand[red](0,0)(0,2)[$i$] \DGCdot{1.01}
\DGCstrand[green](1,0)(1,2)[$j$]
\end{DGCpicture}-
\begin{DGCpicture}
\DGCstrand[red](0,0)(0,2)[$i$]
\DGCstrand[green](1,0)(1,2)[$j$]\DGCdot{1.01}
\end{DGCpicture} \ , \ \ \ \ \ \ \
\begin{DGCpicture}
\DGCstrand[green](0,0)(1,1)(0,2)[$j$]
\DGCstrand[red](1,0)(0,1)(1,2)[$i$]
\end{DGCpicture}
\ = \
\begin{DGCpicture}
\DGCstrand[green](0,0)(0,2)[$j$]
\DGCstrand[red](1,0)(1,2)[$i$]\DGCdot{1.01}
\end{DGCpicture}-
\begin{DGCpicture}
\DGCstrand[green](0,0)(0,2)[$j$]\DGCdot{1.01}
\DGCstrand[red](1,0)(1,2)[$i$]
\end{DGCpicture} \ .
\]

\item Two special Reidemeister III type relations:
\[
\begin{DGCpicture}[scale=0.75]
\DGCstrand[red](0,0)(2,2)[$i$]
\DGCstrand[red](2,0)(0,2)[$i$]
\DGCstrand[green](1,0)(0,1)(1,2)[$j$]
\end{DGCpicture}
-
\begin{DGCpicture}[scale=0.75]
\DGCstrand[red](0,0)(2,2)[$i$]
\DGCstrand[red](2,0)(0,2)[$i$]
\DGCstrand[green](1,0)(2,1)(1,2)[$j$]
\end{DGCpicture}\ = \
\begin{DGCpicture}[scale=0.75]
\DGCstrand[red](0,0)(0,2)[$i$]
\DGCstrand[red](2,0)(2,2)[$i$]
\DGCstrand[green](1,0)(1,2)[$j$]
\end{DGCpicture},
\]

\[
\begin{DGCpicture}[scale=0.75]
\DGCstrand[green](0,0)(2,2)[$j$]
\DGCstrand[green](2,0)(0,2)[$j$]
\DGCstrand[red](1,0)(0,1)(1,2)[$i$]
\end{DGCpicture}
-
\begin{DGCpicture}[scale=0.75]
\DGCstrand[green](0,0)(2,2)[$j$]
\DGCstrand[green](2,0)(0,2)[$j$]
\DGCstrand[red](1,0)(2,1)(1,2)[$i$]
\end{DGCpicture}  = -
\begin{DGCpicture}[scale=0.75]
\DGCstrand[green](0,0)(0,2)[$j$]
\DGCstrand[green](2,0)(2,2)[$j$]
\DGCstrand[red](1,0)(1,2)[$i$]
\end{DGCpicture},
\]

\vspace{0.1in}

and the other Reidemeister III type relations:

\[
\begin{DGCpicture}[scale=0.75]
\DGCstrand[green](0,0)(2,2)[$j$]
\DGCstrand[red](2,0)(0,2)[$i$]
\DGCstrand[green](1,0)(0,1)(1,2)[$j$]
\end{DGCpicture}
=
\begin{DGCpicture}[scale=0.75]
\DGCstrand[green](0,0)(2,2)[$j$]
\DGCstrand[red](2,0)(0,2)[$i$]
\DGCstrand[green](1,0)(2,1)(1,2)[$j$]
\end{DGCpicture} \ , \ \ \
\begin{DGCpicture}[scale=0.75]
\DGCstrand[green](0,0)(2,2)[$j$]
\DGCstrand[red](2,0)(0,2)[$i$]
\DGCstrand[red](1,0)(0,1)(1,2)[$i$]
\end{DGCpicture}
=
\begin{DGCpicture}[scale=0.75]
\DGCstrand[green](0,0)(2,2)[$j$]
\DGCstrand[red](2,0)(0,2)[$i$]
\DGCstrand[red](1,0)(2,1)(1,2)[$i$]
\end{DGCpicture} \ ,
\]

\vspace{0.1in}

\[
\begin{DGCpicture}[scale=0.75]
\DGCstrand[red](0,0)(2,2)[$i$]
\DGCstrand[green](2,0)(0,2)[$j$]
\DGCstrand[green](1,0)(0,1)(1,2)[$j$]
\end{DGCpicture}
=
\begin{DGCpicture}[scale=0.75]
\DGCstrand[red](0,0)(2,2)[$i$]
\DGCstrand[green](2,0)(0,2)[$j$]
\DGCstrand[green](1,0)(2,1)(1,2)[$j$]
\end{DGCpicture} \ , \ \ \
\begin{DGCpicture}[scale=0.75]
\DGCstrand[red](0,0)(2,2)[$i$]
\DGCstrand[green](2,0)(0,2)[$j$]
\DGCstrand[red](1,0)(0,1)(1,2)[$i$]
\end{DGCpicture}
=
\begin{DGCpicture}[scale=0.75]
\DGCstrand[red](0,0)(2,2)[$i$]
\DGCstrand[green](2,0)(0,2)[$j$]
\DGCstrand[red](1,0)(2,1)(1,2)[$i$]
\end{DGCpicture} \ .
\]
\end{itemize}
\end{eg}

\begin{eg}[$A_1\times A_1$]\label{eg-KLR-relation-A1-times-A1-case} This corresponds to the quiver with two vertices and no arrows:
\[
\xymatrix{\underset{i}{\textcolor[rgb]{1.00,0.00,0.00}{\bullet}} & \underset{k}{\textcolor[rgb]{0.00,0.00,1.00}{\bullet}}}.
\]
Such pairs of vertices in a Dynkin diagram are referred to as \emph{distant vertices}. The relations for this case are listed below.
\begin{itemize}
\item The usual nilHecke relations among strands and dots of the same color.

\item Dots slide through $i-k$ and $k-i$ crossings:
\[
\begin{DGCpicture}
\DGCstrand[red](0,0)(1,1)[$i$]\DGCdot{.25}
\DGCstrand[blue](1,0)(0,1)[$k$]
\end{DGCpicture}=
\begin{DGCpicture}
\DGCstrand[red](0,0)(1,1)[$i$]\DGCdot{.75}
\DGCstrand[blue](1,0)(0,1)[$k$]
\end{DGCpicture},
\ \ \ \
\begin{DGCpicture}
\DGCstrand[blue](0,0)(1,1)[$k$]
\DGCstrand[red](1,0)(0,1)[$i$]\DGCdot{.75}
\end{DGCpicture}=
\begin{DGCpicture}
\DGCstrand[blue](0,0)(1,1)[$k$]
\DGCstrand[red](1,0)(0,1)[$i$]\DGCdot{.25}
\end{DGCpicture},
\]

\[
\begin{DGCpicture}
\DGCstrand[blue](0,0)(1,1)[$k$]\DGCdot{.25}
\DGCstrand[red](1,0)(0,1)[$i$]
\end{DGCpicture}=
\begin{DGCpicture}
\DGCstrand[blue](0,0)(1,1)[$k$]\DGCdot{.75}
\DGCstrand[red](1,0)(0,1)[$i$]
\end{DGCpicture},
\ \ \ \
\begin{DGCpicture}
\DGCstrand[red](0,0)(1,1)[$i$]
\DGCstrand[blue](1,0)(0,1)[$k$]\DGCdot{.75}
\end{DGCpicture}=
\begin{DGCpicture}
\DGCstrand[red](0,0)(1,1)[$i$]
\DGCstrand[blue](1,0)(0,1)[$k$]\DGCdot{.25}
\end{DGCpicture}.
\]

\item Two additional Reidemeister II type relations for the new crossings:
\[
\begin{DGCpicture}
\DGCstrand[red](0,0)(1,1)(0,2)[$i$]
\DGCstrand[blue](1,0)(0,1)(1,2)[$k$]
\end{DGCpicture}
\ = \
\begin{DGCpicture}
\DGCstrand[red](0,0)(0,2)[$i$]
\DGCstrand[blue](1,0)(1,2)[$k$]
\end{DGCpicture}, \ \ \ \ \ \ \ \
\begin{DGCpicture}
\DGCstrand[blue](0,0)(1,1)(0,2)[$k$]
\DGCstrand[red](1,0)(0,1)(1,2)[$i$]
\end{DGCpicture}
\ = \
\begin{DGCpicture}
\DGCstrand[blue](0,0)(0,2)[$k$]
\DGCstrand[red](1,0)(1,2)[$i$]
\end{DGCpicture}.
\]

\item Reidemeister III type relations:
\[
\begin{DGCpicture}[scale=0.75]
\DGCstrand[red](0,0)(2,2)[$i$]
\DGCstrand[red](2,0)(0,2)[$i$]
\DGCstrand[blue](1,0)(0,1)(1,2)[$k$]
\end{DGCpicture}
=
\begin{DGCpicture}[scale=0.75]
\DGCstrand[red](0,0)(2,2)[$i$]
\DGCstrand[red](2,0)(0,2)[$i$]
\DGCstrand[blue](1,0)(2,1)(1,2)[$k$]
\end{DGCpicture} \ , \ \ \
\begin{DGCpicture}[scale=0.75]
\DGCstrand[blue](0,0)(2,2)[$k$]
\DGCstrand[blue](2,0)(0,2)[$k$]
\DGCstrand[red](1,0)(0,1)(1,2)[$i$]
\end{DGCpicture}
=
\begin{DGCpicture}[scale=0.75]
\DGCstrand[blue](0,0)(2,2)[$k$]
\DGCstrand[blue](2,0)(0,2)[$k$]
\DGCstrand[red](1,0)(2,1)(1,2)[$i$]
\end{DGCpicture}\ , \ \ \
\begin{DGCpicture}[scale=0.75]
\DGCstrand[red](0,0)(2,2)[$i$]
\DGCstrand[blue](2,0)(0,2)[$k$]
\DGCstrand[red](1,0)(0,1)(1,2)[$i$]
\end{DGCpicture}
=
\begin{DGCpicture}[scale=0.75]
\DGCstrand[red](0,0)(2,2)[$i$]
\DGCstrand[blue](2,0)(0,2)[$k$]
\DGCstrand[red](1,0)(2,1)(1,2)[$i$]
\end{DGCpicture},
\]
\[
\begin{DGCpicture}[scale=0.75]
\DGCstrand[blue](0,0)(2,2)[$k$]
\DGCstrand[red](2,0)(0,2)[$i$]
\DGCstrand[blue](1,0)(0,1)(1,2)[$k$]
\end{DGCpicture}
=
\begin{DGCpicture}[scale=0.75]
\DGCstrand[blue](0,0)(2,2)[$k$]
\DGCstrand[red](2,0)(0,2)[$i$]
\DGCstrand[blue](1,0)(2,1)(1,2)[$k$]
\end{DGCpicture} \ , \ \ \
\begin{DGCpicture}[scale=0.75]
\DGCstrand[blue](0,0)(2,2)[$k$]
\DGCstrand[red](2,0)(0,2)[$i$]
\DGCstrand[red](1,0)(0,1)(1,2)[$i$]
\end{DGCpicture}
=
\begin{DGCpicture}[scale=0.75]
\DGCstrand[blue](0,0)(2,2)[$k$]
\DGCstrand[red](2,0)(0,2)[$i$]
\DGCstrand[red](1,0)(2,1)(1,2)[$i$]
\end{DGCpicture}
\ , \ \ \
\begin{DGCpicture}[scale=0.75]
\DGCstrand[red](0,0)(2,2)[$i$]
\DGCstrand[blue](2,0)(0,2)[$k$]
\DGCstrand[blue](1,0)(0,1)(1,2)[$k$]
\end{DGCpicture}
=
\begin{DGCpicture}[scale=0.75]
\DGCstrand[red](0,0)(2,2)[$i$]
\DGCstrand[blue](2,0)(0,2)[$k$]
\DGCstrand[blue](1,0)(2,1)(1,2)[$k$]
\end{DGCpicture}.
\]
\end{itemize}
\end{eg}

\paragraph{Differentials on KLR algebras.} The polynomial representation $\bpol_\nu$ of
$R(\nu)$ has a $p$-DG module structure specified as follows.
Firstly, let $\pol_\nu:=\bpol_\nu$ with the obvious algebra structure together with with
the $p$-differential defined by $\dif(x_t(\mathbf{i}))=x_t^2(\mathbf{i})$ for all
$\mathbf{i}\in \mathrm{Seq}(\nu)$ and $1\leq t \leq |\nu|$. $(\pol_\nu,\dif)$ is then
a $p$-DG algebra. Choose a family of linear polynomials $g_{\alpha}(\mathbf{i}):=
\sum_{t=1}^{m} \alpha_t(\mathbf{i})x_t(\mathbf{i})$, one for each
$\mathbf{i}\in \mathrm{Seq}(\nu)$, where $\alpha_t(\mathbf{i})\in \F_p$. Let
$\dif_\alpha$ act on an element $f1_\mathbf{i}\in \bpol_\mathbf{i}$ by
\[
\dif_\alpha(f1_\mathbf{i}):=\dif(f)1_\mathbf{i}+fg_{\alpha}(\mathbf{i})1_\mathbf{i}.
\]
We denote this $p$-DG $\pol_\nu$-module by $\bpol_\nu(\alpha)$.
Unlike in the $A_1$ case, we do not make a particular (balanced) choice of the degree to
assign to $1_{\mathbf{i}}\in \bpol_\nu(\alpha)$. This degree will depend on $\mathbf{i}$,
and will be uniquely determined by the choice of degree for any single sequence of labels
in $\mathrm{Seq}(\nu)$.

As in the $A_1$ case, a $p$-differential on $\bpol_\nu$ induces a natural $p$-differential
on the endomorphism algebra $R(\nu)$. We first investigate these differentials on
KLR algebras associated to the only two simply-laced rank-two Cartan data: $A_2$ and
$A_1 \times A_1$, starting with the $A_2$ case.

Consider the polynomial module $\bpol_{i+j}\cong \Bbbk[x_1,x_2]1_{ij}\oplus
\Bbbk[x_1,x_2]1_{ji}$ (we use the shorthand notation $x_1$ for $x_1(ij)$ and $x_1(ji)$ etc.).
Let
\begin{equation}
\dif_\alpha(1_{ij})  = \alpha_1(ij)x_11_{ij}+\alpha_2(ij)x_21_{ij}, \ \ \ \
\dif_\alpha(1_{ji})  =  \alpha_1(ji)x_11_{ji}+\alpha_2(ji)x_21_{ji}.
\end{equation}
An easy computation shows that, diagrammatically, the induced differential on
$R(i+j)$ is given on the $i-j$ and $j-i$ crossings by
\[
\dif_\alpha\left(~
\begin{DGCpicture}
\DGCstrand[red](0,0)(1,1)[$i$]
\DGCstrand[green](1,0)(0,1)[$j$]
\end{DGCpicture}
~\right)=
(1+\alpha_1(ji)-\alpha_2(ij))\begin{DGCpicture}
\DGCstrand[red](0,0)(1,1)[$i$]
\DGCstrand[green](1,0)(0,1)[$j$]\DGCdot{0.75}
\end{DGCpicture}+
(1+\alpha_2(ji)-\alpha_1(ij))\begin{DGCpicture}
\DGCstrand[red](0,0)(1,1)[$i$]\DGCdot{0.75}
\DGCstrand[green](1,0)(0,1)[$j$]
\end{DGCpicture},
\]

\[
\dif_\alpha\left(~
\begin{DGCpicture}
\DGCstrand[green](0,0)(1,1)[$j$]
\DGCstrand[red](1,0)(0,1)[$i$]
\end{DGCpicture}
~\right)=
(\alpha_1(ij)-\alpha_2(ji)){
\begin{DGCpicture}
\DGCstrand[green](0,0)(1,1)[$j$]
\DGCstrand[red](1,0)(0,1)[$i$]\DGCdot{0.75}
\end{DGCpicture}}+
(\alpha_2(ij)-\alpha_1(ji)) {
\begin{DGCpicture}
\DGCstrand[green](0,0)(1,1)[$j$]\DGCdot{0.75}
\DGCstrand[red](1,0)(0,1)[$i$]
\end{DGCpicture}}.
\]
We would like the differential to be local, as in Section~\ref{sec-nilHecke}
and in the sense explained earlier. Let
$r_{ij}:=1+\alpha_1(ji)-\alpha_2(ij)$ and $r_{ji}:=\alpha_1(ij)-\alpha_2(ji)$.
Define a four-parameter family of differentials
$\dif=\dif(a_i, a_j, r_{ij}, r_{ji})$ on
$R(\underset{i}{\textcolor[rgb]{1.00,0.00,0.00}{\bullet}}\rightarrow
\underset{j}{\textcolor[rgb]{0.00,1.00,0.00}{\bullet}})$ as follows:
\[
\dif\left(~
{\begin{DGCpicture}
\DGCstrand[red](0,0)(0,1)[$i$]\DGCdot{0.5}
\end{DGCpicture}}~
\right)=
{\begin{DGCpicture}
\DGCstrand[red](0,0)(0,1)[$i$]\DGCdot{0.5}[r]{$^2$}
\end{DGCpicture}}\ \ , \ \ \ \
\dif \left(~
{\begin{DGCpicture}
\DGCstrand[green](0,0)(0,1)[$j$]\DGCdot{0.5}
\end{DGCpicture}}
~\right)=
{\begin{DGCpicture}
\DGCstrand[green](0,0)(0,1)[$j$]\DGCdot{0.5}[r]{$^2$}
\end{DGCpicture}}\ \ ,
\]

\[
\dif\left(
\begin{DGCpicture}
\DGCstrand[red](0,0)(1,1)[$i$]
\DGCstrand[red](1,0)(0,1)[$i$]
\end{DGCpicture}\right)=
a_i\begin{DGCpicture}
\DGCstrand[red](0,0)(0,1)[$i$]
\DGCstrand[red](1,0)(1,1)[$i$]
\end{DGCpicture}
-(a_i+1)
\begin{DGCpicture}
\DGCstrand[red](0,0)(1,1)[$i$]
\DGCstrand[red](1,0)(0,1)[$i$]\DGCdot{0.75}
\end{DGCpicture}
+(a_i-1)
\begin{DGCpicture}
\DGCstrand[red](0,0)(1,1)[$i$]\DGCdot{0.75}
\DGCstrand[red](1,0)(0,1)[$i$]
\end{DGCpicture},
\]

\[
\dif\left(
\begin{DGCpicture}
\DGCstrand[green](0,0)(1,1)[$j$]
\DGCstrand[green](1,0)(0,1)[$j$]
\end{DGCpicture}\right)=
a_j\begin{DGCpicture}
\DGCstrand[green](0,0)(0,1)[$j$]
\DGCstrand[green](1,0)(1,1)[$j$]
\end{DGCpicture}
-(a_j+1)
\begin{DGCpicture}
\DGCstrand[green](0,0)(1,1)[$j$]
\DGCstrand[green](1,0)(0,1)[$j$]\DGCdot{0.75}
\end{DGCpicture}
+(a_j-1)
\begin{DGCpicture}
\DGCstrand[green](0,0)(1,1)[$j$]\DGCdot{0.75}
\DGCstrand[green](1,0)(0,1)[$j$]
\end{DGCpicture},
\]

\[
\dif\left(
\begin{DGCpicture}
\DGCstrand[red](0,0)(1,1)[$i$]
\DGCstrand[green](1,0)(0,1)[$j$]
\end{DGCpicture}\right)=
r_{ij}\begin{DGCpicture}
\DGCstrand[red](0,0)(1,1)[$i$]
\DGCstrand[green](1,0)(0,1)[$j$]\DGCdot{0.75}
\end{DGCpicture}+
(1-r_{ji})\begin{DGCpicture}
\DGCstrand[red](0,0)(1,1)[$i$]\DGCdot{0.75}
\DGCstrand[green](1,0)(0,1)[$j$]
\end{DGCpicture},
\]

\[
\dif\left(
\begin{DGCpicture}
\DGCstrand[green](0,0)(1,1)[$j$]
\DGCstrand[red](1,0)(0,1)[$i$]
\end{DGCpicture}
\right)=
r_{ji}{
\begin{DGCpicture}
\DGCstrand[green](0,0)(1,1)[$j$]
\DGCstrand[red](1,0)(0,1)[$i$]\DGCdot{0.75}
\end{DGCpicture}}+
(1-r_{ij}) {
\begin{DGCpicture}
\DGCstrand[green](0,0)(1,1)[$j$]\DGCdot{0.75}
\DGCstrand[red](1,0)(0,1)[$i$]
\end{DGCpicture}}.
\]
In this definition we allow the parameters $a_i,a_j,r_{ij},r_{ji}\in \Bbbk$.

\begin{lemma}\label{lemma-dif-on-sl3-parameter}The above differential
$\dif=\dif(a_i,a_j,r_{ij},r_{ji})$ satisfies $\dif^p=0$ on
$R(\underset{i}{\textcolor[rgb]{1.00,0.00,0.00}{\bullet}}\rightarrow\underset{j}{\textcolor[rgb]{0.00,1.00,0.00}{\bullet}})$ if and only if $a_i,a_j,r_{ij},r_{ji}\in \F_p$.
Any homogeneous degree two $p$-nilpotent local derivation on
$R(\underset{i}{\textcolor[rgb]{1.00,0.00,0.00}{\bullet}}\rightarrow\underset{j}{\textcolor[rgb]{0.00,1.00,0.00}{\bullet}})$ is of the form $\lambda\cdot\dif(a_i,a_j,r_{ij},r_{ji})$
for some $\lambda \in \Bbbk$ and $a_i,a_j,r_{ij},r_{ji}\in \F_p$.
\end{lemma}
\begin{proof}[Sketch of proof]When $a_i,a_j,r_{ij},r_{ji}\in \F_p$,
$\dif^p=0$ on
$R(\underset{i}{\textcolor[rgb]{1.00,0.00,0.00}{\bullet}}\rightarrow\underset{j}{\textcolor[rgb]{0.00,1.00,0.00}{\bullet}})$ since $\dif$
was induced from a $p$-DG module. Conversely,
Lemma~\ref{lemma-dif-p-equals-zero-nilhecke} shows that $\dif^p=0$ on the nilHecke
algebra generators of the same color $i$ or $j$ if and only if $a_i,a_j\in \F_p$.
A similar computation for $\dif^p$ as used in that lemma applied to the new generators
$\begin{DGCpicture}[scale=0.65]
\DGCstrand[green](0,0)(1,1)[$^j$]
\DGCstrand[red](1,0)(0,1)[$^i$]
\end{DGCpicture}$ and
$\begin{DGCpicture}[scale=0.65]
\DGCstrand[red](0,0)(1,1)[$^i$]
\DGCstrand[green](1,0)(0,1)[$^j$]
\end{DGCpicture}$
restricts $r_{ij}$, $r_{ji}$ to be in $\F_p$. The classification of all possible derivations follows by an easy but lengthy computation.
\end{proof}

The second case is that of the quiver with two distant vertices ($A_1\times A_1$).
\[
\xymatrix{
\underset{i}{\textcolor[rgb]{1.00,0.00,0.00}{\bullet}} &
\underset{k}{\textcolor[rgb]{0.00,0.00,1.00}{\bullet}}.
}
\]
On the polynomial module $\bpol_{i+k}\cong \Bbbk[x_1,x_2]1_{ik}\oplus \Bbbk[x_1,x_2]1_{ki}$, set
\begin{equation}
\dif_\alpha(1_{ik})  = \alpha_1(ik)x_11_{ik}+\alpha_2(ik)x_21_{ik}, \ \ \ \
\dif_\alpha(1_{ki})  =  \alpha_1(ki)x_11_{ki}+\alpha_2(ki)x_21_{ki}.
\end{equation}
The induced differential on $R(i+k)$ is given on the $i-k$, $k-i$ crossings by
\[
\dif\left(
\begin{DGCpicture}
\DGCstrand[red](0,0)(1,1)[$i$]
 \DGCstrand[blue](1,0)(0,1)[$k$]
\end{DGCpicture}\right)=
(\alpha_1(ki)-\alpha_2(ik))\begin{DGCpicture}
\DGCstrand[red](0,0)(1,1)[$i$]
\DGCstrand[blue](1,0)(0,1)[$k$]\DGCdot{0.75}
\end{DGCpicture}
+(\alpha_2(ki)-\alpha_1(ik))\begin{DGCpicture}
\DGCstrand[red](0,0)(1,1)[$i$]\DGCdot{0.75}
\DGCstrand[blue](1,0)(0,1)[$k$]
\end{DGCpicture},
\]

\[
\dif\left(
{\begin{DGCpicture}
\DGCstrand[blue](0,0)(1,1)[$k$]
\DGCstrand[red](1,0)(0,1)[$i$]
\end{DGCpicture}}\right)=
(\alpha_1(ik)-\alpha_2(ki)){\begin{DGCpicture}
\DGCstrand[blue](0,0)(1,1)[$k$]
\DGCstrand[red](1,0)(0,1)[$i$]\DGCdot{0.75}
\end{DGCpicture}}
+(\alpha_2(ik)-\alpha_1(ki)) {\begin{DGCpicture}
\DGCstrand[blue](0,0)(1,1)[$k$]\DGCdot{0.75}
\DGCstrand[red](1,0)(0,1)[$i$]
\end{DGCpicture}}.
\]
We denote $u_{ik}:=\alpha_1(ki)-\alpha_2(ik)$ and $u_{ki}:=\alpha_1(ik)-\alpha_2(ki)$,
and require that the differential is independent of the position of the crossings.
This results in a
six-para\-meter family ($\mu_i, \mu_k, a_i, a_k, u_{ik},u_{ki}\in \Bbbk$)
of local differentials:
\[
\dif
\left(~
{\begin{DGCpicture}
\DGCstrand[red](0,0)(0,1)[$i$]\DGCdot{0.5}
\end{DGCpicture}}
~\right)=\mu_i~
{\begin{DGCpicture}
\DGCstrand[red](0,0)(0,1)[$i$]
\DGCdot{0.5}[r]{$^2$}
\end{DGCpicture}} \ \ , \ \ \ \
\dif \left(~
{\begin{DGCpicture}
\DGCstrand[blue](0,0)(0,1)[$k$]
\DGCdot{0.5}
\end{DGCpicture}}
~\right)=\mu_k~
{\begin{DGCpicture}
\DGCstrand[blue](0,0)(0,1)[$k$]
\DGCdot{0.5}[r]{$^2$}
\end{DGCpicture}}\ \ ,
\]

\[
\dif\left(
\begin{DGCpicture}
\DGCstrand[red](0,0)(1,1)[$i$]
\DGCstrand[red](1,0)(0,1)[$i$]
\end{DGCpicture}
\right)=\mu_i\left(
a_i\begin{DGCpicture}
\DGCstrand[red](0,0)(0,1)[$i$]
\DGCstrand[red](1,0)(1,1)[$i$]
\end{DGCpicture}
-(a_i+1)
\begin{DGCpicture}
\DGCstrand[red](0,0)(1,1)[$i$]
\DGCstrand[red](1,0)(0,1)[$i$]\DGCdot{0.75}
\end{DGCpicture}
+(a_i-1)
\begin{DGCpicture}
\DGCstrand[red](0,0)(1,1)[$i$]\DGCdot{0.75}
\DGCstrand[red](1,0)(0,1)[$i$]
\end{DGCpicture}
\right),
\]

\[
\dif\left(
\begin{DGCpicture}
\DGCstrand[blue](0,0)(1,1)[$k$]
\DGCstrand[blue](1,0)(0,1)[$k$]
\end{DGCpicture}\right)=
\mu_k\left(a_k\begin{DGCpicture}
\DGCstrand[blue](0,0)(0,1)[$k$]
\DGCstrand[blue](1,0)(1,1)[$k$]
\end{DGCpicture}
-(a_k+1)
\begin{DGCpicture}
\DGCstrand[blue](0,0)(1,1)[$k$]
\DGCstrand[blue](1,0)(0,1)[$k$]\DGCdot{0.75}
\end{DGCpicture}
+(a_k-1)
\begin{DGCpicture}
\DGCstrand[blue](0,0)(1,1)[$k$]\DGCdot{0.75}
\DGCstrand[blue](1,0)(0,1)[$k$]
\end{DGCpicture}\right),
\]

\[
\dif\left(
\begin{DGCpicture}
\DGCstrand[red](0,0)(1,1)[$i$]
 \DGCstrand[blue](1,0)(0,1)[$k$]
\end{DGCpicture}\right)=
u_{ik}\begin{DGCpicture}
\DGCstrand[red](0,0)(1,1)[$i$]
\DGCstrand[blue](1,0)(0,1)[$k$]\DGCdot{0.75}
\end{DGCpicture}
-u_{ki}\begin{DGCpicture}
\DGCstrand[red](0,0)(1,1)[$i$]\DGCdot{0.75}
\DGCstrand[blue](1,0)(0,1)[$k$]
\end{DGCpicture},
\]

\[
\dif\left(
{\begin{DGCpicture}
\DGCstrand[blue](0,0)(1,1)[$k$]
\DGCstrand[red](1,0)(0,1)[$i$]
\end{DGCpicture}}\right)=
u_{ki}{\begin{DGCpicture}
\DGCstrand[blue](0,0)(1,1)[$k$]
\DGCstrand[red](1,0)(0,1)[$i$]\DGCdot{0.75}
\end{DGCpicture}}
-u_{ik} {\begin{DGCpicture}
\DGCstrand[blue](0,0)(1,1)[$k$]\DGCdot{0.75}
\DGCstrand[red](1,0)(0,1)[$i$]
\end{DGCpicture}}.
\]

The proof of the following lemma is entirely analogous to those of
Lemmas~\ref{lemma-dif-p-equals-zero-nilhecke}, \ref{lemma-dif-on-sl3-parameter}.

\begin{lemma}\label{lemma-dif-on-A1*A1-parameter}The above differential $\dif=\dif(a_i,a_k,u_{ik},u_{ki},\mu_i,\mu_k)$ satisfies $\dif^p=0$ on $R(\underset{i}{\textcolor[rgb]{1.00,0.00,0.00}{\bullet}}\quad \underset{k}{\textcolor[rgb]{0.00,0.00,1.00}{\bullet}})$ if and only if $a_i,a_k,u_{ik},u_{ki}\in \F_p$ while $\mu_i,\mu_k \in \Bbbk$. Any homogeneous degree two $p$-nilpotent local derivation on $R(\underset{i}{\textcolor[rgb]{1.00,0.00,0.00}{\bullet}}\quad \underset{k}{\textcolor[rgb]{0.00,0.00,1.00}{\bullet}})$ arises in this way. $\hfill\square$
\end{lemma}

\begin{rmk}By rescaling the differential on the nilHecke algebra, one could have specified that $\dif(x(i))=\mu_i x^2(i)$ by any $\mu_i\in \Bbbk$ for each color $i\in I$, and $\dif(\delta(i))$ would then be modified accordingly (as above, for the $i-i$ and $k-k$ crossings). Inside the KLR algebra, one might attempt to rescale the nilHecke differentials for each color separately. This can be accomplished when the colors are distant, but for adjacent colors the $A_2$ relations force the scalings to be the same. Therefore, in a connected simply-laced quiver, we may and will always assume that the scaling factor $\mu_i=1$ for each color $i\in I$.
\end{rmk}

In a similar way as in Proposition \ref{prop-symmetry-between-a-and-minus-a}, one shows that the symmetries $\psi$, $\sigma$ of $R(\Gamma)$ intertwine differentials.

\begin{prop}\label{prop-symmetry-intertwines-differentials-KLR}The following relations on the differentials hold.
\begin{itemize}
\item[(i)] On $R(\underset{i}{\textcolor[rgb]{1.00,0.00,0.00}{\bullet}}\rightarrow \underset{j}{\textcolor[rgb]{0.00,1.00,0.00}{\bullet}})$,
\[
\psi\circ\dif(a_i,a_j,r_{ij},r_{ji})\circ \psi  = \dif(-a_i, -a_j, 1-r_{ij}, 1-r_{ji}),
\]
\[
\sigma\circ\dif(a_i,a_j,r_{ij},r_{ji})\circ \sigma  = \dif(-a_i, -a_j, 1-r_{ij}, 1-r_{ji}).
\]
\item[(ii)] On $R(\underset{i}{\textcolor[rgb]{1.00,0.00,0.00}{\bullet}}\ \ \ \ \underset{k}{\textcolor[rgb]{0.00,0.00,1.00}{\bullet}})$,
\[
\psi\circ\dif(a_i,a_k,u_{ik},u_{ki},\mu_i,\mu_k)\circ \psi  = \dif(-a_i, -a_k, -u_{ik}, -u_{ki},\mu_i,\mu_k),
\]
\[
\sigma\circ\dif(a_i,a_k,u_{ik},u_{ki},\mu_i,\mu_k)\circ \sigma  = \dif(-a_i, -a_k, -u_{ik}, -u_{ki},\mu_i,\mu_k).
\]
\end{itemize}
\end{prop}
\begin{proof}[Sketch of proof.]To show these relations, it suffices to check them
on the algebra generators. The proof is entirely similar to the argument in the proof of
Proposition~\ref{prop-symmetry-between-a-and-minus-a}. That $a_i, a_j, a_k$ are replaced
respectively by $-a_i, -a_j, -a_k$ under conjugation is implied by the proposition.
For $r_{ij}, r_{ji}, u_{ik}, u_{ki}$ it suffices to compute both sides on two-color
crossings. For instance, denoting $\dif(a_i,a_j,r_{ij},r_{ji})$ by $\dif_{r,s}$
for short, where $r=r_{ji}$ and $s=r_{ij}$, we have
\[
\begin{array}{rl}
\psi\circ\dif_{r,s}\circ \psi
\left(
{\begin{DGCpicture}[scale=0.95]
\DGCstrand[green](0,0)(1,1)[$j$]
\DGCstrand[red](1,0)(0,1)[$i$]
\end{DGCpicture}}
\right) \hspace{-0.1in}
& = \psi\circ\dif_{r,s}\left(
{\begin{DGCpicture}[scale=0.95]
\DGCstrand[red](0,0)(1,1)[$i$]
\DGCstrand[green](1,0)(0,1)[$j$]
\end{DGCpicture}}
\right) =
\psi \left( s
{\begin{DGCpicture}[scale=0.95]
\DGCstrand[red](0,0)(1,1)[$i$]
\DGCstrand[green](1,0)(0,1)[$j$]\DGCdot{0.75}
\end{DGCpicture}}
+(1-r)
{\begin{DGCpicture}[scale=0.95]
\DGCstrand[red](0,0)(1,1)[$i$]\DGCdot{0.75}
\DGCstrand[green](1,0)(0,1)[$j$]
\end{DGCpicture}}
\right) \\
& =(1-r)
\begin{DGCpicture}[scale=0.95]
\DGCstrand[green](0,0)(1,1)[$j$]
\DGCstrand[red](1,0)(0,1)[$i$]\DGCdot{0.75}
\end{DGCpicture}+s
\begin{DGCpicture}[scale=0.95]
\DGCstrand[green](0,0)(1,1)[$j$]\DGCdot{0.75}
\DGCstrand[red](1,0)(0,1)[$i$]
\end{DGCpicture}  =
\dif_{1-r,1-s}\left(
\begin{DGCpicture}[scale=0.95]
\DGCstrand[green](0,0)(1,1)[$j$]
\DGCstrand[red](1,0)(0,1)[$i$]
\end{DGCpicture}
\right).
\end{array}
\]
The rest of the proof follows similarly.
\end{proof}

The proof of the following result is similar to that of the nilHecke case (Proposition \ref{prop-NHn-compact-H-module}).

\begin{prop}\label{prop-finite-dimensional-homology-R-nu}
Let $\Gamma$ be the quiver $\underset{i}{\textcolor[rgb]{1.00,0.00,0.00}{\bullet}}\rightarrow \underset{j}{\textcolor[rgb]{0.00,1.00,0.00}{\bullet}}$ (resp. $\underset{i}{\textcolor[rgb]{1.00,0.00,0.00}{\bullet}}\ \ \ \ \underset{k}{\textcolor[rgb]{0.00,0.00,1.00}{\bullet}}$), and $\dif$ be the differential $\dif(a_i,a_j,r_{ij},r_{ji})$ (resp. $\dif(a_i,a_k, u_{ik}, u_{ki}, \mu_i,\mu_k)$ with $\mu_i,\mu_k \neq 0$). Then for each weight $\nu \in \N[I]$, where $I=\{i, j\}$ (resp. $I=\{i, k\}$), the $p$-complex $(R(\nu),\dif)$ is quasi-isomorphic to a finite dimensional $p$-complex.
\end{prop}
\begin{proof}[Sketch of proof] Since the differentials either preserve or decrease the number the crossings in a diagram, one can filter the $p$-complex $R(\nu)$ by the number of crossings in the diagrams of a basis, so that each subquotient is isomorphic to a rank-one polynomial module over the $p$-DG algebra $\Bbbk[x_1, \cdots, x_m]$ with $\dif(x_i)=x_i^2$. Such polynomial modules are quasi-isomorphic to finite dimensional $p$-complexes by the discussion in Section \ref{subsec-p-der-on-NH} (after equation \eqref{eq-dif-g}).
\end{proof}

\paragraph{Extension to any simply-laced Cartan datum.} Now let $\Gamma$ be a connected,
simply-laced Cartan datum with vertex set $I$. Such data are classified by finite
graphs without loops at vertices or multiple edges. Fix an arbitrary orientation
for $\Gamma$. Define a multi-parameter family of differentials
$$\dif:=\dif (a_i, r_{ij},r_{ji}, u_{ik},u_{ki}),$$
where $i,j,k \in I$, on $R(\Gamma)$ as follows.
\begin{itemize}
\item[(i)] For each vertex $i\in I$ choose $a_i\in \F_p$ and define
\begin{equation}\label{eqn-dif-on-R-nu-general-1}
\dif
\left(~
{\begin{DGCpicture}
\DGCstrand(0,0)(0,1)[$i$]\DGCdot{0.5}
\end{DGCpicture}}
~\right)=
{\begin{DGCpicture}
\DGCstrand(0,0)(0,1)[$i$]
\DGCdot{0.5}[r]{$^2$}
\end{DGCpicture}},
\end{equation}

\begin{equation}\label{eqn-dif-on-R-nu-general-2}
\dif\left(
\begin{DGCpicture}
\DGCstrand(0,0)(1,1)[$i$]
\DGCstrand(1,0)(0,1)[$i$]
\end{DGCpicture}\right)=
a_i\begin{DGCpicture}
\DGCstrand(0,0)(0,1)[$i$]
\DGCstrand(1,0)(1,1)[$i$]
\end{DGCpicture}
-(a_i+1)
\begin{DGCpicture}
\DGCstrand(0,0)(1,1)[$i$]
\DGCstrand(1,0)(0,1)[$i$]\DGCdot{0.75}
\end{DGCpicture}
+(a_i-1)
\begin{DGCpicture}
\DGCstrand(0,0)(1,1)[$i$]\DGCdot{0.75}
\DGCstrand(1,0)(0,1)[$i$]
\end{DGCpicture}.
\end{equation}

\item[(ii)] For each pair of vertices $i, j\in I$ in $\Gamma$ connected by an oriented
edge $i \rightarrow j$ choose
$r_{ij},~r_{ji}\in \F_p$ and define
\begin{equation}\label{eqn-dif-on-R-nu-general-3}
\dif\left(
\begin{DGCpicture}
\DGCstrand(0,0)(1,1)[$i$]
\DGCstrand(1,0)(0,1)[$j$]
\end{DGCpicture}\right)=
r_{ij}
\begin{DGCpicture}
\DGCstrand(0,0)(1,1)[$i$]
\DGCstrand(1,0)(0,1)[$j$]\DGCdot{0.75}
\end{DGCpicture}+
(1-r_{ji})
\begin{DGCpicture}
\DGCstrand(0,0)(1,1)[$i$]\DGCdot{0.75}
\DGCstrand(1,0)(0,1)[$j$]
\end{DGCpicture},
\end{equation}

\begin{equation}\label{eqn-dif-on-R-nu-general-4}
\dif\left(
\begin{DGCpicture}
\DGCstrand(0,0)(1,1)[$j$]
\DGCstrand(1,0)(0,1)[$i$]
\end{DGCpicture}
\right)=
r_{ji}{
\begin{DGCpicture}
\DGCstrand(0,0)(1,1)[$j$]
\DGCstrand(1,0)(0,1)[$i$]\DGCdot{0.75}
\end{DGCpicture}}+
(1-r_{ij}) {
\begin{DGCpicture}
\DGCstrand(0,0)(1,1)[$j$]\DGCdot{0.75}
\DGCstrand(1,0)(0,1)[$i$]
\end{DGCpicture}}.
\end{equation}

\item[(iii)] For each unordered pair of disconnected vertices $i,~k\in I$
choose $u_{ik}, u_{ki}\in \F_p$ and define
\begin{equation}\label{eqn-dif-on-R-nu-general-5}
\dif\left(
\begin{DGCpicture}
\DGCstrand(0,0)(1,1)[$i$]
\DGCstrand(1,0)(0,1)[$k$]
\end{DGCpicture}\right)=
u_{ik}
\begin{DGCpicture}
\DGCstrand(0,0)(1,1)[$i$]
\DGCstrand(1,0)(0,1)[$k$]\DGCdot{0.75}
\end{DGCpicture}
-u_{ki}\begin{DGCpicture}
\DGCstrand(0,0)(1,1)[$i$]\DGCdot{0.75}
\DGCstrand(1,0)(0,1)[$k$]
\end{DGCpicture},
\end{equation}

\begin{equation}\label{eqn-dif-on-R-nu-general-6}
\dif\left(
{\begin{DGCpicture}
\DGCstrand(0,0)(1,1)[$k$]
\DGCstrand(1,0)(0,1)[$i$]
\end{DGCpicture}}\right)=
u_{ki}{\begin{DGCpicture}
\DGCstrand(0,0)(1,1)[$k$]
\DGCstrand(1,0)(0,1)[$i$]\DGCdot{0.75}
\end{DGCpicture}}
-u_{ik} {\begin{DGCpicture}
\DGCstrand(0,0)(1,1)[$k$]\DGCdot{0.75}
\DGCstrand(1,0)(0,1)[$i$]
\end{DGCpicture}}.
\end{equation}
\end{itemize}

This assignment gives rises to a differential on $R(\Gamma)$. We have already checked that the defining relations in $R(\Gamma)$ involving strands of at most two colors are respected by these differentials. It is immediate to see that the only missing relations
\[
\begin{DGCpicture}[scale=0.75]
\DGCstrand(0,0)(2,2)[$i$]
\DGCstrand(2,0)(0,2)[$k$]
\DGCstrand(1,0)(0,1)(1,2)[$j$]
\end{DGCpicture}
=
\begin{DGCpicture}[scale=0.75]
\DGCstrand(0,0)(2,2)[$i$]
\DGCstrand(2,0)(0,2)[$k$]
\DGCstrand(1,0)(2,1)(1,2)[$j$]
\end{DGCpicture},
\]
where $i,j,k$ are pairwise distinct vertices, are preserved under the differential. In these relations, the types of distinct-color crossings involved do not change--only their orders do; hence the differential is well-defined.

\vspace{0.1in}

\subsection{Quantum Serre relations}
Given any simply-laced Cartan datum $\Gamma$ with vertex set $I$, there is an associated (half) quantum Kac-Moody algebra $\mathbf{f}_\Gamma$ \cite[Section 1]{Lu}, built as follows. One starts with a free $\N[I]$-graded associative algebra $^\prime\mathbf{f}_\Gamma$ over the field $\Q(q)$ with generators $E_i$ for all $i\in I$, where $\textrm{deg}(E_i)=i$ as an element of $\N[I]$. Then $\mathbf{f}_\Gamma$ is defined to be $^\prime\mathbf{f}_\Gamma$ quotiented by the two-sided ideal generated by the following relations, known as the quantum Serre relations:
\begin{equation}\label{eqn-qsr-distant-vertices}
E_iE_k=E_kE_i,
\end{equation}
if $i, k$ are distant vertices, while
\begin{equation}\label{eqn-qsr-connected-vertices}
\left\{
\begin{array}{rcl}
{[ 2 ] E_iE_jE_i} & = & E_iE_iE_j+E_jE_iE_i,\\
{[ 2 ] E_jE_iE_j} & = & E_jE_jE_i+E_iE_jE_j,
\end{array}
\right.
\end{equation}
if $i, j$ are connected by one edge. The divided powers $E_i^{(n)}:=E_i^n/[n]!$ for all $i\in I$ and $n\in \N$ generate an $\Z[q,q^{-1}]$ integral subalgebra $\mathbf{f}_{\Gamma,\Z}$ of $\mathbf{f}_\Gamma$ such that
\begin{equation}\label{eqn-integral-version-of-f-Gamma}
\mathbf{f}_{\Gamma,\Z}\o_{\Z[q,q^{-1}]}\Q(q)=\mathbf{f}_\Gamma.
\end{equation}
The relations (\ref{eqn-qsr-connected-vertices}) acquire the divided power form
\begin{equation}\label{eqn-qsr-connected-vertices-divided}
\left\{
\begin{array}{rcl}
{E_iE_jE_i} & = & E_{i^{(2)}}E_j+E_jE_{i^{(2)}},\\
{E_jE_iE_j} & = & E_{j^{(2)}}E_i+E_iE_{j^{(2)}}.
\end{array}
\right.
\end{equation}
We will abbreviate $E_iE_jE_i$ by $E_{iji}$, $E_{i^{(2)}}E_j$ by $E_{i^{(2)}j}$ etc.~in what follows.

The main goal of this section is to show that requiring these relations on the
Grothendieck groups of the derived categories of $p$-DG algebras $(R(\nu), \dif)$ severely restricts
possible choices of parameters in $\dif$ (parameters $a_i,r_{ij},r_{ji},u_{ik},u_{ki} \in \F_p$).
 We will show at the end that for such special parameter values the quantum Serre relations
hold on the level of $K_0$, interpreting these relations on the categorical level as well.
The local nature of the differentials allows us to
consider the two special cases $A_1\times A_1$ and $A_2$ separately
(see Examples~\ref{eg-KLR-relation-A2-case} and \ref{eg-KLR-relation-A1-times-A1-case}).

Below we will use that if $\sum_{\alpha\in J}[M_\alpha]=\sum_{\beta \in K}[N_\beta]$ holds in the Grothendieck group $K_0(R(\Gamma),\dif)$, where $M_\alpha$, $N_\beta$ and $P$ are compact $p$-DG modules over $(R(\Gamma),\dif)$ and $J, K$ are some finite index sets, then
\begin{equation}\label{eqn-equal-symbol-equal-euler}
\sum_{\alpha\in J}[\RHOM_{R(\Gamma)}(P,M_\alpha)]=\sum_{\beta\in K}[\RHOM_{R(\Gamma)}(P,N_\beta)]
\end{equation}
in $K_0(H\udmod)\cong \mathbb{O}_p$. This follows since
\[
\RHOM_{R(\Gamma)}(-,-):\mc{D}^c(R(\Gamma),\dif)\times \mc{D}^c(R(\Gamma),\dif)\lra H\ufmod
\]
is an exact bi-functor, and $R(\Gamma)$ has finite dimensional cohomology (Proposition \ref{prop-finite-dimensional-homology-R-nu}) under the differentials defined by equations (\ref{eqn-dif-on-R-nu-general-1})-(\ref{eqn-dif-on-R-nu-general-6}). When the context is clear, we will just write $\RHOM(-,-)$ for short.

\paragraph{The $A_1\times A_1$ case.}
Consider the $p$-DG algebra $(R(i+k),\dif(a_i,a_k,u_{ik},u_{ki}))$. Define the $p$-DG modules
\[
P_{ik}:=R(i+k)1_{ik}=
\begin{DGCpicture}
\DGCstrand[red](0,0)(0,0.5)[$i$]
\DGCstrand[blue](1,0)(1,0.5)[$k$]
\DGCcoupon(-0.35,0.5)(1.35,1.5){$R(i+k)$}
\end{DGCpicture}=
\left\{~
\begin{DGCpicture}
\DGCstrand[red](0,0)(0,0.5)[$i$]
\DGCstrand[blue](1,0)(1,0.5)[$k$]
\DGCcoupon(-0.35,0.5)(1.35,1.5){$x$}
\end{DGCpicture}~\Bigg|~ x \in R(i+k)
\right\},
\]
\[
P_{ki}:=R(i+k)1_{ki}=
\begin{DGCpicture}
\DGCstrand[blue](0,0)(0,0.5)[$k$]
\DGCstrand[red](1,0)(1,0.5)[$i$]
\DGCcoupon(-0.35,0.5)(1.35,1.5){$R(i+k)$}
\end{DGCpicture}=
\left\{~
\begin{DGCpicture}
\DGCstrand[blue](0,0)(0,0.5)[$k$]
\DGCstrand[red](1,0)(1,0.5)[$i$]
\DGCcoupon(-0.35,0.5)(1.35,1.5){$x$}
\end{DGCpicture}~\Bigg|~ x \in R(i+k)
\right\}.
\]
Both modules are cofibrant since they are $p$-DG direct summands of the free module $R(i+k)$ (Proposition \ref{prop-easy-properties-cofibrant-modules}). When the parameters $u_{ik}=u_{ki}=0$ in the definition of the differential, there are isomorphisms of $p$-DG modules
\[
\xymatrix{
\begin{DGCpicture}
\DGCstrand[red](0,0)(0,0.5)[$i$]
\DGCstrand[blue](1,0)(1,0.5)[$k$]
\DGCcoupon(-0.35,0.5)(1.35,1.5){$R(i+k)$}
\end{DGCpicture} \hspace{-0.35in}  &
&& \lltwocell_{\eta_{ki}}^{\eta_{ik}}{'}
& \hspace{-0.35in}
\begin{DGCpicture}
\DGCstrand[blue](0,0)(0,0.5)[$k$]
\DGCstrand[red](1,0)(1,0.5)[$i$]
\DGCcoupon(-0.35,0.5)(1.35,1.5){$R(i+k)$}
\end{DGCpicture}},
\]
where $\eta_{ki}$ is given by post-composing any diagram in $P_{ik}$ with the crossing
$\begin{DGCpicture}[scale=0.65]
\DGCstrand[blue](0,0)(1,1)[$^k$]
\DGCstrand[red](1,0)(0,1)[$^i$]
\end{DGCpicture}$, since
$\dif\left(
\begin{DGCpicture}[scale=0.65]
\DGCstrand[blue](0,0)(1,1)[$^k$]
\DGCstrand[red](1,0)(0,1)[$^i$]
\end{DGCpicture}\right)=0$ so that multiplication by this element commutes with the differential. Likewise $\eta_{ik}$ is given by attaching the crossing
$\begin{DGCpicture}[scale=0.65]
\DGCstrand[red](0,0)(1,1)[$^i$]
\DGCstrand[blue](1,0)(0,1)[$^k$]
\end{DGCpicture}$
to the bottom of any diagram of $P_{ki}$. Therefore, in $K_0(R(i+k),\dif(a_i,a_k,0,0))$, $[P_{ik}]=[P_{ki}]$. We interpret the $p$-DG isomorphism $P_{ik}\cong P_{ki}$ as a categorical lift of the relation $E_iE_k=E_kE_i$. The next result shows that such a categorical interpretation only exists when $u_{ik}=u_{ki}=0$.

\begin{prop}\label{prop-QSR-forces-specialization-parameter-A1-times-A1}
In the Grothendieck group $K_0(R(i+k),\dif(a_i,a_k,u_{ik},u_{ki}))$, the relation $[P_{ik}]=[P_{ki}]$ holds if and only if $u_{ik}=u_{ki}=0$. Furthermore, when $u_{ik}=u_{ki}=0$, there is isomorphisms of $R(i+k)_\dif$-module
\[
\eta_{ki}:P_{ik}\lra P_{ki}
\]
whose inverse is $\eta_{ik}$.
\end{prop}
\begin{proof}For the ease of notation, we will let $u=u_{ki}$ and $v=u_{ik}$ in the proof. If $[P_{ik}]=[P_{ki}]$, using the $\RHOM$ pairing with the cofibrant modules $P_{ik}$ and $P_{ki}$ (\ref{eqn-equal-symbol-equal-euler}), we have
\begin{equation}\label{eqn-equality-RHOM-symbols-1}
{[\RHOM_{R(i+k)}(P_{ik},P_{ik})]}= {[\RHOM_{R(i+k)}(P_{ik},P_{ki})]},
\end{equation}
\begin{equation}\label{eqn-equality-RHOM-symbols-2}
{[\RHOM_{R(i+k)}(P_{ki},P_{ik})]}={[\RHOM_{R(i+k)}(P_{ki},P_{ki})]}.
\end{equation}

The cofibrance condition implies that the derived hom is isomorphic to the usual hom as $H$-modules. On one hand, the left hand side of the equation (\ref{eqn-equality-RHOM-symbols-1}) equals the $\mathbb{O}_p$-dimension of the $H$-module
\[
\RHOM(P_{ik},P_{ik})\cong \HOM_{R(i+k)}(P_{ik},P_{ik})
\cong \Bbbk
\left\langle ~
\begin{DGCpicture}
\DGCstrand[red](0,0)(0,1)[$i$]\DGCdot{0.75}[l]{$^{n_1}$}
\DGCstrand[blue](1,0)(1,1)[$k$]\DGCdot{0.75}[r]{$^{n_2}$}
\end{DGCpicture}
~\Bigg|~ n_1, n_2 \in \N
\right\rangle ,
\]
which is quasi-isomorphic to the trivial $H$-module $V_0\cong\Bbbk$ whose $\mathbb{O}_p$-dimension (see Notation \ref{notation-symbol-as-O-dim}) is $1$. On the other hand,
\[
\RHOM(P_{ik},P_{ki})\cong \HOM_{R(i+k)}(P_{ik},P_{ki})
\cong \Bbbk
\left\langle ~
\begin{DGCpicture}
\DGCstrand[blue](0,0)(1,1)[$k$]\DGCdot{0.75}[r]{$^{n_2}$}
\DGCstrand[red](1,0)(0,1)[$i$]\DGCdot{0.75}[l]{$^{n_1}$}
\end{DGCpicture}
~\Bigg|~ n_1, n_2 \in \N
\right\rangle ,
\]
and the differential $\dif=\dif(a_i,a_k,u,v,\mu_i,\mu_k)$ acts on any basis element by
\[
\dif\left(
\begin{DGCpicture}
\DGCstrand[blue](0,0)(1,1)[$k$]\DGCdot{0.75}[r]{$^{n_2}$}
\DGCstrand[red](1,0)(0,1)[$i$]\DGCdot{0.75}[l]{$^{n_1}$}
\end{DGCpicture}
\right)=(n_1+u)
\begin{DGCpicture}
\DGCstrand[blue](0,0)(1,1)[$k$]\DGCdot{0.75}[r]{$^{n_2}$}
\DGCstrand[red](1,0)(0,1)[$i$]\DGCdot{0.75}[l]{$^{n_1+1}$}
\end{DGCpicture}+(n_2-v)
\begin{DGCpicture}
\DGCstrand[blue](0,0)(1,1)[$k$]\DGCdot{0.75}[r]{$^{n_2+1}$}
\DGCstrand[red](1,0)(0,1)[$i$]\DGCdot{0.75}[l]{$^{n_1}$}
\end{DGCpicture},
\]
so that as an $H$-module, $\RHOM(P_{ik},P_{ki})$ is quasi-isomorphic to $V_{p-u}\o V_{v}$, whose $\mathbb{O}_p$-dimension is $[\RHOM(P_{ik},P_{ki})]=(1+q^2+\cdots+q^{2(p-u)})(1+q^2+\cdots+q^{2v})$. Therefore we obtain the constraint equation in $\mathbb{O}_p$
\begin{equation}\label{eqn-equality-RHOM-symbols-3}
(1+q^2+\cdots+q^{2(p-u)})(1+q^2+\cdots+q^{2v})=1.
\end{equation}
Likewise, computing the $\mathbb{O}_p$-dimension of both sides of equation (\ref{eqn-equality-RHOM-symbols-2}) gives us
\begin{equation}\label{eqn-equality-RHOM-symbols-4}
(1+q^2+\cdots+q^{2(p-v)})(1+q^2+\cdots+q^{2u})=1.
\end{equation}
Mod $p$ reduction of the equations (\ref{eqn-equality-RHOM-symbols-3}), (\ref{eqn-equality-RHOM-symbols-4}) results in
\[
(1-u)(1+v)\equiv 1, \ \ \ \ (1-v)(1+u)\equiv 1 ~(\mathrm{mod}~p).
\]
Solving these last two equations together gives $u=v=0$, as claimed. The last statement follows from the discussion before the proposition.
\end{proof}

\paragraph{The $A_2$ case.} We first review briefly the proof of the quantum Serre relations in the Grothendieck group $K_0(R(\underset{i}{\textcolor[rgb]{1.00,0.00,0.00}{\bullet}}\rightarrow \underset{j}{\textcolor[rgb]{0.00,1.00,0.00}{\bullet}}))$, as done in \cite[Section 2.5]{KL1}. As an $R(2i+j)$-module, $P_{iji}:=R(2i+j)1_{iji}=\left\{
\begin{DGCpicture}[scale=0.85]
\DGCstrand[red](0,0)(0,1)[$i$]
\DGCstrand[green](0.5,0)(0.5,1)[$j$]
\DGCstrand[red](1,0)(1,1)[$i$]
\DGCcoupon(-0.15,0.5)(1.15,1){$x$}
\end{DGCpicture}
\bigg| x \in R(2i+j)
\right\}$ decomposes into a direct sum of projective modules
$P_{iji}\cong P_{i^{(2)}j}\bigoplus P_{ji^{(2)}}$, where
\[P_{i^{(2)}j}:=
\begin{DGCpicture}
\DGCstrand[red](0,0)(1,1)[$i$]
\DGCstrand[red](1,0)(0,1)[$i$]
\DGCdot{0.25}
\DGCstrand[green](2,0)(2,1)[$j$]
\DGCcoupon(-0.15,1)(2.15,2){$R(2i+j)$}
\end{DGCpicture}
=\left\{~
\begin{DGCpicture}[scale=0.75]
\DGCstrand[red](0,0)(1,1)[$i$]
\DGCstrand[red](1,0)(0,1)[$i$]
\DGCdot{0.25}
\DGCstrand[green](2,0)(2,1)[$j$]
\DGCcoupon(-0.15,1)(2.15,2){$x$}
\end{DGCpicture}~
\bigg| x \in R(2i+j)
\right\},
\]

\[
P_{ji^{(2)}}:=
\begin{DGCpicture}
\DGCstrand[red](1,0)(2,1)[$i$]
\DGCstrand[red](2,0)(1,1)[$i$]\DGCdot{0.25}
\DGCstrand[green](0,0)(0,1)[$j$]
\DGCcoupon(-0.15,1)(2.15,2){$R(2i+j)$}
\end{DGCpicture}
=\left\{~
\begin{DGCpicture}[scale=0.75]
\DGCstrand[red](1,0)(2,1)[$i$]
\DGCstrand[red](2,0)(1,1)[$i$]\DGCdot{0.25}
\DGCstrand[green](0,0)(0,1)[$j$]
\DGCcoupon(-0.15,1)(2.15,2){$x$}
\end{DGCpicture}~
\bigg| x \in R(2i+j)
\right\}.
\]
The isomorphism is realized via the four maps
\[
\xymatrix{
& P_{iji}
\ar[dl]_{\rho_{i^{(2)}j}} \ar[dr]^{\rho_{ji^{(2)}}} & \\
P_{i^{(2)}j}\ar[dr]_{\vartheta_{i^{(2)}j}} && P_{ji^{(2)}}\ar[dl]^{\vartheta_{ji^{(2)}}},\\
& P_{iji} &
}
\]
where $\rho_{i^{(2)}j}$ is the left $R(2i+j)$-module map given by right multiplying any element of $P_{iji}$ by
$\begin{DGCpicture}[scale=0.85]
\DGCstrand[red](0,0)(1,1)[$i$]
\DGCstrand[red](0.5,0)(0,1)[$i$]
\DGCdot{0.125}
\DGCstrand[green](1,0)(0.5,1)[$j$]
\end{DGCpicture}$. We will simply refer to this map as
\begin{equation}\label{eqn-maps-in-QSR-abelian-case-I}
\rho_{i^{(2)}j}=\begin{DGCpicture}
\DGCstrand[red](0,0)(1,1)[$i$]
\DGCstrand[red](0.5,0)(0,1)[$i$]
\DGCdot{0.125}
\DGCstrand[green](1,0)(0.5,1)[$j$]
\end{DGCpicture}.
\end{equation}
Likewise, the other maps in the above diagram are given by
\begin{equation}\label{eqn-maps-in-QSR-abelian-case-II}
\vartheta_{i^{(2)}j} =
\begin{DGCpicture}
\DGCstrand[red](0,0)(0.5,1)[$i$]
\DGCstrand[green](0.5,0)(1,1)[$j$]
\DGCstrand[red](1,0)(0,1)[$i$]
\end{DGCpicture}, \ \ \ \ \
\vartheta_{ji^{(2)}}=
\begin{DGCpicture}
\DGCstrand[red](0,0)(1,1)[$i$]
\DGCstrand[green](0.5,0)(0,1)[$j$]
\DGCstrand[red](1,0)(0.5,1)[$i$]
\end{DGCpicture}, \ \ \ \ \
\rho_{ji^{(2)}}=-
\begin{DGCpicture}
\DGCstrand[green](0,0)(0.5,1)[$j$]
\DGCstrand[red](0.5,0)(1,1)[$i$]
\DGCstrand[red](1,0)(0,1)[$i$]\DGCdot{0.125}
\end{DGCpicture}.
\end{equation}
Moreover, under this decomposition, there are isomorphisms of projective $R(2i+j)$-modules
\begin{equation}\label{eqn-iso-R-2i-plus-j-mod-1}
\xymatrix{
\begin{DGCpicture}
\DGCstrand[red](0,0)(1,1)[$i$]
\DGCstrand[red](1,0)(0,1)[$i$]
\DGCdot{0.25}
\DGCstrand[green](2,0)(2,1)[$j$]
\DGCcoupon(-0.15,1)(2.15,2){$R(2i+j)$}
\end{DGCpicture} \hspace{-0.35in} &
&& \lltwocell<4>_{\vartheta_{i^{(2)}j}}^{\rho_{i^{(2)}j}}{'}
& \hspace{-0.35in}
\begin{DGCpicture}[scale=0.65]
\DGCstrand[red](0,0)(2,2)[$i$]
\DGCstrand[red](2,0)(0,2)[$i$]
\DGCstrand[green](1,0)(2,1)(1,2)[$j$]
\DGCcoupon(-0.35,2)(2.35,3){$R(2i+j)$}
\end{DGCpicture}},
\end{equation}

\begin{equation}\label{eqn-iso-R-2i-plus-j-mod-2}
\xymatrix{
\begin{DGCpicture}
\DGCstrand[red](1,0)(2,1)[$i$]
\DGCstrand[red](2,0)(1,1)[$i$]\DGCdot{0.25}
\DGCstrand[green](0,0)(0,1)[$j$]
\DGCcoupon(-0.15,1)(2.15,2){$R(2i+j)$}
\end{DGCpicture} \hspace{-0.35in} &
&& \lltwocell<4>_{\vartheta_{ji^{(2)}}}^{\rho_{ji^{(2)}}}{'}
& \hspace{-0.35in}
\begin{DGCpicture}[scale=0.65]
\DGCstrand[red](0,0)(2,2)[$i$]
\DGCstrand[red](2,0)(0,2)[$i$]
\DGCstrand[green](1,0)(0,1)(1,2)[$j$]
\DGCcoupon(-0.35,2)(2.35,3){$R(2i+j)$}
\end{DGCpicture}}.
\end{equation}
Hence in the usual Grothendieck group of $R(2i+j)$ spanned by symbols of finitely generated graded projective modules, we have
\[
[P_{iji}]=[P_{i^{(2)}j}]+[P_{ji^{(2)}}].
\]
This is precisely the divided power form of the quantum Serre relation
\[E_{iji}=E_{i^{(2)}j}+E_{ji^{(2)}}.\]

Now consider $R(2i+j)$ as a $p$-DG algebra equipped with the differential $\dif=\dif(a_i,a_j,r_{ij},r_{ji})$ (see Lemma \ref{lemma-dif-on-sl3-parameter}). A categorical lifting of the quantum Serre relation to the category of $p$-DG modules over $(R(2i+j),\dif)$ necessarily means that the following equality of symbols
\[[2][P_{iji}]=[P_{iij}]+[P_{jii}]\]
holds in $K_0(R(2i+j),\dif)$. The modules $P_{iji}$, $P_{iij}$, $P_{jii}$ are compact cofibrant, and we compute the following Cartan matrix with entries in $K_0(H\udmod)\cong \mathbb{O}_p$:
\[
\left(
  \begin{array}{ccc}
  {[\RHOM(P_{iij},P_{iij})]}  & {[\RHOM(P_{iij},P_{iji})]}  & {[\RHOM(P_{iij},P_{jii})]} \\
  {[\RHOM(P_{iji},P_{iij})]}  & {[\RHOM(P_{iji},P_{iji})]}  & {[\RHOM(P_{iji},P_{jii})]} \\
  {[\RHOM(P_{jii},P_{iij})]}  & {[\RHOM(P_{jii},P_{iji})]}  & {[\RHOM(P_{jii},P_{jii})]} \\
  \end{array}
\right).
\]
We give one example to show how these entries are computed.

Until the end of this case ($A_2$), we will write $a:=a_i$, $b:=a_j$, $r:=r_{ji}$, $s:=r_{ij}$ for convenience.

\begin{eg}\label{eg-A2-a-Cartan-entry}
Consider the $(2,2)$-entry ${[\RHOM(P_{iji},P_{iji})]}$. Since $P_{iji}$ is cofibrant, the $p$-complex $\RHOM(P_{iji},P_{iji})$ is isomorphic to $\HOM(P_{iji},P_{iji})\cong 1_{iji}R(2i+j)1_{iji}$ with the induced differential from the algebra $R(2i+j)$ (Proposition \ref{prop-easy-properties-cofibrant-modules}). Diagrammatically,
\[
\begin{array}{rcl}
\HOM(P_{iji},P_{iji})& = &
\Bbbk\left\langle~
\begin{DGCpicture}[scale=0.65]
\DGCstrand[red](0,0)(0,2)[$i$]\DGCdot{1.75}[ur]{$^{k_1}$}
\DGCstrand[green](1,0)(1,2)[$j$]\DGCdot{1.75}[ur]{$^{k_2}$}
\DGCstrand[red](2,0)(2,2)[$i$]\DGCdot{1.75}[ur]{$^{k_3}$}
\end{DGCpicture},~
\begin{DGCpicture}[scale=0.65]
\DGCstrand[red](0,0)(2,2)[$i$]\DGCdot{1.75}[ur]{$^{k_3}$}
\DGCstrand[green](1,0)(0,1)(1,2)[$j$]\DGCdot{1.75}[ur]{$^{k_2}$}
\DGCstrand[red](2,0)(0,2)[$i$]\DGCdot{1.75}[ur]{$^{k_1}$}
\end{DGCpicture}
\Bigg|~k_1, k_2, k_3\in \N
\right\rangle\\
& & \\
& \cong &
\Bbbk[x_1(i),x_2(j),x_3(i)]\left\langle~
\begin{DGCpicture}[scale=0.65]
\DGCstrand[red](0,0)(0,2)[$i$]
\DGCstrand[green](1,0)(1,2)[$j$]
\DGCstrand[red](2,0)(2,2)[$i$]
\end{DGCpicture},~
\begin{DGCpicture}[scale=0.65]
\DGCstrand[red](0,0)(2,2)[$i$]
\DGCstrand[green](1,0)(0,1)(1,2)[$j$]
\DGCstrand[red](2,0)(0,2)[$i$]
\end{DGCpicture}
\right\rangle.
\end{array}
\]
The differential $\dif=\dif(a,b,s,r)$ acts on the polynomial algebra $\pol_{iji}:=\Bbbk[x_1(i),x_2(j),x_3(i)]$ by $\dif(x_1(i))=x_1^2(i)$, $\dif(x_2(j))=x_2^2(j)$, and $\dif(x_3(i))=x_3^2(i)$, while on the module basis elements it has the effect
\[
\begin{array}{rcl}
\dif\left(
\begin{DGCpicture}[scale=0.65]
\DGCstrand[red](0,0)(0,2)[$i$]
\DGCstrand[green](1,0)(1,2)[$j$]
\DGCstrand[red](2,0)(2,2)[$i$]
\end{DGCpicture}
\right) & = &
0,\\
&&\\
\dif\left(
\begin{DGCpicture}[scale=0.65]
\DGCstrand[red](0,0)(2,2)[$i$]
\DGCstrand[green](1,0)(0,1)(1,2)[$j$]
\DGCstrand[red](2,0)(0,2)[$i$]
\end{DGCpicture}
\right) & = &
(r-1-a)
\begin{DGCpicture}[scale=0.65]
\DGCstrand[red](0,0)(2,2)[$i$]
\DGCstrand[green](1,0)(0,1)(1,2)[$j$]
\DGCstrand[red](2,0)(0,2)[$i$]\DGCdot{1.75}
\end{DGCpicture}
+
\begin{DGCpicture}[scale=0.65]
\DGCstrand[red](0,0)(2,2)[$i$]
\DGCstrand[green](1,0)(0,1)(1,2)[$j$]\DGCdot{1.75}
\DGCstrand[red](2,0)(0,2)[$i$]
\end{DGCpicture}
+(a-r)
\begin{DGCpicture}[scale=0.65]
\DGCstrand[red](0,0)(2,2)[$i$]\DGCdot{1.75}
\DGCstrand[green](1,0)(0,1)(1,2)[$j$]
\DGCstrand[red](2,0)(0,2)[$i$]
\end{DGCpicture}\\
&& \\
&&
+(a+1-r)
\left(
\begin{DGCpicture}[scale=0.65]
\DGCstrand[red](0,0)(0,2)[$i$]\DGCdot{1.75}
\DGCstrand[green](1,0)(1,2)[$j$]
\DGCstrand[red](2,0)(2,2)[$i$]
\end{DGCpicture}-
\begin{DGCpicture}[scale=0.65]
\DGCstrand[red](0,0)(0,2)[$i$]
\DGCstrand[green](1,0)(1,2)[$j$]\DGCdot{1.75}
\DGCstrand[red](2,0)(2,2)[$i$]
\end{DGCpicture}
\right).
\end{array}
\]
Therefore, as a $p$-DG module over $(\pol_{iji},\dif)$, $\HOM(P_{iji},P_{iji})$ fits into the short exact sequence,
\[
0 \lra
\begin{DGCpicture}[scale=0.65]
\DGCstrand[red](0,0)(0,2)[$i$]
\DGCstrand[green](1,0)(1,2)[$j$]
\DGCstrand[red](2,0)(2,2)[$i$]
\DGCcoupon(-0.25,2)(2.25,3){$\pol_{iji}$}
\end{DGCpicture}
\lra
\HOM(P_{iji},P_{iji})
\lra
\begin{DGCpicture}[scale=0.65]
\DGCstrand[red](0,0)(2,2)[$i$]
\DGCstrand[green](1,0)(0,1)(1,2)[$j$]
\DGCstrand[red](2,0)(0,2)[$i$]
\DGCcoupon(-0.25,2)(2.25,3){$\pol_{iji}$}
\end{DGCpicture}
\lra 0,
\]
where the two end terms are respectively isomorphic to the $p$-DG $\pol_{iji}$-modules (ideals) $\pol_{iji}\cdot 1$ and $\pol_{iji}\cdot(x_1^{r-1-a}(i)x_2(j)x_3^{a-r}(i))$. It is easy to see that the former is quasi-isomorphic to $\Bbbk$, while the latter module is acyclic. In $\mc{D}(\pol_{iji},\dif)$, we have a corresponding distinguished triangle (Lemma \ref{lemma-ses-lead-to-dt-derived-category}), so that the $\mathbb{O}_p$-dimension of $\RHOM(P_{iji},P_{iji})$ satisfies
\[
{[\RHOM(P_{iji},P_{iji})]=
\left[~\begin{DGCpicture}[scale=0.65]
\DGCstrand[red](0,0)(0,2)[$i$]
\DGCstrand[green](1,0)(1,2)[$j$]
\DGCstrand[red](2,0)(2,2)[$i$]
\DGCcoupon(-0.25,2)(2.25,3){$\pol_{iji}$}
\end{DGCpicture}~\right]+
\left[~
\begin{DGCpicture}[scale=0.65]
\DGCstrand[red](0,0)(2,2)[$i$]
\DGCstrand[green](1,0)(0,1)(1,2)[$j$]
\DGCstrand[red](2,0)(0,2)[$i$]
\DGCcoupon(-0.25,2)(2.25,3){$\pol_{iji}$}
\end{DGCpicture}
~\right]}
=1+0=1.
\]
\end{eg}

By similar computations as in the above example, one calculates each of the Cartan matrix entries to be
\[
\left\{
\begin{array}{rcl}
 {[\RHOM(P_{iij},P_{iij})]}\hspace{-0.1in} & = \hspace{-0.1in}& 1+q^{-2}(a+2)_{q^2}(2-a)_{q^2}, \\
 {[\RHOM(P_{iji},P_{iij})]}\hspace{-0.1in} & = \hspace{-0.1in}& q(1+p-s)_{q^2}(r)_{q^2}+q^{-1}(a+2)_{q^2}(1+p-s)_{q^2}(1+r-a)_{q^2},\\
 {[\RHOM(P_{jii},P_{iij})]}\hspace{-0.1in} & = \hspace{-0.1in}& q^2(1+p-2s)_{q^2}(r)_{q^2}^2 +(1+p-2s)_{q^2}(1+r+a)_{q^2}(1+r-a)_{q^2},\\
 {[\RHOM(P_{iij},P_{iji})]} \hspace{-0.1in}& =\hspace{-0.1in} & q(1+p-r)_{q^2}(s)_{q^2} + q^{-1}(a+2-r)_{q^2}(2-a)_{q^2}(s)_{q^2},\\
 {[\RHOM(P_{iji},P_{iji})]} \hspace{-0.1in}& =\hspace{-0.1in} & 1,\\
 {[\RHOM(P_{jii},P_{iji})]} \hspace{-0.1in}& =\hspace{-0.1in} & q(1+p-s)_{q^2}(r)_{q^2} + q^{-1}(1+p-s)_{q^2}(a+2)_{q^2}(1+r-a)_{q^2},\\
 {[\RHOM(P_{iij},P_{jii})]}\hspace{-0.1in} & =\hspace{-0.1in} & q^2(1+p-r)_{q^2}^2(2s-1)_{q^2}+ (a+2-r)_{q^2}(2-a-r)_{q^2}(2s-1)_{q^2},\\
 {[\RHOM(P_{iji},P_{jii})]}\hspace{-0.1in} & =\hspace{-0.1in} & q(1+p-r)_{q^2}(s)_{q^2}+ q^{-1}(a+2-r)_{q^2}(s)_{q^2}(2-a)_{q^2},\\
 {[\RHOM(P_{jii},P_{jii})]}\hspace{-0.1in} & =\hspace{-0.1in} & 1+q^{-2}(a+2)_{q^2}(2-a)_{q^2},
\end{array}
  \right.
\]
where $(n)_{q^2}$ denotes the unbalanced quantum integer $(n)_{q^2}:=1+q^2+\cdots+q^{2(n-1)}\in \mathbb{O}_p$.
Equation (\ref{eqn-equal-symbol-equal-euler}) yields the following $\mathbb{O}_p$-vector
form of the quantum Serre relation:
\begin{equation}\label{eqn-imposing-QSR-on-Cartan-matrix}
{[2]\cdot(\textrm{the second row})=(\textrm{the first row})+(\textrm{the third row})}.
\end{equation}
When reduced $\textrm{mod}~p$, these equations give rise to the system of equations:
\begin{equation}\label{eqn-mod-p-QSR-1}
\begin{array}{l}
2(1-s)(3r-a+2-a^2+ar) = (5-a^2)+(1-2s)(2r^2+2r+1-a^2),\\
2  =  s(5-3r-a^2+ar)+(1-s)(ar-a^2+3r-a+2),\\
2s(5-3r-a^2+ar)  =  (2s-1)(5-6r+2r^2-a^2)+(5-a^2).
\end{array}
\end{equation}

Transposing $i$ and $j$ above, the quantum Serre relation $[2]E_{jij}=E_{jji}+E_{ijj}$ gives equations:
\begin{equation}\label{eqn-mod-p-QSR-2}
\begin{array}{l}
2(1-r)(3s-b+2-b^2+bs) = (5-b^2)+(1-2r)(2s^2+2s+1-b^2),\\
2  =  r(5-3s-b^2+bs)+(1-r)(bs-b^2+3s-b+2),\\
2r(5-3s-b^2+bs)  =  (2r-1)(5-6s+2s^2-b^2)+(5-b^2).
\end{array}
\end{equation}
Solving equations (\ref{eqn-mod-p-QSR-1}), (\ref{eqn-mod-p-QSR-2}), one obtains the following two solutions:
\begin{equation}
\left\{
\begin{array}{l}
a=b=1,\\
r=s=1.
\end{array}
\right. \ \ \ \
\left\{
\begin{array}{l}
a=b=-1,\\
r=s=0.
\end{array}
\right.
\end{equation}
Notice that the two systems of parameters give rise to the differentials $\dif(1,1,1,1)$, $\dif(-1,-1,0,0)$ which are conjugate to each other under the (anti-)automorphisms $\psi$ and $\sigma$ (Proposition \ref{prop-symmetry-intertwines-differentials-KLR}).

The above discussion shows that these two groups of special parameters are necessary for the quantum Serre relations to hold in $K_0(R(\underset{i}{\textcolor[rgb]{1.00,0.00,0.00}{\bullet}}\rightarrow \underset{j}{\textcolor[rgb]{0.00,1.00,0.00}{\bullet}}),\dif(a,b,r,s))$. Next we show that under these special parameters, we do have the equality of symbols
\[
{[ 2 ] [P_{iji}]}  =  [P_{iij}]+[P_{jii}],\ \ \ \
{[ 2 ] [P_{jij}]}  =  [P_{jji}]+[P_{ijj}],
\]
in the Grothendieck group. In fact, we will prove that the following divided power form of the quantum Serre relations
\[
{ [P_{iji}]}  =  [P_{i^{(2)}j}]+[P_{ji^{(2)}}],\ \ \ \
{ [P_{jij}]}  =  [P_{j^{(2)}i}]+[P_{ij^{(2)}}],
\]
holds in $K_0(R(\underset{i}{\textcolor[rgb]{1.00,0.00,0.00}{\bullet}}\rightarrow \underset{j}{\textcolor[rgb]{0.00,1.00,0.00}{\bullet}}))$ under the differentials $\dif(1,1,1,1)$ or $\dif(-1,-1,0,0)$. Since these two differentials are conjugate to each other, it suffices to do this for $\dif(1,1,1,1)$.

\begin{lemma}\label{lemma-serre-modules-are-cofibrant} Equipped with the differential $\dif(1,1,1,1)$, the module
\[P_{i^{(2)}j}:=
\begin{DGCpicture}
\DGCstrand[red](0,0)(1,1)[$i$]
\DGCstrand[red](1,0)(0,1)[$i$]
\DGCdot{0.25}
\DGCstrand[green](2,0)(2,1)[$j$]
\DGCcoupon(-0.15,1)(2.15,2){$R(2i+j)$}
\end{DGCpicture}
=\left\{~
\begin{DGCpicture}[scale=0.75]
\DGCstrand[red](0,0)(1,1)[$i$]
\DGCstrand[red](1,0)(0,1)[$i$]
\DGCdot{0.25}
\DGCstrand[green](2,0)(2,1)[$j$]
\DGCcoupon(-0.15,1)(2.15,2){$x$}
\end{DGCpicture}~
\bigg| x \in R(2i+j)
\right\},
\]
is compact cofibrant, and likewise for the module $P_{ji^{(2)}}$.
\end{lemma}
\begin{proof}As we have seen in Section 2, when $a=1$, the polynomial module
$P_{2,i}:=
\begin{DGCpicture}[scale=0.75]
\DGCstrand[red](0,0)(1,1)[$^i$]
\DGCstrand[red](1,0)(0,1)[$^i$]
\DGCdot{0.25}
\DGCcoupon(-0.25,1)(1.25,2){$R(2i)$}
\end{DGCpicture}
$ with the differential $\dif_1$ is compact cofibrant (Proposition \ref{prop-NHn-column-module-compact}), and likewise for
$P_{1,j}\cong \Bbbk[x(j)]$. The module in the statement of the lemma is easily seen to be no other than $\Ind_{R(2i)\boxtimes R(j)}^{R(2i+j)}(P_{2,i}\boxtimes P_{1,j})$. Hence the claim follows since $\Ind$ preserves compactness and cofibrance.
\end{proof}

Under $\dif(1,1,1,1)$, one computes easily that
\begin{equation}\label{eqn-dif-on-idempotent-R-nu-special-parameter}
\dif\left(
\begin{DGCpicture}[scale=0.65]
\DGCstrand[red](0,0)(2,2)[$i$]
\DGCstrand[red](2,0)(0,2)[$i$]
\DGCstrand[green](1,0)(2,1)(1,2)[$j$]
\end{DGCpicture}
\right)=
\begin{DGCpicture}[scale=0.65]
\DGCstrand[red](0,0)(2,2)[$i$]
\DGCstrand[red](2,0)(0,2)[$i$]
\DGCstrand[green](1,0)(2,1)(1,2)[$j$]\DGCdot{1.75}
\end{DGCpicture}-
\begin{DGCpicture}[scale=0.65]
\DGCstrand[red](0,0)(2,2)[$i$]
\DGCstrand[red](2,0)(0,2)[$i$]\DGCdot{1.75}
\DGCstrand[green](1,0)(2,1)(1,2)[$j$]
\end{DGCpicture}\ , \ \ \
\dif\left(
\begin{DGCpicture}[scale=0.65]
\DGCstrand[red](0,0)(2,2)[$i$]
\DGCstrand[red](2,0)(0,2)[$i$]
\DGCstrand[green](1,0)(0,1)(1,2)[$j$]
\end{DGCpicture}
\right)=
\begin{DGCpicture}[scale=0.65]
\DGCstrand[red](0,0)(2,2)[$i$]
\DGCstrand[red](2,0)(0,2)[$i$]
\DGCstrand[green](1,0)(2,1)(1,2)[$j$]\DGCdot{1.75}
\end{DGCpicture}
-
\begin{DGCpicture}[scale=0.65]
\DGCstrand[red](0,0)(2,2)[$i$]
\DGCstrand[red](2,0)(0,2)[$i$]\DGCdot{1.75}
\DGCstrand[green](1,0)(2,1)(1,2)[$j$]
\end{DGCpicture}.
\end{equation}
Therefore $P_{iji}$ fits into the short exact sequence of $p$-DG modules
\begin{equation}\label{eqn-ses-QSR-A2}
0 \lra
\begin{DGCpicture}[scale=0.65]
\DGCstrand[red](0,0)(2,2)[$i$]
\DGCstrand[red](2,0)(0,2)[$i$]
\DGCstrand[green](1,0)(2,1)(1,2)[$j$]
\DGCcoupon(-0.35,2)(2.35,3){$R(2i+j)$}
\end{DGCpicture}
\lra
\begin{DGCpicture}[scale=0.65]
\DGCstrand[red](0,0)(0,2)[$i$]
\DGCstrand[red](2,0)(2,2)[$i$]
\DGCstrand[green](1,0)(1,2)[$j$]
\DGCcoupon(-0.35,2)(2.35,3){$R(2i+j)$}
\end{DGCpicture}
\lra
\begin{DGCpicture}[scale=0.65]
\DGCstrand[red](0,0)(2,2)[$i$]
\DGCstrand[red](2,0)(0,2)[$i$]
\DGCstrand[green](1,0)(0,1)(1,2)[$j$]
\DGCcoupon(-0.35,2)(2.35,3){$R(2i+j)$}
\end{DGCpicture}
\lra 0.
\end{equation}
Here the differential acts on the generator of the quotient module trivially:
\[
\dif\left(
\begin{DGCpicture}[scale=0.65]
\DGCstrand[red](0,0)(2,2)[$i$]
\DGCstrand[red](2,0)(0,2)[$i$]
\DGCstrand[green](1,0)(0,1)(1,2)[$j$]
\end{DGCpicture}
\right) \equiv 0
\ \ \ \ (\mathrm{mod}~
\begin{DGCpicture}[scale=0.65]
\DGCstrand[red](0,0)(2,2)[$i$]
\DGCstrand[red](2,0)(0,2)[$i$]
\DGCstrand[green](1,0)(2,1)(1,2)[$j$]
\DGCcoupon(-0.35,2)(2.35,3){$R(2i+j)$}
\end{DGCpicture}).
\]

Now we claim that the isomorphisms of projective $R(2i+j)$-modules defined in equations (\ref{eqn-iso-R-2i-plus-j-mod-1}), (\ref{eqn-iso-R-2i-plus-j-mod-2}) lift to isomorphisms of $p$-DG modules. Here the module
\[
S:=\begin{DGCpicture}[scale=0.65]
\DGCstrand[red](0,0)(2,2)[$i$]
\DGCstrand[red](2,0)(0,2)[$i$]
\DGCstrand[green](1,0)(2,1)(1,2)[$j$]
\DGCcoupon(-0.35,2)(2.35,3){$R(2i+j)$}
\end{DGCpicture}\ \ \ \
(\textrm{resp.} \ Q:=
\begin{DGCpicture}[scale=0.65]
\DGCstrand[red](0,0)(2,2)[$i$]
\DGCstrand[red](2,0)(0,2)[$i$]
\DGCstrand[green](1,0)(0,1)(1,2)[$j$]
\DGCcoupon(-0.35,2)(2.35,3){$R(2i+j)$}
\end{DGCpicture})
\]
is equipped with the $p$-DG submodule structure (resp. quotient module structure) of $P_{iji}$. The differential acts on the module generator
$\begin{DGCpicture}[scale=0.5]
\DGCstrand[red](0,0)(2,2)[$i$]
\DGCstrand[red](2,0)(0,2)[$i$]
\DGCstrand[green](1,0)(2,1)(1,2)[$j$]
\end{DGCpicture}$
of $S$ by the first equation of ($\ref{eqn-dif-on-idempotent-R-nu-special-parameter}$), while $\dif$ acts as zero on
$\begin{DGCpicture}[scale=0.5]
\DGCstrand[red](0,0)(2,2)[$i$]
\DGCstrand[red](2,0)(0,2)[$i$]
\DGCstrand[green](1,0)(0,1)(1,2)[$j$]
\end{DGCpicture}$
of $Q$. We will prove the claim by exhibiting the intertwining relations $\rho_{i^{(2)}j}\circ\dif_S = \dif_{P_{i^{(2)}j}}\circ \rho_{i^{(2)}j}$ on $S$, and $\rho_{ji^{(2)}} \circ \dif_Q=\dif_{P_{ji^{(2)}}} \circ \rho_{ji^{(2)}}$ on the quotient module $Q$. See equations (\ref{eqn-maps-in-QSR-abelian-case-I}), (\ref{eqn-maps-in-QSR-abelian-case-II}) for the maps $\rho_{i^{(2)}j}$, $\rho_{ji^{(2)}}$, $\vartheta_{i^{(2)}j}$ and $\vartheta_{ji^{(2)}}$. Equivalently $\vartheta_{i^{(2)}j} \circ \dif_{P_{i^{(2)}j}} \circ \rho_{i^{(2)}j}=\dif_{S}$ and $\vartheta_{ji^{(2)}} \circ \dif_{P_{ji^{(2)}}} \circ \rho_{ji^{(2)}}=\dif_{Q}$.

Given any element $x$ of $R(2i+j)$,
\[
\begin{array}{l}
\vartheta_{i^{(2)}j} \circ \dif_{P_{i^{(2)}j}} \circ \rho_{i^{(2)}j} \left(~
\begin{DGCpicture}[scale=0.5]
\DGCstrand[red](0,0)(2,2)[$i$]
\DGCstrand[red](2,0)(0,2)[$i$]
\DGCstrand[green](1,0)(2,1)(1,2)[$j$]
\DGCcoupon(-0.35,2)(2.35,3){$x$}
\end{DGCpicture}~
\right)  =
\vartheta_{i^{(2)}j} \circ \dif_{P_{i^{(2)}j}}\left(-
\begin{DGCpicture}[scale=0.5]
\DGCstrand[red](0,0)(2,2)[$i$]
\DGCstrand[red](1,0)(0,2)[$i$]\DGCdot{0.5}
\DGCstrand[green](2,0)(1,2)[$j$]
\DGCcoupon(-0.35,2)(2.35,3){$x$}
\end{DGCpicture}
~\right)\\
 \\
 = -\vartheta_{i^{(2)}j}\left(~
\begin{DGCpicture}[scale=0.5]
\DGCstrand[red](0,0)(2,2)[$i$]
\DGCstrand[red](1,0)(0,2)[$i$]\DGCdot{0.5}
\DGCstrand[green](2,0)(1,2)[$j$]
\DGCcoupon(-0.35,2)(2.35,3){$\dif(x)$}
\end{DGCpicture}+
\begin{DGCpicture}[scale=0.5]
\DGCstrand[red](0,0)(2,2)[$i$]
\DGCstrand[red](1,0)(0,2)[$i$]\DGCdot{0.5}
\DGCstrand[green](2,0)(1,2)[$j$]\DGCdot{1.5}
\DGCcoupon(-0.35,2)(2.35,3){$x$}
\end{DGCpicture}+
\begin{DGCpicture}[scale=0.5]
\DGCstrand[red](0,0)(0,2)[$i$]
\DGCstrand[red](1,0)(2,2)[$i$]\DGCdot{0.5}
\DGCstrand[green](2,0)(1,2)[$j$]
\DGCcoupon(-0.35,2)(2.35,3){$x$}
\end{DGCpicture}-2
\begin{DGCpicture}[scale=0.5]
\DGCstrand[red](0,0)(2,2)[$i$]
\DGCstrand[red](1,0)(0,2)[$i$]\DGCdot{0.5}\DGCdot{1.5}
\DGCstrand[green](2,0)(1,2)[$j$]
\DGCcoupon(-0.35,2)(2.35,3){$x$}
\end{DGCpicture}+
\begin{DGCpicture}[scale=0.5]
\DGCstrand[red](0,0)(2,2)[$i$]
\DGCstrand[red](1,0)(0,2)[$i$]\DGCdot{0.5}[r]{$^2$}
\DGCstrand[green](2,0)(1,2)[$j$]
\DGCcoupon(-0.35,2)(2.35,3){$x$}
\end{DGCpicture}
~\right)\\
 \\
 = -\vartheta_{i^{(2)}j}\left(~
\begin{DGCpicture}[scale=0.5]
\DGCstrand[red](0,0)(2,2)[$i$]
\DGCstrand[red](1,0)(0,2)[$i$]\DGCdot{0.5}
\DGCstrand[green](2,0)(1,2)[$j$]
\DGCcoupon(-0.35,2)(2.35,3){$\dif(x)$}
\end{DGCpicture}+
\begin{DGCpicture}[scale=0.5]
\DGCstrand[red](0,0)(2,2)[$i$]
\DGCstrand[red](1,0)(0,2)[$i$]\DGCdot{0.5}
\DGCstrand[green](2,0)(1,2)[$j$]\DGCdot{1.5}
\DGCcoupon(-0.35,2)(2.35,3){$x$}
\end{DGCpicture}-
\begin{DGCpicture}[scale=0.5]
\DGCstrand[red](0,0)(2,2)[$i$]
\DGCstrand[red](1,0)(0,2)[$i$]\DGCdot{0.5}\DGCdot{1.5}
\DGCstrand[green](2,0)(1,2)[$j$]
\DGCcoupon(-0.35,2)(2.35,3){$x$}
\end{DGCpicture}
~\right)\\
 \\
 =
\begin{DGCpicture}[scale=0.5]
\DGCstrand[red](0,0)(2,2)[$i$]
\DGCstrand[red](2,0)(0,2)[$i$]
\DGCstrand[green](1,0)(2,1)(1,2)[$j$]
\DGCcoupon(-0.35,2)(2.35,3){$\dif(x)$}
\end{DGCpicture}-
\begin{DGCpicture}[scale=0.5]
\DGCstrand[red](0,0)(2,2)[$i$]
\DGCstrand[red](2,0)(0,2)[$i$]\DGCdot{1.75}
\DGCstrand[green](1,0)(2,1)(1,2)[$j$]
\DGCcoupon(-0.35,2)(2.35,3){$x$}
\end{DGCpicture}+
\begin{DGCpicture}[scale=0.5]
\DGCstrand[red](0,0)(2,2)[$i$]
\DGCstrand[red](2,0)(0,2)[$i$]
\DGCstrand[green](1,0)(2,1)(1,2)[$j$]\DGCdot{1.75}
\DGCcoupon(-0.35,2)(2.35,3){$x$}
\end{DGCpicture}.
\end{array}
\]
Comparing this with the first equation of (\ref{eqn-dif-on-idempotent-R-nu-special-parameter}) one obtains the claimed intertwining relation. Likewise, on the quotient module,
\[
\begin{array}{l}
\vartheta_{ji^{(2)}} \circ \dif_{P_{ji^{(2)}}} \circ \rho_{ji^{(2)}} \left(~
\begin{DGCpicture}[scale=0.5]
\DGCstrand[red](0,0)(2,2)[$i$]
\DGCstrand[red](2,0)(0,2)[$i$]
\DGCstrand[green](1,0)(0,1)(1,2)[$j$]
\DGCcoupon(-0.35,2)(2.35,3){$x$}
\end{DGCpicture}
~\right)  =
\vartheta_{ji^{(2)}} \circ \dif_{P_{ji^{(2)}}}\left(-
\begin{DGCpicture}[scale=0.5]
\DGCstrand[red](1,0)(2,2)[$i$]
\DGCstrand[red](2,0)(0,2)[$i$]\DGCdot{0.5}
\DGCstrand[green](0,0)(1,2)[$j$]
\DGCcoupon(-0.35,2)(2.35,3){$x$}
\end{DGCpicture}
~\right)\\
 \\
 =  -\vartheta_{ji^{(2)}}\left(~
\begin{DGCpicture}[scale=0.5]
\DGCstrand[red](1,0)(2,2)[$i$]
\DGCstrand[red](2,0)(0,2)[$i$]\DGCdot{0.5}
\DGCstrand[green](0,0)(1,2)[$j$]
\DGCcoupon(-0.35,2)(2.35,3){$\dif(x)$}
\end{DGCpicture}+
\begin{DGCpicture}[scale=0.5]
\DGCstrand[red](1,0)(2,2)[$i$]
\DGCstrand[red](2,0)(0,2)[$i$]\DGCdot{0.5}\DGCdot{1.65}
\DGCstrand[green](0,0)(1,2)[$j$]
\DGCcoupon(-0.35,2)(2.35,3){$x$}
\end{DGCpicture}-
\begin{DGCpicture}[scale=0.5]
\DGCstrand[red](1,0)(2,2)[$i$]
\DGCstrand[red](2,0)(0,2)[$i$]\DGCdot{0.5}\DGCdot{1.05}
\DGCstrand[green](0,0)(1,2)[$j$]
\DGCcoupon(-0.35,2)(2.35,3){$x$}
\end{DGCpicture}-
\begin{DGCpicture}[scale=0.5]
\DGCstrand[red](1,0)(2,2)[$i$]
\DGCstrand[red](2,0)(0,2)[$i$]\DGCdot{0.5}[r]{$^2$}
\DGCstrand[green](0,0)(1,2)[$j$]
\DGCcoupon(-0.35,2)(2.35,3){$x$}
\end{DGCpicture}+
\begin{DGCpicture}[scale=0.5]
\DGCstrand[red](1,0)(2,2)[$i$]
\DGCstrand[red](2,0)(0,2)[$i$]\DGCdot{0.5}[r]{$^2$}
\DGCstrand[green](0,0)(1,2)[$j$]
\DGCcoupon(-0.35,2)(2.35,3){$x$}
\end{DGCpicture}
~\right)\\
 \\
 = -\vartheta_{ji^{(2)}}\left(~
\begin{DGCpicture}[scale=0.5]
\DGCstrand[red](1,0)(2,2)[$i$]
\DGCstrand[red](2,0)(0,2)[$i$]\DGCdot{0.5}
\DGCstrand[green](0,0)(1,2)[$j$]
\DGCcoupon(-0.35,2)(2.35,3){$\dif(x)$}
\end{DGCpicture}
~\right) =
\begin{DGCpicture}[scale=0.5]
\DGCstrand[red](0,0)(2,2)[$i$]
\DGCstrand[red](2,0)(0,2)[$i$]
\DGCstrand[green](1,0)(0,1)(1,2)[$j$]
\DGCcoupon(-0.35,2)(2.35,3){$\dif(x)$}
\end{DGCpicture}.
\end{array}
\]

The case when $i$, $j$ are transposed is entirely similar and left to the reader. We summarize the above discussion into the following result.
\begin{prop}\label{prop-QSR-forces-specialization-parameter-A2}The following equalities of symbols
\[
\left\{
\begin{array}{rcl}
{[ 2 ] [P_{iji}]} & = & [P_{iij}]+[P_{jii}],\\
{[ 2 ] [P_{jij}]} & = & [P_{jji}]+[P_{ijj}],
\end{array}
\right.
\]
hold in the Grothendieck group $K_0(R(\underset{i}{\textcolor[rgb]{1.00,0.00,0.00}{\bullet}}\rightarrow \underset{j}{\textcolor[rgb]{0.00,1.00,0.00}{\bullet}}),\dif(a_i,a_j,r_{ij},r_{ji}))$ if and only if the parameters $(a_i,a_j,r_{ij},r_{ji})$ are given by either
\[
\left\{
\begin{array}{l}
a_i=a_j=1,\\
r_{ij}=r_{ji}=1,
\end{array}
\right. \ \ \ \ \textrm {or} \ \ \ \
\left\{
\begin{array}{l}
a_{i}=a_j=-1,\\
r_{ij}=r_{ji}=0.
\end{array}
\right.
\]
$\hfill\square$
\end{prop}

The discussion of this subsection gives in fact more than just the relations of symbols in the Grothendieck group. If $\dif$ is parameterized by the first system of coefficients, there are short exact sequences of $R(2i+j)_\dif$-modules (resp. $R(2j+i)_\dif$-modules):
\begin{equation}\label{eqn-ses-R-2i-plus-j-divided}
0 \lra
\begin{DGCpicture}[scale=0.65]
\DGCstrand[red](0,0)(2,2)[$i$]
\DGCstrand[red](1,0)(0,2)[$i$]\DGCdot{0.5}
\DGCstrand[green](2,0)(1,2)[$j$]
\DGCcoupon(-0.35,2)(2.35,3){$R(2i+j)$}
\end{DGCpicture}
\xrightarrow{\vartheta_{i^{(2)}j}}
\begin{DGCpicture}[scale=0.65]
\DGCstrand[red](0,0)(0,2)[$i$]
\DGCstrand[red](2,0)(2,2)[$i$]
\DGCstrand[green](1,0)(1,2)[$j$]
\DGCcoupon(-0.35,2)(2.35,3){$R(2i+j)$}
\end{DGCpicture}
\xrightarrow{\rho_{ji^{(2)}}}
\begin{DGCpicture}[scale=0.65]
\DGCstrand[red](1,0)(2,2)[$i$]
\DGCstrand[red](2,0)(0,2)[$i$]\DGCdot{0.5}
\DGCstrand[green](0,0)(1,2)[$j$]
\DGCcoupon(-0.35,2)(2.35,3){$R(2i+j)$}
\end{DGCpicture}
\lra 0.
\end{equation}

\begin{equation}\label{eqn-ses-R-2j-plus-i-divided}
\left(
\mathrm{resp.}\ \
0 \lra
\begin{DGCpicture}[scale=0.65]
\DGCstrand[green](0,0)(2,2)[$j$]
\DGCstrand[green](1,0)(0,2)[$j$]\DGCdot{0.5}
\DGCstrand[red](2,0)(1,2)[$i$]
\DGCcoupon(-0.35,2)(2.35,3){$R(2j+i)$}
\end{DGCpicture}
\xrightarrow{\vartheta_{j^{(2)}i}}
\begin{DGCpicture}[scale=0.65]
\DGCstrand[green](0,0)(0,2)[$j$]
\DGCstrand[green](2,0)(2,2)[$j$]
\DGCstrand[red](1,0)(1,2)[$i$]
\DGCcoupon(-0.35,2)(2.35,3){$R(2j+i)$}
\end{DGCpicture}
\xrightarrow{\rho_{ij^{(2)}}}
\begin{DGCpicture}[scale=0.65]
\DGCstrand[green](1,0)(2,2)[$j$]
\DGCstrand[green](2,0)(0,2)[$j$]\DGCdot{0.5}
\DGCstrand[red](0,0)(1,2)[$i$]
\DGCcoupon(-0.35,2)(2.35,3){$R(2j+i)$}
\end{DGCpicture}
\lra 0.
\right)
\end{equation}
Likewise if $\dif$ is parameterized by the second system of equations, one obtains similar short exact sequences by applying the symmetry $\sigma$ (Proposition \ref{prop-symmetry-intertwines-differentials-KLR}) to the above ones. The short exact sequences (\ref{eqn-ses-R-2i-plus-j-divided}), (\ref{eqn-ses-R-2j-plus-i-divided}) give rise to exact triangles in the homotopy category and derived category (Lemma \ref{lemma-ses-lead-to-dt-homotopy-category} and \ref{lemma-ses-lead-to-dt-derived-category}):
\begin{equation}\label{eqn-derived-qsr-2i-plus-j}
\begin{DGCpicture}[scale=0.65]
\DGCstrand[red](0,0)(2,2)[$i$]
\DGCstrand[red](1,0)(0,2)[$i$]\DGCdot{0.5}
\DGCstrand[green](2,0)(1,2)[$j$]
\DGCcoupon(-0.35,2)(2.35,3){$R(2i+j)$}
\end{DGCpicture}
\xrightarrow{\vartheta_{i^{(2)}j}}
\begin{DGCpicture}[scale=0.65]
\DGCstrand[red](0,0)(0,2)[$i$]
\DGCstrand[red](2,0)(2,2)[$i$]
\DGCstrand[green](1,0)(1,2)[$j$]
\DGCcoupon(-0.35,2)(2.35,3){$R(2i+j)$}
\end{DGCpicture}
\xrightarrow{\rho_{ji^{(2)}}}
\begin{DGCpicture}[scale=0.65]
\DGCstrand[red](1,0)(2,2)[$i$]
\DGCstrand[red](2,0)(0,2)[$i$]\DGCdot{0.5}
\DGCstrand[green](0,0)(1,2)[$j$]
\DGCcoupon(-0.35,2)(2.35,3){$R(2i+j)$}
\end{DGCpicture}
\lra
\left(~
\begin{DGCpicture}[scale=0.65]
\DGCstrand[red](0,0)(2,2)[$i$]
\DGCstrand[red](1,0)(0,2)[$i$]\DGCdot{0.5}
\DGCstrand[green](2,0)(1,2)[$j$]
\DGCcoupon(-0.35,2)(2.35,3){$R(2i+j)$}
\end{DGCpicture}
~\right)[1]~,
\end{equation}

\begin{equation}\label{eqn-derived-qsr-2j-plus-i}
\left(\mathrm{resp.} \ \
\begin{DGCpicture}[scale=0.65]
\DGCstrand[green](0,0)(2,2)[$j$]
\DGCstrand[green](1,0)(0,2)[$j$]\DGCdot{0.5}
\DGCstrand[red](2,0)(1,2)[$i$]
\DGCcoupon(-0.35,2)(2.35,3){$R(2j+i)$}
\end{DGCpicture}
\xrightarrow{\vartheta_{j^{(2)}i}}
\begin{DGCpicture}[scale=0.65]
\DGCstrand[green](0,0)(0,2)[$j$]
\DGCstrand[green](2,0)(2,2)[$j$]
\DGCstrand[red](1,0)(1,2)[$i$]
\DGCcoupon(-0.35,2)(2.35,3){$R(2j+i)$}
\end{DGCpicture}
\xrightarrow{\rho_{ij^{(2)}}}
\begin{DGCpicture}[scale=0.65]
\DGCstrand[green](1,0)(2,2)[$j$]
\DGCstrand[green](2,0)(0,2)[$j$]\DGCdot{0.5}
\DGCstrand[red](0,0)(1,2)[$i$]
\DGCcoupon(-0.35,2)(2.35,3){$R(2j+i)$}
\end{DGCpicture}
\lra
\left(~
\begin{DGCpicture}[scale=0.65]
\DGCstrand[green](0,0)(2,2)[$j$]
\DGCstrand[green](1,0)(0,2)[$j$]\DGCdot{0.5}
\DGCstrand[red](2,0)(1,2)[$i$]
\DGCcoupon(-0.35,2)(2.35,3){$R(2j+i)$}
\end{DGCpicture}
~\right)[1]
\right)~.
\end{equation}
We regard the exact triangles (\ref{eqn-derived-qsr-2i-plus-j}) and (\ref{eqn-derived-qsr-2j-plus-i}) as categorical liftings of the divided power form of the quantum Serre relations
\[
E_{iji}=E_{i^{(2)}j}+E_{ji^{(2)}}, \ \ \ \ (\mathrm{resp}~~E_{jij}=E_{j^{(2)}i}+E_{ij^{(2)}}).
\]
We extend the results to  general simply-laced Cartan data in what follows.

\begin{rmk}[$\mf{sl}_3$ at $\sqrt{-1}$] At a fourth root of unity $\zeta_4=\pm\sqrt{-1}$, the quantum Serre relations $[2]E_{iji}=E_{iij}+E_{jii}$, $[2]E_{jij}=E_{jji}+E_{ijj}$ for the small quantum group $u^+_{\zeta_4}(\mf{sl}_3)$ are redundant, since in this case $[2]=0$, $E_i^2=0$, and $E_j^2=0$. By computing the endomorphism algebras of the $p$-DG modules as we did in Example \ref{eg-A2-a-Cartan-entry}, one finds that $P_{iij}$, $P_{jii}$, $P_{jji}$ and $P_{ijj}$ are all contractible,  while $P_{iji}$, $P_{jij}$ are not. Furthermore, one can also show that the symbols of $P_{jiji}$, $P_{ijij}$, $P_{ijiji}$ etc. are not zero in the Grothendieck group of $(R(A_2),\dif(1,1,1,1))$.
\end{rmk}

\paragraph{Idempotents and $p$-DG filtrations.}Some of the manipulations on the
previous few pages can be restated more intrinsically as follows.
Let $R$ be a ring and $x,y\in R$ be two elements satisfying
\begin{equation}\label{eq-xyx}
xyx=x, \quad \quad   yxy=y.
\end{equation}
Then $xy$ and $yx$ are idempotents in $R$ and the projective left
$R$-modules $Rxy$ and $Ryx$ are isomorphic via the maps that take
$zxy\in Rxy$ to $zxyx=zx\in Ryx$ and $zyx\in Ryx$ to $zyxy=zy\in Rxy$,
respectively:
\begin{equation}\label{eqn-iso-Rxy-Ryx}
\xymatrix{Rxy  & \ltwocell_{\cdot x}^{\cdot y}{'} Ryx.}
\end{equation}
Assume now that $R$ carries a $p$-nilpotent derivation $\dif$. The projective modules
$Rxy$ and $Ryx$ are $\dif$-closed if and only if
\begin{equation}\label{eqn-condition-proj-mod-dif-closed}
\dif(xy)\in Rxy , \quad \quad \dif(yx)\in Ryx.
\end{equation}
The maps in (\ref{eqn-iso-Rxy-Ryx}) will commute with $\dif$ if and only if
\begin{equation}\label{eqn-condition-maps-intertwines-dif}
\dif(xy)x=\dif(xyx), \quad \quad \dif(yx)y=\dif(yxy).
\end{equation}
Expanding via the Leibniz rule, (\ref{eqn-condition-maps-intertwines-dif}) is equivalent to
\begin{equation}\label{eqn-condition-maps-intertwines-dif-2}
xy\dif(x)=0, \quad \quad yx\dif(y)=0.
\end{equation}
The following conditions are equivalent to (\ref{eqn-condition-proj-mod-dif-closed})
and (\ref{eqn-condition-maps-intertwines-dif-2}):
\begin{equation}\label{eqn-cond-iso-dif-proj-mod}
x\dif(y)=0, \quad y\dif(x)=0.
\end{equation}
Indeed, to see that (\ref{eqn-cond-iso-dif-proj-mod}) implies (\ref{eqn-condition-proj-mod-dif-closed}), one computes
\[
\dif(xy)=\dif(x)y+x\dif(y)=\dif(x)y=\dif(x)yxy\in Rxy.
\]
Likewise, (\ref{eqn-cond-iso-dif-proj-mod}) gives $\dif(yx)\in Ryx$.

Suppose that $R$ has, in addition, elements  $x^\prime$, $y^\prime$ such that
\begin{equation}\label{eq-xyxprime}
x^\prime y^\prime x^\prime=x^\prime, \quad \quad y^\prime x^\prime y^\prime =y^\prime
\end{equation}
and
\begin{equation}\label{eqn-orthogonal-idempotents}
y^\prime x=0, \quad \quad y x^\prime=0.
\end{equation}
Then $x^\prime y^\prime$, $y^\prime x^\prime$ are also idempotents in $R$, and
the projective $R$-modules $Rx^\prime y^\prime$ and $Ry^\prime x^\prime$ are isomorphic.
Furthermore, $xy$ and $x^\prime y^\prime$ are mutually orthogonal,
$e:=xy+ x^\prime y^\prime$ is an idempotent, and there is an isomorphism of
projective left $R$-modules
\[Re\cong Rxy \oplus Rx^\prime y^\prime \cong Ryx \oplus R y^\prime x^\prime ,\]
which can be presented via the diagram
\[
\xymatrix{
& Re
\ar[dl]_{\cdot x} \ar[dr]^{\cdot x^\prime} & \\
Ryx\ar[dr]_{\cdot y} && Ry^\prime x^\prime \ar[dl]^{\cdot y^\prime},\\
& Re &
}
\]
with $\cdot x$ given by right multiplication by $x$ which takes $ze$ to $zex\in Ryx$, etc.

It is, in general, too much to hope that $\dif$ would respect this direct sum
decomposition of $Re$. Instead, we would like $Rxy\subset Re$ to be a $p$-DG submodule isomorphic to $Ryx$, and the quotient $Re/Rxy$ to be $p$-DG isomorphic to
$Ry^\prime x^\prime$. Therefore, in addition to (\ref{eqn-condition-proj-mod-dif-closed}),
we want the maps
\begin{equation}\label{eqn-iso-quotient-to-Ryx-prime}
\xymatrix{Re/Rxy  & \ltwocell_{\cdot x^\prime}^{\cdot y^\prime}{'} Ry^\prime x^\prime.}
\end{equation}
to commute with the $p$-differentials. This is equivalent to the condition that $Ry^\prime x^\prime$ is $\dif$-closed
\begin{equation}\label{eqn-quotient-mod-condition}
\dif(y^\prime x^\prime)\in Ry^\prime x^\prime,
\end{equation}
and that
\begin{equation}\label{eqn-intertwiners-quotient-mod}
\dif(e)x^\prime=\dif(ex^\prime), \ \ \ \ \dif(y^\prime x^\prime) y^\prime \equiv \dif(y^\prime x^\prime y^\prime)~(\mathrm{mod}~Rxy).
\end{equation}
A simple computation shows that condition (\ref{eqn-quotient-mod-condition}) is implied by (\ref{eqn-intertwiners-quotient-mod}), the latter in turn being equivalent to the conditions
\begin{equation}\label{eqn-equivalent-formuation-quotient-intertwiners}
y^\prime\dif(x^\prime)=0, \ \ \ \
x^\prime \dif(y^\prime)\in Rxy.
\end{equation}
Furthermore, $Re$ is $\dif$-closed when (\ref{eqn-equivalent-formuation-quotient-intertwiners}) is satisfied. Then, there is an exact sequence of $R_\dif$-modules
\[
\xymatrix{
0\ar[r] & Rxy~ \ar@{^(->}[r]  & Re \ar[r] & Re/Rxy \ar[r] & 0 .\\
& Ryx \ar[u]^{\cong} &&&}
\]
The surjective map $Re\lra Ry^\prime x^\prime$ given by right multiplication by $x^\prime$
has $Rxy$ as its kernel and commutes with $\dif$ since $y^\prime\dif(x^\prime)=0$. Consequently,
there is a short exact sequence of $(R,\dif)$-modules (equivalently, a $p$-DG filtration
on $Re$)
\begin{equation} \label{eq-shortex}
\xymatrix{
0 \ar[r] & Ryx \ar[r]^{\cdot y} & Re \ar[r]^{\cdot x^\prime} & Ry^\prime x^\prime \ar[r] & 0.
}
\end{equation}
If $Ryx$ and $Ry^\prime x^\prime$ are compact, then so is $Re$ and this filtration leads to a relation in $K_0(R,\dif)$
\begin{equation} \label{eq-rel-in-K0}
{[Re]}={[Ryx]}+{[Ry^\prime x^\prime]}.
\end{equation}

In the case of $R=R(2i+j)$ with the differential $\dif(1,1,1,1)$, let
\begin{equation}
x=
\begin{DGCpicture}
\DGCstrand[red](0,0)(1,1)[$i$]
\DGCstrand[red](0.5,0)(0,1)[$i$]
\DGCdot{0.125}
\DGCstrand[green](1,0)(0.5,1)[$j$]
\end{DGCpicture}, \ \ \ \ \
y =
\begin{DGCpicture}
\DGCstrand[red](0,0)(0.5,1)[$i$]
\DGCstrand[green](0.5,0)(1,1)[$j$]
\DGCstrand[red](1,0)(0,1)[$i$]
\end{DGCpicture}, \ \ \ \ \
x^\prime=-
\begin{DGCpicture}
\DGCstrand[green](0,0)(0.5,1)[$j$]
\DGCstrand[red](0.5,0)(1,1)[$i$]
\DGCstrand[red](1,0)(0,1)[$i$]\DGCdot{0.125}
\end{DGCpicture},\ \ \ \ \
y^\prime=
\begin{DGCpicture}
\DGCstrand[red](0,0)(1,1)[$i$]
\DGCstrand[green](0.5,0)(0,1)[$j$]
\DGCstrand[red](1,0)(0.5,1)[$i$]
\end{DGCpicture}.
\end{equation}
An easy computation shows that conditions (\ref{eq-xyx}),
(\ref{eqn-cond-iso-dif-proj-mod}), (\ref{eq-xyxprime}), (\ref{eqn-orthogonal-idempotents}),
(\ref{eqn-equivalent-formuation-quotient-intertwiners}) hold.
Idempotents $yx,$ $y'x',$ and $e$ give rise to compact $R(2i+j)_\dif$-modules
$P_{i^{(2)}j}$, $P_{ji^{(2)}}$ (Lemma \ref{lemma-serre-modules-are-cofibrant}), and $P_{iji}$ respectively, resulting
in the short exact sequence
\begin{equation} \label{eq-res-short}
0 \lra P_{i^{(2)}j} \lra P_{iji} \lra P_{j i^{(2)}} \lra 0
\end{equation}
and in the relation in the Grothendieck group
\begin{equation} \label{eq-in-Groth}
[P_{iji} ] = [P_{i^{(2)}j}] + [P_{j i^{(2)}}].
\end{equation}

\paragraph{The general case.}In general, let $\Gamma$ be a connected simply-laced Cartan datum with vertex set $I$ and an arbitrary orientation. Define $(R(\Gamma),\dif)$ to be the associated $p$-DG algebra with the differential parameter chosen as in equations (\ref{eqn-dif-on-R-nu-general-1}) through (\ref{eqn-dif-on-R-nu-general-6}). For any sequence of colors $\mathbf{i}=i_1i_2\cdots i_n \in I^n$ such that $\sum_{t=1}^n i_t=\nu\in \N[I]$, define the compact cofibrant $p$-DG module over $(R(\Gamma),\dif)$ by
\[
P_{\mathbf{i}}:=R(\nu)1_{\mathbf{i}}=
\left\{
\begin{DGCpicture}
\DGCstrand(0,0)(0,1)[$i_1$]
\DGCstrand(0.5,0)(0.5,1)[$i_2$]
\DGCstrand(2,0)(2,1)[$i_n$]
\DGCcoupon*(0.75,0.25)(1.75,0.75){$\cdots$}
\DGCcoupon(-0.1,1)(2.1,2){$x$}
\end{DGCpicture}~
\Bigg|x\in R(\nu)
\right\}.\]
The local nature of the $p$-DG algebra together with the rank-two cases
(Propositions~\ref{prop-QSR-forces-specialization-parameter-A1-times-A1} and \ref{prop-QSR-forces-specialization-parameter-A2}) applied to each pair of vertices in $\Gamma$ leads to the following.
\begin{flushleft}
\begin{itemize}
\item[$(I)$] If $i$, $k$ are distant vertices in $\Gamma$, and $u_{ik}=u_{ki}=0$,
then there is an isomorphism of left $(R(\Gamma),\dif)$-modules
    \[\eta_{ ki }:P_{\cdots ik\cdots }\lra P_{\cdots ki\cdots },\]
    where $\eta_{ki }$ is given by post-composing any element in $P_{\cdots ik \cdots}$ with
    \[
    \begin{DGCpicture}
    \DGCstrand(0,0)(0,1)[$i_1$]
    \DGCstrand(2,0)(3,1)[$k$]
    \DGCstrand(3,0)(2,1)[$i$]
    \DGCstrand(5,0)(5,1)[$i_n$]
    \DGCcoupon*(0.5,0.25)(1.5,0.75){$\cdots$}
    \DGCcoupon*(3.5,0.25)(4.5,0.75){$\cdots$}
    \end{DGCpicture}.
    \]
    The inverse map $\eta_{ i k }$ of $\eta_{ ki }$ is given by attaching to any element of $P_{\cdots ki\cdots}$ from below the element
    \[
    \begin{DGCpicture}
    \DGCstrand(0,0)(0,1)[$i_1$]
    \DGCstrand(2,0)(3,1)[$i$]
    \DGCstrand(3,0)(2,1)[$k$]
    \DGCstrand(5,0)(5,1)[$i_n$]
    \DGCcoupon*(0.5,0.25)(1.5,0.75){$\cdots$}
    \DGCcoupon*(3.5,0.25)(4.5,0.75){$\cdots$}
    \end{DGCpicture}.
    \]
\item[$(II)$] If $i$, $j$ are vertices in $\Gamma$ connected by the oriented edge $i \rightarrow j$, and the parameters satisfy
    \[a_i=a_j=r_{ij}=r_{ji}=1,\]
    then there are short exact sequences of left $R(\Gamma)_{\dif}$-modules
    \[
    0\longrightarrow P_{\cdots i^{(2)}j\cdots }\xrightarrow{\vartheta_{i^{(2)}j}} P_{\cdots i j i\cdots } \xrightarrow{\rho_{ji^{(2)}}} P_{\cdots ji^{(2)}\cdots } \longrightarrow 0,
    \]
    \[
    0\longrightarrow P_{\cdots j^{(2)}i\cdots }\xrightarrow{\vartheta_{j^{(2)}i}} P_{\cdots jij \cdots } \xrightarrow{\rho_{ij^{(2)}}} P_{\cdots ij^{(2)}\cdots } \longrightarrow 0.
    \]
    Here $\vartheta_{i^{(2)}j}$, $\rho_{ji^{(2)}}$,  $\vartheta_{j^{(2)}i}$ and $\rho_{ij^{(2)}}$ are respectively right multiplication by the elements
    \[
    \begin{DGCpicture}
    \DGCstrand(-1,0)(-1,1)[${i_1}$]
    \DGCstrand(0,0)(1,1)[${i}$]
    \DGCstrand(1,0)(2,1)[${j}$]
    \DGCstrand(2,0)(0,1)[${i}$]
    \DGCstrand(3,0)(3,1)[${i_{n}}$]
    \DGCcoupon*(-0.9,0.25)(-0.1,0.75){$\cdots$}
    \DGCcoupon*(2.1,0.25)(2.9,0.75){$\cdots$}
    \end{DGCpicture}\ , \ \ \ \ \ \
    \begin{DGCpicture}
    \DGCstrand(-1,0)(-1,1)[${i_1}$]
    \DGCstrand(0,0)(1,1)[$j$]
    \DGCstrand(1,0)(2,1)[$i$]
    \DGCstrand(2,0)(0,1)[$i$]\DGCdot{0.25}
    \DGCstrand(3,0)(3,1)[${i_{n}}$]
    \DGCcoupon*(-0.9,0.25)(-0.1,0.75){$\cdots$}
    \DGCcoupon*(2.1,0.25)(2.9,0.75){$\cdots$}
    \end{DGCpicture}\ ,
    \]
    \[
    \begin{DGCpicture}
    \DGCstrand(-1,0)(-1,1)[${i_1}$]
    \DGCstrand(0,0)(1,1)[${j}$]
    \DGCstrand(1,0)(2,1)[${i}$]
    \DGCstrand(2,0)(0,1)[${j}$]
    \DGCstrand(3,0)(3,1)[${i_{n}}$]
    \DGCcoupon*(-0.9,0.25)(-0.1,0.75){$\cdots$}
    \DGCcoupon*(2.1,0.25)(2.9,0.75){$\cdots$}
    \end{DGCpicture}\ , \ \ \ \ \ \
    \begin{DGCpicture}
    \DGCstrand(-1,0)(-1,1)[${i_1}$]
    \DGCstrand(0,0)(1,1)[$i$]
    \DGCstrand(1,0)(2,1)[$j$]
    \DGCstrand(2,0)(0,1)[$j$]\DGCdot{0.25}
    \DGCstrand(3,0)(3,1)[${i_{n}}$]
    \DGCcoupon*(-0.9,0.25)(-0.1,0.75){$\cdots$}
    \DGCcoupon*(2.1,0.25)(2.9,0.75){$\cdots$}
    \end{DGCpicture} \ .
    \]

\item[$(II^\sigma)$]  If $i$, $j$ are vertices in $\Gamma$ connected by the oriented edge $i \rightarrow j$, and the parameters satisfy
    \[
    1+a_i=1+a_j=r_{ij}=r_{ji}=0,
    \]
    then there are short exact sequences of left $R(\Gamma)_{\dif}$-modules
    \[
    0\longrightarrow P_{\cdots ji^{(2)}\cdots }\xrightarrow{\vartheta_{ji^{(2)}}} P_{\cdots i j i\cdots} \xrightarrow{\rho_{i^{(2)}j}} P_{\cdots i^{(2)}j\cdots} \longrightarrow 0,
    \]
    \[
    0\longrightarrow P_{\cdots ij^{(2)}\cdots }\xrightarrow{\vartheta_{ij^{(2)}}} P_{\cdots jij \cdots} \xrightarrow{\rho_{j^{(2)}i}} P_{\cdots j^{(2)}i\cdots} \longrightarrow 0,
    \]
    where $\vartheta_{ji^{(2)}}$, $\rho_{i^{(2)}j}$, $\vartheta_{ij^{(2)}}$ and $\rho_{j^{(2)}i}$ are respectively given by right multiplication with the elements
    \[
    \begin{DGCpicture}
    \DGCstrand(-1,0)(-1,1)[${i_1}$]
    \DGCstrand(0,0)(2,1)[${i}$]
    \DGCstrand(1,0)(0,1)[${j}$]
    \DGCstrand(2,0)(1,1)[${i}$]
    \DGCstrand(3,0)(3,1)[${i_{n}}$]
    \DGCcoupon*(-0.9,0.25)(-0.1,0.75){$\cdots$}
    \DGCcoupon*(2.1,0.25)(2.9,0.75){$\cdots$}
    \end{DGCpicture}\ , \ \ \ \ \ \
    \begin{DGCpicture}
    \DGCstrand(-1,0)(-1,1)[${i_1}$]
    \DGCstrand(0,0)(2,1)[$i$]\DGCdot{0.25}
    \DGCstrand(1,0)(0,1)[$i$]
    \DGCstrand(2,0)(1,1)[$j$]
    \DGCstrand(3,0)(3,1)[${i_{n}}$]
    \DGCcoupon*(-0.9,0.25)(-0.1,0.75){$\cdots$}
    \DGCcoupon*(2.1,0.25)(2.9,0.75){$\cdots$}
    \end{DGCpicture} \ ,
    \]
    \[
    \begin{DGCpicture}
    \DGCstrand(-1,0)(-1,1)[${i_1}$]
    \DGCstrand(0,0)(2,1)[${j}$]
    \DGCstrand(1,0)(0,1)[${i}$]
    \DGCstrand(2,0)(1,1)[${j}$]
    \DGCstrand(3,0)(3,1)[${i_{n}}$]
    \DGCcoupon*(-0.9,0.25)(-0.1,0.75){$\cdots$}
    \DGCcoupon*(2.1,0.25)(2.9,0.75){$\cdots$}
    \end{DGCpicture}\ , \ \ \ \ \ \
    \begin{DGCpicture}
    \DGCstrand(-1,0)(-1,1)[${i_1}$]
    \DGCstrand(0,0)(2,1)[$j$]\DGCdot{0.25}
    \DGCstrand(1,0)(0,1)[$j$]
    \DGCstrand(2,0)(1,1)[$i$]
    \DGCstrand(3,0)(3,1)[${i_{n}}$]
    \DGCcoupon*(-0.9,0.25)(-0.1,0.75){$\cdots$}
    \DGCcoupon*(2.1,0.25)(2.9,0.75){$\cdots$}
    \end{DGCpicture} \ .
    \]
\end{itemize}
\end{flushleft}
$(II^\sigma)$ follows from $(II)$ by applying the symmetry $\sigma$ (see equation (\ref{eqn-KLR-symmetry-sigma})) and using Proposition \ref{prop-symmetry-intertwines-differentials-KLR}. We leave the verification of these statements to the reader. The following definition summarizes the two parameters of differentials for which the categorical quantum Serre relations hold.

\begin{defn}\label{def-correct-differential-on-KLR}Let $\Gamma$ be a connected simply-laced Cartan datum with the set of vertices $I$ and an arbitrary orientation, and $R(\Gamma)$ be its associated KLR algebra.

\begin{itemize}
\item[$(D_{+})$] The $p$-nilpotent local differential $\dif_1$ on $R(\Gamma)$ acts on the one-strand generators by
\[
 \dif_1
\left(~
{\begin{DGCpicture}
\DGCstrand(0,0)(0,1)[$i$]\DGCdot{0.5}
\end{DGCpicture}}
~\right) =
{\begin{DGCpicture}
\DGCstrand(0,0)(0,1)[$i$]
\DGCdot{0.5}[r]{$^2$}
\end{DGCpicture}} \ ,
\]
while on the two-strand generators,
\[
\begin{array}{l}
\dif_1\left(
\begin{DGCpicture}
\DGCstrand(0,0)(1,1)[$i_1$]
\DGCstrand(1,0)(0,1)[$i_2$]
\end{DGCpicture}
\right)  =
\delta_{i_1,i_2}
\begin{DGCpicture}
\DGCstrand(0,0)(0,1)[$i_1$]
\DGCstrand(1,0)(1,1)[$i_2$]
\end{DGCpicture}
-i_1\cdot i_2
\begin{DGCpicture}
\DGCstrand(0,0)(1,1)[$i_1$]
\DGCstrand(1,0)(0,1)[$i_2$]\DGCdot{0.75}
\end{DGCpicture}
\\
\\
= \left\{
\begin{array}{l}
\begin{DGCpicture}
\DGCstrand(0,0)(0,1)[$i_1$]
\DGCstrand(1,0)(1,1)[$i_2$]
\end{DGCpicture}-2
\begin{DGCpicture}
\DGCstrand(0,0)(1,1)[$i_1$]
\DGCstrand(1,0)(0,1)[$i_2$]\DGCdot{0.75}
\end{DGCpicture}=
-
\begin{DGCpicture}
\DGCstrand(0,0)(1,1)[$i_1$]
\DGCstrand(1,0)(0,1)[$i_2$]\DGCdot{0.75}
\end{DGCpicture}
-
\begin{DGCpicture}
\DGCstrand(0,0)(1,1)[$i_1$]
\DGCstrand(1,0)(0,1)[$i_2$]\DGCdot{0.25}
\end{DGCpicture} \quad \textrm{if $i_1=i_2$,}\\
 \\
\quad \quad  0  \ \quad \quad \textrm{if $i_1, i_2$ are distant,}\\
 \\
\ \
\begin{DGCpicture}
\DGCstrand(0,0)(1,1)[$i_1$]
\DGCstrand(1,0)(0,1)[$i_2$]\DGCdot{0.75}
\end{DGCpicture} \quad  \textrm{if $i_1, i_2$ are connected.}
\end{array}
\right.
\end{array}
\]
Here $i,i_1,i_2\in I$ are vertices of $\Gamma$, and $i_1\cdot i_2$ is the Cartan pairing between $i_1,i_2$.

\vspace{0.05in}

\item[$(D_{-})$]The $p$-nilpotent local differential $\dif_{-1}$ on $R(\Gamma)$ acts on the one-strand generators by
\[
\dif_{-1}
\left(~
{\begin{DGCpicture}
\DGCstrand(0,0)(0,1)[$i$]\DGCdot{0.5}
\end{DGCpicture}}
~\right)=
{\begin{DGCpicture}
\DGCstrand(0,0)(0,1)[$i$]
\DGCdot{0.5}[r]{$^2$}
\end{DGCpicture}} \ ,
\]
while on the two-strand generators,
\[
\begin{array}{l}
\dif_{-1}\left(
\begin{DGCpicture}
\DGCstrand(0,0)(1,1)[$i_1$]
\DGCstrand(1,0)(0,1)[$i_2$]
\end{DGCpicture}
\right)  =
-\delta_{i_1,i_2}
\begin{DGCpicture}
\DGCstrand(0,0)(0,1)[$i_1$]
\DGCstrand(1,0)(1,1)[$i_2$]
\end{DGCpicture}
-i_1\cdot i_2
\begin{DGCpicture}
\DGCstrand(0,0)(1,1)[$i_1$]\DGCdot{0.75}
\DGCstrand(1,0)(0,1)[$i_2$]
\end{DGCpicture} \\
  \\
 =  \left\{
\begin{array}{l}=
-\begin{DGCpicture}
\DGCstrand(0,0)(0,1)[$i_1$]
\DGCstrand(1,0)(1,1)[$i_2$]
\end{DGCpicture}
-2
\begin{DGCpicture}
\DGCstrand(0,0)(1,1)[$i_1$]\DGCdot{0.75}
\DGCstrand(1,0)(0,1)[$i_2$]
\end{DGCpicture}=
-
\begin{DGCpicture}
\DGCstrand(0,0)(1,1)[$i_1$]\DGCdot{0.25}
\DGCstrand(1,0)(0,1)[$i_2$]
\end{DGCpicture}
-
\begin{DGCpicture}
\DGCstrand(0,0)(1,1)[$i_1$]\DGCdot{0.75}
\DGCstrand(1,0)(0,1)[$i_2$]
\end{DGCpicture} \quad \textrm{if $i_1=i_2$,}\\
 \\
\quad \quad  \ 0 \quad \quad  \ \textrm{if $i_1, i_2$ are distant,}\\
 \\
\quad \begin{DGCpicture}
\DGCstrand(0,0)(1,1)[$i_1$]\DGCdot{0.75}
\DGCstrand(1,0)(0,1)[$i_2$]
\end{DGCpicture}  \quad \textrm{if $i_1, i_2$ are connected.}
\end{array}
\right.
\end{array}
\]
\end{itemize}
\end{defn}

The next theorem summarizes the discussion in the subsection.

\begin{thm}\label{thm-QSR-holds-only-for-special-paramters} The relations
\[
[P_{\cdots ik\cdots }]=[P_{\cdots ki\cdots}],
\]
where $i,k$ are distant vertices, and
\[
{[ 2 ] [P_{\cdots iji\cdots }}]  =  [P_{\cdots iij\cdots }]+[P_{\cdots jii\cdots }],
\]
where $i,j$ are vertices connected by one edge in $\Gamma$ with any orientation, hold in the Grothendieck group $K_0(R(\Gamma),\dif)$ if and only if $\dif=\dif_{\pm 1}$. The differentials $\dif_{\pm 1}$ are conjugate to each other by the (anti-)automorphisms $\psi$ and $\sigma$ of $R(\Gamma)$. Furthermore, the following statements hold.
\begin{itemize}
\item[(i)] In the derived category $\mc{D}(R(\Gamma),\dif_1)$, there is an isomorphism of $p$-DG modules
    \[
    P_{\cdots ik\cdots }\cong P_{\cdots ki \cdots},
    \]
    if $i$, $k$ are distant vertices. There is an exact triangle if $i$ and $j$ are vertices connected by one edge
    \[
    P_{\cdots i^{(2)}j\cdots} \lra P_{\cdots i j i\cdots } \lra P_{\cdots ji^{(2)}\cdots} \longrightarrow P_{\cdots i^{(2)}j\cdots}[1].
    \]
\item[(ii)]In the derived category $\mc{D}(R(\Gamma),\dif_{-1})$, there is an isomorphism of $p$-DG modules
    \[
    P_{\cdots ik\cdots }\cong P_{\cdots ki\cdots},
    \]
    if $i$, $k$ are distant vertices, while there is an exact triangle
    \[
    P_{\cdots ji^{(2)}\cdots }\lra P_{\cdots i j i\cdots } \lra P_{\cdots i^{(2)}j\cdots } \lra
    P_{\cdots ji^{(2)}\cdots }[1]
    \]
if $i$ and $j$ are connected by one edge.$\hfill\square$
\end{itemize}
\end{thm}

\subsection{Grothendieck groups for  small weights}
In this subsection we compute Grothendieck groups of $p$-DG module categories over $(R(\nu),\dif)$ for certain small weights $\nu$ in rank two. From now on we will specialize the definition of $\dif$ to the first system of parameters ($D_+$) in Definition \ref{def-correct-differential-on-KLR}.

\begin{eg}\label{eg-Grothendieck-group-R-i-plus-k}Let $i,k$ be distant vertices. The $p$-DG algebra $R(i+k)$ has a basis consisting of elements
\[
\left\{
\begin{DGCpicture}
\DGCstrand[red](0,0)(0,1)[$i$]\DGCdot{0.75}[r]{$^{n_1}$}
\DGCstrand[blue](1,0)(1,1)[$k$]\DGCdot{0.75}[r]{$^{n_2}$}
\end{DGCpicture},~
\begin{DGCpicture}
\DGCstrand[blue](0,0)(0,1)[$k$]\DGCdot{0.75}[r]{$^{n_1}$}
\DGCstrand[red](1,0)(1,1)[$i$]\DGCdot{0.75}[r]{$^{n_2}$}
\end{DGCpicture},~
\begin{DGCpicture}
\DGCstrand[red](0,0)(1,1)[$i$]\DGCdot{0.75}[r]{$^{n_2}$}
\DGCstrand[blue](1,0)(0,1)[$k$]\DGCdot{0.75}[r]{$^{n_1}$}
\end{DGCpicture},~
\begin{DGCpicture}
\DGCstrand[blue](0,0)(1,1)[$k$]\DGCdot{0.75}[r]{$^{n_2}$}
\DGCstrand[red](1,0)(0,1)[$i$]\DGCdot{0.75}[r]{$^{n_1}$}
\end{DGCpicture}
\Bigg|n_1, n_2\in \N \right\},
\]
and the differential $\dif$ acts trivially on diagrams without dots. The inclusion of the $2\times 2$ matrix algebra with trivial differential
\[
\left(\begin{array}{ccc}
\begin{DGCpicture}
\DGCstrand[red](0,0)(0,1)[$i$]
\DGCstrand[blue](1,0)(1,1)[$k$]
\end{DGCpicture} &&
\begin{DGCpicture}
\DGCstrand[red](0,0)(1,1)[$i$]
\DGCstrand[blue](1,0)(0,1)[$k$]
\end{DGCpicture}\\
 && \\
\begin{DGCpicture}
\DGCstrand[blue](0,0)(1,1)[$k$]
\DGCstrand[red](1,0)(0,1)[$i$]
\end{DGCpicture} &&
\begin{DGCpicture}
\DGCstrand[blue](0,0)(0,1)[$k$]
\DGCstrand[red](1,0)(1,1)[$i$]
\end{DGCpicture}
\end{array}\right)
\]
into $R(i+k)$ is a quasi-isomorphism of $p$-DG algebras. By Theorem \ref{thm-qis-algebra-equivalence-derived-categories}, this inclusion induces an equivalence of derived categories, and hence Grothendieck groups. From the toy matrix model in Section 2.3, one concludes that $K_0(R(i+k))\cong \mathbb{O}_p$.
\end{eg}

This example generalizes to any weight space of $A_1\times A_1$, using
Theorem~\ref{thm-p-DG-nilHecke-categorifies-small-sl2}. One concludes that the
Grothendieck group of the $p$-DG algebra $(R(\underset{i}{\textcolor[rgb]{1.00,0.00,0.00}{\bullet}}\ \ \ \ \underset{k}{\textcolor[rgb]{0.00,0.00,1.00}{\bullet}}),\dif)$
can be identified with the twisted bialgebra
\[
K_0(R(\underset{i}{\textcolor[rgb]{1.00,0.00,0.00}{\bullet}}\ \ \ \ \underset{k}{\textcolor[rgb]{0.00,0.00,1.00}{\bullet}}),\dif)\cong u^+_{\mathbb{O}_{p}}(\mf{sl}_2\times \mf{sl}_2).
\]

\vspace{0.2in}

In what follows we will focus on the $A_2$ case $\underset{i}{\textcolor[rgb]{1.00,0.00,0.00}{\bullet}}\rightarrow \underset{j}{\textcolor[rgb]{0.00,1.00,0.00}{\bullet}}$.

\begin{eg}\label{eg-Grothendieck-group-R-i-plus-j}The $\Bbbk$-algebra $R(i+j)$ has a basis
\[
\left\{
\begin{DGCpicture}
\DGCstrand[red](0,0)(0,1)[$i$]\DGCdot{0.75}[r]{$^{n_1}$}
\DGCstrand[green](1,0)(1,1)[$j$]\DGCdot{0.75}[r]{$^{n_2}$}
\end{DGCpicture},~
\begin{DGCpicture}
\DGCstrand[green](0,0)(0,1)[$j$]\DGCdot{0.75}[r]{$^{n_1}$}
\DGCstrand[red](1,0)(1,1)[$i$]\DGCdot{0.75}[r]{$^{n_2}$}
\end{DGCpicture},~
\begin{DGCpicture}
\DGCstrand[red](0,0)(1,1)[$i$]\DGCdot{0.75}[r]{$^{n_2}$}
\DGCstrand[green](1,0)(0,1)[$j$]\DGCdot{0.75}[r]{$^{n_1}$}
\end{DGCpicture},~
\begin{DGCpicture}
\DGCstrand[green](0,0)(1,1)[$j$]\DGCdot{0.75}[r]{$^{n_2}$}
\DGCstrand[red](1,0)(0,1)[$i$]\DGCdot{0.75}[r]{$^{n_1}$}
\end{DGCpicture}
\Bigg|n_1, n_2\in \N \right\},
\]
and the differential $\dif$ acts on the generators by
\[
\begin{array}{ll}
\dif\left(
\begin{DGCpicture}
\DGCstrand[red](0,0)(0,1)[$i$]
\DGCstrand[green](1,0)(1,1)[$j$]
\end{DGCpicture}
\right)=0, &
\dif\left(
\begin{DGCpicture}
\DGCstrand[green](0,0)(0,1)[$j$]
\DGCstrand[red](1,0)(1,1)[$i$]
\end{DGCpicture}
\right)=0,\\
&\\
\dif\left(
\begin{DGCpicture}
\DGCstrand[red](0,0)(1,1)[$i$]
\DGCstrand[green](1,0)(0,1)[$j$]
\end{DGCpicture}
\right)=
\begin{DGCpicture}
\DGCstrand[red](0,0)(1,1)[$i$]
\DGCstrand[green](1,0)(0,1)[$j$]\DGCdot{0.75}
\end{DGCpicture}, &
\dif\left(
\begin{DGCpicture}
\DGCstrand[green](0,0)(1,1)[$j$]
\DGCstrand[red](1,0)(0,1)[$i$]
\end{DGCpicture}
\right)=
\begin{DGCpicture}
\DGCstrand[green](0,0)(1,1)[$j$]
\DGCstrand[red](1,0)(0,1)[$i$]\DGCdot{0.75}
\end{DGCpicture}.
\end{array}
\]
Consider the two-sided ideal
\[
J:=
\Bbbk\left\langle
\begin{DGCpicture}
\DGCstrand[red](0,0)(0,1)[$i$]\DGCdot{0.75}[r]{$^{n_1}$}
\DGCstrand[green](1,0)(1,1)[$j$]\DGCdot{0.75}[r]{$^{n_2}$}
\end{DGCpicture},~
\begin{DGCpicture}
\DGCstrand[green](0,0)(0,1)[$j$]\DGCdot{0.75}[r]{$^{n_1}$}
\DGCstrand[red](1,0)(1,1)[$i$]\DGCdot{0.75}[r]{$^{n_2}$}
\end{DGCpicture}\Bigg|~n_1+n_2\geq 1 \right\rangle \bigoplus
\Bbbk\left\langle
\begin{DGCpicture}
\DGCstrand[red](0,0)(1,1)[$i$]\DGCdot{0.75}[r]{$^{n_2}$}
\DGCstrand[green](1,0)(0,1)[$j$]\DGCdot{0.75}[r]{$^{n_1}$}
\end{DGCpicture},~
\begin{DGCpicture}
\DGCstrand[green](0,0)(1,1)[$j$]\DGCdot{0.75}[r]{$^{n_2}$}
\DGCstrand[red](1,0)(0,1)[$i$]\DGCdot{0.75}[r]{$^{n_1}$}
\end{DGCpicture}
\Bigg|
n_1, n_2 \in \N  \right\rangle
.
\]
It is preserved by the differential and forms a contractible $p$-complex.
Therefore, the quotient $p$-DG algebra $(R(i+j)/J,\dif)$ is isomorphic to
\[
\Bbbk
\begin{DGCpicture}
\DGCstrand[red](0,0)(0,1)[$i$]
\DGCstrand[green](1,0)(1,1)[$j$]
\end{DGCpicture}\times
\Bbbk
\begin{DGCpicture}
\DGCstrand[green](0,0)(0,1)[$j$]
\DGCstrand[red](1,0)(1,1)[$i$]
\end{DGCpicture}\cong \Bbbk \times \Bbbk,
\]
with trivial differential, and the natural projection $\pi:R(i+j)\lra R(i+j)/J$
is a quasi-isomorphism. Alternatively, the inclusion of the algebras
$\Bbbk\times\Bbbk \hookrightarrow R(i+j)$ is a quasi-isomorphism of $p$-DG algebras.
By Theorem~\ref{thm-qis-algebra-equivalence-derived-categories}, there is an
equivalence of triangulated categories $\mc{D}(R(i+j)) \cong \mc{D}(\Bbbk\times\Bbbk)$,
implying that $K_0(R(i+j))\cong \mathbb{O}_p \oplus \mathbb{O}_p$.
\end{eg}

To compute the next example we will use the following observation. Let $(R,\dif)$ be a
$p$-DG algebra and $\epsilon_1, \epsilon_2\in R$ be idempotents such that $R\epsilon_t$
is a $p$-DG submodule of $R$ for $t=1,2$. Necessarily $\dif(\epsilon_t)=r_t\epsilon_t$
for some $r_t\in R$. Then under the isomorphism of $\Bbbk$-spaces
$\HOM_R(R\epsilon_1, R\epsilon_2)\cong \epsilon_1 R \epsilon_2$, the natural
$\dif$-action on the $\HOM$-space translates into an action on $\epsilon_1 R \epsilon_2$,
denoted by ``$\diamond$'', as follows. For any $x \in  R $,
\[
\dif\diamond(\epsilon_1 x \epsilon_2)=\epsilon_1\cdot\dif(x\epsilon_2).
\]
Indeed, if $f\in \HOM_R(R\epsilon_1,R\epsilon_2)$, then
$\dif(f) (y\epsilon_1)=\dif(f(y\epsilon_1))-f(\dif(y\epsilon_1))$ for any $y\in R$.
Now if $f$ is a morphism defined by an element $\epsilon_1 x \epsilon_2$,
$f(y\epsilon_1)=y\epsilon_1 x \epsilon_2$, then
$\dif (f)(y\epsilon_1)=\dif(y\epsilon_1 x \epsilon_2)-\dif(y\epsilon_1)\epsilon_1 x
\epsilon_2)=y\epsilon_1(\dif(x\epsilon_2))$ and the claim follows.

\begin{eg}\label{eg-Grothendieck-group-R-2i-plus-j} We now compute the Grothendieck group of the $p$-DG algebra $(R(2i+j),\dif)$, assuming $\textrm{char}(\Bbbk)\geq 3$. By Proposition \ref{prop-derived-category-compactly-generated}, the triangulated category $\mc{D}(R(2i+j))$ is generated by the $p$-DG module $R(2i+j)$, which is isomorphic to the direct sum of three compact cofibrant modules $P_{iij}$, $P_{iji}$ and $P_{jii}$. Furthermore, $P_{iij}$ has a two step filtration induced from the filtration (\ref{eqn-NH2-ses-sym2-submodules}) of $\NH_2$ by its polynomial representation. The subquotients of the filtration are isomorphic to grading shifts of $P_{i^{(2)}j}$. Likewise $P_{jii}$ is filtered by grading shifts of $P_{ji^{(2)}}$. The short exact sequence (\ref{eqn-ses-QSR-A2}) and the $p$-DG enhanced isomorphisms (\ref{eqn-iso-R-2i-plus-j-mod-1}), (\ref{eqn-iso-R-2i-plus-j-mod-2}) show that $P_{iji}$ fits into an exact triangle
\[
P_{i^{(2)}j}\lra P_{iji} \lra P_{ji^{(2)}} \lra P_{i^{(2)}j}[1].
\]
Hence the modules $P_{i^{(2)}j}$, $P_{ji^{(2)}}$ are compact cofibrant generators of $\mc{D}(R(2i+j))$, and Proposition \ref{prop-criterion-derived-equivalence} shows that there is an equivalence between $\mc{D}(R(2i+j))$ and $\mc{D}((\mathrm{END}_{R(2i+j)}(P_{i^{(2)}j}\oplus P_{ji^{(2)}})^{op})$. We compute the differential on this endomorphism algebra using the remarks made before this example.

First off, since
\[
\HOM(P_{iij},P_{iij}) =
\Bbbk\left\langle~
\begin{DGCpicture}[scale=0.65]
\DGCstrand[red](0,0)(0,2)[$i$]\DGCdot{1.75}[ur]{$^{n_1}$}
\DGCstrand[red](1,0)(1,2)[$i$]\DGCdot{1.75}[ur]{$^{n_2}$}
\DGCstrand[green](2,0)(2,2)[$j$]\DGCdot{1.75}[ur]{$^{n_3}$}
\end{DGCpicture},~
\begin{DGCpicture}[scale=0.65]
\DGCstrand[red](0,0)(1,2)[$i$]\DGCdot{1.75}[ur]{$^{n_2}$}
\DGCstrand[red](1,0)(0,2)[$i$]\DGCdot{1.75}[ur]{$^{n_1}$}
\DGCstrand[green](2,0)(2,2)[$j$]\DGCdot{1.75}[ur]{$^{n_3}$}
\end{DGCpicture}
\Bigg|~n_1, n_2, n_3\in \N
\right\rangle,
\]
the endomorphism space of $P_{i^{(2)}j}$  is spanned by
\[
\HOM(P_{i^{(2)}j},P_{i^{(2)}j}) \cong
\Bbbk\left\langle~
\begin{DGCpicture}[scale=0.85]
\DGCstrand[red](0,0)(1,1)(1,2)(0,3)[$i$]\DGCdot{2.25}\DGCdot{1.75}[r]{$^{n_2}$}
\DGCstrand[red](1,0)(0,1)(0,2)(1,3)[$i$]\DGCdot{1.75}[r]{$^{n_1}$}\DGCdot{0.25}
\DGCstrand[green](2,0)(2,3)[$j$]\DGCdot{1.75}[r]{$^{n_3}$}
\end{DGCpicture},~
\begin{DGCpicture}[scale=0.85]
\DGCstrand[red](0,0)(1,1)(0,2)(1,3)[$i$]\DGCdot{1.75}[r]{$^{n_1}$}
\DGCstrand[red](1,0)(0,1)(1,2)(0,3)[$i$]\DGCdot{2.25}\DGCdot{1.75}[r]{$^{n_2}$}\DGCdot{0.25}
\DGCstrand[green](2,0)(2,3)[$j$]\DGCdot{1.75}[r]{$^{n_3}$}
\end{DGCpicture}
\Bigg|~n_1, n_2, n_3\in \N
\right\rangle.
\]
Terms in the second diagram vanish because there is a double crossing between the $i-i$ strands. To simplify terms in of the first diagram, one uses that two-variable polynomials form a rank-two module over $\sym_2$:
\[
\Bbbk\left\langle
\begin{DGCpicture}[scale=0.65]
\DGCstrand[red](0,0)(0,1.5)[$i$]\DGCdot{0.75}[r]{$^{n_1}$}
\DGCstrand[red](1,0)(1,1.5)[$i$]\DGCdot{0.75}[r]{$^{n_2}$}
\end{DGCpicture}
\right\rangle\cong
\Bbbk\left\langle
\begin{DGCpicture}[scale=0.65]
\DGCstrand[red](0,0)(0,1.5)[$i$]
\DGCstrand[red](1,0)(1,1.5)[$i$]
\DGCcoupon(-0.25,0.75)(1.25,1.25){$s_1$}
\end{DGCpicture},~
\begin{DGCpicture}[scale=0.65]
\DGCstrand[red](0,0)(0,1.5)[$i$]\DGCdot{0.35}
\DGCstrand[red](1,0)(1,1.5)[$i$]
\DGCcoupon(-0.25,0.75)(1.25,1.25){$s_1$}
\end{DGCpicture}
~\Bigg|s_1\in \sym_2\right\rangle,
\]
and therefore
\[
\HOM(P_{i^{(2)}j},P_{i^{(2)}j}) \cong
\Bbbk\left\langle~
\begin{DGCpicture}[scale=0.85]
\DGCstrand[red](0,0)(1,1)(1,2)(0,3)[$i$]\DGCdot{2.25}
\DGCstrand[red](1,0)(0,1)(0,2)(1,3)[$i$]\DGCdot{0.25}
\DGCstrand[green](2,0)(2,3)[$j$]\DGCdot{1.5}[r]{$^{n_3}$}
\DGCcoupon(-0.25,1.25)(1.25,1.75){$s_1$}
\end{DGCpicture},~
\begin{DGCpicture}[scale=0.85]
\DGCstrand[red](0,0)(1,1)(1,2)(0,3)[$i$]\DGCdot{2.25}
\DGCstrand[red](1,0)(0,1)(0,2)(1,3)[$i$]\DGCdot{0.25}\DGCdot{1}
\DGCstrand[green](2,0)(2,3)[$j$]\DGCdot{1.5}[r]{$^{n_3}$}
\DGCcoupon(-0.25,1.25)(1.25,1.75){$s_1$}
\end{DGCpicture}
\Bigg| s_1 \in \sym_2, n_3\in \N
\right\rangle.
\]
Using that symmetric polynomials can slide through crossings without obstruction, the $\HOM$-space simplifies to be
\[
\begin{array}{rcl}
\HOM(P_{i^{(2)}j},P_{i^{(2)}j}) & \cong &
\Bbbk\left\langle~
\begin{DGCpicture}[scale=0.85]
\DGCstrand[red](0,0)(1,1)(1,2)(0,3)\DGCdot{2.25}
\DGCstrand[red](1,0)(0,1)(0,2)(1,3)\DGCdot{0.28}
\DGCstrand[green](2,0)(2,3)\DGCdot{0}[r]{$^{k_1}$}
\DGCcoupon(-0.25,-0.25)(1.25,0.12){$s_1$}
\end{DGCpicture}
\Bigg|~s_1\in \sym_2[x_1(i),x_2(i)], k_1 \in \N
\right\rangle\\
& \cong &
\sym_2[x_1(i),x_2(i)]\o \Bbbk[x_3(j)]\cdot
\begin{DGCpicture}[scale=0.85]
\DGCstrand[red](0,0)(1,1)(1,2)(0,3)[$i$]\DGCdot{2.25}
\DGCstrand[red](1,0)(0,1)(0,2)(1,3)[$i$]\DGCdot{0.25}
\DGCstrand[green](2,0)(2,3)[$j$]
\end{DGCpicture}.
\end{array}
\]
On the module generator $\dif$ acts by
\[
\dif\diamond\left(
\begin{DGCpicture}[scale=0.85]
\DGCstrand[red](0,0)(1,1)(1,2)(0,3)[$i$]\DGCdot{2.25}
\DGCstrand[red](1,0)(0,1)(0,2)(1,3)[$i$]\DGCdot{0.25}
\DGCstrand[green](2,0)(2,3)[$j$]
\end{DGCpicture}
\right)=-
\begin{DGCpicture}[scale=0.85]
\DGCstrand[red](0,0)(1,1)(1,2)(0,3)[$i$]\DGCdot{2.25}
\DGCstrand[red](1,0)(0,1)(0,2)(1,3)[$i$]\DGCdot{0.25}\DGCdot{0.75}
\DGCstrand[green](2,0)(2,3)[$j$]
\end{DGCpicture}=0.
\]
Hence the $p$-DG endomorphism algebra $\HOM(P_{i^{(2)}j},P_{i^{(2)}j})$ is isomorphic to the $p$-DG algebra $\sym_2[x_1(i),x_2(i)]\o \Bbbk[x_3(j)]$ with the usual differential, which in turn is quasi-isomorphic to the ground field $\Bbbk$. Likewise, a similar computation shows that $\HOM(P_{ji^{(2)}},P_{ji^{(2)}})$ is also quasi-isomorphic to the ground field as a $p$-DG algebra.

Next, we compute the $\dif$ action on the space $\HOM(P_{ji^{(2)}}, P_{i^{(2)}j})$. As in the previous case, one shows that the $\HOM$-space is spanned by
\[
\begin{array}{rcl}
\HOM(P_{ji^{(2)}},P_{i^{(2)}j}) & \cong &
\Bbbk\left\langle~
\begin{DGCpicture}[scale=0.85]
\DGCstrand[red](0,0)(1,1)(2,2)(1,3)\DGCdot{2.25}
\DGCstrand[red](1,0)(0,1)(1,2)(2,3)\DGCdot{0.28}
\DGCstrand[green](2,0)(2,1)(0,2)(0,3)\DGCdot{0}[r]{$_{n_3}$}
\DGCcoupon(-0.25,-0.25)(1.25,0.12){$s_1$}
\end{DGCpicture},~
\Bigg|~s_1\in \sym_2[x_1(i),x_2(i)],~n_3\in \N
\right\rangle.\\
&\cong & \sym_2[x_1(i),x_2(i)]\o \Bbbk[x_3(j)]\cdot
\begin{DGCpicture}[scale=0.85]
\DGCstrand[red](0,0)(1,1)(2,2)(1,3)[$i$]\DGCdot{2.25}
\DGCstrand[red](1,0)(0,1)(1,2)(2,3)[$i$]\DGCdot{0.25}
\DGCstrand[green](2,0)(2,1)(0,2)(0,3)[$j$]
\end{DGCpicture},
\end{array}
\]
and $\dif$ acts on the generator by
\[
\dif\diamond
\left(
\begin{DGCpicture}[scale=0.85]
\DGCstrand[red](0,0)(1,1)(2,2)(1,3)[$i$]\DGCdot{2.25}
\DGCstrand[red](1,0)(0,1)(1,2)(2,3)[$i$]\DGCdot{0.25}
\DGCstrand[green](2,0)(2,1)(0,2)(0,3)[$j$]
\end{DGCpicture}
\right)=2~
\begin{DGCpicture}[scale=0.85]
\DGCstrand[red](0,0)(1,1)(2,2)(1,3)[$i$]\DGCdot{2.25}
\DGCstrand[red](1,0)(0,1)(1,2)(2,3)[$i$]\DGCdot{0.25}
\DGCstrand[green](2,0)(2,1)(0,2)(0,3)[$j$]\DGCdot{0.25}
\end{DGCpicture}.
\]
Thus, as a $p$-DG module over $\sym_2[x_1(i),x_2(i)]\o \Bbbk[x_3(i)]$,
$\HOM(P_{ji^{(2)}},P_{i^{(2)}j})$ is isomorphic to the left submodule (ideal)
generated by $1\o x_3^2(j)$ inside $\sym_2[x_1(i),x_2(i)]\o \Bbbk[x_3(i)]$.
Therefore it is quasi-isomorphic to the $p$-complex
\[\Bbbk\o (x_3^2(j)\lra x_3^3(j) \lra \cdots \lra x_3^p(j))\cong V_2\{2\},\]
since $\sym_2$ is quasi-isomorphic to $\Bbbk$.

Lastly, we compute the space $\HOM(P_{i^{(2)}j},P_{ji^{(2)}})$ with the induced $\dif$-action.
\[
\begin{array}{rcl}
\HOM(P_{i^{(2)}j},P_{ji^{(2)}}) & \cong &
\Bbbk\left\langle~
\begin{DGCpicture}[scale=0.85]
\DGCstrand[green](0,0)(0,1)(2,2)(2,3)\DGCdot{0}[r]{$_{k_1}$}
\DGCstrand[red](1,0)(2,1)(1,2)(0,3)\DGCdot{2.25}
\DGCstrand[red](2,0)(1,1)(0,2)(1,3)\DGCdot{0.28}
\DGCcoupon(0.75,-0.25)(2.25,0.12){$s_1$}
\end{DGCpicture},~
\Bigg|~s_1\in \sym_2[x_2(i),x_3(i)],~k_1\in \N
\right\rangle.\\
&\cong & \Bbbk[x_1(j)] \o \sym_2[x_2(i),x_3(i)]\cdot
\begin{DGCpicture}[scale=0.85]
\DGCstrand[green](0,0)(0,1)(2,2)(2,3)[$j$]
\DGCstrand[red](1,0)(2,1)(1,2)(0,3)[$i$]\DGCdot{2.25}
\DGCstrand[red](2,0)(1,1)(0,2)(1,3)[$i$]\DGCdot{0.25}
\end{DGCpicture},
\end{array}
\]
while $\dif$ acts on the module generator by
\[
\dif\diamond\left(
\begin{DGCpicture}[scale=0.85]
\DGCstrand[green](0,0)(0,1)(2,2)(2,3)[$j$]
\DGCstrand[red](1,0)(2,1)(1,2)(0,3)[$i$]\DGCdot{2.25}
\DGCstrand[red](2,0)(1,1)(0,2)(1,3)[$i$]\DGCdot{0.25}
\end{DGCpicture}
\right)=
\begin{DGCpicture}[scale=0.85]
\DGCstrand[green](0,0)(0,1)(2,2)(2,3)[$j$]
\DGCstrand[red](1,0)(2,1)(1,2)(0,3)[$i$]\DGCdot{2.25}
\DGCstrand[red](2,0)(1,1)(0,2)(1,3)[$i$]\DGCdot{0.25}
\end{DGCpicture}\cdot
\left(~
\begin{DGCpicture}[scale=0.85]
\DGCstrand[green](0,0)(0,3)[$j$]
\DGCstrand[red](1,0)(1,3)[$i$]\DGCdot{1.5}
\DGCstrand[red](2,0)(2,3)[$i$]
\end{DGCpicture}+
\begin{DGCpicture}[scale=0.85]
\DGCstrand[green](0,0)(0,3)[$j$]
\DGCstrand[red](1,0)(1,3)[$i$]
\DGCstrand[red](2,0)(2,3)[$i$]\DGCdot{1.5}
\end{DGCpicture}
~\right).
\]
As a $p$-DG module over $\Bbbk[x_1(j)]\o\sym_2[x_2(i),x_3(i)]$, $\HOM(P_{i^{(2)}j},P_{ji^{(2)}})$ is isomorphic to the left submodule in $\Bbbk[x_1(j)]\o\sym_2[x_2(i),x_3(i)]$ generated by the element $1\o e_2(x_2(i),x_3(i))=1\o(x_2(i)x_3(i))$. We now show that this submodule is contractible. Indeed it suffices to check that, under the usual differential on $\sym_2$, $e_2=x_1x_2\in \sym_2$ generates a contractible submodule. Note that $e_2\sym_2$ fits into a short exact sequence
\[
0\lra e_2\sym_2\lra \sym_2\lra \sym_1 \lra 0,
\]
where $\sym_1\cong\Bbbk[e_1]$ has the quotient differential structure $\dif(e_1)=e_1^2$. Therefore the surjection $\sym_2\twoheadrightarrow \sym_1$ of $p$-DG algebras, both quasi-isomorphic to the ground field, is a quasi-isomorphism. The contractibility of the kernel follows.

Given the above computation, it is more convenient to study the endomorphism algebra of the cofibrant module $(P_{i^{(2)}j}\o \widetilde{V}_{p-2}\{-p\})\oplus P_{ji^{(2)}}\cong P_{i^{(2)}j}[1]\oplus P_{ji^{(2)}}$, which after all is also a compact generator. Combining the previous observations, we see that the endomorphism algebra is quasi-isomorphic to the quiver algebra with the trivial differential
\[
\HOM(P_{i^{(2)}j}[1]\oplus P_{ji^{(2)}},P_{i^{(2)}j}[1]\oplus P_{ji^{(2)}})\cong \Bbbk(\bullet\lra {\bullet}).
\]
Here the arrow has degree one and the dots stands for the (quasi-isomorphic) one-dimensional endomorphism spaces spanned by $\Id_{P_{i^{(2)}j}[1]}$ and $\Id_{P_{ji^{(2)}}}$ respectively. By Proposition \ref{prop-criterion-derived-equivalence}, we conclude that $\mc{D}(R(2i+j))\cong \mc{D}(\Bbbk(\bullet\lra {\bullet}))$. It follows from Proposition \ref{prop-K0-of-smooth-basic-artinian-algebras} that
$K_0(R(2i+j))\cong \mathbb{O}_p\oplus\mathbb{O}_p$.
\end{eg}

\paragraph{Further remarks.}Let $\Gamma$ be a connected simply-laced Cartan datum.
So far we have seen that, under the differentials given by the special parameters
of Definition~\ref{def-correct-differential-on-KLR}, shadows of the small
quantum groups appear on the categorified level: relations $E_i^p=0$ are categorified
into the triviality of the derived categories at weights $pi$
(Proposition~\ref{prop-NHp-contractible}), quantum Serre relations hold
(Theorem~\ref{thm-QSR-holds-only-for-special-paramters}), and some direct comparison
with the weight spaces in the simplest cases is shown in the above examples.
We are tempted to propose the following conjecture.

\vspace{0.06in}

\begin{conj} \label{conj-cat}
Let $\Gamma$ be a Cartan datum of ADE type. With the parameters of
the differential $\dif$ specified as in
Definition~\ref{def-correct-differential-on-KLR}, the $p$-DG algebra
$(R(\Gamma),\dif)$ categorifies an $\mathbb{O}_p$-integral form of the positive
half of the small quantum group associated with $\Gamma$.
\end{conj}

\vspace{0.06in}

This conjecture is vague, as we do not describe a particular integral
form of the quantum group over the ring $\mathbb{O}_p$
which $(R(\Gamma),\dif)$ is expected to categorify. One part of the
conjecture states that the natural map
 $$ K_0(R(\nu), \dif) \otimes_{\mathbb{O}_p} K_0(R(\nu'), \dif) \lra
 K_0( R(\nu) \otimes_{\Bbbk} R(\nu'), \dif ) $$
is an isomorphism for any weights $\nu,\nu'$.

It is unclear how to extend the differentials to nonsimply-laced KLR algebras, since
dots there have varying degrees depending on colors.

%%%%%%%%%%%%%%%%%%%%%%%%%%%%%%%%%%%%%%%%%%%%%%%%%%%%%%%%%%%%%%%%%%%%%%%%%%%%%%%
\newstrandstyle{Red}{type=ribbon, style=solid,ribboncolor=lightred,color=red}%%
%%%%%%%%%%%%%%%%%%%%%%%%%%%%%%%%%%%%%%%%%%%%%%%%%%%%%%%%%%%%%%%%%%%%%%%%%%%%%%%

\vspace{0.06in}

The Webster algebras~\cite{Web} have a diagrammatic description analogous to that of the KLR
algebras and their cyclotomic quotients. The diagrams are generated by red and
black strands, such that red strands are labeled by positive weights and never
cross each other, while black strands can carry dots and satisfy relations similar
to nilHecke algebras. For instance, in the $A_1$ case, there are the following
red-black crossings for the weight $k\in \N$:
\[
\begin{DGCpicture}
\DGCstrand(1,0)(0,1)
\DGCstrand[Red](0,0)(1,1)
\DGCdot.{0.15}[dl]{$_k$}
\end{DGCpicture}, \ \ \ \ \ \ \ \
\begin{DGCpicture}
\DGCstrand(0,0)(1,1)
\DGCstrand[Red](1,0)(0,1)
\DGCdot.{0.15}[dl]{$_k$}
\end{DGCpicture}.
\]
subject to some local relations, for instance:
\[
\begin{DGCpicture}
\DGCstrand(1,0)(0,1)(1,2)
\DGCstrand[Red](0,0)(1,1)(0,2)
\DGCdot.{0.15}[dl]{$_k$}
\end{DGCpicture}
\ = \
\begin{DGCpicture}
\DGCstrand(1,0)(1,2)
\DGCdot{1}[r]{$k$}
\DGCstrand[Red](0,0)(0,2)
\DGCdot.{0.15}[dl]{$_k$}
\end{DGCpicture}, \ \ \ \
\begin{DGCpicture}
\DGCstrand(0,0)(1,1)(0,2)
\DGCstrand[Red](1,0)(0,1)(1,2)
\DGCdot.{0.15}[dl]{$_k$}
\end{DGCpicture}
\ = \
\begin{DGCpicture}
\DGCstrand[Red](1,0)(1,2)
\DGCdot.{0.15}[dl]{$_k$}
\DGCstrand(0,0)(0,2)\DGCdot{1}[r]{$k$}
\end{DGCpicture}
\]
We refer the reader to (\cite[Section 2]{Web}) for the details.

One readily sees that, for any $\lambda\in \F_p$, the following differential on the
new generators, together with the differential $\dif_1$ on the nilHecke algebra
(equations (\ref{eqn-dif-dot}), (\ref{eqn-dif-crossing})) defines a $p$-nilpotent
local derivation on the Webster algebra.
\[
\dif\left(
\begin{DGCpicture}
\DGCstrand(1,0)(0,1)
\DGCstrand[Red](0,0)(1,1)
\DGCdot.{0.15}[dl]{$_k$}
\end{DGCpicture}
\right)=k\lambda
\begin{DGCpicture}
\DGCstrand(1,0)(0,1)\DGCdot{0.75}
\DGCstrand[Red](0,0)(1,1)
\DGCdot.{0.15}[dl]{$_k$}
\end{DGCpicture} \ , \ \ \
\dif\left(
\begin{DGCpicture}
\DGCstrand(0,0)(1,1)
\DGCstrand[Red](1,0)(0,1)
\DGCdot.{0.15}[dl]{$_k$}
\end{DGCpicture}
\right)=k(1-\lambda)
\begin{DGCpicture}
\DGCstrand(0,0)(1,1)\DGCdot{0.75}
\DGCstrand[Red](1,0)(0,1)
\DGCdot.{0.15}[dl]{$_k$}
\end{DGCpicture} \ .
\]

%\printindex

%%%%%%%%%%%%%%%%%%%%
%%%%%%%%%%%%%%%%%%%%
%%
%%   REFERENCES
%%
%%%%%%%%%%%%%%%%%%%%
%%%%%%%%%%%%%%%%%%%%

\vspace{0.1in}

\noindent
{\sl \small Mikhail Khovanov, Department of Mathematics, Columbia University, New York, NY 10027}

\noindent
{\tt \small email: khovanov@math.columbia.edu}

\noindent
{\sl \small You Qi, Department of Mathematics, Columbia
University, New York, NY 10027}

\noindent
{\tt \small email: yq2121@math.columbia.edu}

\end{document}